\documentclass[12pt]{amsart}
\usepackage{amsmath}
\usepackage{amssymb}
\usepackage{amsthm}
\usepackage{amscd}
\usepackage{enumerate}
\usepackage{verbatim}
\usepackage[all]{xy}
\usepackage{mathrsfs}

\theoremstyle{plain}
\newtheorem{thm}{Theorem}[section]
\newtheorem{prop}[thm]{Proposition}
\newtheorem{lem}[thm]{Lemma}
\newtheorem{cor}[thm]{Corollary}

\theoremstyle{definition}

\newtheorem{rmk}[thm]{Remark}

\newcommand{\rank}{\mathrm{rank}}

\newcommand{\Hom}{\mathrm{Hom}}
\newcommand{\End}{\mathrm{End}}

\newcommand{\Lie}{\mathrm{Lie}}

\newcommand{\Ker}{\mathrm{Ker}}
\newcommand{\Coker}{\mathrm{Coker}}
\newcommand{\Img}{\mathrm{Im}}
\newcommand{\prjt}{\mathrm{pr}}
\newcommand{\Fil}{\mathrm{Fil}}
\newcommand{\Spf}{\mathrm{Spf}}
\newcommand{\Spv}{\mathrm{Sp}}
\newcommand{\Spec}{\mathrm{Spec}}

\newcommand{\Gal}{\mathrm{Gal}}

\newcommand{\Frac}{\mathrm{Frac}}
\newcommand{\id}{\mathrm{id}}

\newcommand{\Span}{\mathrm{Span}}

\newcommand{\diag}{\mathrm{diag}}

\newcommand{\rig}{\mathrm{rig}}
\newcommand{\HT}{\mathrm{HT}}
\newcommand{\Hdg}{\mathrm{Hdg}}

\newcommand{\ord}{\mathrm{ord}}

\newcommand{\Ann}{\mathrm{Ann}}
\newcommand{\hgt}{\mathrm{ht}}
\newcommand{\Supp}{\mathrm{Supp}}
\newcommand{\Ass}{\mathrm{Ass}}
\newcommand{\ADBT}{\mathrm{ADBT}}

\newcommand{\Tate}{\mathrm{Tate}}
\newcommand{\spc}{\mathrm{sp}}
\newcommand{\univ}{\mathrm{un}}
\newcommand{\Isom}{\mathrm{Isom}}
\newcommand{\Nor}{\mathrm{N}}
\newcommand{\Tr}{\mathrm{Tr}}
\newcommand{\Dec}{\mathrm{Dec}}
\newcommand{\tot}{\mathrm{tot}}

\newcommand{\Kbar}{\bar{K}}
\newcommand{\kbar}{\bar{k}}

\newcommand{\okbar}{\mathcal{O}_{\bar{K}}}

\newcommand{\Fpbar}{\bar{\mathbb{F}}_p}
\newcommand{\oelbar}{\mathcal{O}_{\bar{L}}}

\newcommand{\okey}{\mathcal{O}_K}

\newcommand{\oel}{\mathcal{O}_L}

\newcommand{\cA}{\mathcal{A}}

\newcommand{\cC}{\mathcal{C}}
\newcommand{\cD}{\mathcal{D}}
\newcommand{\cE}{\mathcal{E}}
\newcommand{\cF}{\mathcal{F}}
\newcommand{\cG}{\mathcal{G}}
\newcommand{\cH}{\mathcal{H}}
\newcommand{\cI}{\mathcal{I}}

\newcommand{\cK}{\mathcal{K}}
\newcommand{\cL}{\mathcal{L}}
\newcommand{\cM}{\mathcal{M}}

\newcommand{\cO}{\mathcal{O}}
\newcommand{\cP}{\mathcal{P}}

\newcommand{\cU}{\mathcal{U}}
\newcommand{\cV}{\mathcal{V}}
\newcommand{\cW}{\mathcal{W}}
\newcommand{\cX}{\mathcal{X}}
\newcommand{\cY}{\mathcal{Y}}
\newcommand{\cZ}{\mathcal{Z}}

\newcommand{\Gm}{\mathbb{G}_\mathrm{m}}
\newcommand{\hGm}{\hat{\mathbb{G}}_\mathrm{m}}

\newcommand{\SG}{\mathfrak{S}}

\newcommand{\SGm}{\mathfrak{M}}
\newcommand{\SGn}{\mathfrak{N}}
\newcommand{\SGl}{\mathfrak{L}}
\newcommand{\SGp}{\mathfrak{P}}
\newcommand{\SGq}{\mathfrak{Q}}

\newcommand{\SGy}{\mathfrak{Y}}

\newcommand{\ModSGf}{\mathrm{Mod}_{/\mathfrak{S}_1}^{1,\varphi}}
\newcommand{\ModSGfinf}{\mathrm{Mod}_{/\mathfrak{S}_{\infty}}^{1,\varphi}}

\newcommand{\upi}{\underline{\pi}}

\newcommand{\uv}{\underline{v}}

\newcommand{\oef}{\mathcal{O}_F}

\newcommand{\Cp}{\mathbb{C}_p}
\newcommand{\OCp}{\mathcal{O}_{\mathbb{C}_p}}

\newcommand{\bA}{\mathbb{A}}
\newcommand{\bB}{\mathbb{B}}
\newcommand{\bC}{\mathbb{C}}

\newcommand{\bF}{\mathbb{F}}

\newcommand{\bQ}{\mathbb{Q}}
\newcommand{\bR}{\mathbb{R}}
\newcommand{\bT}{\mathbb{T}}
\newcommand{\bZ}{\mathbb{Z}}

\newcommand{\sC}{\mathscr{C}}
\newcommand{\sD}{\mathscr{D}}

\newcommand{\sS}{\mathscr{S}}

\newcommand{\frA}{\mathfrak{A}}
\newcommand{\fra}{\mathfrak{a}}
\newcommand{\frB}{\mathfrak{B}}
\newcommand{\frb}{\mathfrak{b}}
\newcommand{\frc}{\mathfrak{c}}

\newcommand{\frF}{\mathfrak{F}}

\newcommand{\frl}{\mathfrak{l}}
\newcommand{\frem}{\mathfrak{m}}
\newcommand{\frn}{\mathfrak{n}}

\newcommand{\fro}{\mathfrak{o}}
\newcommand{\frP}{\mathfrak{P}}
\newcommand{\frp}{\mathfrak{p}}

\newcommand{\frq}{\mathfrak{q}}

\newcommand{\frU}{\mathfrak{U}}
\newcommand{\frV}{\mathfrak{V}}
\newcommand{\frX}{\mathfrak{X}}
\newcommand{\frY}{\mathfrak{Y}}

\newcommand{\frIW}{\mathfrak{IW}}
\newcommand{\cIW}{\mathcal{IW}}
\newcommand{\bOmega}{\mathbf{\Omega}}

\makeatletter
    
    \@addtoreset{equation}{section}
\makeatother

\makeatletter
\renewcommand{\p@enumii}{}
\makeatother

\begin{document}

\title[Properness of the Hilbert eigenvariety]{On a properness of the Hilbert eigenvariety at integral weights: the case of quadratic residue fields}
\author{Shin Hattori}
\date{\today}
\email{shin-h@math.kyushu-u.ac.jp}
\address{Faculty of Mathematics, Kyushu University}
%\classification{11S23}
%\keywords{ramification, finite flat group scheme, canonical subgroup}
\thanks{Supported by JSPS KAKENHI Grant Number 26400016.}

\begin{abstract}
Let $p$ be a rational prime. Let $F$ be a totally real number field such that $F$ is unramified over $p$ and the residue degree of any prime ideal of $F$ dividing $p$ is $\leq 2$. In this paper, we show that the eigenvariety for $\mathrm{Res}_{F/\bQ}(\mathit{GL}_2)$, constructed by Andreatta-Iovita-Pilloni, is proper at integral weights for $p\geq 3$. We also prove a weaker result for $p=2$.
\end{abstract}

\maketitle

\tableofcontents

\section{Introduction}\label{intro}

Let $p$ be a rational prime and $N$ a positive integer which is prime to $p$. We fix an algebraic closure $\bar{\bQ}_p$ of $\bQ_p$ and denote its $p$-adic completion by $\bC_p$. Let $\cW_\bQ$ be the weight space for $\mathit{GL}_{2,\bQ}$, which is a rigid analytic variety over $\bQ_p$ such that the set of $\bC_p$-valued points $\cW_\bQ(\bC_p)$ is identified with the set of continuous homomorphisms $\bQ_p^\times\to \bC_p^\times$.

In \cite{CM,Buz}, Coleman-Mazur and Buzzard defined a rigid analytic curve $\cC_N$ with a morphism $\kappa: \cC_N\to \cW_\bQ$ such that the set of $\bC_p$-valued points $\cC_N(\bC_p)$ is in bijection with the set of normalized overconvergent elliptic eigenforms of tame level $N$ which are of finite slopes, in such a way that the eigenform $f$ corresponding to a point $x\in\cC_N(\bC_p)$ is of weight $\kappa(x)$. The curve $\cC_N$ is called the Coleman-Mazur eigencurve, and it has played an important role in arithmetic geometry, since it turned out to be useful to control $p$-adic congruences of elliptic modular forms. After their construction of the eigencurve, much progress has been made to generalize it to the case of automorphic forms on algebraic groups other than $\mathit{GL}_{2,\bQ}$. Now we have, for various algebraic groups $G$ over a number field, a similar rigid analytic variety $\cE$ to the Coleman-Mazur eigencurve over a weight space $\cW^G$ for $G$, which is called the eigenvariety for $G$.

Despite of their importance, we still do not know much about the geometry of eigenvarieties. For example, we do not even know if an eigenvariety has finitely many irreducible components. One of the topics of active research
is the smoothness of eigenvarieties at classical points. For the Coleman-Mazur eigencurve, we know that the smoothness at classical points in many cases \cite{BeCh_lissite, BeD,Hida_GalRep, Ki_FMconj}.
Bella\"{i}che-Chenevier \cite{BeCh} studied tangent spaces of their eigenvariety for unitary groups at certain classical points, and applied it to showing the non-vanishing of a Bloch-Kato Selmer group. On the other hand, Bella\"{i}che proved the non-smoothness of the eigenvariety for $U(3)$ at classical points \cite{Bel_U3}. It is natural to think that such geometric information of eigenvarieties is related to deep $p$-adic properties of automorphic forms. 

Another interesting topic, which this paper concerns with, is a properness of eigenvarieties over weight spaces. Since eigenvarieties are not of finite type over weight spaces, they are not proper in the usual sense. Instead, we consider the following geometric interpretation of the non-existence of holes: Let $\cD_{\bC_p}=\Spv(\bC_p\langle T \rangle)$ be the closed unit disc centered at the origin $O$ and $\cD^\times_{\bC_p}=\cD_{\bC_p}\setminus \{O\}$ the punctured disc. For any quasi-separated rigid analytic variety $\cX$, we write $\cX_{\bC_p}$ for the base extension of $\cX$ to $\Spv(\bC_p)$. Suppose that we have a commutative diagram of rigid analytic varieties
\[
\xymatrix{
\cD^\times_{\bC_p} \ar[r]\ar[d]& \cE_{\bC_p} \ar[d]\\
\cD_{\bC_p} \ar[r]\ar@{.>}[ur] & \cW_{\bC_p}^G,
}
\]
where the vertical arrows are the natural maps. Then we say that the eigenvariety $\cE$ is proper if there exists a morphism $\cD_{\bC_p}\to \cE_{\bC_p}$ such that the above diagram is still commutative with this morphism added. Roughly speaking, this means that any family of overconvergent eigenforms of finite slopes on $G$ parametrized by the punctured disc can always be extended to the puncture. However, note that what eigenvarieties parametrize are in general not eigenforms themselves but eigensystems occurring in the space of overconvergent automorphic forms.
We also note that the naive interpretation of the non-existence of holes that any $p$-adically convergent sequence of overconvergent eigenforms of finite slopes converges to an overconvergent eigenform of finite slope, is false \cite[Theorem 2.1]{ColSt}.

For the properness of the Coleman-Mazur eigencurve $\cC_N$, Buzzard-Calegari first proved the properness of $\cC_N$ for the case where $p=2$ and $N=1$ \cite{BuzCal}. It was followed by Calegari's result \cite{Cal} on the properness of $\cC_N$ at integral weights: he showed the existence of the map $\cD_{\bC_p}\to \cC_{N,{\bC_p}}$ as in the definition of the properness if the image of the puncture $O$ in the weight space corresponds to a classical weight. One of the key points of their proofs is to show that any non-zero overconvergent elliptic eigenform of infinite slope does not converge on a certain region of a modular curve, while any overconvergent elliptic eigenform of finite slope does converge on a larger region. In \cite{BuzCal}, they deduced the former from the theory of canonical subgroups, especially a behavior of the $U_p$-correspondence for elliptic curves with Hodge height $p/(p+1)$, while the latter is a consequence of a standard analytic continuation argument via the $U_p$-operator. Recently, the properness of the Coleman-Mazur eigencurve was proved in full generality by Diao-Liu \cite{DL} by using $p$-adic Hodge theory, especially the theory of trianguline $p$-adic representations in families.

For algebraic groups other than $\mathit{GL}_{2,\bQ}$, the properness of eigenvarieties has not been known. Note that in Diao-Liu's proof of the properness of the Coleman-Mazur eigencurve, in order to apply $p$-adic Hodge theory, it seems crucial that we have a Galois representation, not just a Galois pseudo-representation, over (the normalization of) the eigencurve. This is a consequence of the fact that we can convert pseudo-representations into representations over smooth rigid analytic curves \cite[Remark after Theorem 5.1.2]{CM}. Thus at present it is unclear if their proof can be generalized to show the properness of eigenvarieties of dimension greater than one on the components where the residual Galois representations attached to automorphic forms are absolutely reducible. 

The aim of this paper is to generalize the method of Buzzard and Calegari to the case of Hilbert modular forms and to obtain the properness of the Hilbert eigenvariety constructed by Andreatta-Iovita-Pilloni \cite{AIP2} at integral weights in some cases. 

To state the main theorem, we fix some notation. For any totally real number field $F$ with ring of integers $\cO_F$, put $G=\mathrm{Res}_{F/\bQ}(\mathit{GL}_{2})$ and $\bT=\mathrm{Res}_{\oef/\bZ}(\Gm)$. Let $K/\bQ_p$ be a finite extension such that $F\otimes K$ splits completely. Let $\cW^G$ be the weight space for $G$ over $K$ as in \cite[\S 4.1]{AIP2}. By definition, we have
\[
\cW^G=\Spf(\okey[[\bT(\bZ_p)\times \bZ_p^\times]])^\rig
\]
and the set of $\bC_p$-valued points $\cW^G(\bC_p)$ can be identified with the set of pairs of continuous characters 
\[
\nu:\bT(\bZ_p)\to \bC_p^\times,\quad w:\bZ_p^\times\to \bC_p^\times.
\]
We say that the weight $(\nu,w)$ is $1$-integral if its restriction to $1+p(\cO_F\otimes \bZ_p)\times (1+p\bZ_p)$ comes from an algebraic character $\bT\times \Gm \to \Gm$. This restriction corresponds to a pair $((k_\beta)_\beta,k_0)$ of a tuple $(k_\beta)_\beta$ of integers indexed by the set of embeddings $\beta: F\to K$ and an integer $k_0$. We say that a $1$-integral weight is $1$-even if every $k_\beta$ and $k_0$ are even. Then the main theorem in this paper is the following.

\begin{thm}[Theorem \ref{main}]\label{intromain}
Let $F$ be a totally real number field which is unramified over $p$. Let $K/\bQ_p$ be a finite extension in $\bar{\bQ}_p$ such that $F\otimes K$ splits completely. Let $N\geq 4$ be an integer prime to $p$. Let $\cE\to \cW^G$ be the Hilbert eigenvariety of tame level $N$ over $K$ constructed in \cite[\S 5]{AIP2}. 

Suppose that for any prime ideal $\frp$ of $F$ dividing $p$, the residue degree $f_\frp$ of $\frp$ satisfies $f_\frp\leq 2$ ({\it resp.} $p$ splits completely in $F$) if $p$ is odd ({\it resp.} even). Then $\cE$ is proper at $1$-integral ({\it resp.} $1$-even) weights. Namely, any commutative diagram of rigid analytic varieties over $\bC_p$
\[
\xymatrix{
\cD^\times_{\bC_p} \ar[r]^\varphi\ar[d]& \cE_{\bC_p} \ar[d]\\
\cD_{\bC_p} \ar[r]_-{\psi}\ar@{.>}[ur] & \cW^G_{\bC_p}
}
\]
can be filled with the dotted arrow if $\psi(O)$ corresponds to a $1$-integral ({\it resp.} $1$-even) weight.
\end{thm}

For the proof, we basically follow the idea of Buzzard and Calegari \cite{BuzCal, Cal}. Thus the key step in our case is also to show that any non-zero overconvergent Hilbert eigenform $f$ of $1$-integral weight and infinite slope does not converge on the locus where all the partial Hodge heights are no more than $1/(p+1)$ in a Hilbert modular variety. 

Let us explain briefly how to show this non-convergence property, following \cite{BuzCal}. For simplicity, we assume that $f$ is of integral weight, namely the weight $(\nu,w)$ corresponds to an algebraic character $\bT\times \Gm\to \Gm$. For any Hilbert-Blumenthal abelian variety (HBAV) $A$ with an $\cO_F$-action over the integer ring $\oel$ of a finite extension $L/K$, we say that a finite flat closed $\cO_F$-subgroup scheme $\cH$ of $A$ over $\oel$ is $p$-cyclic if its generic fiber is etale locally isomorphic to the constant group scheme $\underline{\cO_F/p\cO_F}$. We say that $A$ is critical if every $\beta$-Hodge height of $A$ is equal to $p/(p+1)$ for any embedding $\beta: F\to K$. Then we show that for any critical $A$ and any $p$-cyclic subgroup scheme $\cH$ of $A$, the quotient $A/\cH$ has the canonical subgroup $A[p]/\cH$ of level one and its $\beta$-Hodge heights are all $1/(p+1)$ (Proposition \ref{critisog}). This is where the assumption on residue degrees is used in the most crucial way. It is unclear if the claim holds without this assumption: At least, we have a counterexample of a similar assertion for truncated Barsotti-Tate groups if we drop the assumption on $f_\frp$ (Remark \ref{counterexamplef3}).

Consider the Hilbert modular variety classifying pairs $(A,\cH)$ of a HBAV $A$ and its $p$-cyclic subgroup scheme $\cH$. Let $\cU$ be the locus where $\cH$ is the canonical subgroup of $A$.
Another thing we need here is to show that for any $(A,\cH)$ with $A$ critical, the corresponding point $[(A,\cH)]$ of the Hilbert modular variety has a connected admissible affinoid open neighborhood intersecting $\cU$ such that, if an overconvergent Hilbert eigenform $f$ of integral weight converges on the locus where all the $\beta$-Hodge heights are $\leq 1/(p+1)$, then we can evaluate $U_pf$ on this neighborhood (Proposition \ref{ConnCrit}). This implies that, if $f$ is in addition of infinite slope,
then we have $(U_pf)(A,\cH)=0$ for any critical $A$ and any $p$-cyclic subgroup scheme $\cH$. From this, by a combinatorial argument (Lemma \ref{combinat}), 
we obtain $f(A/\cH,A[p]/\cH)=0$ for any such $(A,\cH)$, which yields $f=0$ and the above non-convergence property follows. 
It seems that this argument using a connected neighborhood cannot be generalized immediately to the case where $f$ is not of locally algebraic weight, since in this case $U_p f$ is defined only on the locus $\cU$ (even after taking a finite etale cover) and it cannot be evaluated for any critical $A$.

Note that the theory of canonical subgroups of level one for the Hilbert case was established by Goren-Kassaei \cite{GK}. In this paper, we re-interpret and slightly generalize their result using the Breuil-Kisin classification of finite flat group schemes, following the author's previous works \cite{Ha_cansub, Ha_cansubZ} and Tian's \cite{Ti_2}. This construction of canonical subgroups via the Breuil-Kisin classification gives a more precise theory of canonical subgroups of higher level than in \cite{AIP2}. This enables us to enlarge the locus in the Hilbert modular variety where the sheaves of overconvergent Hilbert modular forms are defined from the original locus given in \cite{AIP2}, and to include the case of $p<5$ in the main theorem.

What the Hilbert eigenvariety $\cE$ of \cite{AIP2} parametrizes are eigensystems in the space of overconvergent Hilbert modular forms. Thus, to follow the strategy of Buzzard and Calegari to reduce the properness to the above non-convergence property of overconvergent modular forms, we have to convert a family of eigensystems of finite slopes, or a morphism from a rigid analytic variety to $\cE$, into a family of eigenforms and vice versa. The latter direction can be treated (Proposition \ref{Chenevier}) as in the proof of \cite[Proposition 7.2.8]{BeCh}. For the former direction, we first prove that any family of eigensystems over any smooth rigid analytic variety over $\Cp$ can be lifted locally to a family of eigenforms (Proposition \ref{DSL}). This can be considered as a version of Deligne-Serre's lifting lemma \cite[Lemme 6.11]{DS}. Then we glue the local eigenforms using a weak multiplicity one result, after we normalize the local eigenforms with respect to the first $q$-expansion coefficient (Proposition \ref{GlobalEigenform}). This use of the weak multiplicity one and the normalization via a $q$-expansion coefficient hinders us from generalizing the main theorem to the case of $\mathit{GSp_{2g}}$ where the sheaf of overconvergent Siegel modular forms and the Siegel eigenvariety are constructed in a similar way \cite{AIP}.

Once we have a family of overconvergent Hilbert eigenforms $f$ of finite slopes parametrized by $\cD^\times_{\bC_p}$ associated to the family of eigensystems $\varphi:\cD^\times_{\bC_p} \to \cE_{\bC_p}$, we extend its domain of definition in the Hilbert modular variety as large as possible by an analytic continuation using the $U_p$-operator. Since the Hecke eigenvalues are of absolute values bounded by one, we can show that the $q$-expansion defines a rigid analytic function around a cusp parametrized by $\cD^\times_{\bC_p}$ which is of absolute value bounded by one. Such a function automatically extends to the puncture, and a gluing shows that $f$ also extends to the puncture (Proposition \ref{gluecusp}). Since we analytically continued $f$ to a large region, the specialization $f(O)$ at the puncture is also defined over the same large region. Thus the non-convergence property of eigenforms of infinite slope mentioned above implies that $f(O)$ is also of finite slope, which gives us an extended map $\cD_{\bC_p}\to \cE_{\bC_p}$.

The organization of this paper is as follows. In Section \ref{BuzEV}, we recall Buzzard's eigenvariety machine \cite{Buz} on which the construction of the Hilbert eigenvariety in \cite{AIP2} relies, and we prove results to convert a family of eigensystems into local eigenforms and vice versa. Section \ref{SecCanSub} is devoted to developing the theory of canonical subgroups using the Breuil-Kisin classification of finite flat group schemes. In particular, we prove the key result on a behavior of the $U_p$-correspondence on the critical locus. In Section \ref{SecHilbEV}, we recall the definition of overconvergent Hilbert modular forms and the construction of the Hilbert eigenvariety, both due to Andreatta-Iovita-Pilloni \cite{AIP2}, including generalizations of some of their results to the case over $\Cp$. We also give a connected neighborhood of any critical point in a Hilbert modular variety, which is another key ingredient of the proof of Theorem \ref{intromain}. In Section \ref{SecQexp}, we prove properties of the $q$-expansion for overconvergent Hilbert modular forms. These are used to produce a global eigenform by gluing local eigenforms obtained from a family of eigensystems, and also to extend a family of overconvergent Hilbert eigenforms over the punctured unit disc to the puncture. Combining these results, we prove Theorem \ref{intromain} in Section \ref{mainpf}.

\noindent
{\bf Acknowledgments.} The author would like to thank Fabrizio Andreatta, Ruochuan Liu and Vincent Pilloni for kindly answering his questions on their works, and Tadashi Ochiai for helpful comments on an earlier draft. He also would like to thank Shu Sasaki for enlightening discussions on $p$-adic modular forms and encouragements.

%---------------------------------------

%---------------------------------------

\section{Lemmata on Buzzard's eigenvariety}\label{BuzEV}

Let $p$ be a rational prime and $K$ a finite extension of $\bQ_p$ in $\bar{\bQ}_p$. In this section, we establish two lemmata on Buzzard's eigenvariety machine \cite{Buz}. In the first lemma, we show that any family of Hecke eigensystems over a smooth rigid analytic variety over $\Cp$ lifts locally to a family of eigenforms.
The second one enables us to convert any family of Hecke eigensystems of finite slopes over a reduced rigid analytic variety into a morphism to the eigenvariety.

\subsection{Buzzard's eigenvariety machine}\label{SecBuzEV}

First we briefly recall the construction of Buzzard's eigenvariety. Let $R$ be a reduced $K$-affinoid algebra. Let $M$ be a Banach $R$-module satisfying the condition $(\mathit{Pr})$ of \cite[\S 2]{Buz}. We write $\End_R^{\mathrm{cont}}(M)$ for the $R$-algebra of continuous $R$-endomorphisms of $M$. 
Let $\bT$ be a commutative $K$-algebra endowed with a $K$-algebra homomorphism $\bT\to \End_R^{\mathrm{cont}}(M)$. Let $\phi$ be an element of $\bT$. Suppose that $\phi$ acts on $M$ as a compact operator. We call such a quadruple $(R,M,\bT,\phi)$ an input data for the eigenvariety machine over $K$.

For such $M$ and $M'$, a continuous $R$-linear $\bT$-module homomorphism $\alpha:M'\to M$ is called a primitive link if there exists a compact $R$-linear $\bT$-module homomorphism $c: M\to M'$ such that $\phi$ acts on $M$ as $\alpha\circ c$ and it acts on $M'$ as $c\circ \alpha$. A continuous $R$-linear $\bT$-module homomorphism $\alpha:M'\to M$ is called a link if it is the composite of a finite number of primitive links.

Let $P(T)=1+\sum_{n\geq 1} c_n T^n$ be the characteristic power series of $\phi$ acting on $M$, which is an element of the ring $R\{\{T\}\}$ of entire functions over $R$. The spectral variety $Z_{\phi}$ for $\phi$ is the closed analytic subvariety of $\Spv(R)\times \bA^1$ defined by $P(T)$. 
We denote the projection $Z_\phi\to \Spv(R)$ by $f$.

The eigenvariety $\cE$ associated to $(R,M,\bT,\phi)$ is the rigid analytic variety over $Z_\phi$ defined as follows: Let $\cC$ be the set of admissible affinoid open subsets $Y$ of $Z_\phi$ satisfying the condition that there exists an affinoid subdomain $X$ of $\Spv(R)$ such that $Y\subseteq f^{-1}(X)$ and the map $Y\to X$ induced by $f$ is finite and surjective. We can show that $\cC$ is an admissible covering of $Z_\phi$ \cite[\S 4, Theorem]{Buz}, and we refer to $\cC$ as the canonical admissible covering of $Z_\phi$. 

Let $Y=\Spv(B)$ be an element of $\cC$ and $X=\Spv(A)$ as above. Suppose that $X$ is connected. Then the $A$-algebra $B$ is projective of constant rank $d$. In the ring of entire functions $A\{\{T\}\}$ over $A$, we can show that $P(T)$ can be written as $P(T)=Q(T)S(T)$ with some $S(T)\in A\{\{T\}\}$ and a polynomial $Q(T)$ of degree $d$ over $A$ with constant term one, and that we have a natural isomorphism $A[T]/(Q(T))\simeq B$. Put $Q^*(T)=T^d Q(T^{-1})$. By the Riesz theory \cite[Theorem 3.3]{Buz}, the restriction $M_A$ of $M$ to $X=\Spv(A)$ can be uniquely decomposed as $M_A=N\oplus F$, where $N$ is a projective $A$-module of rank $d$ such that $Q^*(\phi)$ acts on $N$ as the zero map and it acts on $F$ as an isomorphism. Since $Q^*(0)\neq 0$, the operator $\phi$ is invertible on $N$. Let $\bT(Y)$ be the $A$-subalgebra of $\End_A^{\mathrm{cont}}(N)$ generated by the image of $\bT$. Then the $A$-algebra $\bT(Y)$ is finite and thus a $K$-affinoid algebra. Moreover, we have a natural $A$-algebra homomorphism $A[T]/(Q(T))\simeq B\to \bT(Y)$ sending $T$ to $(\phi|_N)^{-1}$. Put $\cE(Y)=\Spv(\bT(Y))$. If $X$ is not connected, by decomposing $X$ into connected components as $X=\coprod_i X_i$, we put $\cE(Y)=\coprod_i \cE(Y|_{X_i})$. Then these local pieces can be glued along the admissible covering $\cC$ and define the eigenvariety $\cE\to Z_\phi$ \cite[\S 5]{Buz}. By \cite[Lemma 5.3]{Buz}, the rigid analytic varieties $\cE$ and $Z_\phi$ are separated.

By the construction, the natural map $\cE\to Z_\phi$ is finite and the structure morphism $\cE\to \Spv(R)$ is locally (with respect to both the source and the target) finite. Moreover, we have a $K$-algebra homomorphism $\bT\to \cO(\cE)$ such that, for any admissible affinoid open subset $V$ of $Z_\phi$, the induced map $\bT\otimes_K \cO(V)\to \cO(\cE|_V)$ is surjective.

In some cases we can glue this construction to define the eigenvariety over a non-affinoid base space. Let $\cW$ be a reduced rigid analytic variety over $K$. Let $\bT$ be a commutative $K$-algebra and $\phi$ an element of $\bT$. Suppose that, for any admissible affinoid open subset $X\subseteq \cW$, we are given a Banach $\cO(X)$-module $M_X$ satisfying the condition $(\mathit{Pr})$ with a $K$-algebra homomorphism $\bT\to \End_{\cO(X)}^{\mathrm{cont}}(M_X)$ such that the image of $\phi$ is a compact operator. Suppose also that for any admissible affinoid open subsets $X_1\subseteq X_2\subseteq \cW$, we have a continuous $\cO(X_1)$-module homomorphism $\alpha: M_{X_1}\to M_{X_2}\hat{\otimes}_{\cO(X_2)} \cO(X_1)$ which is a link and satisfies a cocycle condition. Then the eigenvarieties for $(\cO(X), M_X, \bT, \phi)$ can be patched into the eigenvariety $\cE\to Z_\phi \to \cW$ \cite[Construction 5.7]{Buz}, where $Z_\phi$ denotes the spectral variety over $\cW$ constructed by gluing the spectral varieties over $X$. 

Let $L/K$ be an extension of complete valuation fields (of height one). For any quasi-separated rigid analytic variety $\cX$ over $K$ and any coherent $\cO_{\cX}$-module $\cF$, we can define base extensions $\cX_{L}:=\cX \hat{\otimes}_K L$ and $\cF_L$ of $\cX$ and $\cF$ functorially (see \cite[9.3.6]{BGR} and \cite[\S 3.1]{Con_irr}). If the extension $L/K$ is finite, then they are just the fiber product and the pull-back in the usual sense. Otherwise, it seems unclear if it has usual properties as a fiber product: for an open immersion $j: \cU\to \cX$, what we know in this case is that the base extension $j_L:\cU_L \to \cX_L$ is also an open immersion if $j$ is quasi-compact (for example, if $\cX$ is quasi-separated and $\cU$ is an admissible affinoid open subset) or a Zariski open immersion. At any rate, \cite[Proposition 9.3.6/1 and Corollary 9.3.6/2]{BGR} implies that the base extension takes any admissible affinoid covering of $\cX$ to that of $\cX_L$. We write the set of $L$-valued points $\cX_L(L)$ also as $\cX(L)$. 

We say that a $K$-algebra homomorphism $\lambda: \bT\to L$ is an $L$-valued eigensystem in $M$ if there exist an admissible affinoid open subset $X\subseteq \cW$, an element $x\in X(L)$ given by a $K$-algebra homomorphism $x^*: \cO(X)\to L$ and a non-zero element $m$ of $M_X\hat{\otimes}_{\cO(X), x^*}L$ such that we have $h m=\lambda(h) m$ for any $h\in \bT$. It is said to be of finite slope if $\lambda(\phi)\neq 0$. Then there exists a natural bijection between $\cE(L)$ and the set of $L$-valued eigensystems $\lambda$ in $M$ of finite slopes \cite[Lemma 5.9]{Buz}. We state the following lemma for the reference, which is in fact shown in \cite{Buz}.

\begin{lem}\label{LemInN}
Let $(R,M,\bT,\phi)$ be an input data for the eigenvariety machine over $K$ and let $\cE\to Z_\phi$ be the associated eigenvariety over $X=\Spv(R)$. Let $L/K$ be an extension of complete valuation fields and take $z\in \cE(L)$. Let $x\in X(L)$ and $y\in Z_\phi(L)$ be the images of $z$. Let $\lambda: \bT\to L$ be the $L$-valued eigensystem in $M$ corresponding to $z$. Let $m$ be a non-zero element of $M\hat{\otimes}_{R,x^*} L$ satisfying $h m=\lambda(h) m$ for any $h\in \bT$. Take an admissible affinoid open subset $V$ in the canonical admissible covering of $Z_\phi$ satisfying $y\in V(L)$. Put $W=f(V)=\Spv(A)$. Suppose $W$ is connected. Let $P(T)$ be the characteristic power series of $\phi$ acting on $M$, $Q(T)$ the factor of $P(T)$ in $A\{\{T\}\}$ associated to $V$ and $M_A=N\oplus F$ the corresponding decomposition of $M_A$, as above. 
\begin{enumerate}
\item\label{LemInN_EV} $\lambda(h)=h(z)$ in $L$, where $h(z)$ is the specialization at $z$ of the image of $h$ by the map $\bT\to \cO(\cE)$.
\item\label{LemInN_Dec} The decomposition 
\[
M\hat{\otimes}_{R, x^*}L=N\otimes_{A, x^*}L\oplus F\hat{\otimes}_{A,x^*}L
\] 
is the one corresponding to the factor $Q_x(T)$ of $P_x(T)$, where $P_x(T)$ and $Q_x(T)$ are the images of $P(T)$ and $Q(T)$ in $L\{\{T\}\}$ by $x^*$, respectively. 
\item\label{LemInN_InN} $Q_x(\lambda(\phi)^{-1})=0$ and $m\in N{\otimes}_{A,x^*} L$.
\end{enumerate}
\begin{proof}
The first assertion follows from the proof of \cite[Lemma 5.9]{Buz}. The second one follows from \cite[Lemma 2.13]{Buz} and the uniqueness of the decomposition in \cite[Theorem 3.3]{Buz}. For the third one, note that the definition of the map $\cE(V)\to V$ implies $Q_x(\lambda(\phi)^{-1})=Q_x^*(\lambda(\phi))=0$. Since $Q_x^*(\phi)m=Q_x^*(\lambda(\phi))m=0$, the second assertion implies $m \in N{\otimes}_{A,x^*} L$.
\end{proof}
\end{lem}

%---------------------------------------------------------------------

%---------------------------------------------------------------------

\subsection{Lifting lemma \`{a} la Deligne-Serre}\label{SecDSL}

In this subsection, we consider the problem of converting a family of eigensystems into a family of eigenforms. First we show the following local lemma.

\begin{lem}\label{DSloc}
Let $L$ be a complete valuation field which is algebraically closed. Let $A$ be an $L$-affinoid algebra and let $N$ be a projective $A$-module of finite rank. Let $T$ be a finite $A$-algebra equipped with an $A$-algebra homomorphism $T\to \End_A(N)$. Let $S$ be an $L$-affinoid algebra which is an integral domain and let $\varphi: T\to S$ be a homomorphism of $L$-affinoid algebras. For any $x\in \Spv(S)$, we write $m_x$ for the associated maximal ideal of $S$. Assume that, for any $x\in \Spv(S)$, the induced map 
\[
\varphi(-)(x): T\to S/m_x 
\]
is an $S/m_x$-valued eigensystem in $N$. Namely, we assume that, for any $x\in \Spv(S)$, there exists a non-zero element $f_x\in N\otimes_A S/m_x$ satisfying $(h\otimes 1) f_x=(1\otimes\varphi(h)(x))f_x$ for any $h\in T$. 
\end{lem} 

\begin{enumerate}
\item\label{DSloc1} There exists a non-zero element $F\in N\otimes_A S$ satisfying $(h\otimes 1)F=(1\otimes\varphi(h))F$ for any $h\in T$.
\item\label{DSloc2} Assume moreover that $S$ is a principal ideal domain. We write $F(x)$ for the image of $F$ in $N\otimes_A S/m_x$. Then there exists $F$ as in (\ref{DSloc1})
satisfying $F(x)\neq 0$ for any $x\in \Spv(S)$.
\end{enumerate}
\begin{proof}
Put $P=\Ker(\varphi: T\to S)$, which is a prime ideal of $T$. Consider the multiplication map $\mu:T\otimes_{A}T/P\to T/P$, and put 
\[
Q=\Ker(\mu)=\Ker(T\otimes_A T/P\to T/P \to S). 
\]
Then the ideal $Q$ is a minimal prime ideal. Indeed, since the $A$-algebra $T$ is finite, the $T/P$-algebra $T\otimes_{A} T/P$ is also finite and thus the latter ring is a finite extension of a quotient of $T/P$. Since the quotient $(T\otimes_{A} T/P)/Q$ is isomorphic to $T/P$, we have the inequality
\[
\dim(T/P)\geq \dim (T\otimes_{A} T/P)\geq \hgt(Q)+\dim (T/P),
\]
which implies $\hgt(Q)=0$.

The ideals $n_x=\varphi^{-1}(m_x)$ and $\bar{n}'_x=(\varphi\circ \mu)^{-1}(m_x)$ are maximal ideals of the rings $T$ and $T\otimes_{A} T/P$, respectively. We write $\bar{n}_x$ for the inverse image of $m_x$ by the map $T/P\to S$, which is also a maximal ideal. Via the map $1\otimes \varphi: T\otimes_{A} T/P\to T\otimes_{A} S$, the ring $T\otimes_{A} T/P$ acts on $N\otimes_A S/m_x$ for any $x\in \Spv(S)$.

First we claim that $\bar{n}'_x=\Ann_{T\otimes_{A} T/P}(f_x)$. Since $\bar{n}'_x$ is a maximal ideal and $f_x\neq 0$, it is enough to show $\bar{n}'_x\subseteq \Ann_{T\otimes_{A} T/P}(f_x)$. Since $L$ is algebraically closed, the ideal $\Img(n_x\otimes_{A} T/P)+\Img(T\otimes_{A} \bar{n}_x)$ is a maximal ideal contained in $\bar{n}'_x$, and thus they are equal. For any $h\in T$, we denote its image in $T/P$ by $\bar{h}$. Take elements $h\in T$ and $h'\in n_x$. We have $(h\otimes \bar{h}') f_x=0$. On the other hand, we also have $(h'\otimes 1)f_x=(1\otimes\varphi(h')(x))f_x=0$ by assumption. This implies $(h'\otimes\bar{h})f_x=0$ and the claim follows.

Next we claim that the localization $(N\otimes_A T/P)_Q$ of the $T\otimes_A T/P$-module $N\otimes_A T/P$ at $Q$ is non-zero. Suppose the contrary. Since  the $T\otimes_A T/P$-module $N\otimes_A T/P$ is finite, we can find $s\notin Q$ satisfying $s(N\otimes_A T/P)=0$. Take any $x\in \Spv(S)$. We have $s(N\otimes_A T/n_x)=0$. Since $L$ is algebraically closed, we have $L=T/n_x=S/m_x$ and we also see that $s(N\otimes_A S/m_x)=0$. In particular, we have $sf_x=0$ and $s\in \Ann_{T\otimes_A T/P}(f_x)=\bar{n}'_x$. Thus we obtain
\[
s\in \bigcap_{x \in \Spv(S)} \bar{n}'_x=\bigcap_{x \in \Spv(S)} (\varphi\circ \mu)^{-1}(m_x)= (\varphi\circ \mu)^{-1}(\bigcap_{x \in \Spv(S)} m_x).
\]
The assumption that $S$ is a reduced $L$-affinoid algebra implies 
\[
\bigcap_{x \in \Spv(S)} m_x=0. 
\]
Hence $s\in \Ker(\varphi\circ \mu)=Q$, which is a contradiction.

Therefore we obtain $Q\in \Supp_{T\otimes_A T/P}(N\otimes_A T/P)$. Since $Q$ is a minimal prime ideal, it is also contained in $\Ass_{T\otimes_A T/P}(N\otimes_A T/P)$. Namely, the prime ideal $Q$ is written as $Q=\Ann_{T\otimes_A T/P}(G)$ with some non-zero element $G$ of $N\otimes_A T/P$. Since the $A$-module $N$ is projective, the natural map $1\otimes \varphi: N\otimes_A T/P\to N\otimes_A S$ is an injection. Thus the image $F=(1\otimes \varphi)(G)$ is non-zero. Moreover, since $h\otimes 1-1\otimes \bar{h}\in Q$ for any $h\in T$, we have the equality $(h\otimes 1)G=(1\otimes \bar{h})G$. Hence we obtain $(h\otimes 1)F=(1\otimes\varphi(h))F$
and the assertion (\ref{DSloc1}) follows.

Now assume that $S$ is a principal ideal domain. Then each maximal ideal $m_x$ of $S$ is generated by a single element $t_x$. Put
\[
\Sigma(F)=\{x\in \Spv(S)\mid F(x)=0\}.
\]
Since the $A$-module $N$ is projective and the Krull dimension of $S$ is no more than one, we see that $\Sigma(F)$ is a finite set. For any $x\in \Sigma(F)$, the element $F$ lies in $\Ker(N\otimes_A S\to N\otimes_A S/m_x)=m_x(N\otimes_A S)$. By Krull's intersection theorem, there exists a positive integer $c_x$ satisfying $F\in t_x^{c_x} (N\otimes_A S)\setminus t_x^{c_x+1} (N\otimes_A S)$. Put $F=t_x^{c_x} H$ with some non-zero element $H$ of $N\otimes_A S$. We have $H(x)\neq 0$ and $\Sigma(H)\subsetneq \Sigma(F)$. Since the $S$-module $N\otimes_A S$ is torsion free, the element $H$ also satisfies $(h\otimes 1)H=(1\otimes\varphi(h))H$ for any $h\in T$. Repeating this, we can find $F$ as in the assertion (\ref{DSloc1}) satisfying $\Sigma(F)=\emptyset$. 
\end{proof}

\begin{rmk}\label{DSloc2-disk}
Let $\Spv(S)$ be a connected affinoid subdomain of the unit disc $\cD_{\bC_p}=\Spv(\bC_p\langle T\rangle)$. Note that $\bC_p\langle T\rangle$ is a principal ideal domain, since it is a unique factorization domain of Krull dimension one. \cite[Proposition 7.2.2/1]{BGR} implies that $S$ is a regular ring of Krull dimension no more than one such that every maximal ideal is principal. Since $\Spv(S)$ is connected, we see that $S$ is a principal ideal domain. Hence the assumption of Lemma \ref{DSloc} (\ref{DSloc2}) is satisfied in this case.
\end{rmk}

We say that a rigid analytic variety $X$ is principally refined if any admissible covering of $X$ has a refinement by an admissible affinoid covering $X=\bigcup_{i\in I} U_i$ such that the affinoid algebra of each affinoid open subset $U_i$ in the refined covering is a principal ideal domain.

\begin{rmk}\label{DSL-disc}
Remark \ref{DSloc2-disk} implies that any open subvariety of $\cD_{\bC_p}$ is principally refined.
\end{rmk}

For the eigenvariety associated to an input data $(R,M,\bT,\phi)$, the above lemma implies the following proposition.

\begin{prop}\label{DSL}
Let $(R,M,\bT,\phi)$ be an input data for the eigenvariety machine over $K$ and let $\cE\to Z_\phi\to \Spv(R)$ be the associated eigenvariety. Let $L/K$ be an extension of complete valuation fields such that $L$ is algebraically closed. Let $X$ be a smooth rigid analytic variety over $L$ and let $\varphi: X \to \cE_{L}=\cE\hat{\otimes}_K L$ be a morphism of rigid analytic varieties over $L$.
\begin{enumerate}
\item\label{DSL1} There exist an admissible affinoid covering $X=\bigcup_{i\in I} U_i$ and a non-zero element $F_i\in M\hat{\otimes}_R \cO(U_i)$ for each $i\in I$ satisfying $(h\otimes 1)F_i=(1\otimes\varphi^*(h))F_i$ for any $h\in \bT$, where $\varphi^*: \bT\to \cO(\cE)\to \cO(U_i)$ is the map induced by $\varphi$.

\item\label{DSL2} Assume moreover that $X$ is principally refined. We write $k(x)$ for the residue field of $x\in U_i$ and $F_i(x)$ for the image of $F_i$ in $M\hat{\otimes}_R k(x)$. Then we can find $F_i$ as in (\ref{DSL1}) satisfying $F_i(x)\neq 0$ for any $x\in U_i$.
\end{enumerate}
\end{prop}
\begin{proof}
Let $\cC$ be the canonical admissible covering of $Z_\phi$. For any $V\in \cC$, we have the $K$-affinoid variety $\cE(V)=\Spv(\bT(V))$, as before. Then $\cE_{L}=\bigcup_{V\in \cC} \cE(V)_{L}$ is an admissible affinoid covering of $\cE_{L}$. Let $f:Z_\phi\to \Spv(R)$ be the natural projection and write as $f(V)=\Spv(A)$. For any $V\in \cC$ such that $f(V)$ is connected, take an admissible affinoid covering $\varphi^{-1}(\cE(V)_{L})=\bigcup_{i\in I_V} U_i$ such that $U_i=\Spv(S_i)$ is connected for any $i\in I_V$. From the construction of the eigenvariety, we have a natural decomposition $M\hat{\otimes}_R A=N\oplus F$ into closed $A$-submodules $N$ and $F$. Note that the $A$-module $N$ is finite and projective. Since the complete tensor product commutes with the direct sum, the $S_i$-module $N\otimes_A S_i$ is a submodule of $M\hat{\otimes}_R S_i$.

For any $i\in I_V$, consider the natural map $\varphi^*: \bT\to \bT(V)\to S_i$. For any $x\in U_i=\Spv(S_i)$, the composite $\Spv(k(x))\to U_i\to \cE_L$ corresponds to a $k(x)$-valued eigensystem of $\bT$ in $M$ of finite slope. Namely, there exists a non-zero element $g_{x}$ of $M\hat{\otimes}_R k(x)=N{\otimes}_A k(x)\oplus F\hat{\otimes}_A k(x)$ satisfying $(h\otimes 1)g_{x}=(1\otimes \varphi^*(h)(x)) g_{x}$ for any $h\in \bT$ and $(\phi\otimes 1)g_{x}\neq 0$. Lemma \ref{LemInN} (\ref{LemInN_InN}) implies $g_x\in N{\otimes}_A k(x)$.
Since $U_i$ is connected and smooth, the ring $S_i$ is an integral domain. Applying Lemma \ref{DSloc} (\ref{DSloc1}) to $(A\hat{\otimes}_K L,N\hat{\otimes}_K L,\bT(V)\hat{\otimes}_K L, S_i)$, we obtain a non-zero element $G_i\in N\otimes_A S_i=(N\hat{\otimes}_K L)\hat{\otimes}_{A\hat{\otimes}_K L} S_i$ satisfying $(h\otimes 1)G_i=(1\otimes\varphi^*(h))G_i$ for any $h\in \bT$. Setting $F_i$ to be the image of $G_i$ by the injection $N\otimes_A S_i \to M\hat{\otimes}_R S_i$, the assertion (\ref{DSL1}) follows.

For the assertion (\ref{DSL2}), by assumption we may assume that each $S_i$ is a principal domain. Then Lemma \ref{DSloc} (\ref{DSloc2}) allows us to find $G_i$ satisfying in addition $G_i(x)\neq 0$ for any $x\in U_i$. Since we have a commutative diagram
\[
\xymatrix{
N\otimes_A S_i \ar@{^{(}->}[r]\ar[d] & M\hat{\otimes}_R S_i \ar[d]\\
N\otimes_A k(x) \ar@{^{(}->}[r] & M\hat{\otimes}_R k(x)
}
\]
such that the horizontal arrows are injective, we obtain $F_i(x)\neq 0$ for any $x\in U_i$.
\end{proof}

%---------------------------------------------------------------------

%---------------------------------------------------------------------

\subsection{Bella\"{i}che-Chenevier's argument}\label{SecChenevier}

Let $(R,M,\bT,\phi)$ be an input data for the eigenvariety machine over $K$ and let $\cE\to Z_\phi\to \Spv(R)$ be the associated eigenvariety. Let $L/K$ be an extension of complete valuation fields. Put $R_L=R\hat{\otimes}_K L$. Let $X$ be a rigid analytic variety over $L$ equipped with a morphism $\kappa: X\to \Spv(R_L)$. For any $x\in X$, we have a natural ring homomorphism $\kappa^*(x): R\to k(x)$. A ring homomorphism $\varphi: \bT\to \cO(X)$ is said to be a family of eigensystems in $M$ over $X$ if, for any $x\in X$, there exists a non-zero element $f_x$ of $M\hat{\otimes}_{R,\kappa^*(x)} k(x)$ such that $(h\otimes 1)f_x=(1\otimes \varphi(h)(x))f_x$ for any $h\in \bT$. It is said to be of finite slopes if $\varphi(\phi)(x)\neq 0$ for any $x\in X$. This is the same as saying that $\varphi(\phi)\in \cO(X)^\times$. In this subsection, we show that we can convert a family of eigensystems of finite slopes over a reduced base space into a morphism to the eigenvariety, following \cite[Proposition 7.2.8]{BeCh}. First we recall the following lemma.

\begin{lem}\label{LemRedFactor}
\begin{enumerate}
\item\label{LemRedFactor1} Let $f:X\to Y$ be a morphism of rigid analytic varieties over $L$ with $X$ reduced. Let $Z$ be a closed analytic subvariety of $Y$. Suppose $f(X)\subseteq Z$. Then $f$ factors through $Z$. 
\item\label{LemRedFactor2} Let $f,f':X\to Y$ be two morphisms of rigid analytic varieties over $L$ with $X$ reduced and $Y$ separated. Suppose that these morphisms define the same map between the underlying sets. Then $f=f'$.  
\end{enumerate}
\end{lem}
\begin{proof}
For the first assertion, we may assume that $X=\Spv(R_1)$, $Y=\Spv(R_2)$ and $Z=\Spv(R_2/I)$ for some ideal $I$ of $R_2$. Consider the associated ring homomorphism $f^*: R_2\to R_1$ and put $J=\Ker(f^*)$. 
By assumption, every maximal ideal $m$ of $R_1$ satisfies $(f^*)^{-1}(m)\supseteq I$. Since $R_1$ is Jacobson and reduced, we obtain
\[
I\subseteq \bigcap_{m\in \Spv(R_1)} (f^*)^{-1}(m)=(f^*)^{-1}(\bigcap_{m\in \Spv(R_1)} m)=(f^*)^{-1}(0)=J.
\]
Hence the assertion (\ref{LemRedFactor1}) follows. The second assertion follows from the first one applied to $(f,f'):X\to Y\times_L Y$ and the diagonal $Y\to Y\times_L Y$.
\end{proof}

\begin{prop}\label{Chenevier}
Let $(R,M,\bT,\phi)$ be an input data for the eigenvariety machine over $K$ and let $\cE\to Z_\phi\to \Spv(R)$ be the associated eigenvariety. Let $L/K$ be an extension of complete valuation fields. Let $X$ be a reduced rigid analytic variety over $L$ equipped with a morphism $\kappa: X\to \Spv(R_L)$. Suppose that we have a family of eigensystems of finite slopes $\varphi: \bT\to \cO(X)$ in $M$ over $X$. Then there exists a unique morphism $\Phi:X\to \cE_L$ such that the diagram
\[
\xymatrix{
X \ar[r]^{\Phi} \ar[rd]_{\kappa}& \cE_L\ar[d]\\
& \Spv(R_L)
}
\]
is commutative and, for any $x\in X$, the eigensystem over $k(x)$ corresponding to $\Spv(k(x))\to X\overset{\Phi}{\to} \cE_L$ is the map $\varphi(-)(x): \bT\to k(x)$. 
\end{prop}

\begin{proof}
Let $\cC$ be the canonical admissible covering of $Z_\phi$. Take any $V=\Spv(B)\in \cC$ and put $f(V)=\Spv(A)$ as in the proof of Proposition \ref{DSL}. 
Let $I$ be a finite subset of $\bT$ such that its image in $\bT(V)$ is a system of generators of the finite $B$-algebra $\bT(V)$. We denote by $\bA^I_{V_L}$ the affine space over $V_L=V\hat{\otimes}_K L$ whose variables are indexed by $I$. We have a morphism of rigid analytic varieties
\[
i_{V,I}: \cE(V)_L\to \bA^I_{V_L},\quad z\mapsto (h(z))_{h\in I}.
\]  
From the definition of $I$, we see that the map $i_{V,I}$ is a closed immersion.

On the other hand, we also have a morphism of rigid analytic varieties
\[
\mu: X\to \Spv(R_L)\times \bA_L^1, \quad x\mapsto (\kappa(x),\varphi(\phi)^{-1}(x)).
\]
Let $P(T)\in R\{\{T\}\}$ be the characteristic power series of $\phi$ acting on $M$. For any $x\in X$, let $P_x(T)$ be the image of $P(T)$ in $k(x)\{\{T\}\}$ by the map $\kappa^*(x): R\to k(x)$. By \cite[Lemma 2.13]{Buz}, it is the characteristic power series of $\phi$ acting on $M\hat{\otimes}_{R,\kappa^*(x)} k(x)$. By assumption, there exists a non-zero element $g_x$ of $M\hat{\otimes}_{R,\kappa^*(x)} k(x)$ satisfying
\[
(h\otimes 1) g_x=(1\otimes \varphi(h)(x))g_x. 
\]
Then Lemma \ref{LemInN} (\ref{LemInN_InN}) implies $P_x(\varphi(\phi)(x)^{-1})=0$. By using the assumption that $X$ is reduced and Lemma \ref{LemRedFactor} (\ref{LemRedFactor1}), we see that the morphism $\mu$ factors through $Z_{\phi,L}$. 

For any $V\in \cC$, put $X_{V_L}=\mu^{-1}(V_L)$. For any $I$ as above, we consider the morphism of rigid analytic varieties over $V_L$
\[
j_{V,I}: X_{V_L} \to \bA^I_{V_L},\quad x\mapsto (\varphi(h)(x))_{h\in I}.
\]
By \cite[Lemma 5.9]{Buz} and Lemma \ref{LemInN} (\ref{LemInN_EV}), for any $x\in X_{V_L}$ there exists a unique point $z_x\in \cE(k(x))$ satisfying $\varphi(h)(x)=h(z_x)$ for any $h\in \bT$. 
We claim that $z_x\in \cE(V)_L$. Indeed, we may assume that $f(V)$ is connected. Let $Q(T)$ be the factor of $P(T)$ corresponding to $V$ and $Q_x(T)$ its image by $\kappa^*(x)$. Let $N$ be the direct summand of $M_A$ corresponding to $V$. For any $x\in X_{V_L}$, we have $\mu(x)\in V_L$ and $Q_x(\varphi(\phi^{-1})(x))=Q^*_x(\varphi(\phi)(x))=0$. Hence $Q^*_x(\phi)g_x=0$ and thus $g_x\in N\otimes_A k(x)$. From the proof of \cite[Lemma 5.9]{Buz}, this implies $z_x\in \cE(V)_L$ and the claim follows.

In particular, we have $j_{V,I}(x)=i_{V,I}(z_x)$ for any $x\in X_{V_L}$ and thus $j_{V,I}(X_{V_L})\subseteq i_{V,I}(\cE(V)_L)$. Since $i_{V,I}$ is a closed immersion and $X_{V_L}$ is reduced, Lemma \ref{LemRedFactor} (\ref{LemRedFactor1}) yields a unique morphism $\Phi_{V,I}: X_{V_L}\to \cE(V)_L$ over $V_L$ which makes the following diagram commutative.
\[
\xymatrix{
X_{V_L} \ar[r]^{\Phi_{V,I}}\ar[rd]_{j_{V,I}} & \cE(V)_L \ar@{^{(}->}[d]^{i_{V,I}}\\
 & \bA^I_{V_L}
}
\]
We claim that the morphism $\Phi_{V,I}$ is independent of the choice of a finite subset $I$ of $\bT$ as above. Indeed, for any $x\in X_{V_L}$, we have $\Phi_{V,I}(x)=i_{V,I}^{-1}(j_{V,I}(x))=z_x$, which depends only on $x$. Since $X$ is reduced and $\cE$ is separated, Lemma \ref{LemRedFactor} (\ref{LemRedFactor2}) implies the claim. Moreover, by the same reason we can glue the morphisms $\Phi_{V,I}$ along $V\in \cC$ and obtain a morphism $\Phi: X\to \cE_L$. Since the requirement on $\Phi$ in the proposition is the same as $\Phi(x)=z_x$, it is satisfied by the morphism $\Phi$ we have constructed. Lemma \ref{LemRedFactor} (\ref{LemRedFactor2}) ensures the uniqueness.
\end{proof}

%---------------------------------------------------------------------

%---------------------------------------------------------------------

\section{Theory of canonical subgroups}\label{SecCanSub}

The theory of canonical subgroups is a powerful tool to study overconvergent modular forms and the dynamics of the $U_p$-correspondence on Shimura varieties. For the Hilbert modular varieties, such a theory was established successfully by Goren-Kassaei \cite{GK}, using the geometry of the modular varieties over a field of characteristic $p$. In order to obtain more precise information on the $U_p$-correspondence which was needed for an application, Tian \cite{Ti_2} combined Goren-Kassaei's work with the approach in \cite{Ha_cansub} of using the Breuil-Kisin classification of finite flat group schemes over complete discrete valuation rings. In this section, we recall their theory of canonical subgroups and give a slight generalization, including its higher level version, which is necessary for the sequel. Moreover, we prove a key property of the $U_p$-correspondence on the critical locus.

%---------------------------------------------------------------------

%---------------------------------------------------------------------

\subsection{Breuil-Kisin modules}\label{SecBKMod}

Let $k$ be a perfect field of characteristic $p$ and $W=W(k)$ the Witt ring of $k$. Put $W_n=W/p^nW$. We denote by $\sigma$ both the $p$-th power Frobenius map on $k$ and its natural lift on $W$. Let $K$ be a finite totally ramified extension of $\Frac(W)$ of degree $e$. We denote its ring of integers by $\okey$. Let $\pi$ be a uniformizer of $K$. Let $v_p$ be the additive valuation on $K$ normalized as $v_p(p)=1$. For any non-negative real number $i$, we put
\[
m_K^{\geqslant i}=\{x\in \okey\mid v_p(x)\geq i\}, \quad \cO_{K,i}=\cO_K/m_K^{\geqslant i}.
\]
For any extension $L/K$ of valuation fields, we consider the valuation on $L$ extending $v_p$ and define $m_{L}^{\geqslant i}$ and $\cO_{L,i}$ similarly. 
We write as
\[
\sS_i=\Spec(\cO_{K,i}),\quad \sS_{L,i}=\Spec(\cO_{L,i}).
\]
We denote the maximal ideal of $\oel$ by $m_L$.
For any element $x\in \cO_{L,1}$, we define the truncated valuation $v_p(x)$ by
\[
v_p(x)=\min\{v_p(\hat{x}),1\}
\]
with any lift $\hat{x}\in \cO_L$ of $x$.
For any $x\in L$, we define the absolute value of $x$ by $|x|=p^{-v_p(x)}$.
Let us fix an algebraic closure $\Kbar$ of $K$ and extend $v_p$ naturally to $\Kbar$. 
Put $G_K=\Gal(\Kbar/K)$. We fix a system $(\pi_n)_{n\geq 0}$ of $p$-power roots of $\pi$ in $\Kbar$ such that $\pi_0=\pi$ and $\pi_{n+1}^p=\pi_n$ for any $n$. Put $K_\infty=\bigcup_{n\geq 0} K(\pi_n)$ and $G_{K_\infty}=\Gal(\Kbar/K_\infty)$.

Let $E(u)\in W[u]$ be the monic Eisenstein polynomial for $\pi$ and set $c_0=p^{-1}E(0)\in W^\times$. Put $\SG=W[[u]]$ and $\SG_n=\SG/p^n \SG$. The ring $\SG_1=k[[u]]$ is a complete discrete valuation ring with additive valuation $v_u$ normalized as $v_u(u)=1$. We also denote by $\varphi$ the $\sigma$-semilinear continuous ring homomorphism $\varphi:\SG\to \SG$ defined by $u\mapsto u^p$. 

An $\SG$-module $\SGm$ is said to be a Breuil-Kisin module (of $E$-height $\leq 1$) if $\SGm$ is a finitely generated $\SG$-module equipped with a $\varphi$-semilinear map $\varphi_\SGm:\SGm\to \SGm$ such that the cokernel of the linearization 
\[
1\otimes\varphi_{\SGm}: \varphi^*\SGm=\SG\otimes_{\varphi,\SG} \SGm \to \SGm
\]
is killed by $E(u)$. We refer to $\varphi_\SGm$ as the Frobenius map of the Breuil-Kisin module $\SGm$ and often write as $\varphi$ abusively. A morphism of Breuil-Kisin modules is defined as an $\SG$-linear map compatible with Frobenius maps. Let $\ModSGf$ be the category of Breuil-Kisin modules $\SGm$ such that the underlying $\SG$-module $\SGm$ is free of finite rank over $\SG_1$. We denote by $\ModSGfinf$ the category of Breuil-Kisin modules $\SGm$ such that the underlying $\SG$-module $\SGm$ is finitely generated, $p$-power torsion and $u$-torsion free.

The category $\ModSGfinf$ admits a natural notion of duality, which is denoted by $\SGm\mapsto \SGm^\vee$ \cite[Proposition 3.1.7]{Li_FC}. Here we give its explicit definition for the full subcategory $\ModSGf$. For any object $\SGm$ of $\ModSGf$, let $e_1,\ldots,e_h$ be its basis. Write as
\[
\varphi_{\SGm}(e_1,\ldots,e_h)=(e_1,\ldots,e_h)A
\]
with some $A\in M_h(\SG_1)$. Consider its dual $\SGm^\vee=\Hom_{\SG_1}(\SGm,\SG_1)$ with the dual basis $e_1^\vee,\ldots, e_h^\vee$. We give $\SGm^\vee$ a structure of a Breuil-Kisin module by 
\[
\varphi_{\SGm^\vee}(e_1^\vee,\ldots,e_h^\vee)=(e_1^\vee,\ldots,e_h^\vee)\left(\frac{E(u)}{c_0}\right){}^t\! A^{-1},
\]
which is independent of the choice of a basis.

Consider the inverse limit ring
\[
R=\varprojlim_{n\geq 0}(\cO_{\Kbar,1}\gets \cO_{\Kbar,1} \gets\cdots),
\]
where every transition map is the $p$-th power Frobenius map. The absolute Galois group $G_K$ acts on $R$ via the natural action on each entry. We define an element $\upi$ of $R$ by $\upi=(\pi_0,\pi_1,\ldots)$. The ring $R$ is a complete valuation ring of characteristic $p$ with algebraically closed fraction field, and we normalize the additive valuation $v_R$ on $R$ by $v_R(\upi)=1/e$. We define $m_R^{\geqslant i}$ and $R_{i}=R/m_R^{\geqslant i}$
as before, using $v_R$. We consider the Witt ring $W(R)$ as an $\SG$-algebra by the continuous $W$-linear map defined by $u\mapsto [\upi]$. Then we have the following classification of finite flat group schemes over $\okey$ 
\cite{Br_AZ, Ki_Fcrys, Kim_2, Lau_2,Li_2}.

\begin{thm}\label{BKC}
\begin{enumerate}
\item\label{BKC_eq} There exists an exact anti-equivalence 
\[
\cG\mapsto \SGm^*(\cG)
\]
from the category of finite flat group schemes over $\okey$ killed by some $p$-power to the category $\ModSGfinf$. If $\cG$ is a truncated Barsotti-Tate group of level $n$ over $\okey$, then the $\SG$-module $\SGm^*(\cG)$ is free over $\SG_n$.

\item\label{BKC_Gal} Let $n$ be a positive integer satisfying $p^n\cG=0$. Then there exists a natural isomorphism of $G_{K_\infty}$-modules 
\[
\cG(\okbar)\to \Hom_{\SG,\varphi}(\SGm^*(\cG),W_n(R)).
\]
Moreover, we have $\lg_{\bZ_p}(\cG(\okbar))=\lg_{\SG}(\SGm^*(\cG))$.
\item\label{BKC_dual} Let $\cG^\vee$ be the Cartier dual of $\cG$. Then there exists a natural isomorphism $\SGm^*(\cG^\vee)\to \SGm^*(\cG)^\vee$ which, combined with the natural isomorphism of (\ref{BKC_Gal}), identifies the pairing of Cartier duality 
\[
\langle -,-\rangle_{\cG}:\cG(\okbar)\times\cG^\vee(\okbar)\to \mu_{p^n}(\okbar)
\]
with the natural perfect pairing
\[
\Hom_{\SG,\varphi}(\SGm^*(\cG),W_n(R))\times \Hom_{\SG,\varphi}(\SGm^*(\cG)^\vee,W_n(R))\to W_n(R).
\]

\item\label{BKC_ram} For any non-negative rational number $i$, we define the $i$-th lower ramification subgroup $\cG_i$ of $\cG$ as the scheme-theoretic closure in $\cG$ of $\Ker(\cG(\okbar)\to \cG(\cO_{\Kbar,i}))$. Then there exists an ideal $I_{n,i}$ of $W_n(R)$ such that the isomorphism of (\ref{BKC_Gal}) induces an isomorphism
\[
\cG_i(\okbar) \simeq \Hom_{\SG,\varphi}(\SGm^*(\cG),I_{n,i})
\]
for any $i\leq 1$.
Moreover, we have $I_{1,i}=m_R^{\geqslant i}$.
\end{enumerate}
\end{thm}
\begin{proof}
The assertions (\ref{BKC_eq}) and (\ref{BKC_Gal}) are contained in \cite[Corollary 4.3]{Kim_2}: the assertion on truncated Barsotti-Tate groups of level $n$ over $\okey$ follows from the fact that they are $p^n$-torsion parts of $p$-divisible groups \cite[Th\'{e}or\`{e}me 4.4 (e)]{Il}, and the equality on the length follows from the natural isomorphism of (\ref{BKC_Gal}). 
The assertion (\ref{BKC_dual}) follows from a similar assertion on $p$-divisible groups \cite[\S 5.1]{Kim_2} and a d\'{e}vissage argument as in \cite[Proposition 4.4]{Ha_ramcorr}. The assertion (\ref{BKC_ram}) is \cite[Theorem 1.1 and Corollary 3.3]{Ha_lowram}.
\end{proof}

Next we recall, for any extension $L/K$ of complete valuation fields, the definitions of invariants associated to a finite flat group scheme $\cG$ over $\oel$ which is killed by $p^n$ with some positive integer $n$. For any finitely presented torsion $\oel$-module $M$, write as $M\simeq \bigoplus_i \oel/(a_i)$ with some $a_i\in \oel$ and put $\deg(M)=\sum_i v_p(a_i)$. Since $\cG$ is finitely presented over $\oel$, the module $\omega_{\cG}$ of invariant differentials of $\cG$ is a finitely presented $\oel$-module. We put $\deg(\cG)=\deg(\omega_{\cG})$, and refer to it as the degree of $\cG$. 

Let $\bar{L}$ be an algebraic closure of $L$. Note that any element $x\in \cG(\oelbar)$ defines a homomorphism
\[
x: \cG^\vee\times_{\oel}\Spec(\oelbar)\to \Gm\times_{\oel}\Spec(\oelbar)
\]
by Cartier duality. We define the Hodge-Tate map by 
\[
\HT_\cG: \cG(\oelbar)\to \omega_{\cG^\vee}\otimes_{\oel} \oelbar,\quad x\mapsto x^*\frac{dT}{T}
\]
and, for any positive rational number $i$, the $i$-th Hodge-Tate map by the composite
\[
\HT_{\cG,i}: \cG(\oelbar)\overset{\HT_\cG}{\to} \omega_{\cG^\vee}\otimes_{\oel} \cO_{\bar{L}} \to \omega_{\cG^\vee}\otimes_{\oel} \cO_{\bar{L},i}
\]
of $\HT_\cG$ and the reduction map. We often denote them by $\HT$ and $\HT_i$. 

Suppose that $\cG$ is a truncated Barsotti-Tate group of level $n$, height $h$ and dimension $d$ over $\oel$. Consider the $p$-torsion part $\cG[p]$. 
Note that the Lie algebra $\Lie(\cG^\vee[p] \times \sS_{L,1})$ is a free $\cO_{L,1}$-module of rank $h-d$. The Verschiebung of $\cG^\vee[p]\times \sS_{L,1}$ induces a map on the Lie algebra
\[
\Lie(V_{\cG^\vee[p]\times \sS_{L,1}}): \Lie(\cG^\vee[p]\times \sS_{L,1})^{(p)} \to \Lie(\cG^\vee[p]\times \sS_{L,1}).
\]
The truncated valuation for $v_p$ of the determinant of a representing matrix of this map is independent of the choice of a basis of the Lie algebra, which we call the Hodge height of $\cG$ and denote by $\Hdg(\cG)$. Finally, for any truncated Barsotti-Tate group $\cG$ of level one over $\okey$ and any element $i$ of $e^{-1} \bZ_{\geq 0}$, the quotient $\SGm^*(\cG)_i=\SGm^*(\cG)/u^{e i}\SGm^*(\cG)$ has a natural structure of a $\varphi$-module induced by $\varphi_\SGm$. We put
\[
\Fil^1\SGm^*(\cG)_i=\Img(1\otimes\varphi: \varphi^* \SGm^*(\cG)_i \to \SGm^*(\cG)_i).
\]
It also has a natural structure of a $\varphi$-module induced by $\varphi_\SGm$. By the isomorphisms of $k$-algebras $\SG_1/(u^e)\to \cO_{K,1}$ defined by $u\mapsto \pi$ and $R_i\to \cO_{\Kbar,i}$ defined by the zeroth projection $\prjt_0$ for $i\leq 1$, we identify the both sides. 
For any $x\in \SG_1/(u^e)$, we define the truncated valuation $v_u(x)$ by $v_u(x)=\min\{v_u(\hat{x}),e\}$ with any lift $\hat{x}\in \SG_1$ of $x$.
Then these invariants of $\cG$ on the side of differentials can be read off from the associated Breuil-Kisin module, as follows.

\begin{prop}\label{DegHdg}
\begin{enumerate}
\item\label{DegHdg_deg} For any finite flat group scheme $\cG$ over $\okey$ killed by $p$, there exists a natural isomorphism
\[
\SGm^*(\cG)/(1\otimes \varphi_{\SGm^*(\cG)})(\varphi^*\SGm^*(\cG))\to \omega_{\cG}
\]
and we have 
\[
\deg(\cG)=e^{-1}v_u(\det(\varphi_{\SGm^*(\cG)})).
\]
\item\label{DegHdg_Hdg} Suppose that $\cG$ is a truncated Barsotti-Tate group of level one. Then we have a natural isomorphism
\[
\Lie(\cG^\vee\times \sS_{1})\to \Fil^1 \SGm^*(\cG)_1.
\]
The $\cO_{K,1}$-module $\Fil^1 \SGm^*(\cG)_1$ is a direct summand of $\SGm^*(\cG)_1$ of rank $h-d$. Moreover, we have the equality of truncated valuations
\[
\Hdg(\cG)=e^{-1}v_u(\det(\varphi_{\Fil^1 \SGm^*(\cG)_1})).
\]
\item\label{DegHdg_HT} Suppose that $\cG$ is a truncated Barsotti-Tate group of level one. For any positive rational number $i\leq 1$, the $i$-th Hodge-Tate map coincides with the composite
\begin{align*}
\cG(\okbar)&\to \Hom_{\SG,\varphi}(\SGm^*(\cG),R)\to \Hom_{\SG}(\Fil^1\SGm^*(\cG)_1,R_i)\\
&\simeq \Hom_{\cO_{K}}(\Lie(\cG^\vee\times \sS_{1}),\cO_{\Kbar,i})\simeq \omega_{\cG^\vee}\otimes_{\okey} \cO_{\Kbar,i}.
\end{align*}
\end{enumerate}
\end{prop}
\begin{proof}
The first isomorphism is shown in \cite[Proposition 3.2]{Ti_2} and the others are in \cite[\S 2.3]{Ha_cansub}. Note that though \cite{Ha_cansub} assumes $p>2$, the same proof remains valid also for $p=2$ by using \cite{Kim_2} instead of \cite{Ki_Fcrys}.
\end{proof}

%---------------------------------------------------------------------

%---------------------------------------------------------------------

%---------------------------------------------------------------------

%---------------------------------------------------------------------

\subsection{$\bZ_{p^f}$-groups}\label{SecZfgps}

Let $f$ be a positive integer. We assume that the residue field $k$ of $K$ contains the finite field $\bF_{p^f}$. Let $\bB_f$ be the set of embeddings of $\bF_{p^f}$ into $k$. Any $\beta\in \bB_f$ has the canonical lifts $\bZ_{p^f}\to \okey$ and $\bQ_{p^f}\to K$, which we also denote by $\beta$. Then any $W\otimes \bZ_{p^f}$-module $M$ is decomposed as 
\[
M=\bigoplus_{\beta\in \bB_f} M_\beta
\]
according with the decomposition $W\otimes \bZ_{p^f}\simeq \prod_{\beta\in \bB_f} W$. 

Let $L/K$ be any extension of complete valuation fields and $\bar{L}$ an algebraic closure of $L$. A group scheme $\cG$ over $\oel$ is said to be a $\bZ_{p^f}$-group if it is equipped with an action of the ring $\bZ_{p^f}$. Then we have the decompositions
\[
\omega_{\cG}=\bigoplus_{\beta\in \bB_f} \omega_{\cG,\beta},\quad \Lie(\cG\times \sS_{L,n})=\bigoplus_{\beta\in \bB_f} \Lie(\cG\times \sS_{L,n})_{\beta}.
\]
When $\cG$ is finite and flat over $\oel$, we define the $\beta$-degree of $\cG$ by $\deg_\beta(\cG)=\deg(\omega_{\cG,\beta})$. We have $\deg(\cG)=\sum_{\beta\in \bB_f} \deg_{\beta}(\cG)$. Moreover, for any exact sequence of finite flat $\bZ_{p^f}$-groups over $\oel$
\[
\xymatrix{
0 \ar[r] & \cG' \ar[r] & \cG \ar[r] & \cG''\ar[r] & 0,
}
\]
the equality $\deg_{\beta}(\cG)=\deg_{\beta}(\cG')+\deg_{\beta}(\cG'')$ holds.

Let $n$ be a positive integer. A $\bZ_{p^f}$-group $\cG$ over $\oel$ is said to be a truncated Barsotti-Tate $\bZ_{p^f}$-group of level $n$ if $\cG$ is a truncated Barsotti-Tate group of level $n$, height $2f$ and dimension $f$ such that $\omega_{\cG}$ is a free $\cO_{L,n} \otimes\bZ_{p^f}$-module of rank one. Note that for such $\cG$, we have $\deg_{\beta}(\cG)=n$. We say that such $\cG$ is $\bZ_{p^f}$-alternating self-dual if it is equipped with an isomorphism of $\bZ_{p^f}$-groups $i:\cG\simeq \cG^\vee$ over $\oel$ such that the perfect pairing defined via Cartier duality 
\[
\cG(\oelbar)\times \cG(\oelbar)\overset{1\times i}\to \cG(\oelbar)\times \cG^\vee(\oelbar)\overset{\langle-,-\rangle_{\cG}}{\to} \mu_{p^n}(\oelbar)
\]
satisfies $\langle x, i(ax)\rangle_{\cG}=1$ for any $x\in \cG(\oelbar)$ and $a\in \bZ_{p^f}$. In this case, we also say that the isomorphism $i$ is $\bZ_{p^f}$-alternating. Then the map $i$ is skew-symmetric: namely, we have the commutative diagram
\[
\xymatrix{
\cG \ar[r]^-{i}\ar[d]_{\mathrm{can}.} & \cG^\vee \ar[d]^{-1}\\
\cG^{\vee\vee} \ar[r]_-{i^\vee} & \cG^\vee.
}
\]
For $p\neq 2$, an isomorphism of $\bZ_{p^f}$-groups $i:\cG\simeq \cG^\vee$ is $\bZ_{p^f}$-alternating if and only if it is skew-symmetric. We abbreviate $\bZ_{p^f}$-alternating self-dual truncated Barsotti-Tate $\bZ_{p^f}$-group of level $n$ as $\bZ_{p^f}\text{-}\ADBT_n$. For a $\bZ_{p^f}\text{-}\ADBT_n$ $\cG$ over $\oel$, the $\cO_{L,n} \otimes \bZ_{p^f}$-modules 
$\omega_{\cG}$, $\Lie(\cG\times \sS_{L,n})$, $\omega_{\cG^\vee}$, $\Lie(\cG^\vee\times \sS_{L,n})$
are all free of rank one. Moreover, the action of the Verschiebung on $\Lie(V_{\cG^\vee[p]\times\sS_{L,1}})$ can be written as the direct sum of $\sigma$-semilinear maps
\[
\Lie(V_{\cG^\vee[p]\times\sS_{L,1}})_\beta: \Lie(\cG^\vee[p]\times\sS_{L,1})_{\sigma^{-1}\circ\beta}\to \Lie(\cG^\vee[p]\times\sS_{L,1})_{\beta}.
\]
Note that the both sides are free $\cO_{L,1}$-modules of rank one, and by choosing their bases, this map is identified with the multiplication by an element $a_{\beta}\in \cO_{L,1}$. We define the $\beta$-Hodge height $\Hdg_\beta(\cG)$ of $\cG$ as the truncated valuation of $a_\beta$, namely
\[
\Hdg_{\beta}(\cG)=v_p(a_{\beta}),
\]
which is independent of the choice of bases. From the diagram in the proof of \cite[Proposition 2]{Fa} and \cite[Lemma 2.3.7]{Con}, we obtain the equality
\[
\Hdg_{\beta}(\cG)=\Hdg_{\beta}(\cG^\vee).
\]

A Breuil-Kisin module $\SGm$ is called a $\bZ_{p^f}$-Breuil-Kisin module if $\SGm$ is equipped with an $\SG$-linear action of the ring $\bZ_{p^f}$ commuting with $\varphi_\SGm$. A morphism of $\bZ_{p^f}$-Breuil-Kisin modules is that of Breuil-Kisin modules compatible with $\bZ_{p^f}$-action. The $\bZ_{p^f}$-Breuil-Kisin modules whose underlying $\SG$-modules are free of finite rank over $\SG_1$ ({\it resp.} finitely generated, $p$-power torsion and $u$-torsion free) form a category, which we denote by $\bZ_{p^f}\text{-}\ModSGf$ ({\it resp.} $\bZ_{p^f}$-$\ModSGfinf$). Note that $\SGm\mapsto \SGm^\vee$ defines a notion of duality also for these categories. The anti-equivalence $\SGm^*(-)$ of the Breuil-Kisin classification induces an anti-equivalence from the category of finite flat $\bZ_{p^f}$-groups over $\okey$ killed by some $p$-power to $\bZ_{p^f}$-$\ModSGfinf$. 

To give an object $\SGm$ of $\bZ_{p^f}$-$\ModSGf$ ({\it resp.} $\bZ_{p^f}$-$\ModSGfinf$) is the same as to give a free $\SG_1$-module $\SGm$ of finite rank ({\it resp.} a finitely generated $\SG$-module $\SGm$ which is $p$-power torsion and $u$-torsion free) equipped with a decomposition into $\SG$-submodules $\SGm=\bigoplus_{\beta\in \bB_f} \SGm_{\beta}$ and a $\varphi$-semilinear map $\varphi_{\SGm,\beta}: \SGm_{\sigma^{-1}\circ\beta}\to \SGm_{\beta}$, which we often write as $\varphi_{\beta}$, for each $\beta\in \bB_f$ such that the cokernel of the linearization $1\otimes \varphi_{\beta}: \varphi^* \SGm_{\sigma^{-1}\circ\beta} \to \SGm_{\beta}$ is killed by $E(u)$. Since $1\otimes\varphi: \varphi^*\SGm \to \SGm$ is injective, the map $1\otimes \varphi_\beta$ is also injective. Hence we see that if $\SGm\neq 0$, then $\SGm_\beta\neq 0$ for any $\beta\in \bB_f$. Since $E(u) \SGm\subseteq (1\otimes\varphi)(\varphi^* \SGm)$, we have $E(u) \SGm_{\beta}\subseteq (1\otimes\varphi_{\beta})(\varphi^* \SGm_{\sigma^{-1}\circ \beta})$. 

Let $\SGm$ be any object of $\bZ_{p^f}$-$\ModSGf$. The last inclusion implies that the free $\SG_1$-modules $\SGm_\beta$ have the same rank for any $\beta\in \bB_f$, which is equal to
\[
f^{-1}\rank_{\SG_1}(\SGm)=\dim_{\bF_{p^f}}(\Hom_{\SG_1,\varphi}(\SGm,R)).
\]
Moreover, Proposition \ref{DegHdg} (\ref{DegHdg_deg}) implies that, if $\cG$ is the finite flat group scheme over $\okey$ corresponding to $\SGm$, then we have
\begin{equation}\label{EqnDegBeta}
\deg_\beta(\cG)=e^{-1}\lg_{\SG_1}(\Coker(1\otimes\varphi_{\beta}: \varphi^*\SGm_{\sigma^{-1}\circ \beta}\to \SGm_{\beta})).
\end{equation}

\begin{lem}\label{degbeta}
Let $\cG$ be a finite flat $\bZ_{p^f}$-group over $\okey$. Then we have
\[
\deg_\beta(\cG)+\deg_\beta(\cG^\vee)=\lg_{\SG}(\SGm^*(\cG)_\beta).
\]
\end{lem}
\begin{proof}
Let $\cH$ be the scheme-theoretic closure in $\cG$ of $\cG(\okbar)[p]$. It is a finite flat closed $\bZ_{p^f}$-subgroup of $\cG$ killed by $p$. Since the both sides are additive with respect to exact sequences of finite flat $\bZ_{p^f}$-groups over $\okey$, by an induction we may assume that $\cG$ is killed by $p$. 

Put $\SGm=\SGm^*(\cG)$. Let $A_\beta$ be the representing matrix of the map $\varphi_{\SGm, \beta}:\SGm_{\sigma^{-1}\circ\beta}\to \SGm_{\beta}$ with some bases. From the definition of the dual, we see that the representing matrix of the map $\varphi_{\SGm^\vee,\beta}$ with the dual bases is $c_0^{-1}E(u) {}^t\! A_\beta^{-1}$. Then the equality (\ref{EqnDegBeta}) implies 
\begin{align*}
\deg_{\beta}(\cG)+\deg_{\beta}(\cG^\vee)&=e^{-1}v_u(\det(A_{\beta})\det (c_0^{-1}E(u) {}^t\! A_{\beta}^{-1}))\\
&=\rank_{\SG_1}(\SGm_\beta).
\end{align*}
This concludes the proof.
\end{proof}

For any $\SGm\in \bZ_{p^f}$-$\ModSGf$ and $i\in e^{-1}\bZ\cap [0,1]$, we put $\SGm_{\beta,i}=\SGm_{\beta}/u^{e i} \SGm_{\beta}$. Then the map $\varphi_{\beta}$ induces an $\SG_1$-semilinear map $\SGm_{\sigma^{-1}\circ\beta,i}\to \SGm_{\beta,i}$, which we denote also by $\varphi_{\beta}$. We define
\[
\Fil^1 \SGm_{\beta,i}=\Img(1\otimes \varphi_{\beta}:\varphi^*\SGm_{\sigma^{-1}\circ \beta,i} \to \SGm_{\beta,i}).
\]

%---------------------------------------------------------------------

%---------------------------------------------------------------------

\subsection{Tian's construction}\label{SecCansub1}

Let $\cG$ be a $\bZ_{p^f}\text{-}\ADBT_1$ over $\okey$ of $\beta$-Hodge height $w_\beta$. Put $\SGm=\SGm^*(\cG)$ and $\SGm_1=\SGm/u^e \SGm$. Then each $\SGm_\beta$ is a free $\SG_1$-module of rank two. By Proposition \ref{DegHdg} (\ref{DegHdg_deg}) and (\ref{DegHdg_Hdg}), we have an exact sequence of $\varphi$-modules over $\cO_{K,1}\otimes \bZ_{p^f}$
\[
\xymatrix{
0 \ar[r] & \Fil^1 \SGm_1 \ar[r] & \SGm \ar[r] & \SGm_1/\Fil^1\SGm_1 \ar[r] & 0,
}
\]
where the $\cO_{K,1}\otimes \bZ_{p^f}$-modules $\Fil^1 \SGm_1$ and $\SGm_1/\Fil^1\SGm_1$ are free of rank one. In particular, this splits as a sequence of $\cO_{K,1}\otimes \bZ_{p^f}$-modules. Hence we also have a split exact sequence of $\cO_{K,1}$-modules
\[
\xymatrix{
0 \ar[r] & \Fil^1 \SGm_{\beta,1} \ar[r] & \SGm_{\beta,1} \ar[r] & \SGm_{\beta,1}/\Fil^1\SGm_{\beta,1} \ar[r] & 0,
}
\]
where the modules on the left-hand side and the right-hand side are free of rank one. As in the proof of \cite[Theorem 3.1]{Ha_cansub}, we can choose a basis $\{e_\beta,e'_\beta\}$ of $\SGm_\beta$ satisfying $e_\beta\in  (1\otimes\varphi)(\varphi^* \SGm_{\sigma^{-1}\circ \beta})$ such that the image of $e_\beta$ in $\SGm_{\beta,1}$ is a basis of $\Fil^1\SGm_{\beta,1}$ and the image of $e'_{\beta}$ in $\SGm_{\beta,1}/\Fil^1\SGm_{\beta,1}$ gives its basis. Then we can write as
\begin{equation}\label{EqnTian}
\varphi(e_{\sigma^{-1}\circ \beta},e'_{\sigma^{-1}\circ \beta})=(e_{\beta},e'_{\beta})\begin{pmatrix}a_{\beta,1} & a_{\beta,2} \\ u^e a_{\beta,3} & u^e a_{\beta,4}\end{pmatrix}
\end{equation}
with some invertible matrix
\[
\begin{pmatrix}a_{\beta,1} & a_{\beta,2} \\ a_{\beta,3} & a_{\beta,4}\end{pmatrix}\in \mathit{GL}_2(\SG_1).
\]

For any $\bZ_{p^f}$-group $\cG$ over $\okey$ killed by $p$, a finite flat closed $\bZ_{p^f}$-subgroup $\cH$ of $\cG$ over $\okey$ is said to be cyclic if the $\bF_{p^f}$-vector space $\cH(\okbar)$ is of rank one. Note that, for such $\cH$, the free $\SG_1$-module $\SGm^*(\cH)_\beta$ is of rank one for any $\beta\in \bB_f$. If $\cG$ is a $\bZ_{p^f}$-$\ADBT_1$, then the proof of \cite[Lemma 2.1.1]{GK} shows that the $\bF_{p^f}$-subspace $\cH(\okbar)$ is automatically isotropic with respect to the $\bZ_{p^f}$-alternating perfect pairing on $\cG(\okbar)$. Moreover, since finite flat closed subgroup schemes of $\cG$ over $\okey$ are determined by their generic fibers, this implies that the $\bZ_{p^f}$-alternating isomorphism $i:\cG\simeq \cG^\vee$ induces an isomorphism $\cH\simeq (\cG/\cH)^\vee$.

Now the existence theorem of the canonical subgroup of level one for a $\bZ_{p^f}\text{-}\ADBT_1$ over $\okey$ is as follows.

\begin{thm}\label{cansub1}
Let $\cG$ be a $\bZ_{p^f}\text{-}\ADBT_1$ over $\okey$ with $\beta$-Hodge height $w_\beta$. Put $w=\max\{w_\beta\mid \beta\in \bB_f\}$. Suppose that the inequality
\[
w_\beta +p w_{\sigma^{-1} \circ \beta} <p
\]
holds for any $\beta\in \bB_f$. Then there exists a finite flat closed cyclic $\bZ_{p^f}$-subgroup $\cC$ of $\cG$ over $\okey$ satisfying 
\[
\deg_\beta(\cG/\cC)=w_\beta.
\]
Moreover, the group scheme $\cC$ is the unique finite flat closed cyclic $\bZ_{p^f}$-subgroup of $\cG$ over $\okey$ satisfying
\[
\deg_{\beta}(\cC)+p\deg_{\sigma^{-1}\circ \beta}(\cC)>1
\]
for any $\beta\in \bB_f$.
We refer to $\cC$ as the canonical subgroup of $\cG$. It has the following properties:
\begin{enumerate}
\item\label{cansub1_isom} Let $\cG'$ be a $\bZ_{p^f}$-$\ADBT_1$ over $\okey$ satisfying the same condition on the $\beta$-Hodge heights as above and $\cC'$ the canonical subgroup of $\cG'$. Then any isomorphism of $\bZ_{p^f}$-groups $j:\cG\to \cG'$ over $\okey$ induces an isomorphism $\cC\simeq \cC'$.
\item\label{cansub1_BC} $\cC$ is compatible with base extension of complete discrete valuation rings with perfect residue fields.
\item\label{cansub1_dual} $\cC$ is compatible with Cartier duality. Namely, $(\cG/\cC)^\vee$ is the canonical subgroup of $\cG^\vee$.
\item\label{cansub1_Frob} The kernel of the Frobenius map of $\cG\times\sS_{1-w}$ coincides with $\cC\times \sS_{1-w}$.
\item\label{cansub1_ram} If $w<p/(p+1)$, then $\cC=\cG_{(1-w)/(p-1)}$. 
\item\label{cansub1_HT} If $w<(p-1)/p$, then $\cC(\okbar)$ coincides with $\Ker(\HT_i)$ for any rational number $i$ satisfying $w/(p-1)< i\leq 1-w$. 
\item\label{cansub1_ram2} If $w<(p-1)/p$, then $\cC=\cG_{i}$ for any rational number $i$ satisfying $1/(p(p-1))\leq i \leq (1-w)/(p-1)$.
\end{enumerate}
\end{thm}
\begin{proof}
Note that, since we have $w<1$ by assumption, Proposition \ref{DegHdg} (\ref{DegHdg_Hdg}) implies 
\[
w_\beta=e^{-1}v_u(a_{\beta,1}).
\]

The existence and the uniqueness in the theorem are due to Tian \cite[Theorem 3.10]{Ti_2}: the $\bZ_{p^f}$-subgroup $\cC$ is defined as the finite flat closed $\bZ_{p^f}$-subgroup of $\cG$ over $\okey$ corresponding to the quotient $\SGn=\SGm/\SGl$ via the Breuil-Kisin classification, where $\SGl=\bigoplus_{\beta\in \bB_f} \SGl_\beta$ is the unique $\bZ_{p^f}$-Breuil-Kisin submodule of $\SGm$ satisfying $\SGl_{\beta, 1-w_\beta}=\Fil^1\SGm_{\beta, 1-w_\beta}$ for any $\beta\in\bB_f$. In particular, the $\SG_1$-module $\SGl_\beta$ is generated by 
\[
\delta_\beta=e_{\beta}+u^{e(1-w_\beta)} y_\beta e'_\beta
\] 
with some $y_\beta\in \SG_1$. The assertions (\ref{cansub1_isom}) and (\ref{cansub1_BC}) follow from the uniqueness.

Let us prove the assertion (\ref{cansub1_dual}). Note that, since $\Hdg_\beta(\cG)=\Hdg_\beta(\cG^\vee)$, the $\bZ_{p^f}\text{-}\ADBT_1$ $\cG^\vee$ over $\okey$ also has the canonical subgroup $\cC'$. By Lemma \ref{degbeta}, we have 
\[
\deg_\beta((\cG/\cC)^\vee)=1-\deg_\beta(\cG/\cC)=1-w_\beta
\]
and the uniqueness assertion of the theorem and the assumption on $w_\beta$ imply $\cC'=(\cG/\cC)^\vee$.

The assertion (\ref{cansub1_Frob}) is also due to Tian \cite[Remark 3.11]{Ti_2}. Here we give a short proof for the convenience of the reader. Since $1-w\leq 1-w_\beta$, the construction of $\SGl$ implies
\begin{equation}\label{Fil1-w}
\SGl_{1-w}=\Fil^1\SGm_{1-w}. 
\end{equation}
Then Proposition \ref{DegHdg} (\ref{DegHdg_deg}) shows that the natural map
\[
\omega_{\cG/\cC}\otimes \cO_{K,1-w}\to \omega_{\cG}\otimes \cO_{K,1-w}
\]
is zero. 
By \cite[Proposition 1]{Fa}, the closed subgroup scheme $(\cG/\cC)^\vee\times \sS_{1-w}$ of $\cG^\vee\times \sS_{1-w}$ is killed by the Frobenius. Comparing the rank, the former coincides with the kernel of the Frobenius of the latter. Since $\cG^\vee\times \sS_{1-w}$ is a truncated Barsotti-Tate group of level one, we see by duality and \cite[Remark 1.3 (b)]{Il} that $\cC\times \sS_{1-w}$ also coincides with the kernel of the Frobenius of $\cG\times \sS_{1-w}$.

Next we consider the assertion (\ref{cansub1_ram}). It can be shown similarly to \cite[Theorem 3.1 (c)]{Ha_cansub}.
For any $\SG_1$-algebra $A$, we define an abelian group $\cH(\SGm)(A)$ by 
\[
\cH(\SGm)(A)=\Hom_{\SG_1,\varphi}(\SGm,A),
\]
where we consider $A$ as a $\varphi$-module with the $p$-th power Frobenius map. If we take the basis $\{e_\beta,e'_\beta\}_{\beta\in \bB_f}$ of $\SGm$ as above, it is identified with the set of $f$-tuples of elements $(x_\beta,x'_\beta)\in A^2$ satisfying
\begin{equation}\label{eqnA}
(x_{\sigma^{-1}\circ\beta}^p, (x'_{\sigma^{-1}\circ\beta})^p)=(x_{\beta},x'_{\beta})\begin{pmatrix}a_{\beta,1} & a_{\beta,2} \\ u^e a_{\beta,3} & u^e a_{\beta,4}\end{pmatrix}.
\end{equation}
We define the subgroup $\cH(\SGm)_i(R)$ of $\cH(\SGm)(R)$ by
\[
\cH(\SGm)_i(R)=\Ker(\cH(\SGm)(R)\to \cH(\SGm)(R_i)). 
\]
Similarly, we have the subgroup $\cH(\SGn)(R)$ of $\cH(\SGm)(R)$. Note that we have an exact sequence
\[
\xymatrix{
0\ar[r] & \cH(\SGn)(R) \ar[r] & \cH(\SGm)(R)\ar[r] & \cH(\SGl)(R) \ar[r] & 0
}
\]
which can be identified with the exact sequence of abelian groups
\[
\xymatrix{
0\ar[r] & \cC(\okbar) \ar[r] & \cG(\okbar)\ar[r] & (\cG/\cC)(\okbar) \ar[r] & 0.
}
\]
Since $\deg_\beta(\cG/\cC)=w_\beta$, the basis $\delta_\beta$ of $\SGl_\beta$ satisfies
\[
\varphi_{\beta}(\delta_{\sigma^{-1}\circ\beta})=\lambda_\beta \delta_{\beta}\text{ with }v_R(\lambda_\beta)=w_\beta.
\]
Thus any element of $\cH(\SGl)(R)$ can be identified with an $f$-tuple $(z_\beta)_{\beta\in \bB_f}$ in $R$ satisfying
\begin{equation}\label{eqnL}
z_{\sigma^{-1}\circ\beta}^p=\lambda_\beta z_{\beta}
\end{equation}
for any $\beta\in \bB_f$.

\begin{lem}\label{lowram}
For any element $(z_\beta)_{\beta\in \bB_f}\neq 0$ of $\cH(\SGl)(R)$, we have
\[
v_R(z_\beta)\leq \frac{w}{p-1}\text{ for any }\beta\in \bB_f.
\]
In particular, $\cH(\SGl)_i(R)=0$ for any $i>w/(p-1)$.
\end{lem}
\begin{proof}
From the equation (\ref{eqnL}), we see that $(z_\beta)_{\beta\in \bB_f}\neq 0$ if and only if $z_\beta\neq 0$ for any $\beta\in \bB_f$. This equation also implies
\[
p v_R(z_{\sigma^{-1}\circ\beta})=v_R(z_{\beta})+ w_\beta
\]
for any $\beta\in \bB_f$ and thus
\[
v_R(z_{\sigma^{-1}\circ\beta})=\frac{1}{p^f-1}\sum_{l=0}^{f-1}p^{f-1-l}w_{\sigma^l \circ \beta} \leq \frac{w}{p-1},
\]
which concludes the proof.
\end{proof}

We claim that $\cH(\SGn)(R)= \cH(\SGm)_{(1-w)/(p-1)}(R)$. Indeed, take any element $(x_\beta,x'_\beta)_{\beta\in \bB_f}$ of the left-hand side. Take the element $y_\beta\in \SG_1$ such that 
\[
\delta_\beta=e_\beta+ u^{e(1-w_\beta)}y_\beta e'_\beta
\] 
is a basis of the $\SG_1$-module $\SGl_\beta$. Then $(x_\beta,x'_\beta)_{\beta\in \bB_f}\in \cH(\SGn)(R)$ if and only if $x_\beta+u^{e(1-w_\beta)}y_\beta x'_\beta=0$ for any $\beta\in \bB_f$. The equation (\ref{eqnA}) implies
\[
(x'_{\sigma^{-1}\circ\beta})^p=x'_{\beta}(-a_{\beta,2}u^{e(1-w_\beta)}y_{\beta} +u^e a_{\beta,4})
\]
and thus
\[
p v_R(x'_{\sigma^{-1}\circ\beta})\geq v_R(x'_{\beta})+1-w_\beta.
\]
Hence we have
\[
v_R(x'_{\sigma^{-1}\circ\beta})\geq \frac{1}{p^f-1}\sum_{l=0}^{f-1} p^{f-1-l}(1-w_{\sigma^l\circ \beta})\geq \frac{1-w}{p-1}
\]
for any $\beta\in \bB_f$ and we obtain $(x_\beta,x'_\beta)_{\beta\in \bB_f}\in \cH(\SGm)_{(1-w)/(p-1)}(R)$.

Conversely, let $(x_\beta,x'_\beta)_{\beta\in \bB_f}$ be an element of $\cH(\SGm)_{(1-w)/(p-1)}(R)$. 
By the equation (\ref{eqnA}), we have
\begin{equation}\label{eqnAinv}
(x_{\beta}, u^{e}x'_{\beta})=(x_{\sigma^{-1}\circ\beta}^p, (x'_{\sigma^{-1}\circ\beta})^p)\begin{pmatrix}a_{\beta,1} & a_{\beta,2} \\ a_{\beta,3} & a_{\beta,4}\end{pmatrix}^{-1}
\end{equation}
for any $\beta\in \bB_f$.
Recall that the matrix on the right-hand side is an element of $\mathit{GL}_2(\SG_1)$. Hence we obtain 
\[
v_R(x_{\beta})\geq \frac{p(1-w)}{p-1}
\]
and the element $z_\beta=x_\beta+u^{e(1-w_\beta)}y_\beta x'_\beta$ satisfies
\[
v_R(z_\beta)\geq \frac{p(1-w)}{p-1}
\]
for any $\beta\in \bB_f$. By the assumption $w<p/(p+1)$, we have 
\[
\frac{w}{p-1}<\frac{p(1-w)}{p-1} 
\]
and Lemma \ref{lowram} implies $z_\beta=0$ for any $\beta\in \bB_f$. Therefore we obtain $(x_\beta,x'_\beta)_{\beta\in \bB_f}\in \cH(\SGn)(R)$, from which the claim follows. Now Theorem \ref{BKC} (\ref{BKC_ram}) shows the assertion (\ref{cansub1_ram}).

Let us show the assertion (\ref{cansub1_HT}). This is shown similarly to \cite[Theorem 3.1 (2)]{Ha_cansub}. Since $i\leq 1-w$, Proposition \ref{DegHdg} (\ref{DegHdg_HT}) and (\ref{Fil1-w}) show that the kernel of the map $\HT_i$ is equal to the kernel of the natural map
\[
\cG(\okbar)\simeq \cH(\SGm)(R)\to \cH(\SGl)(R)\to \cH(\SGl)(R_i).
\]
Since $i>w/(p-1)$, Lemma \ref{lowram} implies $\cH(\SGl)_i(R)=0$ and the right arrow in the above map is injective. Thus $\Ker(\HT_i)$ coincides with the inverse image of 
\[
\cH(\SGn)(R)=\Ker(\cH(\SGm)(R)\to \cH(\SGl)(R))
\]
by the isomorphism $\cG(\okbar)\simeq \cH(\SGm)(R)$, which is $\cC(\okbar)$. The assertion (\ref{cansub1_ram2}) follows from the lemma below.
\end{proof}

\begin{lem}\label{canlowram1}
Let $\cG$ be a $\bZ_{p^f}\text{-}\ADBT_1$ over $\okey$ with $\beta$-Hodge height $w_\beta$. Put $w=\max\{w_\beta\mid \beta\in \bB_f\}$ and
\[
i_n=\frac{1}{p^{n-1}(p-1)}-\frac{w}{p-1},\quad i'_n=\frac{1}{p^n(p-1)}.
\]
Suppose $w<(p-1)/p^n$ for some positive integer $n$. Let $\cC$ be the canonical subgroup of $\cG$, which exists by Theorem \ref{cansub1}. Then we have
\[
\cC=\cG_{i_m}=\cG_{i'_m}
\] 
for any $1\leq m\leq n$.
\end{lem}
\begin{proof}
This can be shown in the same way as \cite[Lemma 5.2]{Ha_lowram}. 
We follow the notation in the proof of Theorem \ref{cansub1}. By Theorem \ref{BKC} (\ref{BKC_ram}) and Theorem \ref{cansub1} (\ref{cansub1_ram}), it is enough to show 
\[
\Hom_{\SG,\varphi}(\SGm, m_R^{\geqslant i'_n})\subseteq \cH(\SGn)(R).
\]
We identify an element $x$ of the left-hand side with a solution $(x_\beta, x'_\beta)_{\beta\in \bB_f}$ of the equation (\ref{eqnA}) in $R$ satisfying $v_R(x_\beta),v_R(x'_\beta)\geq i'_n$ for any $\beta\in \bB_f$. From the equality (\ref{eqnAinv}),
we have $v_R(x_\beta)\geq p i'_n>w/(p-1)$. Since we have $1-w_\beta\geq 1-w >w/(p-1)$, the element $z_\beta=x_\beta+u^{e(1-w_\beta)} y_\beta x'_\beta$ satisfies 
\[
v_R(z_\beta)>w/(p-1).
\]
Thus Lemma \ref{lowram} implies $z_\beta=0$ for any $\beta\in \bB_f$ and $x\in \cH(\SGn)(R)$.
\end{proof}

The description of the Hodge-Tate map via the Breuil-Kisin classification also yields a torsion property of the Hodge-Tate cokernel, as follows. 

\begin{lem}\label{cok1}
Let $\cG$ be a $\bZ_{p^f}\text{-}\ADBT_1$ over $\okey$ with $\beta$-Hodge height $w_\beta$. Put $w=\max\{w_\beta\mid \beta\in \bB_f\}$. Suppose $w<(p-1)/p$. Then the cokernel of the linearization of the Hodge-Tate map
\[
\HT\otimes 1: \cG(\okbar)\otimes \okbar \to \omega_{\cG^\vee}\otimes_{\okey} \okbar
\]
is killed by $m_{\Kbar}^{\geqslant w/(p-1)}$.
\end{lem}
\begin{proof}
For this, we first show the following lemma.

\begin{lem}\label{NAK}
Let $M$ be a finitely generated $\okbar$-module. Let $N$ be an $\okbar$-submodule of $M$.
Suppose that there exist positive rational numbers $r>s$ satisfying $m^{\geqslant s}_{\Kbar} M\subseteq N+m^{\geqslant r}_{\Kbar}M$. Then we have $m^{\geqslant s}_{\Kbar} M\subseteq N$.
\end{lem}
\begin{proof}
Put $Q=m^{\geqslant s}_{\Kbar} (M/N)$. Since the assumption implies $m^{\geqslant r-s}_{\Kbar}Q=Q$, Nakayama's lemma shows $Q=0$ and $m^{\geqslant s}_{\Kbar} M\subseteq N$.
\end{proof}

Put $M=\omega_{\cG^\vee}\otimes_{\okey}\okbar$ and
\[
N=\Img(\HT\otimes 1: \cG(\okbar)\otimes \okbar \to \omega_{\cG^\vee}\otimes_{\okey}\okbar).
\]
We claim that
\[
m_{\Kbar}^{\geqslant w/(p-1)} \Coker(\HT_{1-w}\otimes 1: \cG(\okbar)\otimes \okbar\to \omega_{\cG^\vee}\otimes_{\okey} \cO_{\Kbar,1-w})=0.
\]
This is equivalent to the inclusion 
\[
m_{\Kbar}^{\geqslant w/(p-1)} M \subseteq N+m_{\Kbar}^{\geqslant 1-w} M.
\]
The assumption $w<(p-1)/p$ implies $w/(p-1)<1-w$ and thus Lemma \ref{cok1} follows from the claim and Lemma \ref{NAK}.

Now let us prove the claim. Consider the basis $\delta_\beta$ of $\SGl_\beta$ as in the proof of Theorem \ref{cansub1}. Using this, we identify each element of $\cH(\SGl)(R)$ with an $f$-tuple $(z_{\beta})_{\beta\in \bB_f}$ in $R$ satisfying the equation (\ref{eqnL}). By Proposition \ref{DegHdg} (\ref{DegHdg_HT}) and (\ref{Fil1-w}), the cokernel of the claim is identified with the cokernel of the natural map
\begin{align*}
\cH(\SGl)(R)\otimes R&\to \Hom_{\SG_1}(\SGl, R_{1-w})=\SGl^\vee\otimes R_{1-w},\\
 (z_\beta)_{\beta\in \bB_f}\otimes 1 &\mapsto \sum_{\beta\in \bB_f} \delta_{\beta}^\vee \otimes z_{\beta}.
\end{align*}
Note that the abelian group $\cH(\SGl)(R)$ has a natural action of the ring $\bF_{p^f}$ defined by 
\[
\alpha(z_\beta)_{\beta\in \bB_f}=(\beta(\alpha) z_\beta)_{\beta\in \bB_f}\text{ for any }\alpha\in \bF_{p^f}.
\]
Take a generator $\alpha_0$ of the extension $\bF_{p^f}/\bF_p$ and a non-zero element $(z_\beta)_{\beta\in \bB_f}$ of $\cH(\SGl)(R)$. Then the subset 
$\{\alpha_0^l(z_\beta)_{\beta\in \bB_f} \}_{l=0,1,\ldots,f-1}$ forms a basis of the $\bF_p$-vector space $\cH(\SGl)(R)$. Hence the image of the natural map above is generated by the entries of the $f$-tuple
\[
(\delta_{\beta}^\vee\otimes 1)_{\beta\in \bB_f} (\beta(\alpha_0^l)z_\beta)_{\beta, l}=(\delta_{\beta}^\vee\otimes 1)_{\beta\in \bB_f} \diag(z_\beta)_{\beta\in \bB_f} (\beta(\alpha_0^l))_{\beta,l}.
\]
Since the matrix $(\beta(\alpha_0^l))_{\beta,l}$ is invertible in $\mathit{M}_f(R)$, the cokernel is isomorphic as an $R$-module to
\[
\bigoplus_{\beta\in \bB_f} R_{1-w}/(z_\beta).
\]
Thus the claim follows from Lemma \ref{lowram}.
\end{proof}

%---------------------------------------------------------------------

%---------------------------------------------------------------------

\subsection{Goren-Kassaei's theory}\label{SecGKVar}

Here we analyze the variation of $\beta$-Hodge heights by taking quotients with cyclic $\bZ_{p^f}$-subgroups. For the case of abelian varieties, it was obtained by Goren-Kassaei \cite[Lemma 5.3.4 and Lemma 5.3.6]{GK}.

\begin{lem}\label{LemGKVar}
Let $\cG$ be a $\bZ_{p^f}\text{-}\ADBT_1$ over $\okey$ with $\beta$-Hodge height $w_\beta$. Let $\cH$ be a finite flat closed cyclic $\bZ_{p^f}$-subgroup of $\cG$ over $\okey$. Put $v_\beta=\deg_\beta(\cG/\cH)$.
\begin{enumerate}
\item\label{LemGKVarV} If we have
\[
v_\beta+ p v_{\sigma^{-1}\circ \beta} <p \text{ for any } \beta\in \bB_f, 
\]
then $w_\beta=v_\beta$ and $\cG$ has the canonical subgroup, which is equal to $\cH$.
\item\label{LemGKVarW} If we have
\[
v_\beta+ p v_{\sigma^{-1}\circ \beta} >p \text{ for any } \beta\in \bB_f, 
\]
then $w_\beta=p(1- v_{\sigma^{-1}\circ\beta})$ and $\cG$ has the canonical subgroup, which is not equal to $\cH$. We refer to any $\cH$ satisfying this inequality as an anti-canonical subgroup of $\cG$.
\item\label{LemGKVarVW} If both of the inequalities in (\ref{LemGKVarV}) and (\ref{LemGKVarW}) are not satisfied, then 
\[
w_\beta+ p w_{\sigma^{-1}\circ \beta} \geq p \text{ for some } \beta\in \bB_f.
\]
\end{enumerate}
\end{lem}
\begin{proof}
Let $\SGp$ and $\SGq$ be the Breuil-Kisin modules corresponding to $\cG/\cH$ and $\cH$, respectively. We have an exact sequence of $\SG_1$-modules
\[
\xymatrix{
0\ar[r] & \SGp_\beta \ar[r] & \SGm_\beta \ar[r] & \SGq_\beta \ar[r] & 0
}
\]
for any $\beta\in \bB_f$. Note that $\SG_1$-modules $\SGp_\beta$ and $\SGq_\beta$ are free of rank one. Let $\{f_\beta, f'_\beta\}$ be a basis of the free $\SG_1$-module $\SGm_\beta$ such that $f_\beta$ is a basis of $\SGp_\beta$ and the image of $f'_\beta$ is a basis of $\SGq_\beta$. We can write as
\begin{equation}\label{EqnPQ}
\varphi_{\beta} (f_{\sigma^{-1}\circ \beta}, f'_{\sigma^{-1}\circ \beta})=(f_\beta, f'_\beta) \begin{pmatrix}a_\beta & b_\beta \\ 0 & c_\beta \end{pmatrix}
\end{equation}
with some $a_\beta, b_\beta, c_\beta\in \SG_1$ such that $v_R(a_\beta)=v_\beta$ and $v_R(c_\beta)=\deg_\beta(\cH)=1-v_\beta$. Thus we obtain
\[
\Fil^1 \SGm_{\beta,1}=\langle a_\beta f_\beta, b_\beta f_\beta+ c_\beta f'_\beta \rangle.
\]
Since it is a direct summand of $\SGm_{\beta,1}$ of rank one over $\cO_{K,1}$, we have 
\begin{equation}\label{EqnTrichot}
\left\{
\begin{array}{ll}
v_R(a_\beta)=0 &  (v_\beta=0),\\
v_R(b_\beta)=0 & (0<v_\beta<1),\\
v_R(c_\beta)=0 & (v_\beta=1)
\end{array}
\right.
\end{equation}
and
\[
\Fil^1\SGm_{\beta,1}=
\left\{\begin{array}{ll}
\langle f_\beta \rangle & (v_\beta=0),\\
\langle b_\beta f_\beta+ c_\beta f'_\beta \rangle & (v_\beta>0).\\
\end{array}\right.
\]
Moreover, in $\SGm_{\sigma\circ \beta,1}$ we have
\begin{equation}\label{EqnHdg}
\left\{
\begin{aligned}
\varphi_{\sigma\circ\beta}(f_\beta)&=a_{\sigma\circ \beta} f_{\sigma\circ \beta},\\
\varphi_{\sigma\circ\beta}(b_\beta f_\beta+ c_\beta f'_\beta)&=(b_\beta^p a_{\sigma\circ \beta}+c_\beta^p b_{\sigma\circ \beta})f_{\sigma\circ \beta}+c_\beta^p c_{\sigma\circ \beta} f'_{\sigma\circ \beta}.
\end{aligned}
\right.
\end{equation}

Now let us consider the assertion (\ref{LemGKVarV}). The assumption implies $v_{\sigma\circ\beta}<1$ for any $\beta\in \bB_f$. Hence $w_{\sigma\circ \beta}$ is equal to the valuation of the coefficient of $f_{\sigma\circ \beta}$ of the right-hand side of the equality (\ref{EqnHdg}) in both cases.
\begin{itemize}
\item If $v_\beta=0$, then $w_{\sigma\circ \beta}=v_R(a_{\sigma\circ \beta})=v_{\sigma\circ \beta}$.
\item If $0<v_\beta<1$, then $v_R(b_\beta)=0$ and
\[
w_{\sigma\circ \beta}=v_R(b_\beta^p a_{\sigma\circ \beta}+c_\beta^p b_{\sigma\circ \beta}).
\]
The assumption also yields $v_{\sigma\circ \beta} <p(1-v_\beta)$ and thus we have $w_{\sigma\circ \beta}=v_{\sigma\circ \beta}$.
\end{itemize}
Thus $\cG$ satisfies the assumption of Theorem \ref{cansub1} and has the canonical subgroup $\cC$. Since $\deg_\beta(\cH)=1-v_\beta$, the uniqueness of the theorem implies $\cH=\cC$.

Next we treat the assertion (\ref{LemGKVarW}). The assumption implies $v_{\sigma\circ\beta}>0$ for any $\beta\in \bB_f$. 
\begin{itemize}
\item If $0<v_{\sigma\circ \beta}<1$, then $v_R(b_{\sigma\circ \beta})=0$ and
\[
w_{\sigma\circ \beta}=v_R(b_\beta^p a_{\sigma\circ \beta}+c_\beta^p b_{\sigma\circ \beta}).
\]
By assumption we have $v_{\sigma\circ \beta} >p(1-v_\beta)$ and thus we obtain $w_{\sigma\circ \beta}=p(1-v_\beta)$.

\item If $v_{\sigma\circ \beta}=1$, then $v_R(c_{\sigma\circ \beta})=0$. This implies that $w_{\sigma\circ \beta}$ is equal to the valuation of the coefficient of $f'_{\sigma\circ \beta}$ of the equality (\ref{EqnHdg}), namely
\[
w_{\sigma\circ \beta}=v_R(c_\beta^p c_{\sigma\circ \beta})=p(1-v_\beta).
\]
\end{itemize}
From this we see that $\cG$ has the canonical subgroup $\cC$. If $\cC=\cH$, then we have
\[
v_\beta=\deg_\beta(\cG/\cH)=w_\beta=p(1-v_{\sigma^{-1}\circ \beta})
\]
for any $\beta\in \bB_f$, which contradicts the assumption.

The assertion (\ref{LemGKVarVW}) can be shown as in \cite[Corollary 5.3.7]{GK}: Take $\beta'\in \bB_f$ such that $v_{\beta'}+p v_{\sigma^{-1}\circ \beta'}\leq p$. Let $i\geq 1$ be the minimal integer satisfying $v_{\sigma^i \circ \beta'}+ pv_{\sigma^{i-1}\circ \beta'}\geq p$. The minimality shows that $\beta=\sigma^{i-1} \circ \beta'$ satisfies
\begin{equation}\label{EqnVW}
v_\beta+p v_{\sigma^{-1}\circ \beta}\leq p,\quad v_{\sigma \circ\beta}+p v_{\beta}\geq p.
\end{equation}
We claim that 
\[
w_\beta\geq v_\beta, \quad w_{\sigma\circ \beta}\geq p(1-v_\beta). 
\]
Indeed, if $v_\beta=0$ then the first inequality is trivial. If $0<v_\beta<1$, then (\ref{EqnHdg}) implies 
\[
w_{\beta}=\left\{
\begin{array}{ll}
v_R(a_{\beta}) & (v_{\sigma^{-1}\circ \beta} =0),\\
v_R(b_{\sigma^{-1}\circ \beta}^p a_{\beta}+c_{\sigma^{-1}\circ \beta}^p b_{\beta})& (v_{\sigma^{-1}\circ \beta} >0).
\end{array}
\right.
\]
From this and (\ref{EqnVW}), we obtain $w_\beta\geq v_\beta$. If $v_\beta=1$, then (\ref{EqnHdg}) gives
\[
w_{\beta}=\left\{
\begin{array}{ll}
1 & (v_{\sigma^{-1}\circ \beta} =0),\\
v_R(c_{\sigma^{-1}\circ\beta}^p c_{\beta})=p(1-v_{\sigma^{-1}\circ \beta})& (v_{\sigma^{-1}\circ \beta} >0)
\end{array}
\right.
\]
and the inequality $w_\beta\geq v_\beta$ follows from (\ref{EqnVW}).

Let us consider the second inequality. If $v_{\sigma\circ \beta}=0$, then (\ref{EqnVW}) implies $v_\beta=1$ and the inequality is trivial. If $0<v_{\sigma\circ \beta}<1$, then (\ref{EqnHdg}) implies 
\[
w_{\sigma \circ \beta}=\left\{
\begin{array}{ll}
v_R(a_{\sigma\circ \beta}) & (v_{\beta} =0),\\
v_R(b_\beta^p a_{\sigma\circ \beta}+c_\beta^p b_{\sigma\circ \beta})& (v_{\beta} >0)
\end{array}
\right.
\]
and from (\ref{EqnVW}) we obtain $w_{\sigma\circ \beta}\geq p(1-v_\beta)$ for both cases. If $v_{\sigma\circ \beta}=1$, then we have $v_\beta\geq (p-1)/p>0$ and
\[
w_{\sigma \circ\beta}=v_R(c_{\beta}^p c_{\sigma \circ\beta})=p(1-v_{\beta}).
\]
This concludes the proof of the claim. Now we have
\[
w_{\sigma\circ \beta}+ p w_\beta  \geq p(1-v_\beta)+ p v_\beta=p 
\]
and the assertion (\ref{LemGKVarVW}) follows.
\end{proof}

\begin{lem}\label{cncisoglem}
Let $n$ be a positive integer. Let $\cG$ be a $\bZ_{p^f}\text{-}\ADBT_{n+1}$ over $\okey$ with $\bZ_{p^f}$-alternating isomorphism $i:\cG\simeq \cG^\vee$. Let $\cH$ be a finite flat closed cyclic $\bZ_{p^f}$-subgroup of $\cG[p]$ over $\okey$. Then $p^{-n}\cH/\cH$ is a $\bZ_{p^f}\text{-}\ADBT_n$ over $\okey$, with its $\bZ_{p^f}$-alternating isomorphism of self-duality induced by $i$.
\end{lem}
\begin{proof}
We know that $p^{-n}\cH/\cH$ is a truncated Barsotti-Tate group of level $n$ over $\okey$. Put $\cH'=(\cG[p]/\cH)^\vee$. Since $\cH(\okbar)$ is isotropic, the map $i$ induces an isomorphism $\cH\simeq \cH'$. On the other hand, Cartier duality gives a natural isomorphism $j:p^{-n}\cH'/\cH'\simeq (p^{-n}\cH/\cH)^\vee$ satisfying
\[
\langle\bar{x}, j(\bar{y}) \rangle_{p^{-n}\cH/\cH}=\langle x,y \rangle_{\cG}
\]
for any $x\in p^{-n}\cH(\okbar)$ and $y\in p^{-n}\cH'(\okbar)$, which can be shown as in \cite[\S 4, Proof of Theorem 1.1(b)]{Ha_cansub}. Thus these maps induce a $\bZ_{p^f}$-alternating isomorphism 
\[
p^{-n}\cH/\cH \overset{i}{\to} p^{-n}\cH'/\cH' \overset{j}{\simeq} (p^{-n}\cH/\cH)^\vee.
\]

It remains to prove that the $\cO_{K,n}\otimes \bZ_{p^f}$-module $\omega_{p^{-n}\cH/\cH}$ is free of rank one. Consider the decomposition
\[
\omega_{p^{-n}\cH/\cH}=\bigoplus_{\beta\in \bB_f} \omega_{p^{-n}\cH/\cH,\beta}.
\]
Since we know that the left-hand side is free of rank $f$ as an $\cO_{K,n}$-module, each $\omega_{p^{-n}\cH/\cH,\beta}$ is a free $\cO_{K,n}$-module of rank $f_\beta$ with some non-negative integer $f_\beta$. For $n=1$, we have exact sequences
\[
\xymatrix{
0 \ar[r] & \omega_{\cH,\beta} \ar[r]^-{\times p} & \omega_{p^{-1}\cH,\beta} \ar[r] & \omega_{\cG[p],\beta} \ar[r] & 0,
}
\]
\[
\xymatrix{
0 \ar[r] & \omega_{p^{-1}\cH/\cH,\beta} \ar[r] & \omega_{p^{-1}\cH,\beta} \ar[r] & \omega_{\cH,\beta} \ar[r] & 0
}
\]
and thus $\lg_{\okey}(\omega_{p^{-1}\cH/\cH,\beta})=\lg_{\okey}(\omega_{\cG[p],\beta})$. Since the $\cO_{K,1}$-module $\omega_{\cG[p],\beta}$ is free of rank one, we obtain $f_\beta=1$ and the lemma follows.
\end{proof}

\begin{cor}\label{cncisog}
Let $\cG$ be a $\bZ_{p^f}\text{-}\ADBT_2$ over $\okey$ with $\beta$-Hodge height $w_\beta$. Suppose that the inequality
\[
w_\beta +p w_{\sigma^{-1} \circ \beta} <p
\]
holds for any $\beta\in \bB_f$. Theorem \ref{cansub1} ensures that the canonical subgroup $\cC$ of $\cG[p]$ exists. 
\begin{enumerate}
\item\label{noncanisog} For any finite flat closed cyclic $\bZ_{p^f}$-subgroup $\cH\neq \cC$ of $\cG[p]$ over $\okey$, we have
\[
\Hdg_{\beta}(p^{-1}\cH/\cH)=p^{-1} w_{\sigma \circ \beta} \text{ for any }\beta\in \bB_f.
\]
Moreover, $p^{-1}\cH/\cH$ has the canonical subgroup $\cG[p]/\cH$.

\item\label{canisog} Suppose that the inequality
\[
w_\beta +p w_{\sigma^{-1} \circ \beta} <1
\]
holds for any $\beta\in \bB_f$. Consider the $\bZ_{p^f}\text{-}\ADBT_1$ $p^{-1}\cC/\cC$ over $\okey$. Then we have
\[
\Hdg_{\beta}(p^{-1}\cC/\cC)=p w_{\sigma^{-1} \circ \beta} \text{ for any }\beta\in \bB_f.
\]
Moreover, $\cG[p]/\cC$ is an anti-canonical subgroup of $p^{-1}\cC/\cC$.
\end{enumerate}
\end{cor}
\begin{proof}
For the assertion (\ref{noncanisog}), Lemma \ref{LemGKVar} (\ref{LemGKVarVW}) implies that $\cH$ is an anti-canonical subgroup and
\begin{align*}
&\deg_\beta(\cG[p]/\cH)+p\deg_{\sigma^{-1}\circ\beta}(\cG[p]/\cH)>p,\\
&\Hdg_\beta(\cG[p])=p(1-\deg_{\sigma^{-1}\circ \beta}(\cG[p]/\cH))=p\deg_{\sigma^{-1}\circ \beta}(\cH).
\end{align*}
Hence we have
\[
\deg_\beta((p^{-1}\cH/\cH)/(\cG[p]/\cH))+p\deg_{\sigma^{-1}\circ\beta}((p^{-1}\cH/\cH)/(\cG[p]/\cH))<1.
\]
Lemma \ref{LemGKVar} (\ref{LemGKVarV}) shows that $p^{-1}\cH/\cH$ has the canonical subgroup $\cG[p]/\cH$ and
\[
\Hdg_\beta(p^{-1}\cH/\cH)=\deg_\beta(\cH)=p^{-1}\Hdg_{\sigma\circ\beta}(\cG).
\]

Let us consider the assertion (\ref{canisog}). Since $\deg_\beta(\cG[p]/\cC)=w_\beta$, we have
\[
\deg_\beta((p^{-1}\cC/\cC)/(\cG[p]/\cC))=\deg_\beta(\cC)=1-w_\beta.
\]
The assumption implies
\[
\deg_\beta((p^{-1}\cC/\cC)/(\cG[p]/\cC))+p \deg_{\sigma^{-1}\circ \beta}((p^{-1}\cC/\cC)/(\cG[p]/\cC))>p
\]
and Lemma \ref{LemGKVar} (\ref{LemGKVarW}) yields the assertion. 
\end{proof}

%---------------------------------------------------------------------

%---------------------------------------------------------------------

\subsection{Critical locus}\label{SecQuadRes}

In this subsection, we investigate the behavior of the $U_p$-correspondence at the locus where all the $\beta$-Hodge heights are $p/(p+1)$.

\begin{prop}\label{critisog}
Suppose $f\leq 2$. Let $\cG$ be a $\bZ_{p^f}$-$\ADBT_2$ over $\okey$ with $\Hdg_\beta(\cG)=w_\beta$. Suppose $w_\beta=p/(p+1)$ for any $\beta\in \bB_f$. Then, for any finite flat closed cyclic $\bZ_{p^f}$-subgroup $\cH$ of $\cG[p]$ over $\okey$, we have
\[
\deg_\beta(\cG[p]/\cH)=\frac{p}{p+1},\quad \Hdg_\beta(p^{-1}\cH/\cH)=\frac{1}{p+1}
\]
and $p^{-1}\cH/\cH$ has the canonical subgroup $\cG[p]/\cH$.
\end{prop} 
\begin{proof}
Put $\SGm=\SGm^*(\cG[p])$ and $\SGp=\SGm^*(\cG[p]/\cH)$. We take a basis $\{e_\beta, e'_\beta\}$ of the $\SG_1$-module $\SGm_\beta$ as in \S\ref{SecCansub1} and consider the equation (\ref{EqnTian}). Take $x_\beta,y_\beta\in \SG_1$ such that the element $f_\beta=x_\beta e_\beta+y_\beta e'_\beta$ is a basis of the free $\SG_1$-module $\SGp_\beta$ of rank one. Then there exists an $f$-tuple $(\lambda_\beta)_{\beta\in \bB_f}$ in $\SG_1$ satisfying
\begin{equation}\label{EqnCrit}
\begin{pmatrix}a_{\beta,1} & a_{\beta,2} \\ u^e a_{\beta,3} & u^e a_{\beta,4} 
\end{pmatrix}
\begin{pmatrix}
x^p_{\sigma^{-1}\circ \beta}\\y^p_{\sigma^{-1}\circ \beta}
\end{pmatrix}
=\lambda_\beta
\begin{pmatrix}
x_\beta\\ y_\beta
\end{pmatrix}
\end{equation}
for any $\beta\in \bB_f$. Note that $v_R(a_{\beta,1})=p/(p+1)$. Since the matrix
\[
\begin{pmatrix}a_{\beta,1} & a_{\beta,2} \\ a_{\beta,3} & a_{\beta,4} 
\end{pmatrix}
\] 
is an element of $\mathit{GL}_2(\SG_1)$, we have $v_R(a_{\beta,2})=v_R(a_{\beta,3})=0$. 

We claim that the inequalities $0<w_\beta<1$ for any $\beta\in \bB_f$ imply $v_R(y_{\sigma^{-1}\circ \beta})>0$. Indeed, if $v_R(y_{\sigma^{-1}\circ \beta})=0$, then $v_R(\lambda_\beta)=v_R(x_\beta)=0$ and $v_R(y_\beta)\geq 1$. This is a contradiction if $f=1$. Using (\ref{EqnCrit}) for $\sigma\circ \beta$ yields $v_R(\lambda_{\sigma\circ \beta})\leq w_{\sigma\circ \beta}$ and thus $v_R(y_{\sigma\circ \beta})\geq 1-w_{\sigma\circ \beta}$. This is a contradiction if $f=2$ and the claim follows.

Since $x_\beta e_\beta+y_\beta e'_\beta$ generates the direct summand $\SGp_\beta$ of the $\SG_1$-module $\SGm_\beta$, the claim implies $v_R(x_\beta)=0$. Replacing $f_\beta$ by $x_\beta^{-1}f_\beta$, we may assume $x_\beta=1$ for any $\beta\in \bB_f$. Then $(y_\beta)_{\beta\in \bB_f}$ satisfies the equation
\begin{equation}\label{EqnCritY}
\left\{
\begin{array}{l}
a_{\beta,1}+a_{\beta,2} y^p_{\sigma^{-1}\circ \beta} = \lambda_\beta,\\
u^e(a_{\beta,3}+a_{\beta,4} y^p_{\sigma^{-1}\circ \beta}) = \lambda_\beta y_\beta.
\end{array}
\right.
\end{equation}

Next we claim that every solution $(y_\beta)_{\beta\in \bB_f}$ of this equation satisfies $v_R(\lambda_\beta)=p/(p+1)$ for any $\beta\in \bB_f$.
For $f=1$, we see that $y=y_\beta$ satisfies the equation
\[
 y^{p+1}-u^e a_{\beta,2}^{-1} a_{\beta,4} y^p+ a_{\beta,2}^{-1}a_{\beta,1}y-u^ea_{\beta,2}^{-1}a_{\beta,3}=0.
\]
An inspection of its Newton polygon shows 
$v_R(y)=1/(p+1)$. Then the second equation of (\ref{EqnCritY}) implies $v_R(\lambda_\beta)=p/(p+1)$.

Let us consider the case $f=2$. Take any $\beta\in \bB_f$. Put
\[
\begin{pmatrix}A & B\\ C& D
\end{pmatrix}
=\begin{pmatrix}a_{\sigma\circ\beta,1} & a_{\sigma\circ\beta,2} \\ u^e a_{\sigma\circ\beta,3} & u^e a_{\sigma\circ\beta,4} 
\end{pmatrix}
\begin{pmatrix}a_{\beta,1}^p & a_{\beta,2}^p \\ u^{p e} a_{\beta,3}^p & u^{p e} a_{\beta,4}^p 
\end{pmatrix}.
\]
Note that
\[
v_R(A)\geq p ,\quad v_R(B)=\frac{p}{p+1},\quad v_R(C)=1+\frac{p^2}{p+1},\quad v_R(D)=1.
\]
We have
\[
\begin{pmatrix}A & B\\ C& D
\end{pmatrix}
\begin{pmatrix}
1 \\ y_{\sigma^{-1}\circ \beta}^{p^2}
\end{pmatrix}
= \lambda_{\sigma\circ \beta}\lambda_\beta^p
\begin{pmatrix}
1 \\ y_{\sigma^{-1}\circ \beta}
\end{pmatrix}
\]
and thus $y=y_{\sigma^{-1}\circ \beta}$ satisfies the equation
\[
y^{p^2+1}-B^{-1}D y^{p^2} + B^{-1}A y - B^{-1}C=0,
\]
where the coefficients are all integral. An inspection of its Newton polygon shows 
$v_R(y_{\sigma^{-1}\circ\beta})=1/(p+1)$. 
Then the second equation of (\ref{EqnCritY}) yields $v_R(\lambda_{\sigma^{-1}\circ \beta})=p/(p+1)$. Since $\beta\in \bB_f$ is arbitrary, we obtain
\[
v_R(y_\beta)=\frac{1}{p+1},\quad v_R(\lambda_\beta)=\frac{p}{p+1}
\]
for any $\beta\in \bB_f$.

Now the claim shows
\begin{align*}
\deg_\beta((p^{-1}\cH/\cH)/(\cG[p]/\cH))&=\deg_\beta(\cH)=1-v_R(\lambda_\beta)=1/(p+1),\\
\deg_\beta(\cG[p]/\cH)&=\deg_\beta(p^{-1}\cH/\cH)-1/(p+1)=p/(p+1)
\end{align*}
for any $\beta\in \bB_f$. Then Lemma \ref{LemGKVar} (\ref{LemGKVarV}) implies that
\[
\Hdg_\beta(p^{-1}\cH/\cH)=1/(p+1)
\]
and that $\cG[p]/\cH$ is the canonical subgroup of $p^{-1}\cH/\cH$.
\end{proof}

\begin{rmk}\label{counterexamplef3}
A naive generalization of Proposition \ref{critisog} has a counterexample for $f=3$, if $p\neq 2$. Suppose $k=\kbar$ and $p+1\mid e$.
Replacing the uniformizer $\pi$ by a scalar multiple, we may assume that $c_0=p^{-1}E(0)$ satisfies $c_0\equiv 1\bmod p$.
Let $r$ be a positive integer. Fix $\beta\in \bB_3$ and consider the following elements of $M_2(\SG)$.
\[
\begin{array}{c}
\hat{A}_\beta= \begin{pmatrix} u^{\frac{p e}{p+1}} & 1\\ c_0^{-1}E(u) & c_0^{-1}E(u) \end{pmatrix},\quad \hat{A}_{\sigma\circ\beta}= \begin{pmatrix} u^{\frac{p e}{p+1}} & -1\\ c_0^{-1}E(u) & c_0^{-1}E(u) \end{pmatrix},\\
 \hat{A}_{\sigma^2\circ\beta}= \begin{pmatrix} u^{\frac{p e}{p+1}} & 1\\ c_0^{-1}E(u) & c_0^{-1}u^{ r}E(u) \end{pmatrix}
\end{array}
\]
We define the $\bZ_{p^3}$-Breuil-Kisin module $\hat{\SGm}=\bigoplus_{\beta\in \bB_3}\hat{\SGm}_\beta$ by
\[
\hat{\SGm}_\beta=\SG \hat{e}_\beta \oplus \SG \hat{e}'_\beta,\quad \varphi_{\beta}(\hat{e}_{\sigma^{-1}\circ \beta}, \hat{e}'_{\sigma^{-1}\circ \beta})=(\hat{e}_\beta, \hat{e}'_\beta)\hat{A}_\beta.
\]
Take $\hat{\alpha}_\beta \in \SG^\times$ for each $\beta\in \bB_3$ satisfying 
\[
\varphi(\hat{\alpha}_{\sigma^{-1}\circ \beta})=c_0E(u)^{-1}\det(\hat{A}_\beta)\hat{\alpha}_\beta\text{ for any }\beta\in \bB_3. 
\]
Then the map
\[
(\hat{e}_\beta,\hat{e}'_\beta)\mapsto (\hat{e}^\vee_\beta, (\hat{e}'_\beta)^\vee) \hat{\alpha}_\beta\begin{pmatrix}0 & 1 \\ -1 & 0\end{pmatrix}
\]
gives a skew-symmetric isomorphism $\hat{\SGm}\to \hat{\SGm}^\vee$. Since $\hat{\SGm}$ corresponds to a Barsotti-Tate group $\Gamma$ over $\okey$, we see that $\hat{\SGm}/p^2\hat{\SGm}$ corresponds to a $\bZ_{p^3}$-$\ADBT_2$ $\cG=\Gamma[p^2]$ over $\okey$. Let $\cH$ be any cyclic $\bZ_{p^3}$-subgroup of $\cG[p]$ over $\okey$. We write the images of $\hat{e}_\beta, \hat{e}'_\beta$ in $\hat{\SGm}/p\hat{\SGm}_\beta$ as $e_\beta, e'_\beta$ and a basis of $\SGm^*(\cG/\cH)_\beta$ as $x_\beta e_\beta+ y_\beta e'_\beta$. 

Suppose $v_R(y_\beta)>0$ for any $\beta\in \bB_3$. Then we may assume $x_\beta=1$, and we see that $y=y_{\sigma^{2}\circ \beta}$ is a root of the equation
\[
y^{p^3+1}-B^{-1}Dy^{p^3}+B^{-1}A y -B^{-1}C=0,
\]
where we put
\begin{align*}
A&=u^{\frac{e(p^3+p^2+p)}{p+1}}+u^{e(p^2+p)}, \quad B=2u^{ep}-u^{\frac{e(p^3+p^2+p)}{p+1}}+u^{e(p^2+p)},\\
C&=u^{\frac{e(p^3+p^2+2p+1)}{p+1}+r}+u^{e(p^2+p+1)+r},\\
D&=u^{\frac{e(p^2+p+1)}{p+1}}-u^{e(p^2+1)}+u^{e(p+1)+r}+u^{e(p^2+p+1)+r}.
\end{align*}
An inspection of its Newton polygon and derivation shows that this equation has exactly $p^3$ roots satisfying $v_R(y)=1/(p+1)$ and one root satisfying $v_R(y)=1+e^{-1}r$. 

The latter case does not occur, since it contradicts the second equation of (\ref{EqnCritY}). In the former case, put $y=u^{e/(p+1)}\eta$. Then $\eta$ satisfies a monic polynomial of degree $p^3+1$ whose reduction modulo $u$ is $X(X^{p^3}-2^{-1}X^{p^3-1}+2^{-1})$. Hensel's lemma and the assumption on $k$ imply $y\in \SG_1$. Thus $\cG[p]$ has exactly $p^3$ cyclic $\bZ_{p^3}$-subgroups over $\okey$ such that, for any $\beta\in \bB_3$, we have $v_R(y_\beta)>0$. 

By the assumption $k=\kbar$, there exist exactly $p^3-1$ characters $G_K\to \bF_{p^3}^\times$. Hence, among these $p^3$ cyclic $\bZ_{p^3}$-subgroups, two define the same character on the generic fiber. This means that $G_K$ acts on $\cG[p](\okbar)$ via this character. In particular, any $\bF_{p^3}$-subspace of $\cG[p](\okbar)$ is $G_K$-stable. Taking the scheme-theoretic closure, we see that $\cG[p]$ has one cyclic $\bZ_{p^3}$-subgroup $\cH$ over $\okey$ satisfying $v_R(y_{\sigma^{-1}\circ\beta_0})=0$ for some $\beta_0\in \bB_3$. 

For this $\cH$, the equation (\ref{EqnCrit}) gives $v_R(\lambda_{\beta_0})=v_R(x_{\beta_0})=0$ and $v_R(y_{\beta_0})\geq 1$. This in turn gives $v_R(\lambda_{\sigma\circ\beta_0})\leq p/(p+1)$ and $v_R(y_{\sigma\circ \beta_0})\geq 1/(p+1)$. Since $x_{\sigma\circ \beta_0}e_{\sigma\circ \beta_0}+y_{\sigma\circ \beta_0} e'_{\sigma\circ \beta_0}$ generates a direct summand, we have $v_R(x_{\sigma\circ \beta_0})=0$ and this implies $v_R(\lambda_{\sigma\circ \beta_0})=p/(p+1)$. 
Thus we obtain 
\[
\deg_{\sigma\circ \beta_0}(\cH)+ p\deg_{\beta_0}(\cH) =p+\frac{1}{p+1}
\]
and $\cG[p]/\cH$ is not the canonical subgroup of $p^{-1}\cH/\cH$.
\end{rmk}

%---------------------------------------------------------------------

%---------------------------------------------------------------------

\subsection{Canonical subgroup of higher level}\label{SecCansubh}

We derive from Theorem \ref{cansub1} the existence of the canonical subgroup  of level $n$ for a $\bZ_{p^f}$-$\ADBT_n$, by following an argument of Fargues-Tian \cite[\S 7]{Fa} as in \cite[\S 4]{Ha_cansub}. A similar result was also obtained by Goren-Kassaei \cite[Proposition 5.4.5]{GK} except the compatibility with the Hodge-Tate kernel and lower ramification subgroups. This compatibility shown here will be used to enlarge the locus where the sheaf of overconvergent Hilbert modular forms is defined from that of \cite{AIP2}.

\begin{thm}\label{cansubh}
Let $\cG$ be a $\bZ_{p^f}$-$\ADBT_n$ over $\okey$ with $\beta$-Hodge height $w_\beta$. Put $w=\max\{w_\beta\mid \beta\in \bB_f\}$. 
Suppose that we have
\[
w_\beta +p w_{\sigma^{-1} \circ \beta} <p^{2-n}
\]
for any $\beta\in \bB_f$.
Then there exists a finite flat closed $\bZ_{p^f}$-subgroup $\cC_n$ of $\cG$ of rank $p^{nf}$ over $\okey$ satisfying 
\[
\deg_\beta(\cG/\cC_n)=\sum_{l=0}^{n-1} p^{l} w_{\sigma^{-l}\circ \beta}.
\]
We refer to $\cC_n$ as the canonical subgroup of level $n$ of $\cG$. It has the following properties:
\begin{enumerate}
\item\label{cansubh_isom} Let $\cG'$ be a $\bZ_{p^f}$-$\ADBT_n$ over $\okey$ satisfying the same condition on the $\beta$-Hodge heights as above and $\cC'_n$ the canonical subgroup of level $n$ of $\cG'$. Then any isomorphism of $\bZ_{p^f}$-groups $j:\cG\to \cG'$ over $\okey$ induces an isomorphism $\cC_n\simeq \cC'_n$.
\item\label{cansubh_BC} $\cC_n$ is compatible with base extension of complete discrete valuation rings with perfect residue fields.
\item\label{cansubh_dual} $\cC_n$ is compatible with Cartier duality. Namely, $(\cG/\cC_n)^\vee$ is the canonical subgroup of level $n$ of $\cG^\vee$.
\item\label{cansubh_Frob} The kernel of the $n$-th iterated Frobenius map of $\cG\times\sS_{1-p^{n-1}w}$ coincides with $\cC_n\times \sS_{1-p^{n-1}w}$.
\item\label{cansubh_free}
The $\bZ_{p^f}/p^n\bZ_{p^f}$-module $\cC_n(\okbar)$ is free of rank one. 
\item\label{cansubh_cl}
The scheme-theoretic closure of $\cC_n(\okbar)[p^i]$ in $\cC_n$ is the canonical subgroup $\cC_i$ of level $i$ of $\cG[p^i]$ for any $0\leq i\leq n-1$.

\item\label{cansubh_HT} If $w<(p-1)/p^n$, then $\cC_n(\okbar)$ coincides with $\Ker(\HT_i)$ for any rational number $i$ satisfying 
\[
n-1+\frac{w}{p-1}< i\leq n-\frac{w(p^n-1)}{p-1}.
\] 
\item\label{cansubh_ram} If $w<(p-1)/p^n$, then $\cC_n=\cG_{i}$ for any rational number $i$ satisfying
\[
i'_n=\frac{1}{p^n(p-1)}\leq i \leq i_n=\frac{1}{p^{n-1}(p-1)}-\frac{w}{p-1}.
\] 
\end{enumerate}
\end{thm}
\begin{proof}
We proceed by induction on $n$. The case $n=1$ is Theorem \ref{cansub1}. Suppose that $n\geq 2$ and the assertions hold for $n-1$. Let $\cG$ be a $\bZ_{p^f}$-$\ADBT_n$ satisfying the assumption. Then we have the canonical subgroup $\cC_1$ of the $\bZ_{p^f}$-$\ADBT_1$ $\cG[p]$ and Lemma \ref{cncisoglem} implies that 
$p^{1-n}\cC_1/\cC_1$ is also a $\bZ_{p^f}$-$\ADBT_{n-1}$. By Corollary \ref{cncisog} (\ref{canisog}), we have
\[
\Hdg_\beta(p^{1-n}\cC_1/\cC_1)=p w_{\sigma^{-1}\circ \beta}
\]
and by the induction hypothesis, $p^{1-n}\cC_1/\cC_1$ has the canonical subgroup of level $n-1$, which we write as $\cC_n/\cC_1$ with some $\bZ_{p^f}$-subgroup $\cC_n$ of $\cG$. Then we have
\begin{align*}
\deg_\beta(\cG/\cC_n)&= \deg_\beta(\cG/p^{1-n}\cC_1)+\deg_\beta((p^{1-n}\cC_1/\cC_1)/(\cC_n/\cC_1))\\
&=\deg_\beta(\cG[p]/\cC_1)+\sum_{l=0}^{n-2} p^l (p w_{\sigma^{-l-1}\circ \beta})=\sum_{l=0}^{n-1} p^{l} w_{\sigma^{-l}\circ \beta}.
\end{align*}
The assertions (\ref{cansubh_isom}) and (\ref{cansubh_BC}) follow from the construction and the induction hypothesis. The assertions (\ref{cansubh_dual}) and (\ref{cansubh_Frob}) can be shown exactly in the same way as \cite[Theorem 1.1 (b) and (1)]{Ha_cansub}, using the assertion (\ref{cansubh_isom}). 

Let us show the assertion (\ref{cansubh_free}). By an induction, we can show $\cC_{n-1}\subseteq \cC_n$. By the induction hypothesis, it suffices to show $\cC_n(\okbar)\cap \cG[p](\okbar)=\cC_{1}(\okbar)$ for any $n\geq 2$. From the assertion (\ref{cansubh_cl}) for $p^{1-n}\cC_1/\cC_1$, we see that $(\cC_n/\cC_1)(\okbar)[p]$ is the generic fiber of the canonical subgroup of $p^{-1}\cC_1/\cC_1$. On the other hand, Corollary \ref{cncisog} (\ref{canisog}) implies that $\cG[p]/\cC_1$ is not the canonical subgroup of $p^{-1}\cC_1/\cC_1$. Then we have $(\cC_n/\cC_1)(\okbar)\cap (\cG[p]/\cC_1)(\okbar)=0$ and thus $\cC_n(\okbar)\cap \cG[p](\okbar)\subseteq \cC_1(\okbar)$, from which the assertion (\ref{cansubh_free}) follows. 
The assertion (\ref{cansubh_cl}) follows from $\cC_{n-1}\subseteq \cC_n$ and the assertion (\ref{cansubh_free}).

Next we show the assertion (\ref{cansubh_HT}). Let $i$ be as in the assertion. Put $\epsilon=n-i$. Since we have 
\[
w/(p-1)<1-\epsilon\leq 1-w,
\]
by using Theorem \ref{cansub1} (\ref{cansub1_HT}) we can show $\sharp \Ker(\HT_i)\leq p^{nf}$ as in the proof of \cite[Proposition 13]{Fa}. On the other hand, since $\deg_\beta(\cC_1^\vee)=w_\beta$, the $\cO_{\Kbar}$-module $\omega_{\cC_1^\vee}\otimes \cO_{\Kbar}$ is killed by $m_{\Kbar}^{\geqslant w}$. Take any element $x\in \cC_n(\okbar)$ and denote its image in $(\cG/\cC_1)(\okbar)$ by $\bar{x}$. By the induction hypothesis, we have $\HT_j(\bar{x})=0$ for any $j$ satisfying
\[
n-2+pw/(p-1)< j\leq n-1-w(p^n-p)/(p-1).
\]
Thus we obtain 
\[
m_{\Kbar}^{\geqslant n-1-j}\HT(\bar{x})=0,\quad m_{\Kbar}^{\geqslant n-1-j+w}\HT(x)=0
\]
and $\HT_{1-w+j}(x)=0$, which yields $\cC_n(\okbar)\subseteq \Ker(\HT_i)$. Then the assertion (\ref{cansubh_HT}) follows from $\sharp\cC_n(\okbar)=p^{nf}$.

Finally, we show the assertion (\ref{cansubh_ram}) following the proof of \cite[Theorem 1.2]{Ha_lowram}. Using Lemma \ref{canlowram1} and Theorem \ref{cansub1} (\ref{cansub1_Frob}), the same argument as in the proof of \cite[Lemma 5.4]{Ha_lowram} shows $\cG_{i'_n}\subseteq \cC_n$. For the reverse inclusion, we need the following variant of \cite[Proposition 5.5]{Ha_lowram}.

\begin{lem}\label{Hilblowram}
The image of the map $\cG_{i_n}(\okbar)\overset{\times p}{\to} \cG[p^{n-1}]_{pi_n}(\okbar)$ contains $\cG[p^{n-1}]_{i_{n-1}}(\okbar)$.
\end{lem}
\begin{proof}
Note that the map in the lemma is well-defined by \cite[Lemma 5.3]{Ha_lowram}. Put $\SGm=\SGm^*(\cG[p])$. Consider the basis $\{\delta_\beta,e'_\beta\}$ of the $\SG_1$-module $\SGm$ as in the proof of Theorem \ref{cansub1}. Write as
\[
\varphi(\delta_{\sigma^{-1}\circ\beta}, e'_{\sigma^{-1}\circ\beta})=(\delta_\beta, e'_\beta)\begin{pmatrix}\lambda_\beta& \mu_\beta \\ 0 & \nu_\beta \end{pmatrix}.
\]
We have $v_R(\lambda_\beta)=w_\beta$ and $v_R(\nu_\beta)=1-w_\beta$. Then, in the same way as in the proof of \cite[Proposition 5.5]{Ha_lowram}, we reduce ourselves to showing that for any $\xi_\beta, \eta_\beta\in m_R^{\geqslant i_{n-1}}$, there exist $\zeta_\beta, \omega_\beta\in m_R^{\geqslant i_{n}}$ satisfying
\[
(\xi_{\beta}, \eta_{\beta})+ (\zeta_{\sigma^{-1}\circ\beta}^p,\omega_{\sigma^{-1}\circ\beta}^p)=(\zeta_\beta,\omega_\beta)\begin{pmatrix}\lambda_\beta& \mu_\beta \\ 0 & \nu_\beta \end{pmatrix}
\]
for any $\beta\in \bB_f$. We can show by recursion that the equation on $\zeta_\beta$'s has a solution satisfying $v_R(\zeta_\beta)\geq pi_n$ for any $\beta\in \bB_f$. Fixing such $\zeta_\beta$'s, we obtain the system of equations on $\omega_\beta$'s 
\[
\omega_{\sigma^{-1}\circ\beta}^p- \nu_\beta \omega_\beta - \mu_\beta \zeta_\beta +\eta_{\beta}=0.
\]
Take any $a\in R$ satisfying $v_R(a)=i_n$ and put $\omega_\beta=a \alpha_\beta$. Then $(\alpha_\beta)_{\beta\in \bB_f}$ is a solution of the system of equations
\[
\alpha_{\sigma^{-1}\circ\beta}^p- \frac{\nu_\beta}{a^{p-1}} \alpha_\beta - \frac{\mu_\beta \zeta_\beta}{a^p} +\frac{\eta_{\beta}}{a^p}=0,
\]
where all the coefficients are contained in $R$. This system defines a finite $R$-algebra which is free of rank $p^f$. Since $\Frac(R)$ is algebraically closed and $R$ is normal, we can find a solution $(\alpha_\beta)_{\beta\in \bB_f}$ in $R$ and the lemma follows.
\end{proof}

By the induction hypothesis, we have $\cG[p^{n-1}]_{i_{n-1}}=\cC_{n-1}$. By Lemma \ref{canlowram1}, we also have $\cG[p]_{i_n}=\cC_1$. Then Lemma \ref{Hilblowram} implies $\sharp \cG_{i_n}(\okbar)\geq \sharp \cC_n(\okbar)$. Now the assertion (\ref{cansubh_ram}) follows from the inclusions $\cG_{i_n} \subseteq \cG_{i'_n}\subseteq \cC_n$. This concludes the proof of Theorem \ref{cansubh}.
\end{proof}

\begin{cor}\label{noncanh}
Let $n$ be a positive integer. Let $\cG$ be a $\bZ_{p^f}$-$\ADBT_{n+1}$ over $\okey$ with $\beta$-Hodge height $w_\beta$ satisfying 
\[
w_\beta+p w_{\sigma^{-1}\circ \beta} < p^{2-n}
\]
for any $\beta\in \bB_f$. 
Let $\cC_{n-1}$ and $\cC_1$ be the canonical subgroups of level $n-1$ and level one of $\cG[p^{n-1}]$ and $\cG[p]$, respectively. Let $\cH\neq \cC_1$ be a finite flat closed cyclic $\bZ_{p^f}$-subgroup of $\cG[p]$ over $\okey$. Then the $\bZ_{p^f}$-$\ADBT_n$ $p^{-n}\cH/\cH$ has the canonical subgroup $p^{-1}\cC_{n-1}/\cH$. Moreover, 
the natural map $\cC_n \to p^{-1}\cC_{n-1}/\cH$ is an isomorphism over $K$.
\end{cor}
\begin{proof}
By Corollary \ref{cncisog} (\ref{noncanisog}), the $\bZ_{p^f}$-$\ADBT_i$ $p^{-i}\cH/\cH$ has the canonical subgroup of level $i$ for any positive integer $i\leq n$, which we denote by $\bar{\cC}_i$. Moreover, we have $\bar{\cC}_1=\cG[p]/\cH$. By the construction of the canonical subgroup in Theorem \ref{cansubh}, the quotient $\bar{\cC}_n/\bar{\cC}_1$ is equal to the canonical subgroup of level $n-1$ of the $\bZ_{p^f}$-$\ADBT_{n-1}$ $p^{1-n}\bar{\cC}_1/\bar{\cC}_1$. We have the map
\[
p^{1-n}\bar{\cC}_1/\bar{\cC}_1=p^{1-n}(\cG[p]/\cH)/(\cG[p]/\cH)=(\cG[p^n]/\cH)/(\cG[p]/\cH)\overset{\times p}{\to} \cG[p^{n-1}],
\]
where the last arrow is an isomorphism. By Theorem \ref{cansubh} (\ref{cansubh_isom}), we obtain $\bar{\cC}_n=p^{-1}\cC_{n-1}/\cH$. 
Moreover, Theorem \ref{cansubh} (\ref{cansubh_cl}) implies $\cC_n(\okbar)\cap \cH(\okbar)=0$ and the map $\cC_n \to p^{-1}\cC_{n-1}/\cH$ is an injection over $K$. Since the both sides have the same rank over $\okey$, the last assertion follows.
\end{proof}

Finally, we show the following generalization of \cite[Proposition 3.2.1]{AIP} to our setting.

\begin{prop}\label{cokn}
Let $\cG$ be a $\bZ_{p^f}$-$\ADBT_n$ over $\okey$ with $\beta$-Hodge height $w_\beta$. Put $w=\max\{w_\beta\mid \beta\in \bB_f\}$. Suppose $w<(p-1)/p^n$. Let $\cC_n$ be the canonical subgroup of $\cG$ of level $n$, which exists by Theorem \ref{cansubh}. 
\begin{enumerate}
\item For any $i\in e^{-1}\bZ_{\geq 0}$ satisfying $i\leq n-w(p^n-1)/(p-1)$, the natural map 
\[
\omega_{\cG}\otimes_{\okey} \cO_{K, i}\to \omega_{\cC_n}\otimes_{\okey} \cO_{K, i}
\]
is an isomorphism.
\item The cokernel of the linearization of the Hodge-Tate map
\[
\HT\otimes 1: \cC_n^\vee(\okbar)\otimes \okbar \to \omega_{\cC_n}\otimes_{\okey} \okbar
\]
is killed by $m_{\Kbar}^{\geqslant w/(p-1)}$.
\end{enumerate}
\end{prop}
\begin{proof}
Put $b=n-w(p^n-1)/(p-1)$. For the first assertion, consider the exact sequence
\[
\xymatrix{
0 \ar[r] & \omega_{\cG/\cC_n} \ar[r] & \omega_{\cG} \ar[r] & \omega_{\cC_n} \ar[r] & 0
}
\]
and the decompositions
\[
\omega_{\cG/\cC_n}=\bigoplus_{\beta\in\bB_f}\omega_{\cG/\cC_n,\beta},\quad \omega_{\cG}=\bigoplus_{\beta\in\bB_f}\omega_{\cG,\beta}.
\]
Note that $\omega_{\cG,\beta}\simeq \cO_{K,n}$. Theorem \ref{cansubh} implies
\[
\deg_{\beta}(\cG/\cC_n)=\sum_{l=0}^{n-1}p^l w_{\sigma^{-l}\circ \beta}\leq \frac{w(p^n-1)}{p-1}\leq n-i.
\]
Thus the image of the natural map $\omega_{\cG/\cC_n, \beta}\to \omega_{\cG,\beta}$ is contained in $m_K^{\geqslant i}\omega_{\cG,\beta}$ for any $\beta\in \bB_f$ and the first assertion follows.

For the second assertion, consider the commutative diagram
\[
\xymatrix{
\cG^\vee(\okbar) \ar[d]_{p^{n-1}} \ar[r]^{\HT_{\cG^\vee}}& \omega_{\cG}\otimes_{\okey}\okbar \ar[r]\ar[d]& \omega_{\cG}\otimes_{\okey}\cO_{\Kbar,1}\ar[d]\\
\cG^\vee[p](\okbar) \ar[r]^{\HT_{\cG^\vee[p]}} & \omega_{\cG[p]}\otimes_{\okey}\okbar \ar@{=}[r]& \omega_{\cG[p]}\otimes_{\okey}\cO_{\Kbar,1},
}
\]
where the horizontal composites are the first Hodge-Tate maps and the left vertical arrow is surjective. Since the right vertical arrow is an isomorphism, the map $\HT_{\cG^\vee,1}$ factors through $\cG^\vee[p](\okbar)$ and we obtain a natural isomorphism of $\okbar$-modules
\[
\Coker(\HT_{\cG^\vee,1}\otimes 1) \simeq \Coker(\HT_{\cG^\vee[p],1}\otimes 1).
\]
By Lemma \ref{cok1}, they are killed by $m_{\Kbar}^{\geqslant w/(p-1)}$ and thus
\[
m_{\Kbar}^{\geqslant w/(p-1)} (\omega_{\cG}\otimes_{\okey}\okbar) \subseteq \Img(\HT_{\cG^\vee}\otimes 1) +p(\omega_{\cG}\otimes_{\okey}\okbar).
\]
Since $w<1$, Lemma \ref{NAK} implies that the $\okbar$-module $\Coker(\HT_{\cG^\vee}\otimes 1)$ is killed by $m_{\Kbar}^{\geqslant w/(p-1)}$.

On the other hand, we have a commutative diagram
\[
\xymatrix{
\cG^\vee(\okbar) \ar[d] \ar[r]^{\HT_{\cG^\vee}}& \omega_{\cG}\otimes_{\okey}\okbar \ar[r]\ar[d]& \omega_{\cG}\otimes_{\okey}\cO_{\Kbar,b}\ar[d]\\
\cC_n^\vee(\okbar) \ar[r]^{\HT_{\cC_n^\vee}} & \omega_{\cC_n}\otimes_{\okey}\okbar \ar[r]& \omega_{\cC_n}\otimes_{\okey}\cO_{\Kbar,b},
}
\]
where the left vertical arrow is surjective.
By a base change argument using Theorem \ref{cansubh} (\ref{cansubh_BC}), the first assertion implies that the right vertical arrow is an isomorphism. Thus we have a surjection of $\okbar$-modules
\[
\Coker(\HT_{\cG^\vee}\otimes 1) \twoheadrightarrow \Coker(\HT_{\cG^\vee,b}\otimes 1) \simeq \Coker(\HT_{\cC_n^\vee,b}\otimes 1)
\]
and $\Coker(\HT_{\cC_n^\vee,b}\otimes 1)$ is also killed by $m_{\Kbar}^{\geqslant w/(p-1)}$. This is equivalent to the inclusion
\[
m_{\Kbar}^{\geqslant w/(p-1)} (\omega_{\cC_n}\otimes_{\okey} \okbar) \subseteq \Img(\HT_{\cC_n^\vee}\otimes 1) + m_{\Kbar}^{\geqslant b} (\omega_{\cC_n}\otimes_{\okey} \okbar). 
\]
Since $w<(p-1)/p^n$, we have $b>w/(p-1)$ and the proposition follows from Lemma \ref{NAK}. 
\end{proof}

%---------------------------------------------------------------------

%---------------------------------------------------------------------

\section{Hilbert eigenvariety}\label{SecHilbEV}

%---------------------------------------------------------------------

%---------------------------------------------------------------------

\subsection{Hilbert modular varieties}\label{SecHilbMV}

Let $p$ be a rational prime. Let $F$ be a totally real number field of degree $g$ which is unramified over $p$. We denote its ring of integers by $\fro=\cO_F$ and its different by $\cD_F$. For any integer $N$, we put
\[
U_N=\{\epsilon\in \cO_F^\times\mid \epsilon\equiv 1 \bmod N\}. 
\]
We fix once and for all a representative $[\mathrm{Cl}^+(F)]^{(p)}=\{\frc_1=\fro,\frc_2,\ldots, \frc_{h^+}\}$ of the strict class group $\mathrm{Cl}^+(F)$ such that every $\frc_i$ is prime to $p$.

For any prime ideal $\frp \mid p$ of $\cO_F$, let $f_{\frp}$ be the residue degree of $\frp$. Fix a finite extension $K/\bQ_p$ in $\bar{\bQ}_p$ such that $F\otimes K$ splits completely. Let $k$ be the residue field of $K$ and we follow the notation in \S\ref{SecBKMod}. We denote by $\bB_F$ the set of embeddings $F\to K$ and by $\bB_{\frp}$ the subset consisting of embeddings which factor through the completion $F_{\frp}$. Then we can identify $\bB_{\frp}$ with $\bB_{f_{\frp}}$. The set $\bB_F$ is decomposed as
\[
\bB_F=\coprod_{\frp\mid p} \bB_{\frp}.
\]

For any subset $X$ of $F$, we denote by $X^+$ the subset of totally positive elements of $X$. Put $F_\bR=F\otimes \bR$ and $F^*_\bR=\Hom_{\bQ}(F,\bR)$. We denote by $F^{*,+}_\bR$ the subset of $F^*_\bR$ consisting of linear forms which maps the subset $F^{\times,+}$ to $\bR_{>0}$. The group $U_N$ acts on $F$ and $F^{*,+}_\bR$ through $\epsilon\mapsto \epsilon^2$. 

Let $\frc$ be any non-zero fractional ideal of $F$. For any fractional ideals $\fra$, $\frb$ of $F$ satisfying $\fra\frb^{-1}=\frc$, we denote by $\Dec(\fra,\frb)$ the set of rational polyhedral cone decompositions $\sC=\{\sigma\}_{\sigma\in \sC}$ of $F^{*,+}_\bR$ which is projective and smooth with respect to the lattice $\Hom(\fra\frb,\bZ)$ such that the elements of $\sC$ are permuted by the action of $U_N$, the set $\sC/U_N$ is finite and for any $\epsilon\in U_N$ and $\sigma\in \sC$, $\epsilon(\sigma)\cap \sigma\neq \emptyset$ implies $\epsilon=1$, as in \cite[\S 4.1.4]{Hida_PAF}. Here we adopt the convention that $\sigma$ is an open cone. Note that any two elements of $\Dec(\fra,\frb)$ have a common refinement which belongs to $\Dec(\fra,\frb)$. For any such pair $(\fra,\frb)$, we fix once and for all a rational polyhedral cone decomposition $\sC(\fra,\frb)\in \Dec(\fra,\frb)$ and put $\sD(\frc)=\{\sC(\fra,\frb)\mid \fra\frb^{-1}=\frc\}$.

%---------------------------------------------------------------------

%---------------------------------------------------------------------

\subsubsection{Hilbert-Blumenthal abelian varieties}\label{SecHBAV}

Let $N\geq 4$ be an integer with $p\nmid N$ and $\frc$ a non-zero fractional ideal of $F$. Let $S$ be a scheme over $\okey$. A Hilbert-Blumenthal abelian variety over $S$, which we abbreviate as HBAV, is a quadruple $(A,\iota,\lambda,\psi)$ such that
\begin{itemize}
\item $A$ is an abelian scheme over $S$ of relative dimension $g$.
\item $\iota: \cO_F\to \End_S(A)$ is a ring homomorphism.
\item $\lambda$ is a $\frc$-polarization. Namely, $\lambda:A\otimes_{\oef} \frc\simeq A^\vee$ is an isomorphism of abelian schemes to the dual abelian scheme $A^\vee$ compatible with $\oef$-action such that the map 
\[
\Hom_{\oef}(A,A^\vee)\simeq \Hom_{\oef}(A,A\otimes_{\oef} \frc),\quad f\mapsto \lambda^{-1}\circ f 
\]
induces an isomorphism of $\cO_F$-modules with notion of positivity $(\cP_A,\cP_A^+)\simeq (\frc,\frc^+)$. Here $\cP_A$ denotes the $\cO_F$-module of symmetric $\cO_F$-homomorphisms from $A$ to $A^\vee$, $\cP_A^+$ is the subset of $\cO_F$-linear polarizations and any element $\gamma\in \frc$ is identified with the element $(x\mapsto x\otimes \gamma)$ of $\Hom_{\oef}(A,A\otimes_{\oef} \frc)$.
\item $\psi: \cD_F^{-1}\otimes \mu_N\to A$ is an $\cO_F$-linear closed immersion of group schemes, which we call a $\Gamma_{00}(N)$-structure.
\end{itemize}
Note that for such data, the $\cO_F\otimes \cO_S$-module $\Lie(A)$ is locally free of rank one \cite[Corollaire 2.9]{DP}. 

Let $(A,\iota,\lambda,\psi)$ be a HBAV over $S$ with $\frc$-polarization $\lambda$ and $\Gamma_{00}(N)$-structure $\psi$. Let $\fra$ be an ideal of $\cO_F$. Let $\cH$ be a finite locally free closed subgroup scheme of $A$ over $S$ which is stable under the $\cO_F$-action such that $\cH$ is isomorphic, etale locally on $S$, to the constant group scheme $\underline{\cO_F/\fra}$ and that $\Img(\psi)\cap \cH=0$. Then we can define on $A/\cH$ a natural structure of a HBAV $(A/\cH,\bar{\iota}, \bar{\lambda},\bar{\psi})$ with $\frc\fra$-polarization $\bar{\lambda}$ \cite[\S 1.9]{KL}.

For any HBAV $(A,\iota,\lambda,\psi)$ over $S$, the group scheme $A[p^n]$ is decomposed as 
\[
A[p^n]=\bigoplus_{\frp\mid p} A[p^n]_\frp=\bigoplus_{\frp\mid p} A[\frp^n]
\]
according with the decomposition
\[
\cO_F\otimes \bZ_p=\prod_{\frp\mid p} \cO_{F_\frp}.
\]
If $S=\Spec(\oel)$ with some extension $L/K$ of complete valuation fields, then each $A[\frp^n]$ is a truncated Barsotti-Tate group of level $n$, height $2f_{\frp}$ and dimension $f_{\frp}$. Moreover, for any $\frp\mid p$, the $\cO_{F_\frp}$-module $\frc\otimes_{\cO_F}\cO_{F_\frp}$ and the $\cO_F/p^n\cO_F$-module $\frc/p^n\frc$ are free of rank one. This implies that any element of $\frc$ which generates the $\cO_F$-module $\frc/p^n\frc$ and the given $\frc$-polarization on $A$ define isomorphisms 
\[
i_n:A[p^n] \to A^\vee[p^n] \simeq A[p^n]^\vee,\quad i_{n,\frp}: A[\frp^n]\to A[\frp^n]^\vee
\]
which are skew-symmetric by \cite[Corollary 1.3 (ii)]{Oda}. Then $i_{n,\frp}$ is $\cO_{F_\frp}$-alternating if $p\neq 2$ and $\langle x,i_{n,\frp}(ax)\rangle_{A[\frp^n]}^2=1$ for any $x\in A[\frp^n](\oelbar)$ and $a\in \cO_{F_\frp}$ if $p=2$, where $\bar{L}$ is an algebraic closure of $L$. For $p=2$, by choosing a generator of the $\cO_F$-module $\frc/p^{n+1}\frc$, we may assume that the isomorphisms $i_{n,\frp}$ and $i_{n+1,\frp}$ are compatible with each other. In this case, for any lift $\hat{x}$ of $x$ in $A[\frp^{n+1}](\oelbar)$ and $a\in \cO_{F_\frp}$, we have
\begin{align*}
\langle x, i_{n,\frp}(ax) \rangle_{A[\frp^n]}&=\langle x, pi_{n+1,\frp}(a\hat{x})\rangle_{A[\frp^n]}=\langle x, i_{n+1,\frp}(a\hat{x})\rangle_{A[\frp^{n+1}]}\\
&=\langle \hat{x}, i_{n+1,\frp}(a\hat{x})\rangle_{A[\frp^{n+1}]}^p=1.
\end{align*}
Thus $i_{n,\frp}$ is $\cO_{F_{\frp}}$-alternating also for $p=2$. Hence each $A[\frp^n]$ is an $\cO_{F_{\frp}}$-$\ADBT_n$ over $\oel$. From this, we see that $i_n$ is $\cO_F$-alternating, namely $\langle x, i_{n}(ax) \rangle_{A[p^n]}=1$ for any $x\in A[p^n](\oelbar)$ and $a\in \cO_{F}$. For any $\beta\in \bB_{\frp}$, we put $\Hdg_\beta(A)=\Hdg_\beta(A[\frp])$.

On the other hand, for any finite flat group scheme $\cH$ over $\oel$ with an $\cO_F$-action, we have the decompositions
\[
\cH=\bigoplus_{\frp\mid p} \cH_{\frp},\quad \omega_{\cH}=\bigoplus_{\beta\in \bB_F} \omega_{\cH,\beta}
\]
as above such that $\cH_{\frp}$ is a finite flat closed subgroup scheme of $\cH$ over $\oel$ and $\omega_{\cH,\beta}=\omega_{\cH_\frp,\beta}$ for any $\beta\in \bB_{\frp}$. 
Since the $i$-th Hodge-Tate map $\HT_i:\cH(\oelbar)\to \omega_{\cH^\vee}\otimes_{\oel}\cO_{\bar{L},i}$ is $\cO_F$-linear, it is also decomposed as the direct sum of the maps
\[
\HT_i=\bigoplus_{\frp\mid p} \HT_{\cH_\frp,i},\quad \HT_{\cH_\frp,i}: \cH_{\frp}(\oelbar)\to \omega_{\cH^\vee_\frp}\otimes_{\oel}\cO_{\bar{L},i}.
\]
For any ideal $\fra\subseteq \cO_F$, we say that $\cH$ is $\fra$-cyclic if the $\cO_F$-module $\cH(\oelbar)$ is isomorphic to $\cO_F/\fra$. Note that, for any HBAV $A$ over $\oel$ as above and any finite flat $p^n$-cyclic $\cO_F$-subgroup scheme $\cH$ of $A$ over $\oel$, we can define on $A/\cH$ a structure of a HBAV $(A/\cH,\bar{\iota},\bar{\lambda},\bar{\psi})$ with $\frc$-polarization $\bar{\lambda}$ as in \cite[\S 2.1]{Ti_2}.

\begin{prop}\label{cansubAV}
Let $L/K$ be a finite extension in $\Kbar$. Let $\frc$ be a non-zero fractional ideal of $F$. Let $A$ be a HBAV over $\oel$ with a $\frc$-polarization. Put $w_\beta=\Hdg_\beta(A)$ and $w=\max\{w_\beta\mid \beta\in \bB_F\}$. Suppose that 
\[
w_\beta +p w_{\sigma^{-1} \circ \beta} <p^{2-n}
\]
holds for any $\beta\in \bB_F$.
For any $\frp\mid p$, let $\cC_{n,\frp}$ be the canonical subgroup of the $\cO_{F_\frp}$-$\ADBT_n$ $A[\frp^n]$ of level $n$, which exists by Theorem \ref{cansubh}. The finite flat closed subgroup scheme
\[
\cC_n(A)=\bigoplus_{\frp\mid p} \cC_{n,\frp}
\]
of $A[p^n]$ over $\oel$ is stable under the $\cO_F$-action. We call $\cC_n(A)$
the canonical subgroup of $A$ of level $n$. It satisfies
\[
\deg_\beta(A[p^n]/\cC_n(A))=\sum_{l=0}^{n-1} p^{l} w_{\sigma^{-l}\circ \beta}
\]
for any $\beta\in \bB_F$ and the $\cO_F/p^n \cO_F$-module $\cC_n(A)(\okbar)$ is free of rank one. Then $\cC_n=\cC_n(A)$ also satisfies the following.
\begin{enumerate}
\item\label{cansubAV_isom}Let $A'$ be a HBAV over $\oel$ satisfying the same condition on the $\beta$-Hodge heights as above. Then any isomorphism of HBAV's $j:A\to A'$ over $\oel$ induces an isomorphism $\cC_n(A)\simeq \cC_n(A')$.
\item\label{cansubAV_BC} $\cC_n$ is compatible with base extension of complete discrete valuation rings with perfect residue fields.
\item\label{cansubAV_isot} $\cC_n$ is isotropic with respect to the $\cO_F$-alternating isomorphism $i_n:A[p^n]\simeq A[p^n]^\vee$ defined by the $\frc$-polarization of $A$ and any element $x\in \frc$ generating the $\cO_F$-module $\frc/p^n\frc$.
\item\label{cansubAV_Frob} The kernel of the $n$-th iterated Frobenius map of $A[p^n] \times\sS_{L,1-p^{n-1}w}$ coincides with $\cC_n\times \sS_{L,1-p^{n-1}w}$.
\item\label{cansubAV_cl}
The scheme-theoretic closure of $\cC_n(\okbar)[p^i]$ in $\cC_n$ is the canonical subgroup $\cC_i$ of level $i$ of $A[p^i]$ for any $0\leq i\leq n-1$.
\end{enumerate}

Put $b=n-w(p^n-1)/(p-1)$. If $w<(p-1)/p^n$, then $\cC_n=\cC_n(A)$ also has the following properties:
\begin{enumerate}
\setcounter{enumi}{5}
\item\label{cansubAV_HT} $\cC_n(\okbar)$ coincides with $\Ker(\HT_i)$ for any rational number $i$ satisfying
\[
n-1+\frac{w}{p-1}< i\leq b.
\]
\item\label{cansubAV_ram} $\cC_n=A[p^n]_i$ for any rational number $i$ satisfying
\[
\frac{1}{p^n(p-1)}\leq i \leq \frac{1}{p^{n-1}(p-1)}-\frac{w}{p-1}.
\]
\item\label{cansubAV_omega} For any $i\in v_p(\oel)$ satisfying $i\leq b$, the natural map 
\[
\omega_A \otimes_{\oel} \cO_{L,i}\to \omega_{\cC_n}\otimes_{\oel} \cO_{L,i}
\]
is an isomorphism.
\item\label{cansubAV_CokHT} The cokernel of the map
\[
\HT\otimes 1: \cC_n^\vee(\okbar)\otimes \okbar \to \omega_{\cC_n} \otimes_{\oel}\okbar
\]
is killed by $m_{\Kbar}^{\geqslant w/(p-1)}$.
\item\label{cansubAV_noncan} For any $\frp\mid p$ and any finite flat closed $\frp$-cyclic $\cO_F$-subgroup scheme $\cH\neq \cC_{1,\frp}$ of $A[p]$ over $\oel$, the HBAV $A/\cH$ has the canonical subgroup $\cC_n(A/\cH)$ of level $n$, which is equal to 
\[
(\bigoplus_{\frq\mid p,\frq\neq \frp}\cC_{n,\frq}) \oplus (p^{-1}\cC_{n-1,\frp}/\cH).
\]
Moreover, the natural map $A\to A/\cH$ induces a map $\cC_n(A)\to \cC_n(A/\cH)$ which is an isomorphism over $L$.
\end{enumerate} 
\end{prop}
\begin{proof}
The assertion on $\deg_\beta$ follows from that of Theorem \ref{cansubh}, since we have
\[
\deg_\beta(A[p^n]/\cC_n(A))=\deg_\beta(A[p^n]_\frp/\cC_{n,\frp})
\]
for $\frp\mid p$ satisfying $\beta\in \bB_\frp$. The assertion on the freeness follows from Theorem \ref{cansubh} (\ref{cansubh_free}). The assertions (\ref{cansubAV_isom}), (\ref{cansubAV_BC}) and (\ref{cansubAV_cl}) also follow from Theorem \ref{cansubh}. 

Let us show the assertion (\ref{cansubAV_isot}). Theorem \ref{cansubh} (\ref{cansubh_dual}) implies that $(A[p^n]/\cC_n)^\vee$ can be identified with the canonical subgroup of $A[p^n]^\vee$. By Theorem \ref{cansubh} (\ref{cansubh_isom}), the isomorphisms
\[
A[p^n]\overset{x}{\simeq} A^\vee[p^n] \simeq A[p^n]^\vee
\]
preserve the canonical subgroups, and thus their composite induces an isomorphism $\cC_n \simeq (A[p^n]/\cC_n)^\vee$. This shows the assertion (\ref{cansubAV_isot}). Put $w_\frp=\max\{w_\beta\mid \beta\in \bB_\frp\}$. Since we have $1-p^{n-1}w\leq 1-p^{n-1}w_\frp$, the assertion (\ref{cansubAV_Frob}) follows from Theorem \ref{cansubh} (\ref{cansubh_Frob}). 

Suppose $w<(p-1)/p^n$. Then we have
\[
n-1+\frac{w_\frp}{p-1}\leq n-1+\frac{w}{p-1}<n-\frac{w(p^n-1)}{p-1} \leq n-\frac{w_\frp(p^n-1)}{p-1}
\]
for any $\frp\mid p$. Since the map $\HT_b$ is the direct sum of the maps
\[
\HT_{A[\frp^n],b}: A[\frp^n](\okbar)\to \omega_{A[\frp^n]^\vee}\otimes_{\oel}\cO_{\Kbar,b},
\]
Theorem \ref{cansubh} (\ref{cansubh_HT}) implies the assertion (\ref{cansubAV_HT}). Since the formation of lower ramification subgroups commutes with product, the assertion (\ref{cansubAV_ram}) follows from Theorem \ref{cansubh} (\ref{cansubh_ram}). Similarly, the assertions (\ref{cansubAV_omega}) and (\ref{cansubAV_CokHT}) follow from Proposition \ref{cokn}. Since we have the decomposition
\[
(A/\cH)[p^n]=(\bigoplus_{\frq|p,\frq\neq \frp}A[\frq^n])\oplus p^{-n}\cH/\cH,
\]
Corollary \ref{noncanh} shows the assertion (\ref{cansubAV_noncan}). This concludes the proof.
\end{proof}

%---------------------------------------------------------------------

%---------------------------------------------------------------------

\subsubsection{Moduli spaces and toroidal compactifications}\label{SecToroidal}

Let $M(\mu_N,\frc)$ be the Hilbert modular variety over $\okey$ which parametrizes the isomorphism classes of HBAV's $(A,\iota, \lambda, \psi)$ such that $\lambda$ is a $\frc$-polarization and $\psi$ is a $\Gamma_{00}(N)$-structure. The scheme $M(\mu_N,\frc)$ is smooth over $\okey$ \cite[Chapter 3, Theorem 6.9]{Gor}. We denote by $A^\univ$ the universal HBAV over $M(\mu_N,\frc)$. 

An unramified cusp for $M(\mu_N,\frc)$ is a triple $(\fra,\frb, \phi_N)$ of fractional ideals $\fra,\frb$ of $F$ satisfying $\fra\frb^{-1}=\frc$ and an isomorphism of $\cO_F$-modules
\[
\phi_N: \fra^{-1}/N\fra^{-1}\simeq \cO_F/N\cO_F.
\]
For each cusp, we have a Tate object $\Tate_{\fra,\frb}(q)$ over a certain base scheme  \cite[\S 4]{Rap}, which is used to construct a toroidal compactification $\bar{M}(\mu_N,\frc)$ of $M(\mu_N,\frc)$. We recall the definition for unramified cusps. Put $M=\fra \frb$, $M_\bR=M\otimes \bR$ and $M^*_{\bR}=\Hom(M,\bR)$. We identify $M\otimes \bQ$ with $F$. Then any $\sC\in \Dec(\fra,\frb)$ gives a rational polyhedral cone decomposition of 
\[
M^{*,+}_{\bR}=\{f\in M^*_{\bR}\mid f(M^+)\subseteq \bR_{>0}\}. 
\]
For each $\sigma\in \sC$, put 
\[
\sigma^\vee = \{ m\in M_\bR\mid l(m)\geq 0 \text{ for any } l\in \sigma\}.
\]
Then we have an affine torus embedding
\[
S=\Spec(\okey[q^m\mid m\in M])\to S_\sigma= \Spec(\okey[q^m\mid m \in M\cap \sigma^\vee]).
\]
The affine schemes $\{S_\sigma\}_{\sigma\in \sC}$ can be glued via $S_\sigma\cap S_\tau=S_{\sigma\cap\tau}$ to define a torus embedding $S \to S_{\sC}$. We denote by $S^\infty_\sigma$ and $S^\infty_{\sC}=\bigcup_{\sigma\in \sC} S^\infty_{\sigma}$ the complements of $S$ in these embeddings with reduced structures. The formal completions along these closed subschemes are denoted by $\hat{S}_\sigma=\Spf(\hat{R}_\sigma)$ and $\hat{S}_{\sC}$. By assumption, we can construct the quotient $\hat{S}_{\sC}/U_N$ by re-gluing $\{\hat{S}_{\sigma}\}_{\sigma\in \sC}$ via the action $\epsilon: \hat{S}_{\sigma}\simeq \hat{S}_{\epsilon\sigma}$ for any $\epsilon\in U_N$.
The closed subscheme $S^\infty_\sigma$ is defined by a principal ideal $\hat{I}_\sigma$ of the ring $\hat{R}_\sigma$ satisfying $\sqrt{\hat{I}_\sigma}=\hat{I}_\sigma$. The ring $\hat{R}_\sigma$ is a Noetherian normal excellent ring which is complete with respect to the $\hat{I}_\sigma$-adic topology. Put $\bar{S}_\sigma=\Spec(\hat{R}_\sigma)$, $\bar{S}_\sigma^\infty=V(\hat{I}_\sigma)$ and $\bar{S}^0_\sigma=\bar{S}_\sigma\setminus \bar{S}_\sigma^\infty$, where the latter is an affine scheme and we denote its affine ring by $\hat{R}_\sigma^0$.

Note that the torus with character group $\fra$ is $(\fra\cD_F)^{-1}\otimes \Gm$. For any $\eta\in \fra$, we denote by $\frX^\eta$ the element of $\cO((\fra\cD_F)^{-1}\otimes \Gm)$ which the character $\eta$ defines. We have an $\cO_F$-linear homomorphism
\[
q: \frb \to (\fra\cD_F)^{-1}\otimes \Gm(\bar{S}_\sigma^0)
\] 
defined by $\xi\mapsto (\frX^\eta\mapsto q^{\xi\eta})$ with $\xi\in \frb$ and $\eta\in \fra$. By Mumford's construction, we obtain the semi-abelian scheme $\Tate_{\fra,\frb}(q)$ over $\bar{S}_\sigma$ such that its restriction to $\bar{S}_\sigma^0$ is an abelian scheme \cite[\S 4]{Rap}. It admits a natural $\cO_F$-action. Over $\bar{S}_\sigma^0$, we have a natural exact sequence
\[
\xymatrix{
0 \ar[r] & \frac{1}{N}(\fra\cD_F)^{-1}\otimes \mu_N\ar[r] & \Tate_{\fra,\frb}(q)[N]|_{\bar{S}_\sigma^0}\ar[r] &\frac{1}{N}\frb/\frb\ar[r] & 0,
}
\]
which defines, for any unramified cusp $(\fra,\frb,\phi_N)$, a $\Gamma_{00}(N)$-structure on $\Tate_{\fra,\frb}(q)|_{\bar{S}_\sigma^0}$ using $\phi_N$. Moreover, the natural isomorphism
\[
 ((\fra\cD_F)^{-1}\otimes \Gm)\otimes_{\oef}\frc \to (\frb\cD_F)^{-1}\otimes \Gm
\]
induces a $\frc$-polarization
\[
\lambda_{\fra,\frb}: \Tate_{\fra,\frb}(q)|_{\bar{S}_\sigma^0} \otimes_{\oef} \frc\to \Tate_{\frb,\fra}(q)|_{\bar{S}_\sigma^0} \simeq (\Tate_{\fra,\frb}(q)|_{\bar{S}_\sigma^0})^\vee .
\]
By these data we consider the Tate object $\Tate_{\fra,\frb}(q)|_{\bar{S}_\sigma^0}$ as a HBAV over $\bar{S}_\sigma^0$, which yields a morphism $\bar{S}_\sigma^0\to M(\mu_N,\frc)$. Then the toroidal compactification $\bar{M}^{\sD(\frc)}(\mu_N,\frc)$ of $M(\mu_N,\frc)$ over $\okey$ with respect to $\sD(\frc)$, which we also denote by $\bar{M}(\mu_N,\frc)$ if no confusion will occur, is constructed in such a way as to satisfy the following \cite[Th\'{e}or\`{e}me 6.18]{Rap}:
\begin{itemize}
\item $\bar{M}(\mu_N,\frc)$ is projective and smooth over $\okey$.
\item $M(\mu_N,\frc)$ is an open subscheme of $\bar{M}(\mu_N,\frc)$ which is fiberwise dense and the complement $D$ of $M(\mu_N,\frc)$ is a normal crossing divisor. In particular, $M(\mu_N,\frc)$ is quasi-compact.
\item The formal completion $\bar{M}(\mu_N,\frc)|^\wedge_{D}$ of $\bar{M}(\mu_N,\frc)$ along the boundary divisor $D$ is isomorphic to 
\[
\coprod \hat{S}_{\sC(\fra,\frb)}/U_N,
\]
where the disjoint union runs over the set of isomorphism classes of cusps.
\item The universal HBAV $A^\univ$ over $M(\mu_N,\frc)$ extends to a semi-abelian scheme $\bar{A}^\univ$ with $\cO_F$-action over $\bar{M}(\mu_N,\frc)$ such that, for any $\sigma\in \sC(\fra,\frb)$, the pull-back of $\bar{A}^\univ$ by the restriction to $\bar{S}_\sigma^0$ of the unique algebraization $\bar{S}_\sigma\to \bar{M}(\mu_N,\frc)$ of the map $\hat{S}_\sigma\to \bar{M}(\mu_N,\frc)|_{D}^\wedge$ for any cusp $(\fra,\frb,\phi_N)$ is isomorphic to $\Tate_{\fra,\frb}(q)|_{\bar{S}_\sigma^0}$.
\end{itemize}

Let $\bar{\SGm}(\mu_N,\frc)$ be the $p$-adic formal completion of $\bar{M}(\mu_N,\frc)$. Let $\bar{\cM}(\mu_N,\frc)$ be its Raynaud generic fiber. Let $\cM(\mu_N,\frc)$ be the analytification of the scheme $M(\mu_N,\frc)\otimes_{\okey}K$, which is a Zariski open subvariety of $\bar{\cM}(\mu_N,\frc)$. 
The semi-abelian scheme $\bar{A}^{\univ}$ defines semi-abelian objects $\bar{\frA}^\univ$ over $\bar{\SGm}(\mu_N,\frc)$ and $\bar{\cA}^\univ$ over $\bar{\cM}(\mu_N,\frc)$ by taking the $p$-adic completion and the Raynaud generic fiber. For the zero section $e$ of $\bar{A}^\univ$, put $\omega_{\bar{A}^\univ}=e^*\Omega^1_{\bar{A}^\univ/\bar{M}(\mu_N,\frc)}$. For any $g$-tuple $\kappa=(k_\beta)_{\beta\in \bB_F}$ in $\bZ$, we define 
\[
\omega_{\bar{A}^\univ,\beta}=\omega_{\bar{A}^\univ}\otimes_{\cO_F, \beta}\okey,\quad \omega^{\kappa}_{\bar{A}^\univ}=\bigotimes_{\beta\in \bB_F} \omega_{\bar{A}^\univ,\beta}^{\otimes k_\beta}.
\]
We also define $\omega_{\bar{\cA}^\univ,\beta}$ and $\omega^{\kappa}_{\bar{\cA}^\univ}$ similarly.
For any $\beta\in \bB_F$, let $h_{\beta}$ be the $\beta$-partial Hasse invariant, which is a section of the invertible sheaf $\omega_{\bar{A}^\univ,\sigma^{-1}\circ \beta}^p\otimes \omega_{\bar{A}^\univ,\beta}^{-1}$ on $\bar{M}(\mu_N,\frc)\times \sS_1$ \cite[\S 2.5]{GK} (see also \cite[\S 7]{AGo}).  For any extension $L/K$ of complete valuation fields, any HBAV $A$ over $\cO_L$ and any $\beta\in \bB_{F}$, consider the element $P$ of $\bar{\cM}(\mu_N,\frc)(L)$ induced by $A$ and a lift $\tilde{h}_\beta$ of $h_\beta$ as a section of $\omega_{\bar{A}^\univ,\sigma^{-1}\circ \beta}^p\otimes \omega_{\bar{A}^\univ,\beta}^{-1}$ over an open neighborhood of $P$. Then we have the equality of truncated valuations
\[
\Hdg_{\beta}(A)=v_p(\tilde{h}_{\beta}(P)).
\]
If $P\in \bar{\cM}(\mu_N,\frc)(L)$ corresponds to a semi-abelian scheme $A$ over $\oel$ which is not an abelian scheme, then we put $\Hdg_{\beta}(A)=v_p(\tilde{h}_{\beta}(P))=0$.

Let $\uv=(v_\beta)_{\beta\in \bB_F}$ be a $g$-tuple in $[0,1]\cap \bQ$. We denote by $\bar{\cM}(\mu_N,\frc)(\uv)$ and $\cM(\mu_N,\frc)(\uv)$ be the admissible open subsets of $\bar{\cM}(\mu_N,\frc)$ and $\cM(\mu_N,\frc)$ defined by $v_p(\tilde{h}_\beta(P))\leq v_\beta$ for any $\beta \in \bB_F$, respectively. Note that $\bar{\cM}(\mu_N,\frc)(\uv)$ is quasi-compact.
We define its integral model $\bar{\SGm}(\mu_N,\frc)(\uv)$ as follows: write $v_\beta=a_\beta/b_\beta$ with non-negative integers $a_\beta$ and $b_\beta\neq 0$. Take a formal open covering $\bar{\SGm}(\mu_N,\frc)=\bigcup \frU_i$ such that every $h_\beta$ lifts to a section $\tilde{h}_\beta$ on each $\frU_i$. Consider the formal scheme whose restriction to each $\frU_i$ is the admissible blow-up of $\frU_i$ along the ideal $(p^{a_\beta},\tilde{h}_\beta^{b_\beta})$, and its locus where this ideal is generated by $\tilde{h}_\beta^{b_\beta}$. Repeat this for any $\beta\in \bB_F$ and define $\bar{\SGm}(\mu_N,\frc)(\uv)$ as the normalization in $\bar{\cM}(\mu_N,\frc)(\uv)$ of the resulting formal scheme. We denote the special fibers of $\bar{M}(\mu_N,\frc)$ and $\bar{\SGm}(\mu_N,\frc)(\uv)$ by $\bar{M}(\mu_N,\frc)_k$ and $\bar{\SGm}(\mu_N,\frc)(\uv)_k$, respectively. We also denote by $\SGm(\mu_N,\frc)(\uv)$ the complement in $\bar{\SGm}(\mu_N,\frc)(\uv)$ of the boundary divisor of the special fiber.

Let $v$ be an element of $[0,1]\cap \bQ$. When $v_\beta=v$ for any $\beta\in \bB_F$, we write $\bar{\cM}(\mu_N,\frc)(\uv)$ as $\bar{\cM}(\mu_N,\frc)(v)$ . Moreover, we denote by $\bar{\cM}(\mu_N,\frc)(v_\tot)$ the quasi-compact admissible open subset defined similarly to $\bar{\cM}(\mu_N,\frc)(v)$ with the usual Hasse invariant
\[
h=\prod_{\beta\in \bB_F}h_\beta
\]
instead of $h_\beta$'s. We also define similar spaces for these two variants, such as $\bar{\SGm}(\mu_N,\frc)(v)$ and $\bar{\SGm}(\mu_N,\frc)(v_\tot)$.
Note that $\bar{\SGm}(\mu_N,\frc)(0)$ is just the formal open subscheme of $\bar{\SGm}(\mu_N,\frc)$ over which all the $\beta$-partial Hasse invariants are invertible.

Let $R$ be a topological $\okey$-algebra which is idyllic with respect to the $p$-adic topology \cite[D\'{e}finition 1.10.1]{Abbes}. By \cite[Corollaire 2.13.9]{Abbes}, any morphism $\hat{f}:\Spf(R)\to \bar{\SGm}(\mu_N,\frc)$ has a unique algebraization $f:\Spec(R)\to \bar{M}(\mu_N,\frc)$, and we have a semi-abelian scheme $G_R=f^*\bar{A}^\univ$ over $\Spec(R)$. Taking the reduction modulo $p$, we see that $\hat{f}$ factors through $\bar{\SGm}(\mu_N,\frc)(0)$ if and only if $G_R$ is ordinary.

Let $\mathbf{NAdm}$ be the category of admissible $p$-adic formal $\okey$-algebras $R$ such that $R$ is normal. Note that we have $R[1/p]^\circ=R$ by \cite[Remark after Proposition 6.3.4/1]{BGR}. By \cite[Lemme 3.1]{Rap}, we can see as in \cite[Proposition 5.2.1.1]{AIP} that any morphism $\Spv(R[1/p])\to \SGm(\mu_N,\frc)(\uv)^\rig$ corresponds uniquely to an isomorphism class of a HBAV $A$ over $\Spec(R)$ such that $\Hdg_\beta(A_x)\leq v_\beta$ for $x\in \Spv(R[1/p])$.

We give a proof of the following lemma for lack of a reference.

\begin{lem}\label{FvPLem}
Let $L/K$ be an extension of complete valuation fields. Let $\cX$ be a connected smooth rigid analytic variety over $L$ and $\cF$ an invertible sheaf on $\cX$. Suppose that $f\in \cF(\cX)$ vanishes on a non-empty admissible open subset $\cU$ of $\cX$. Then $f=0$.
\end{lem}
\begin{proof}
Take an admissible affinoid covering $\cX=\bigcup_{i\in I} \cX_i$ such that $\cX_i$ is connected and $\cF$ is trivial on $\cX_i$ for any $i\in I$. We have $\cX_{i_0}\cap \cU\neq \emptyset$ for some $i_0$. Then \cite[Exercise 4.6.3]{FvP} implies $f|_{\cX_{i_0}}=0$. 

Put $I_0=\{i\in I\mid f|_{\cX_i}=0\}$, which is non-empty. \cite[Exercise 4.6.3]{FvP} also implies that $\cX_i\cap \cX_j=\emptyset$ for any $i\in I_0$ and $j\in I_1:=I\setminus I_0$. Then for the subsets
\[
\cX_0=\bigcup_{i\in I_0} \cX_i,\quad \cX_1=\bigcup_{i\in I_1} \cX_i
\]
and $s\in\{0,1\}$, the intersection $\cX_s\cap \cX_i$ equals $\cX_i$ if $i\in I_s$ and $\emptyset$ if $i\notin I_s$. Hence $\cX=\cX_0\coprod \cX_1$ is an admissible covering of $\cX$. Since $\cX$ is connected, we obtain $\cX=\cX_0$ and $f=0$.
\end{proof}

\begin{lem}\label{Mvconn}
The rigid analytic variety $\bar{\cM}(\mu_N,\frc)(v)_{\bC_p}$
is connected for any $v\in [0,1]\cap \bQ$.
\end{lem}
\begin{proof}
Since $\bar{\cM}(\mu_N,\frc)(v)$ is separated, it is enough to show that for any sufficiently large finite extension $K'/K$, the base extension $\bar{\cM}(\mu_N,\frc)(v)_{K'}$
is connected \cite[Theorem 3.2.1]{Con_irr}. Replacing $K$ by $K'$, we may assume $K'=K$.

By Ribet's theorem (see \cite[Chapter 3, Theorem 6.19]{Gor}), the ordinary locus $\bar{\SGm}(\mu_N,\frc)(0)_k$ is geometrically connected. Since the rigid analytic variety $\bar{\cM}(\mu_N,\frc)(0)$ is the tube for the immersion $\bar{\SGm}(\mu_N,\frc)(0)_{k}\to \bar{\SGm}(\mu_N,\frc)$, \cite[Proposition 1.3.3]{Berthelot} implies that $\bar{\cM}(\mu_N,\frc)(0)$ is connected.

Consider the case of $v>0$. Suppose that $\bar{\cM}(\mu_N,\frc)(v)$ is not connected. Then we can take its connected component $\cU$ which does not intersect $\bar{\cM}(\mu_N,\frc)(0)$. Since $\cU$ is quasi-compact, there exists a finite admissible affinoid covering $\cU=\bigcup_{i=1}^m \cU_i$ of $\cU$ such that any $\beta$-partial Hasse invariant can be lifted to a section over $\cU_i$. Using the maximal modulus principle on each $\cU_i$, we see that there exists a positive rational number $\delta$ satisfying
\[
\max\{\Hdg_\beta(x)\mid \beta\in \bB_F\}\geq \delta 
\]
for any $x\in \cU$. Then, for any rational number $\varepsilon$ satisfying $0<\varepsilon <\delta$, we have $\bar{\cM}(\mu_N,\frc)(\varepsilon)\cap \cU=\emptyset$. 

On the other hand, let us consider the specialization map 
\[
\spc:\bar{\cM}(\mu_N,\frc)\to \bar{\SGm}(\mu_N,\frc)_{k}
\]
with respect to $\bar{\SGm}(\mu_N,\frc)$. Take any $P\in \cU$ and consider its specialization $\bar{P}=\spc(P)$. Since $P\notin \bar{\cM}(\mu_N,\frc)(0)$, it corresponds to a HBAV. Then \cite[(2.5.1)]{GK} and \cite[Lemma 7.2.5]{deJong} give an identification
\begin{equation}\label{Xlocal}
\spc^{-1}(\bar{P})=\prod_{\beta\in\bB_F} \cA_\beta[0,1), 
\end{equation}
where $\cA_\beta[\rho,\rho')$ is the annulus with parameter $t_\beta$ defined by $\rho\leq |t_\beta| < \rho'$. By \cite[\S 4.2]{GK}, we may assume that the parameter $t_\beta$ satisfies
\begin{equation}\label{GK_tbeta}
\Hdg_\beta(A)=\left\{\begin{array}{ll}v_p(t_\beta(Q))  & (\beta\in \tau(\bar{P})) \\ 0 & (\beta\notin \tau(\bar{P}))\end{array} \right.
\end{equation}
for any $Q\in \spc^{-1}(\bar{P})$ and for any $\beta\in \bB_F$, 
where $A$ is the HBAV corresponding to $Q$ and $\tau(\bar{P})$ is defined by \cite[(2.3.3)]{GK}. In particular, we have $\Hdg_\beta(A)\leq v_p(t_\beta(Q))$ for any $\beta\in \bB_F$. For any positive rational number $\varepsilon$, put
\[
\spc^{-1}(\bar{P})(\varepsilon)'=\prod_{\beta\in\tau(\bar{P})} \cA_\beta[p^{-\varepsilon},1)\times \prod_{\beta\notin\tau(\bar{P})} \cA_\beta[0,1).
\]
Since $\spc^{-1}(\bar{P})(v)'$ is a connected admissible open subset of $\bar{\cM}(\mu_N,\frc)(v)$ containing $P$, it is contained in $\cU$. 
However, for any $\varepsilon$ satisfying $\varepsilon<\min\{\delta,v\}$, we have
\[
\emptyset \neq \spc^{-1}(\bar{P})(\varepsilon)'\subseteq \bar{\cM}(\mu_N,\frc)(\varepsilon)\cap \cU,
\]
which is a contradiction.
\end{proof}

%---------------------------------------------------------------------

%---------------------------------------------------------------------

\subsubsection{Canonical subgroups over moduli spaces}\label{SecCansubFam}

Let $n$ be a positive integer. Let $\uv=(v_\beta)_{\beta\in \bB_F}$ be a $g$-tuple satisfying 
\[
v_\beta\in [0,(p-1)/p^{n})\cap \bQ 
\]
for any $\beta\in \bB_F$. Note that the $1/(p^n(p-1))$-st lower ramification subgroups can be patched into a rigid analytic family \cite[Lemma 5.6]{Ha_lowram}. 
Let $R$ be an object of $\mathbf{NAdm}$ and put $\cU=\Spv(R[1/p])$. Let $\cU\to \SGm(\mu_N, \frc)(\uv)^\rig$ be any morphism of rigid analytic varieties over $K$. This defines a HBAV $\bar{A}^\univ|_R$ over $\Spec(R)$. For any rig-point $x\in \Spec(R)$, we have the canonical subgroup $\cC_n((\bar{A}^\univ|_R)_x)$. Theorem \ref{cansubAV} (\ref{cansubAV_ram}) implies that they can be patched into an admissible open subgroup of $\bar{\cA}^\univ[p^n]|_{\cU}$. By \cite[Proposition 4.1.3]{AIP}, it uniquely extends to a finite flat subgroup scheme $\cC_n$ of $\bar{A}^\univ|_R$ over $\Spec(R)$.

On the other hand, on a formal open neighborhood $\frU$ of a point of the boundary satisfying $\frU\subseteq \bar{\SGm}(\mu_N,\frc)(0)$, the unit component $\bar{\frA}^\univ[p^n]^0|_{\frU}$ is quasi-finite and flat over $\frU$ with constant degree on each fiber by \cite[p.297 (ii)]{Rap}. Thus it is finite and flat. Then, by gluing along $\SGm(\mu_N,\frc)(0)$, we obtain a finite flat formal subgroup scheme $\cC_n$ of $\bar{\frA}^\univ$ over $\bar{\SGm}(\mu_N,\frc)(\uv)$ and its generic fiber $C_n$, which we refer to as the canonical subgroup of level $n$.

Let $R$ be a topological $\okey$-algebra which is quasi-idyllic with respect to the $p$-adic topology \cite[1.10.1.1]{Abbes}. Since any finitely generated $R$-module is automatically $p$-adically complete \cite[Proposition 1.10.2]{Abbes}, any finitely presented flat formal group scheme over $\Spf(R)$ can be identified with a finitely presented flat group scheme over $\Spec(R)$.
Thus we have a theory of Cartier duality for any finitely presented flat formal group scheme over any quasi-idyllic $p$-adic formal scheme. Then, from the construction, we see that the restriction of the Cartier dual $\cC_n^\vee|_{\bar{\SGm}(\mu_N,\frc)(0)}$ to the ordinary locus is finite and etale.

We have the following variant of \cite[Proposition 4.2.1 and Proposition 4.2.2]{AIP}.

\begin{lem}\label{famcanlem}
Let $\uv=(v_\beta)_{\beta\in\bB_F}$ be a $g$-tuple of non-negative rational numbers satisfying 
\[
v:=\max\{v_\beta\mid \beta\in \bB_F\}< (p-1)/p^n.
\]
Let $R$ be an object of $\mathbf{NAdm}$. For any morphism of admissible formal schemes $\hat{f}:\Spf(R)\to \bar{\SGm}(\mu_N,\frc)(\uv)$ over $\okey$, consider the pull-back $G=\bar{A}^\univ|_R$ by the unique algebraization $\Spec(R)\to \bar{M}(\mu_N,\frc)$ of $\hat{f}$ and $\cH_n=\cC_n|_{\Spf(R)}$, which is a subgroup scheme of the formal completion of $G$.
\begin{enumerate}
\item\label{famcanlem_omega}
For any rational number $i\in e^{-1}\bZ_{\geq 0}$ satisfying $i\leq n-v(p^n-1)/(p-1)$, the natural map $\omega_{G}\otimes_{\okey} \cO_{K,i}\to \omega_{\cH_n}\otimes_{\okey} \cO_{K,i}$ is an isomorphism.
\item\label{famcanlem_HT}
Assume that we have an isomorphism of $\cO_F$-modules $\cH_n^\vee(R)\simeq \cO_F /p^n\cO_F$. Then the cokernel of the linearization of the Hodge-Tate map
\[
\HT_{\cH^\vee_n}\otimes 1: \cH_n^\vee(R)\otimes R\to \omega_{\cH_n}
\]
is killed by $m_K^{\geqslant v/(p-1)}$.
\end{enumerate}
\end{lem}
\begin{proof}
Since the ordinary case is trivial, by a gluing argument we may assume that $\hat{f}$ factors through $\SGm(\mu_N,\frc)(\uv)$.
By replacing $\Spf(R)$ with its formal affine open subscheme, we may assume that $R$ is an integral domain and $\omega_{G}$ is a free $\oef\otimes R$-module of rank one. The first assertion follows by reducing it to Proposition \ref{cansubAV} (\ref{cansubAV_omega}) in the same way as \cite[Proposition 4.2.1]{AIP}. For the second assertion, take surjections $R^g \to \cH_n^\vee(R)\otimes R \simeq (R/p^n R)^g$ and $R^g\simeq \omega_{G}\to \omega_{\cH_n}$. Then the map $\HT_{\cH^\vee_n}\otimes 1$ can be identified with the reduction of the map defined by some matrix $\gamma\in M_g(R)$. It suffices to show $m_K^{\geqslant v/(p-1)} R^g\subseteq \gamma(R^g)$. Let $\frp$ be a prime ideal of $R$ of height one and $\hat{R}_\frp$ the completion of the local ring $R_\frp$. Proposition \ref{cansubAV} (\ref{cansubAV_CokHT}) implies $m_K^{\geqslant v/(p-1)} \hat{R}_\frp^g\subseteq \gamma(\hat{R}_\frp^g)$. This shows $m_K^{\geqslant v/(p-1)} R_\frp^g\subseteq \gamma(R_\frp^g)$ and $\det(\gamma)\neq 0$. Since $R$ is normal, $\gamma(R^g)$ is the intersection of $\gamma(R_\frp^g)$ for every such $\frp$ and the assertion follows.
\end{proof}

%---------------------------------------------------------------------

%---------------------------------------------------------------------

\subsection{Connected neighborhoods of critical points}\label{SecConnCritLocus}

Let $Y_{\frc,p}$ be the moduli scheme parametrizing the isomorphism classes of pairs $(A, \cH)$ over schemes $S/\Spec(\okey)$, where $A$ is a HBAV over $S$ with a $\frc$-polarization and a $\Gamma_{00}(N)$-structure, and $\cH$ is a finite locally free closed $\cO_F$-subgroup scheme of $A[p]$ of rank $p^g$ over $S$ such that $\cH$ is isotropic in the sense of \cite[\S 2.1]{GK}. Then $Y_{\frc,p}$ is projective over $M(\mu_N,\frc)$ \cite[p.415]{Stam}. For $S=\Spec(\oel)$ with some extension $L/K$ of complete valuation fields, $\cH$ is isotropic in this sense if and only if $\cH$ is $p$-cyclic.

Let $\SGy_{\frc,p}$ be the $p$-adic formal completion of $Y_{\frc,p}$ and $\cY_{\frc,p}$ its Raynaud generic fiber. Note that they are separated. By \cite[Lemme 3.1]{Rap}, we have $\cY_{\frc,p}(L)=\frY_{\frc,p}(\oel)=Y_{\frc,p}(\oel)$ for any extension $L/K$ of complete discrete valuation fields.
In this subsection, we construct a connected admissible affinoid open neighborhood of a point $Q=[(A,\cH)]$ of $\cY_{\frc,p}$ satisfying $\Hdg_\beta(A)=p/(p+1)$ for any $\beta\in \bB_F$ inside the base extension $\cY_{\frc,p,\bC_p}=\cY_{\frc,p}\hat{\otimes}_K \bC_p$, assuming $f_\frp\leq 2$ for any $\frp\mid p$.

\begin{lem}\label{CritNonempty}
There exists a point of $\bar{\cM}(\mu_N,\frc)$ corresponding to a HBAV $A$ over the integer ring $\oel$ of a finite extension $L/K$ satisfying $\Hdg_\beta(A)=p/(p+1)$ for any $\beta\in \bB_F$.
\end{lem}
\begin{proof}
Consider the stratum $W_{\bB_F}$ of the special fiber $M(\mu_N,\frc)_k$ as in \cite[\S 2.5]{GK}. Since $W_{\bB_F}$ is non-empty, there exists a point $P\in \bar{\cM}(\mu_N,\frc)$ such that $\bar{P}=\spc(P)\in W_{\bB_F}$ for the specialization map $\spc: \bar{\cM}(\mu_N,\frc)\to \bar{M}(\mu_N,\frc)_k$ as before. Since $\tau(\bar{P})=\bB_F$, the identification (\ref{Xlocal}) and (\ref{GK_tbeta}) yield the lemma.
\end{proof}

\begin{prop}\label{ConnCrit}
Suppose $f_\frp\leq 2$ for any $\frp\mid p$. Let $L/K$ be a finite extension in $\bar{\bQ}_p$ and $l$ the residue field of $L$. Let $K'$ be the composite field of $K$ and $\Frac(W(l))$ in $\bar{\bQ}_p$. Let $[(A,\cH)]$ be an element of $Y_{\frc,p}(\oel)$ satisfying $\Hdg_\beta(A)=p/(p+1)$ for any $\beta\in \bB_F$ and $Q$ the element of $\cY_{\frc,p}(L)$ it defines.
Let 
\[
\spc:\cY_{\frc,p}\to (Y_{\frc,p})_{k}=Y_{\frc,p}\times_{\okey}\Spec(k)
\]
be the specialization map with respect to $\SGy_{\frc,p}$ and put $\bar{Q}=\spc(Q)$. We define
\begin{align*}
\cV_Q=\{Q'=[(A',\cH')] \in \spc^{-1}(\bar{Q})\mid\ &\frac{1}{p+1}\leq\deg_\beta(A'[p]/\cH')\leq \frac{p}{p+1},\quad \\
&\Hdg_\beta(A')\leq \frac{p}{p+1}\text{ for any }\beta\in \bB_F\},
\end{align*}
\[
\cV_Q(\tfrac{1}{p+1})=\{Q'=[(A',\cH')] \in \cV_Q \mid \Hdg_\beta(A')\leq \frac{1}{p+1}\text{ for any }\beta\in \bB_F\}.
\]
Then they are admissible affinoid open subsets of $\cY_{\frc,p}$ defined over $K'$ such that $\cV_Q\hat{\otimes}_{K'}\bC_p$ is connected.
\end{prop}
\begin{proof}
By the assumption $f_{\frp}\leq 2$ and Proposition \ref{critisog}, we have the equality $\deg_\beta(A[p]/\cH)=p/(p+1)$ for any $\beta\in \bB_F$. 
\cite[Proposition 4.2]{Ti_2} shows that this value is equal to the one denoted by $\nu_\beta(Q)$ in \cite[\S 4.2]{GK}. In particular, the definition of $\nu_\beta(Q)$ in \cite[\S 4.2]{GK} implies $I(\bar{Q})=\bB_F$ with the notation of \cite[(2.3.2)]{GK}.

We claim that the complete local ring $\hat{\cO}_{Y_{\frc,p},\bar{Q}}$ of $Y_{\frc,p}$ at $\bar{Q}$ is isomorphic to the ring
\begin{equation}\label{frBdef}
\frB'=\cO_{K'}[[X_\beta,Y_\beta\mid \beta\in \bB_F]]/(X_\beta Y_\beta-p\mid \beta\in \bB_F)
\end{equation}
and there exists $g_\beta\in (\frB')^\times$ such that for any finite extension $E/K'$ and any $\cO_{K'}$-algebra homomorphism $x: \frB'\to \cO_E$, the corresponding $\cO_E$-valued point $[(A',\cH')]$ of $Y_{\frc,p}$ satisfies
\[
\deg_\beta(A'[p]/\cH')=v_p(X_\beta(x)),\quad \Hdg_\beta(A')=v_p((X_\beta+g_\beta Y_{\sigma^{-1}\circ \beta}^p)(x)).
\]
Indeed, let $Y_\frc$ be a moduli scheme over $W$ similar to $Y_{\frc,p}$ considered in \cite[\S 2.1]{GK}. Let $R$ be the affine algebra of an affine open neighborhood of $\bar{Q}$ in $Y_\frc$ and $m_{\bar{Q}}$ the maximal ideal of $R$ corresponding to $\bar{Q}$. The ring $\hat{\cO}_{Y_{\frc,p},\bar{Q}}$ is equal to the completion of the local ring of $R\otimes_W \okey$ at the kernel $n_{\bar{Q}}$ of the map $R\otimes_W \okey \to l$ associated to $m_{\bar{Q}}$. Since $K/\Frac(W)$ is finite totally ramified and $p\in m_{\bar{Q}}$, the ring $R_{m_{\bar{Q}}}\otimes_W \cO_K$ is local with maximal ideal $n_{\bar{Q}}(R_{m_{\bar{Q}}}\otimes_W \cO_K)$ and thus it is equal to the localization $(R\otimes_W \cO_K)_{n_{\bar{Q}}}$. We also see that the $m_{\bar{Q}}$-adic topology on the local ring $R_{m_{\bar{Q}}}\otimes_W \cO_K$ is the same as the topology defined by its maximal ideal.

By Stamm's theorem \cite{Stam} (see also \cite[Theorem 2.4.1']{GK}), the $m_{\bar{Q}}$-adic completion $\hat{R}_{m_{\bar{Q}}}$ of the localization $R_{m_{\bar{Q}}}$ is isomorphic to the ring
\[
\frB=W(l)[[X_\beta,Y_\beta\mid \beta\in \bB_F]]/(X_\beta Y_\beta-p\mid \beta\in \bB_F).
\]
Moreover, since $\Hdg_\beta(A)\neq 0$ for any $\beta\in \bB_F$, (\ref{GK_tbeta}) implies $\tau(\bar{Q})=\bB_F$. Thus, for any finite extension $E/\Frac(W(l))$ and any $W(l)$-algebra homomorphism $x:\frB\to \cO_E$, the corresponding HBAV $A'$ satisfies $v(t_\beta(x))=\Hdg_\beta(A')$. By \cite[Lemma 2.8.1]{GK} and the definition of $\nu_\beta(Q)$ in \cite[\S 4.2]{GK}, the isomorphism $\hat{R}_{m_{\bar{Q}}}\simeq \frB$ gives an identification of $\deg_\beta$ and $\Hdg_\beta$ for the ring $\frB$ as claimed before.

Since the ring $\frB/m_{\bar{Q}}^i \frB$ is finite over $W$, we have
\[
(R_{m_{\bar{Q}}}/m_{\bar{Q}}^i R_{m_{\bar{Q}}})\otimes_W \okey \simeq (\frB/m_{\bar{Q}}^i \frB)\otimes_W \okey \simeq \frB'/m_{\bar{Q}}^i \frB'.
\]
Since the $m_{\bar{Q}}$-adic topology on the ring $\frB'$ is the same as the topology defined by its maximal ideal, we obtain the claim. 

By \cite[Lemma 7.2.5]{deJong}, we have
\[
\spc^{-1}(\bar{Q})=(\Spf(\frB'))^\rig.
\]
Thus $\cV_Q$ is the $K'$-affinoid variety whose affinoid ring is the quotient of the Tate algebra 
\[
K' \langle X_\beta, Y_\beta, U_\beta, V_\beta, W_\beta \mid \beta\in \bB_F \rangle
\]
by the ideal generated by
\[
X_\beta^{p+1}-p U_\beta,\quad X_\beta^{p+1} V_\beta - p^p,\quad X_\beta Y_\beta-p,\quad W_\beta(X_\beta +g_\beta Y_{\sigma^{-1}\circ \beta}^p)^{p+1} - p^p
\]
for any $\beta \in \bB_F$. From this, we also obtain a similar description of $\cV_Q(\tfrac{1}{p+1})$ as a $K'$-affinoid variety.

Next we prove that the base extension $\cV_Q\hat{\otimes}_{K'} \bC_p$ is connected. Put $r=1/(p+1)$ and $s=p/(p+1)$. Fix a $(p+1)$-st root $\varpi=p^{1/(p+1)}$ of $p$ in $\bar{\bQ}_p$.
Then the affinoid ring $B_{Q,\bC_p}$ of $\cV_Q\hat{\otimes}_{K'} \bC_p$ is also isomorphic to the quotient of the Tate algebra 
\[
\bC_p \langle X_\beta, Y_\beta, U_\beta, V_\beta, W_\beta \mid \beta\in \bB_F \rangle
\]
by the ideal generated by
\[
X_\beta-\varpi U_\beta,\quad X_\beta V_\beta - \varpi^p,\quad X_\beta Y_\beta-\varpi^{p+1},\quad W_\beta(X_\beta +g_\beta Y_{\sigma^{-1}\circ \beta}^p) - \varpi^p
\]
for any $\beta \in \bB_F$. Note that in the ring $B_{Q,\bC_p}$ we also have $Y_\beta-\varpi V_\beta=0$. Hence $B_{Q,\Cp}$ is isomorphic to the quotient of the ring
\[
\bC_p \langle U_\beta, V_\beta, W_\beta \mid \beta\in \bB_F \rangle
\]
by the ideal generated by
\[
U_\beta V_\beta -\varpi^{p-1},\quad F_\beta:=W_\beta(U_\beta+\varpi^{p-1} g'_\beta V_{\sigma^{-1}\circ \beta}^p)-\varpi^{p-1}
\]
for any $\beta \in \bB_F$ with some $g'_\beta\in \frA_{Q,\bC_p}^\times$, where
\[
\frA_{Q,\bC_p}=\cO_{\bC_p}\langle U_\beta,V_\beta\mid \beta\in \bB_F\rangle/(U_\beta V_\beta-\varpi^{p-1}\mid \beta\in \bB_F).
\]

From these equations, we see that
\[
G_\beta:=V_\beta-W_\beta(1+g'_\beta V_\beta V_{\sigma^{-1}\circ \beta}^p)=0 
\]
in this quotient. Since 
\[
F_\beta\equiv -U_\beta G_\beta \bmod U_\beta V_\beta-\varpi^{p-1},
\]
we obtain 
\[
B_{Q,\bC_p}\simeq \bC_p\langle U_\beta, V_\beta, W_\beta\mid \beta\in \bB_F\rangle/(U_\beta V_\beta-\varpi^{p-1}, G_\beta\mid \beta\in \bB_F).
\]

Note that the ring 
\[
\frB_{Q,\bC_p}=\cO_{\bC_p}\langle U_\beta, V_\beta, W_\beta \mid \beta\in \bB_F\rangle/(U_\beta V_\beta-\varpi^{p-1}, G_\beta\mid \beta\in \bB_F)
\]
is a flat $\cO_{\bC_p}$-algebra. Indeed, consider the polynomial ring $\frA_{Q,\Cp}[W_\beta]$. Since the coefficients of $G_\beta$ as a polynomial of $W_\beta$ generate the unit ideal $\frA_{Q,\bC_p}$, by a limit argument reducing to the Noetherian case and using \cite[(20.F), Corollary 2]{Matsumura} we see that the $\frA_{Q,\bC_p}$-algebra
\[
\frA_{Q,\bC_p}[W_\beta\mid \beta\in \bB_F]/(G_\beta\mid \beta\in \bB_F)
\]
is flat. By \cite[Proposition 1.10.2 (ii)]{Abbes}, the $p$-adic completion of this algebra is $\frB_{Q,\bC_p}$. Since the $\cO_{\bC_p}$-algebra $\frA_{Q,\bC_p}$ is flat, the $p$-adic completion $\frB_{Q,\bC_p}$ is also flat over $\cO_{\bC_p}$.

Put $\bar{G}_\beta=G_\beta \bmod m_{\bC_p}$ and
\[
\bar{R}=\Fpbar [U_\beta, V_\beta, W_\beta\mid \beta\in \bB_F],\quad \bar{J}=(U_\beta V_\beta, \bar{G}_\beta\mid \beta\in \bB_F).
\]
Next we claim that the reduction $\bar{\frB}_{Q,\bC_p}=\bar{R}/\bar{J}$ of $\frB_{Q,\bC_p}$ is reduced and $\Spec(\bar{\frB}_{Q,\bC_p})$ is connected. For the reducedness, it suffices to show that the localization at every maximal ideal is reduced. Let $\SGm$ be any maximal ideal of $\bar{R}$ containing $\bar{J}$. Then we have 
\[
1+g'_\beta V_\beta V_{\sigma^{-1}\circ \beta}^p \notin \SGm
\]
since, supposing the contrary, $\bar{G}_\beta\in \SGm$ implies $V_\beta\in \SGm$ and $1\in \SGm$, which is a contradiction. Thus, in the ring $\bar{R}_{\SGm}$ we have
\[
W_\beta-V_\beta(1+g'_\beta V_\beta V_{\sigma^{-1}\circ \beta}^p)^{-1}\in \bar{J} \bar{R}_{\SGm}
\]
for any $\beta\in \bB_F$. Hence the localization $(\bar{\frB}_{Q,\bC_p})_{\SGm}$ is isomorphic to the localization of the ring
\[
\Fpbar[U_\beta, V_\beta\mid \beta\in \bB_F]/(U_\beta V_\beta\mid \beta\in \bB_F)
\]
at the pull-back of $\SGm$, which is reduced.

Let us show the connectedness. Let $\bB_F=\bB_U \coprod \bB_V$ be a decomposition into the disjoint union of two subsets. Consider the closed subscheme $F_{\bB_U,\bB_V}$ of $\Spec(\bar{\frB}_{Q,\bC_p})$ defined by $U_\beta=0$ for $\beta\in \bB_U$ and $V_\beta=0$ for $\beta\in \bB_V$. Since every $F_{\bB_U,\bB_V}$ contains the point defined by $U_\beta=V_\beta=W_\beta=0$ for any $\beta\in \bB_F$, it is enough to show that $F_{\bB_U, \bB_V}$ is connected for any such decomposition of $\bB_F$. Put 
\[
\bar{\frA}_{\bB_U,\bB_V}=\Fpbar[U_\beta, V_\beta\mid \beta\in\bB_F]/(U_\beta\ (\beta\in \bB_U), V_\beta\ (\beta\in\bB_V)).
\]
Note that the $\bar{\frA}_{\bB_U,\bB_V}$-algebra
\[
\bar{\frA}_{\bB_U,\bB_V}[W_\beta\mid \beta\in \bB_F]/(\bar{G}_\beta\mid \beta\in \bB_F)
\]
is flat. From this we see that the affine algebra of $F_{\bB_U,\bB_V}$ can be identified with the subring
\[
\bar{\frA}_{\bB_U,\bB_V} [\frac{V_\beta}{1+g'_\beta V_\beta V_{\sigma^{-1}\circ\beta}}\mid \beta\in \bB_F]
\]
of $\Frac(\bar{\frA}_{\bB_U,\bB_V})$, which is an integral domain. Hence we obtain the connectedness of $\bar{\frB}_{Q,\bC_p}$. By \cite[Lemma 7.1.9]{deJong}, $\spc^{-1}(\bar{Q})$ is reduced and \cite[Lemma 3.3.1 (1)]{Con_irr} shows that $\cV_Q\hat{\otimes}_{K'} \Cp$ is also reduced. Then \cite[Proposition 1.1]{BLR4} and \cite[Remark after Proposition 6.3.4/1]{BGR} imply that $\frB_{Q,\bC_p}$ is integrally closed in $B_{Q,\bC_p}$ and thus we have
\[
\pi_0(\cV_{Q}\hat{\otimes}_{K'}\bC_p)\simeq \pi_0(\Spec(\frB_{Q,\bC_p}))\simeq \pi_0(\Spec(\bar{\frB}_{Q,\bC_p})).
\]
This shows that $\cV_{Q}\hat{\otimes}_{K'}\bC_p$ is connected.
\end{proof}

\begin{lem}\label{critcanlem}
Suppose $f_\frp\leq 2$ for any $\frp\mid p$. Let $L/K$ be a finite extension. Let $[(A,\cH)]$ be an element of $Y_{\frc,p}(\oel)$ satisfying
\[
\deg_\beta(A[p]/\cH)\leq p/(p+1),\quad \Hdg_\beta(A)\leq p/(p+1)
\]
for any $\beta\in\bB_F$. Then, for any $\frp\mid p$, we have either $A[p]_\frp$ has the canonical subgroup of level one which is equal to $\cH_\frp$, or $\Hdg_\beta(A)=p/(p+1)$ for any $\beta\in \bB_\frp$.
\end{lem}
\begin{proof}
Suppose $\Hdg_{\beta_0}(A)<p/(p+1)$ for some $\beta_0\in \bB_\frp$. Since we have $\Hdg_\beta(A)\leq p/(p+1)$ for any $\beta\in \bB_F$, the assumption on $f_\frp$ implies that the inequality
\[
\Hdg_\beta(A)+ p  \Hdg_{\sigma^{-1}\circ\beta}(A)<p
\]
holds for any $\beta\in \bB_\frp$. By Theorem \ref{cansub1}, the $\cO_{F_\frp}$-$\ADBT_1$ $A[p]_\frp$ has the canonical subgroup $\cC_\frp$. 

Suppose $\cH_\frp\neq \cC_\frp$. For any $\beta\in \bB_{\frp}$, Corollary \ref{cncisog} (\ref{noncanisog}) implies that
\[
\Hdg_{\beta}(p^{-1}\cH_{\frp}/\cH_{\frp})=p^{-1}\Hdg_{\sigma\circ\beta}(A[p]_{\frp})=p^{-1}\Hdg_{\sigma\circ\beta}(A[p])\leq 1/(p+1)
\]
and that $A[p]_{\frp}/\cH_{\frp}$ is the canonical subgroup of $p^{-1}\cH_{\frp}/\cH_{\frp}$. Thus we have
\[
\deg_\beta(A[p]/\cH)=\deg_\beta(A[p]_{\frp}/\cH_{\frp})=1-\Hdg_{\beta}(p^{-1}\cH_{\frp}/\cH_{\frp})\geq p/(p+1),
\]
which yields $\deg_\beta(A[p]/\cH)=p/(p+1)$ and $\Hdg_\beta(A[p])=p/(p+1)$
for any $\beta\in \bB_{\frp}$. This contradicts the choice of $\beta_0$.
\end{proof}

\begin{cor}\label{UpExt}
Suppose $f_\frp\leq 2$ for any $\frp\mid p$. Let $L/K$ be a finite extension. 
Let $[(A',\cH')]$ be an element of $Y_{\frc,p}(\oel)$ such that $[(A'_L,\cH'_L)]\in \cV_Q(L)$. Then, 
for any finite flat closed $p$-cyclic $\cO_F$-subgroup scheme $\cD$ of $A'[p]$ over $\oel$ satisfying $\cD_L\cap\cH'_L=0$, we have 
\[
\Hdg_\beta(A'/\cD)\leq 1/(p+1)
\]
for any $\beta\in \bB_F$ and $A'[p]/\cD$ is the canonical subgroup of $A'/\cD$ of level one.
\end{cor}
\begin{proof}
Write as  $\cD=\bigoplus_{\frp\mid p} \cD_\frp$. The assumption implies $\cD_\frp\neq \cH'_\frp$ for any $\frp\mid p$.
If $\cH'_\frp$ is the canonical subgroup of $A'[p]_\frp$, then Corollary \ref{cncisog} (\ref{noncanisog}) implies that
\[
\Hdg_\beta(A'/\cD)=\Hdg_\beta(p^{-1}\cD_\frp/\cD_\frp)=p^{-1}\Hdg_{\sigma\circ\beta}(A'[p])\leq 1/(p+1)
\] 
for any $\beta\in \bB_\frp$ and that $A'[p]_\frp/\cD_\frp$ is the canonical subgroup of $p^{-1}\cD_\frp/\cD_\frp=(A'/\cD)[p]_\frp$. Otherwise, Lemma \ref{critcanlem} yields $\Hdg_\beta(A')=p/(p+1)$ for any $\beta\in \bB_\frp$. By Proposition \ref{critisog}, we see that
\[
\deg_\beta(A'[p]_\frp/\cD_\frp)=p/(p+1),\quad \Hdg_\beta((A'/\cD)[p]_\frp)=1/(p+1)
\]
for any $\beta\in\bB_\frp$ and that $(A'/\cD)[p]_\frp$ has the canonical subgroup $A'[p]_\frp/\cD_\frp$. Hence the HBAV $A'/\cD$ satisfies
\[
\Hdg_\beta(A'/\cD)\leq 1/(p+1)
\]
for any $\beta\in \bB_F$ and it has the canonical subgroup
\[
A'[p]/\cD=\bigoplus_{\frp\mid p} A'[p]_\frp/\cD_\frp
\]
of level one. This concludes the proof of the corollary.
\end{proof}

\begin{lem}\label{VQcan}
Suppose $f_\frp\leq 2$ for any $\frp\mid p$. Then we have 
\[
\cV_Q(\tfrac{1}{p+1})\neq \emptyset.
\]
Moreover, for any finite extension $L/K$ and any element $[(A',\cH')]$ of $Y_{\frc,p}(\oel)$ satisfying $[(A'_L,\cH'_L)]\in \cV_Q(\tfrac{1}{p+1})(L)$, the HBAV $A'$ has the canonical subgroup $\cH'$. 
\end{lem}
\begin{proof}
Recall that we have $\spc^{-1}(\bar{Q})=(\Spf(\frB'))^\rig$ with the ring $\frB'$ of (\ref{frBdef}) in the proof of Proposition \ref{ConnCrit}. From the description of $\deg_\beta$ in terms of the parameter $X_\beta$ of the ring $\frB'$, we see that there exists a point $[(A',\cH')]\in Y_{\frc,p}(\oel)$ with some finite extension $L/K$ such that $[(A'_L,\cH'_L)]\in \spc^{-1}(\bar{Q})$ and
\[
\deg_\beta(A'[p]/\cH')=1/(p+1)
\]
for any $\beta\in \bB_F$. Then Lemma \ref{LemGKVar} (\ref{LemGKVarV}) implies that $\Hdg_\beta(A')=1/(p+1)$ for any $\beta\in \bB_F$ and thus $[(A'_L,\cH'_L)]\in \cV_{Q}(\tfrac{1}{p+1})(L)$. The last assertion follows from Theorem \ref{cansub1}.
\end{proof}

Since $\cY_{\frc,p}$ is separated, Proposition \ref{ConnCrit} implies that the base extension $\cV_{Q,\bC_p}=\cV_Q\hat{\otimes}_K \bC_p$ is an admissible affinoid open subset of $\cY_{\frc,p,\bC_p}$ whose connected components are all isomorphic to $\cV_Q\hat{\otimes}_{K'} \bC_p$. Each connected component contains an affinoid subdomain of $\cV_Q(\tfrac{1}{p+1})\hat{\otimes}_K \bC_p$ which is isomorphic to $\cV_Q(\tfrac{1}{p+1})\hat{\otimes}_{K'} \bC_p$. By Lemma \ref{VQcan}, we have
\[
\cV_Q(\tfrac{1}{p+1})\hat{\otimes}_{K'} \bC_p\neq \emptyset.
\] 
The point $Q\in \cY_{\frc,p}(L)$ defines a point of $\cY_{\frc,p,\bC_p}(\bC_p)$ by the natural inclusion $L\to \bC_p$, which we also denote by $Q$. Let $\cV^0_{Q,\bC_p}$ be the connected component of $\cV_{Q,\bC_p}$ containing $Q$ and $\cV^0_{Q,\bC_p}(\tfrac{1}{p+1})$ be a copy of $\cV_Q(\tfrac{1}{p+1})\hat{\otimes}_{K'} \bC_p$ which is contained in $\cV^0_{Q,\bC_p}$. These are both non-empty admissible affinoid open subsets of $\cY_{\frc,p,\bC_p}$.

%---------------------------------------------------------------------

%---------------------------------------------------------------------

\subsection{Overconvergent Hilbert modular forms and the eigenvariety}\label{SecOCHMF}

In this subsection, we recall the construction of sheaves of overconvergent Hilbert modular forms and the associated eigenvariety, due to Andreatta-Iovita-Pilloni \cite{AIP2}.

%---------------------------------------------------------------------

%---------------------------------------------------------------------

\subsubsection{Overconvergent modular forms over Hilbert modular varieties}\label{SecOCHMFDef}

Put $\bT=\mathrm{Res}_{\cO_F/\bZ}(\Gm)$. Let $\hat{\bT}$ be its formal completion along the unit section. For any $w \in e^{-1}\bZ_{\geq 1}$, let $\bT_w^0$ be the formal subgroup scheme of $\hat{\bT}$ over $\Spf(\okey)$ representing the functor
\[
\frB \mapsto \Ker(\bT(\frB)\to \bT(\frB/\pi^{e w}\frB)).
\]
on the category of admissible formal $\okey$-algebras $\frB$. Then $\bT^0_w$ is a quasi-compact admissible formal group scheme over $\okey$.

Let $\cW$ be the Berthelot generic fiber of $\Spf(\cO_K[[\bT(\bZ_p)]])$ and we denote the universal character on this space by
\[
\kappa^\univ: \bT(\bZ_p) \to \cO^\circ(\cW)^\times=\cO_K[[\bT(\bZ_p)]]^\times.
\]
Here $\cO^\circ$ is the sheaf of rigid analytic functions with absolute value bounded by one and the last equality follows from \cite[Theorem 7.4.1]{deJong}.
For any morphism $\cX\to \cW$ of rigid analytic varieties over $K$, we denote by $\kappa^\cX$ the restriction 
\[
\kappa^\cX:\bT(\bZ_p)\overset{\kappa^\univ}{\to}\cO^\circ(\cW)^\times \to \cO^\circ(\cX)^\times
\]
of $\kappa^\univ$ to $\cX$. Consider the case where $\cX$ is a reduced $K$-affinoid variety $\cU=\Spv(A)$. Then the subring $A^\circ$ of power-bounded elements is $p$-adically complete. For any positive integer $n$, put $q_n=2$ if $p=2$ and $n=1$, and $q_n=1$ otherwise. When we consider the case of $p=2$ and $n=1$, we assume that $2$ splits completely in $F$. The character $\kappa^\cU$ is said to be $n$-analytic if the restriction to $\bT^0_n(\bZ_p)$ factors as
\[
\xymatrix{
\bT^0_n(\bZ_p)\ar@{=}[r]&1+p^n(\cO_F\otimes \bZ_p) \ar[r]^-{\kappa^\cU}\ar[d]_{\log} & (A^\circ)^\times \\
&q_n p^n(\cO_F\otimes \bZ_p)\ar[r]_-{\psi} & q_n^2 p^n A^\circ\ar[u]_{\exp}
}
\]
with some $\bZ_p$-linear map $\psi$. In this case, we also say that the morphism $\cU\to \cW$ is $n$-analytic. Any $\kappa^{\cU}$ is $n$-analytic for some $n$ by the maximal modulus principle. Note that any $n$-analytic character defines an analytic character $\bT^0_n(\bZ_p)\to A^\times$, even for the case of $p=2$ and $n=1$.

Proposition \ref{cansubAV} and Lemma \ref{famcanlem} enable us to generalize the construction in \cite[\S 3.3]{AIP2}. Let $n$ be a positive integer and $\uv=(v_\beta)_{\beta\in \bB_F}$ a $g$-tuple in $[0,(p-1)/p^{n})\cap \bQ$. Put $v=\max\{v_\beta\mid \beta\in \bB_F\}$. Let $C_n$ be the canonical subgroup of $\bar{\cA}^\univ$ of level $n$ over $\bar{\cM}(\mu_N,\frc)(\uv)$, as before. Put
\[
\bar{\cM}(\Gamma_1(p^n),\mu_N,\frc)(\uv)=\Isom_{\bar{\cM}(\mu_N,\frc)(\uv)}(C_n, \cD_F^{-1}\otimes \mu_{p^n}).
\]
We denote by $\bar{\SGm}(\Gamma_1(p^n),\mu_N,\frc)(\uv)$ the normalization of $\bar{\SGm}(\mu_N,\frc)(\uv)$ in $\bar{\cM}(\Gamma_1(p^n), \mu_N, \frc)(\uv)$. Note that, since $\cC^\vee_n$ is finite and etale over $\bar{\SGm}(\mu_N,\frc)(0)$, we have
\begin{equation}\label{ordmoduli}
\bar{\SGm}(\Gamma_1(p^n),\mu_N,\frc)(0)=\Isom_{\bar{\SGm}(\mu_N,\frc)(0)}(\cC_n, \cD_F^{-1}\otimes \mu_{p^n}),
\end{equation}
which is a $\bT(\bZ/p^n\bZ)$-torsor over $\bar{\SGm}(\mu_N,\frc)(0)$.

Let $\cF$ be the locally free $\cO_F\otimes \cO_{\bar{\SGm}(\Gamma_1(p^n),\mu_N,\frc)(\uv)}$-module of rank one constructed as in \cite[Proposition 3.3]{AIP2}. Let $w$ be an element of $e^{-1}\bZ$ satisfying $n-1\leq w< n-p^n v/(p-1)$, which exists for a sufficiently large $K$. 
Let
\[
\gamma_w: \frIW_w^+ \to \bar{\SGm}(\Gamma_1(p^n),\mu_N, \frc)(\uv)
\]
be the $p$-adic formal $\bT^0_w$-torsor over 
$\bar{\SGm}(\Gamma_1(p^n),\mu_N,\frc)(\uv)$ classifying, for any $R\in\mathbf{NAdm}$ and any morphism of $p$-adic formal schemes $\gamma:\Spf(R)\to \bar{\SGm}(\Gamma_1(p^n),\mu_N,\frc)(\uv)$, the isomorphisms $\alpha: \gamma^*\cF\to \cO_F\otimes R$ such that the composite
\[
\underline{\cO_F/p^n \cO_F}(R) \overset{\gamma}{\simeq} \cC_n^\vee(R)\overset{\HT_w}{\to} \gamma^*\cF/\pi^{ew} \gamma^*\cF \overset{\alpha}{\simeq} \cO_F\otimes R/\pi^{e w} R
\]
sends $1$ to $1$ \cite[\S 3.4]{AIP2}. We also write $\frIW_w^+$ as $\frIW^+_{w,\frc}(\uv)$. We denote the Raynaud generic fiber of $\frIW_w^+$ by $\cIW_w^+$ and also by $\cIW^+_{w,\frc}(\uv)$. 
From (\ref{ordmoduli}), we see that the moduli interpretation of $\frIW_{w,\frc}^+(0)$ as above is also valid for the category of quasi-idyllic $p$-adic $\okey$-algebras $R$.

For the structure morphism
\[
h_n: \bar{\SGm}(\Gamma_1(p^n),\mu_N,\frc)(\uv)\to \bar{\SGm}(\mu_N,\frc)(\uv),
\]
we put $\pi_w=h_n\circ \gamma_w$. We denote by $\gamma_w^\rig$, $h_n^\rig$ and $\pi_w^\rig$ the induced morphisms on the Raynaud generic fibers.
Let $\bT_w$ be the formal subgroup scheme of $\hat{\bT}$ over $\Spf(\okey)$ whose set of $\frB$-valued points are the inverse image of $\bT(\bZ/p^n\bZ)$ by the map $\bT(\frB)\to \bT(\frB/\pi^{e w}\frB)$ for any admissible formal $\okey$-algebra $\frB$. The natural action of $\bT_w^0$ on $\frIW_w^+$ induces an action of $\bT_w$ on $\frIW_w^+$ over $\bar{\SGm}(\mu_N,\frc)(\uv)$ and also on the Raynaud generic fiber $\cIW_w^+$ over $\bar{\cM}(\mu_N,\frc)(\uv)$. Then, for any reduced $K$-affinoid variety $\cU$ and $n$-analytic morphism $\cU\to \cW$, we define
\[
\Omega^{\kappa^{\cU}}=(\pi_w^\rig)_*(\cO_{\cIW_w^+\times \cU})[-\kappa^{\cU}].
\]
By \cite[Proposition 3.12]{AIP2}, it is an invertible sheaf which is independent of the choices of $n$ and $w$. Let $D$ be the boundary divisor of $\bar{\cM}(\mu_N,\frc)$. We also put
\begin{align*}
M(\mu_N, \frc, \kappa^{\cU})(\uv)&=H^0(\bar{\cM}(\mu_N,\frc)(\uv)\times \cU, \Omega^{\kappa^{\cU}}),\\
S(\mu_N,\frc, \kappa^{\cU})(\uv)&=H^0(\bar{\cM}(\mu_N,\frc)(\uv)\times \cU, \Omega^{\kappa^{\cU}}(-D)).
\end{align*}

Note the equality $\SGm(\mu_N,\frc)(\uv)^\rig=\bar{\cM}(\mu_N,\frc)(\uv)\setminus \spc^{-1}(D_k)$, where $D_k$ is the boundary divisor of the special fiber $\bar{\SGm}(\mu_N,\frc)(\uv)_k$.
For any $R\in \mathbf{NAdm}$, let us consider tuples $(A, \iota,\lambda,\psi,u,\alpha)$ over $R$ consisting of a HBAV $(A,\iota, \lambda,\psi)$ over $\Spec(R)$ such that $\Hdg_\beta(A_x)\leq v_\beta$ for any $x\in \Spv(R[1/p])$, an isomorphism of $\cO_F$-group schemes
\[
u: \cC_n|_{R[1/p]}\simeq  \cD_F^{-1}\otimes\mu_{p^n}
\] 
for the canonical subgroup $\cC_n$ of $A$ and an isomorphism 
\[
\alpha:\gamma^*\cF\simeq \cO_F\otimes R
\]
satisfying the compatibility with $u$ as above. Then any element $f\in H^0(\SGm(\mu_N,\frc)(\uv)^\rig, \Omega^{\kappa^\cU})$
can be identified with a rule functorially associating, with any such tuple over $R$ endowed with a map $\Spv(R[1/p])\to \cU$, an element $f(A,\iota,\lambda,\psi,u,\alpha)$ of $R[1/p]$ satisfying
\[
f(A,\iota,\lambda,\psi,t^{-1}u,t^{-1}\alpha)=\kappa^{\cU}(t) f(A,\iota,\lambda,\psi,u,\alpha)
\]
for any $t\in \bT(\bZ_p)$. Similarly, any element $f\in H^0(\bar{\SGm}(\mu_N,\frc)(0)^\rig, \Omega^{\kappa^\cU})$ has a similar description as a rule over any quasi-idyllic $p$-adic $\okey$-algebra $R$ endowed with a morphism $\Spf(R)\to \bar{\SGm}(\mu_N,\frc)(0)$.

For a later use, we also recall the definition of an integral structure of the sheaf $\Omega^{\kappa^\cU}$ for an $n$-analytic map $\kappa^\cU: \cU=\Spv(A)\to \cW$ with some reduced $K$-affinoid algebra $A$. Note that $A^\circ$ is topologically of finite type \cite[Corollary 6.4.1/6]{BGR} and thus $\frU=\Spf(A^\circ)$ is an admissible formal scheme over $\Spf(\okey)$. The map $\kappa^\cU$ extends to a formal character
\[
\boldsymbol{\kappa}^{\cU}: \bT_w\times \frU\to \hGm\times \frU.
\]
We put
\[
\bOmega^{\kappa^\cU}=(\pi_w)_*(\cO_{\frIW_w^+\times \frU})[-\boldsymbol{\kappa}^\cU].
\]
It is a coherent $\cO_{\bar{\SGm}(\mu_N,\frc)(\uv)\times\frU}$-module which is independent of the choice of $w$ such that its Raynaud generic fiber is $\Omega^{\kappa^\cU}$ \cite[Proposition 3.12]{AIP2}. Since the map $h_n$ is an etale $\bT(\bZ/p^n\bZ)$-torsor over the ordinary locus $\bar{\SGm}(\mu_N,\frc)(0)$, the restriction of $\bOmega^{\kappa^\cU}$ to $\bar{\SGm}(\mu_N,\frc)(0)\times\frU$ is an invertible sheaf.

Let $\kappa:\bT(\bZ_p)\to K^\times$ be a weight character which is integral, namely it is written as
\[
\bT(\bZ_p)=(\cO_F\otimes \bZ_p)^\times \ni t\otimes 1\mapsto \prod_{\beta\in \bB_F}\beta(t)^{k_\beta} \in K^\times
\]
with some $g$-tuple of integers $(k_\beta)_{\beta\in \bB_F}$. In this case, the sheaf $\Omega^{\kappa}$ is isomorphic to the classical automorphic sheaf \cite[Corollary 3.9]{AIP2}. Indeed, consider $\cI=\Isom_{\bar{\cM}(\mu_N,\frc)}(\cO_F\otimes \cO_{\bar{\cM}(\mu_N,\frc)}, \omega_{\bar{\cA}^\univ})$. Since the Raynaud generic fiber of the sheaf $\cF$ is $\omega_{\bar{\cA}^\univ}$, we have a natural map $\cIW_w^+\to \cI$, which induces an isomorphism $\omega_{\bar{\cA}^\univ}^{\kappa}\to \Omega^{\kappa}$. We also say that an integral weight $\kappa$ is even if every $k_\beta$ is even.

Moreover, we say that a weight character $\kappa:\bT(\bZ_p)\to K^\times$ is $n$-integral ({\it resp.} $n$-even) if its restriction to $\bT^0_n(\bZ_p)$ is equal to the restriction of a character of some integral ({\it resp.} even) weight $(k_\beta)_{\beta\in \bB_F}$. Then, from the construction of the sheaf $\Omega^{\kappa}$, we see that the pull-back $(h^{\rig}_n)^*\Omega^{\kappa}$ to $\bar{\cM}(\Gamma_1(p^n),\mu_N,\frc)(\uv)$ is isomorphic to $(h^{\rig}_n)^* (\bigotimes_{\beta\in\bB_F} \omega^{\otimes k_\beta}_{\bar{\cA}^\univ,\beta})$.
Note that for the case where $p=2$ splits completely in $F$, a $1$-integral weight is $1$-analytic if and only if it is $1$-even.

%---------------------------------------------------------------------

%---------------------------------------------------------------------

\subsubsection{Overconvergent arithmetic Hilbert modular forms}\label{SecArithOCHMF}

We define the weight space $\cW^G$ for overconvergent Hilbert modular forms as the Berthelot generic fiber of $\Spf(\cO_K[[\bT(\bZ_p)\times \bZ_p^\times]])$. Any morphism $\cX\to \cW^G$ defines a pair $(\nu^{\cX}, w^{\cX})$ of continuous characters
\[
\nu^{\cX}: \bT(\bZ_p)\to \cO^\circ(\cX)^\times,\quad w^{\cX}:\bZ_p^\times \to \cO^\circ(\cX)^\times
\]
with respect to the supremum semi-norm on $\cX$. The map
\[
\bT(\bZ_p)\to \bT(\bZ_p)\times \bZ_p^\times,\quad t\mapsto (t^2, \Nor_{F/\bQ}(t)) 
\]
induces a morphism $k: \cW^G\to \cW$. For any morphism $\cX\to \cW^G$, put $\kappa^{\cX}=k(\nu^{\cX},w^{\cX})$. 
When $\cX$ is a reduced $K$-affinoid variety, we say that $(\nu^{\cX},w^{\cX})$ is $n$-analytic if $\nu^{\cX}$ and $w^{\cX}$ are both $n$-analytic. Note that if $(\nu^{\cX},w^{\cX})$ is $n$-analytic, then $\kappa^{\cX}$ is also $n$-analytic.
We say that a character $(\nu,w): \bT(\bZ_p)\times \bZ_p^\times\to K^\times$ is integral if it comes from an algebraic character $\bT\times \Gm\to \Gm$. Then it is written as 
\[
\bT(\bZ_p)\times \bZ_p^\times\to K^\times,\quad (t\otimes 1,s)\mapsto \prod_{\beta\in \bB_F} \beta(t)^{k_\beta}s^{k_0}
\]
with some $g$-tuple of integers $(k_\beta)_{\beta\in \bB_F}$ and an integer $k_0$. We say that it is even if every $k_\beta$ and $k_0$ are even. We also say that $(\nu,w)$ is $n$-integral ({\it resp.} $n$-even) if its restriction to $\bT^0_n(\bZ_p)\times (1+p^n\bZ_p)$ is equal to the restriction of some integral ({\it resp.} even) character. 
If $(\nu,w)$ is $n$-integral ({\it resp.} $n$-even), then $k(\nu,w)$ is also $n$-integral ({\it resp.} $n$-even).

Let $\cU$ be a reduced $K$-affinoid variety and $\cU\to \cW^G$ an $n$-analytic morphism. Note that for any $\frc$-polarization $\lambda: A\otimes_{\oef} \frc\to A^\vee$ and any $x\in F^{\times,+}$, the multiplication by $x$ gives an $x^{-1}\frc$-polarization 
\[
x\lambda: A\otimes_{\oef} x^{-1}\frc\overset{\times x}{\simeq} A\otimes_{\oef} \frc\overset{\lambda}{\simeq} A^\vee.
\]
Then the group $\Delta=\cO_F^{\times,+}/U_N^2$ acts on the space $M(\mu_N, \frc, \kappa^{\cU})(\uv)$ by
\[
([\epsilon].f)(A,\iota,\lambda,\psi, u,\alpha)=\nu^\cU(\epsilon)f(A,\iota,\epsilon^{-1}\lambda,\psi, u,\alpha)
\]
for any $f\in M(\mu_N, \frc, \kappa^{\cU})(\uv)$ and $\epsilon\in \cO_F^{\times,+}$. We define
\begin{align*}
M^G(\mu_N,\frc, (\nu^{\cU},w^{\cU}))(\uv)&=M(\mu_N, \frc, \kappa^{\cU})(\uv)^{\Delta},\\
S^G(\mu_N,\frc, (\nu^{\cU},w^{\cU}))(\uv)&=S(\mu_N, \frc, \kappa^{\cU})(\uv)^{\Delta}.
\end{align*}

Let $F^{\times, +,(p)}$ be the subgroup of $F^{\times,+}$ consisting of $p$-adic units. For any $x\in F^{\times,+,(p)}$, we define a map
\[
L_x: M^G(\mu_N,\frc, (\nu^{\cU},w^{\cU}))(\uv) \to M^G(\mu_N,x^{-1}\frc, (\nu^{\cU},w^{\cU}))(\uv)
\]
by the formula
\[
(L_x(f))(A,\iota,\lambda,\psi,u,\alpha)=\nu^{\cU}(x) f(A,\iota, x^{-1}\lambda,\psi,u,\alpha).
\]
Let $\Frac(F)^{(p)}$ be the group of fractional ideals of $F$ which are prime to $p$. Then the spaces
\[
M^G(\mu_N, (\nu^{\cU},w^{\cU}))(\uv),\quad S^G(\mu_N, (\nu^{\cU},w^{\cU}))(\uv)
\]
of arithmetic overconvergent Hilbert modular forms and cusp forms are defined as the quotients
\begin{align*}
&\left(\bigoplus_{\frc\in \Frac(F)^{(p)}} M^G(\mu_N,\frc, (\nu^{\cU},w^{\cU}))(\uv)\right)/(L_x(f)-f\mid x\in F^{\times, +,(p)}),\\
&\left(\bigoplus_{\frc\in \Frac(F)^{(p)}} S^G(\mu_N,\frc, (\nu^{\cU},w^{\cU}))(\uv)\right)/(L_x(f)-f\mid x\in F^{\times, +,(p)}).
\end{align*}
By the same construction, we also have the spaces  
\[
M^G(\mu_N, (\nu^{\cU},w^{\cU}))(v_\tot),\quad S^G(\mu_N, (\nu^{\cU},w^{\cU}))(v_\tot).
\]

%---------------------------------------------------------------------

%---------------------------------------------------------------------

\subsubsection{Hecke operators and the Hilbert eigenvariety}\label{SecHeckeOp}

Next we recall the definition of Hecke operators on the space of overconvergent Hilbert modular forms, following \cite[\S 3.7]{AIP2}. Let $n$, $\uv$, $v$ and $w$ be as above. For any HBAV $(A,\iota,\lambda,\psi)$ over a base scheme $S/\Spec(\okey)$, the closed immersion $\psi: \cD^{-1}_F\otimes\mu_N \to A$ gives a subgroup scheme $\Img(\psi)$ of $A$ which is etale locally isomorphic to $\underline{\cD_F^{-1}/N\cD_F^{-1}}$. Let $\frl$ be any non-zero ideal of $\cO_F$. We define 
\[
\cY'_{\frc,\frl}(\uv)\subseteq \cM(\mu_N,\frc)(\uv)\times \cM(\mu_N,\frl\frc)(\uv)
\]
as the subvariety classifying pairs $((A,\iota,\lambda,\psi),(A',\iota',\lambda',\psi'))$ and an isogeny $\pi_{\frl}:A\to A'$ compatible with the other data such that $\Ker(\pi_{\frl})$ is etale locally isomorphic to $\underline{\cO_F/\frl \cO_F}$, $\Ker(\pi_{\frl})\cap \Img(\psi)=0$ and $\Ker(\pi_{\frl})\cap C_1=0$, where $C_1$ is the canonical subgroup of $A$ of level one. Consider the projections
\[
p_1: \cY'_{\frc,\frl}(\uv)\to \cM(\mu_N,\frc)(\uv),\quad p_2: \cY'_{\frc,\frl}(\uv)\to \cM(\mu_N,\frl\frc)(\uv).
\]
Note that the map $p_1$ is finite and etale. For the case where $\frl$ is a prime ideal dividing $p$, we suppose that $p^{-1}v_{\sigma\circ \beta}\leq v_\beta$ for any $\beta\in \bB_{\frl}$. Set $\uv'=(v'_\beta)_{\beta\in \bB_F}$ by $v'_\beta=v_\beta$ for $\beta\notin \bB_\frl$ and $v'_\beta=p^{-1}v_{\sigma\circ \beta}$ for $\beta\in\bB_\frl$. Then Corollary \ref{cncisog} (\ref{noncanisog}) implies that the map $p_2$ factors through the admissible open subset $\cM(\mu_N,\frl\frc)(\uv')\subseteq \cM(\mu_N,\frl\frc)(\uv)$.

Let $\cU$ be a reduced $K$-affinoid variety and $\cU\to \cW$ an $n$-analytic map. Then Proposition \ref{cansubAV} (\ref{cansubAV_noncan})
and the proof of \cite[Corollary 3.25]{AIP2} (see also \cite[Lemma 6.1.1]{AIP}) show that the map $\pi_{\frl}^*: \omega_{A'}\to \omega_A$ induces an isomorphism 
\[
\pi_{\frl}: p_2^* \cIW_{w,\frl \frc}^+(v) \simeq p_1^*\cIW_{w,\frc}^+(v),
\]
which in turn defines an isomorphism
\[
\pi_{\frl}^*: p_1^*(\Omega^{\kappa^{\cU}})\simeq p_2^*(\Omega^{\kappa^{\cU}}).
\]
This gives the Hecke operator
\begin{align*}
H^0(\cM(\mu_N,\frl\frc)(\uv)\times \cU,\Omega^{\kappa^{\cU}})& \overset{p_2^*}{\to} H^0(\cY'_{\frc,\frl}(\uv)\times \cU,p_2^*\Omega^{\kappa^\cU})\\
& \overset{(\pi_{\frl}^*)^{-1}}{\to} H^0(\cY'_{\frc, \frl}(\uv)\times \cU,p_1^*\Omega^{\kappa^\cU})\\
& \overset{\Nor_{F/\bQ}(\frl)^{-1}\Tr_{p_1}}{\to} H^0(\cM(\mu_N,\frc)(\uv)\times \cU,\Omega^{\kappa^\cU}),
\end{align*}
which can be seen as a map $M(\mu_N,\frl\frc,\kappa^\cU)(\uv)\to M(\mu_N,\frc,\kappa^\cU)(\uv)$ by \cite[Theorem 1.6]{Lut}. We denote this map by $T_\frl$ if $(\frl,p)=1$ and $T'_\frl$ otherwise.

On the other hand, for any ideal $\frl$ with $(\frl,pN)=1$, we have a map
\[
s_\frl: \cM(\mu_N,\frc)(\uv)\to \cM(\mu_N, \frl^2 \frc)(\uv),\quad (A,\iota,\lambda,\psi)\mapsto (A\otimes_{\oef} \frl^{-1}, \iota',\frl^2\lambda, \psi').
\]
Here $\iota'$ and $\psi'$ are induced by $\iota$ and $\psi$ via the natural isogeny $A\to A/A[\frl]\simeq A\otimes_{\oef} \frl^{-1}$, and $\frl^2\lambda$ is the $\frl^2\frc$-polarization on $A\otimes_{\oef} \frl^{-1}$ defined by 
\[
(A\otimes_{\oef} \frl^{-1})\otimes_{\oef} \frl^2\frc =(A\otimes _{\oef}\frc)\otimes_{\oef}\frl\overset{\lambda\otimes 1}{\simeq} A^\vee\otimes_{\oef}\frl\simeq (A\otimes_{\oef}\frl^{-1})^\vee.
\]
Then we can show that there exists a natural isomorphism $\pi_\frl^*: \Omega^{\kappa^\cU} \simeq s_\frl^* \Omega^{\kappa^\cU}$ as in \cite[Lemma 6.1.1]{AIP} and we define the operator 
\[
S_\frl: M(\mu_N,\frl^2\frc,\kappa^\cU)(\uv)\to M(\mu_N,\frc,\kappa^\cU)(\uv)
\]
by $S_\frl=\Nor_{F/\bQ}(\frl)^{-2}(\pi_\frl^*)^{-1}\circ s_\frl^*$. This operator satisfies $S_\frl^m=1$ for some positive integer $m$.

To define arithmetic Hecke operators for $\frl$ with $(\frl,p)\neq 1$, let $v_\frp$ be the normalized additive valuation for any $\frp\mid p$. We fix once and for all elements $x_\frp\in F^{\times,+}$ such that $v_\frp(x_\frp)=1$ and $v_{\frp'}(x_\frp)=0$ for any $\frp'\neq \frp$ satisfying $\frp'\mid p$. We define a map 
\[
x_\frp^*: M(\mu_N, x_\frp^{-1}\frc, \kappa^\cU)(\uv)\to M(\mu_N, \frc,\kappa^\cU)(\uv)
\]
by $f\mapsto ((A,\iota,\lambda,\psi)\mapsto f(A,\iota,x_\frp\lambda,\psi))$. Then we denote the composite 
\[
\prod_{\frp\mid p}(x_\frp^*)^{v_\frp(\frl)}\circ T'_\frl:M(\mu_N, \prod_{\frp\mid p}x_\frp^{-v_\frp(\frl)}\frl\frc, \kappa^\cU)(\uv)\to M(\mu_N,\frc,\kappa^\cU)(\uv)
\] 
by $T_\frl$. We also write it as $U_\frl$ if $\frl$ divides a power of $p$.
Then the operators $T_\frl$ for any $\frl$ and $S_\frl$ for $(\frl, pN)=1$ define actions on $M^G(\mu_N,(\nu^\cU,w^\cU))(\uv)$ and $S^G(\mu_N,(\nu^\cU,w^\cU))(\uv)$ which commute with each other. Note that $T_{\frl\frl'}=T_\frl T_{\frl'}$ if $(\frl,\frl')=1$ and that Proposition \ref{cansubAV} (\ref{cansubAV_noncan}) implies
\begin{equation}\label{Heckes}
T_\frem T_{\frem^{s-1}}=\left\{
\begin{array}{ll}
T_{\frem^s}+\Nor_{F/\bQ}(\frem)S_\frem T_{\frem^{s-2}} & (\frem\nmid Np)\\
T_{\frem^s} & (\frem \mid Np)
\end{array}
\right.
\end{equation}
for any maximal ideal $\frem$.

Let $v$ an element of $\bQ\cap(0,\frac{p-1}{p})$. Note that the above definitions of Hecke operators are also valid for $S^G(\mu_N,(\nu^\cU,w^\cU))(v_\tot)$. Then the operator $U_p$ is a compact operator acting on $S^G(\mu_N,(\nu^\cU,w^\cU))(v_\tot)$ which factors as
\[
S^G(\mu_N,(\nu^\cU,w^\cU))(v_\tot)\subseteq S^G(\mu_N,(\nu^\cU,w^\cU))(p^{-1}v_\tot)\to S^G(\mu_N,(\nu^\cU,w^\cU))(v_\tot)
\]
and, for $v<(p-1)/p^2$, also as
\[
S^G(\mu_N,(\nu^\cU,w^\cU))(v_\tot)\to S^G(\mu_N,(\nu^\cU,w^\cU))(p v_\tot)\subseteq  S^G(\mu_N,(\nu^\cU,w^\cU))(v_\tot).
\]

Let $\bT$ be the polynomial ring over $K$ with variables $T_\frl$ for any $\frl$ and $S_\frl$ for $(\frl, pN)=1$. Then the ring $\bT$ acts on $S^G(\mu_N,(\nu^\cU,w^\cU))(v)$ and $S^G(\mu_N,(\nu^\cU,w^\cU))(v_\tot)$ via the Hecke operators defined above.

Now we can construct the eigenvariety from these data, as in \cite[\S 5]{AIP2}. For any positive integer $n$, we fix a positive rational number $v_n<(p-1)/p^{n}$ satisfying $v_n\geq v_{n+1}$ for any $n$. For any admissible affinoid open subset $\cU\subseteq \cW^G$, we put
\[
n(\cU)=\min\{n\in \bZ_{>0}\mid (\nu^\cU,w^\cU)\text{ is $n$-analytic}\}.
\]
We define a Banach $\cO(\cU)$-module $M_{\cU}$ with $\bT$-action as
\[
M_{\cU}=S^G(\mu_N,(\nu^\cU,w^\cU))(v_{n(\cU),\tot}),
\]
on which $U_p$ acts as a compact operator. The proof of \cite[Theorem 4.4]{AIP2} remains valid also for $p=2$ and implies that the $\cO(\cU)$-module $M_\cU$ satisfies the condition $(\mathit{Pr})$. For admissible affinoid open subsets $\cU_1\subseteq \cU_2$ of $\cW^G$, we have $n(\cU_1)\leq n(\cU_2)$ and \cite[Proposition 3.13]{AIP2} yields a map
\[
\alpha_{\cU_1,\cU_2}: M_{\cU_1}\to S^G(\mu_N,(\nu^{\cU_1},w^{\cU_1}))(v_{n(\cU_2),\tot})\simeq M_{\cU_2}\hat{\otimes}_{\cO(\cU_2)}\cO(\cU_1),
\]
where the first arrow is the restriction map. Note that, for any positive rational numbers $v,v'$ satisfying $v'\leq v < pv'<(p-1)/p$, the restriction map
\[
S^G(\mu_N,(\nu^\cU,w^\cU))(v_\tot)\to S^G(\mu_N,(\nu^\cU,w^\cU))(v'_\tot)
\]
is a primitive link. Thus the map $\alpha_{\cU_1,\cU_2}$ is a link satisfying the cocycle condition. Hence, by applying the eigenvariety machine \cite[Construction 5.7]{Buz}, we obtain the Hilbert eigenvariety $\cE\to \cW^G$ as in \cite[Theorem 5.1]{AIP2}.

%---------------------------------------------------------------------

%---------------------------------------------------------------------

\subsection{The case over $\bC_p$}\label{SecCompCp}

Since we are ultimately interested in overconvergent Hilbert modular forms over $\bC_p$, we need to give a slight generalization of the construction in \cite{AIP2} over $\bC_p$. As before, for any quasi-separated rigid analytic variety $\cX$ over $K$ and any coherent $\cO_{\cX}$-module $\cF$, we denote the base extensions of $\cX$ and $\cF$ to $\bC_p$ by $\cX_{\bC_p}$ and $\cF_{\bC_p}$, respectively. Similarly, for any quasi-separated admissible formal scheme $\frX$ over $\Spf(\okey)$ and any coherent $\cO_{\frX}$-module $\frF$, we denote their pull-backs to $\Spf(\OCp)$ by $\frX_{\OCp}$ and $\frF_{\OCp}$, respectively. Then, on the Raynaud generic fiber, we have
\[
(\frX^\rig)_{\Cp}=(\frX_{\OCp})^\rig,\quad (\frF^\rig)_{\Cp}=(\frF_{\OCp})^\rig.
\]

Let $\cU=\Spv(A)$ be a reduced $\bC_p$-affinoid variety. From \cite[Theorem 1.2]{BLR4} and \cite[Proposition 1.10.2 (iii)]{Abbes}, we see that $A^\circ$ is an admissible formal $\OCp$-algebra. Put $\frU=\Spf(A^\circ)$. For any morphism $\cU\to \cW_{\Cp}$ or $\cU\to \cW^G_{\Cp}$, we have an associated character $\kappa^\cU$ or $(\nu^\cU,w^\cU)$ and a notion of $n$-analyticity defined in the same way as above. Consider the base extensions
\[
\pi_{w,\OCp}: \frIW_{w,\OCp}^+ \overset{\gamma_{w,\OCp}}{\to} \bar{\SGm}(\Gamma_1(p^n),\mu_N, \frc)(\uv)_{\OCp}\overset{h_{n,\OCp}}{\to} \bar{\SGm}(\mu_N,\frc)(\uv)_{\OCp}
\]
of the maps $\gamma_w$, $h_n$ and $\pi_w$.
Then, for any $n$-analytic morphism $\cU\to \cW_{\Cp}$, we can define the sheaves
\[
\Omega^{\kappa^\cU}=(\pi^\rig_{w,\bC_p})_*(\cO_{\cIW_{w,\Cp}^+\times \cU})[-\kappa^\cU],\quad \bOmega^{\kappa^\cU}=(\pi_{w,\OCp})_*(\cO_{\frIW_{w,\OCp}^+\times \frU})[-\boldsymbol{\kappa}^\cU]
\]
such that $\Omega^{\kappa^\cU}=(\bOmega^{\kappa^\cU})^\rig$ is an invertible $\cO_{\bar{\cM}(\mu_N,\frc)(v)_{\Cp}\times \cU}$-module, as before. \cite[Proposition 1.9.14 and Proposition 1.10.2 (iii)]{Abbes} implies that $\bOmega^{\kappa^\cU}$ is coherent and that its restriction to $\bar{\SGm}(\mu_N,\frc)(0)_{\OCp}$ is invertible: The latter follows from a similar argument to the proof of \cite[\S 7, Proposition 2]{Mum_AV} combined with the fact that $h_{n,\OCp}$ is a $\bT(\bZ/p^n\bZ)$-torsor over $\bar{\SGm}(\mu_N,\frc)(0)_{\OCp}$. Using $\Omega^{\kappa^\cU}$, we define $M(\mu_N,\frc,\kappa^\cU)(\uv)$ and its variants in the same way as the case over $K$.

For any reduced $K$-affinoid variety $\cV$ and any $n$-analytic morphism $\cV\to \cW$, consider the base extension $\cV_{\Cp}\to \cW_{\Cp}$ and the associated character $\kappa^{\cV_{\Cp}}$. Then we can show that there exist natural isomorphisms
\begin{equation}\label{OmegaBC}
(\Omega^{\kappa^\cV})_{\Cp}\simeq \Omega^{\kappa^{\cV_{\Cp}}},\quad \Omega^{\kappa^\cV}(-D)_{\Cp}\simeq \Omega^{\kappa^{\cV_{\Cp}}}(-D)
\end{equation}
in the same way as the proof of \cite[Proposition 3.13]{AIP2}. Similarly, for any morphism $f:\cU'\to \cU$ of reduced $\Cp$-affinoid varieties,
we have natural isomorphisms
\begin{equation}\label{OmegaBC2}
f^*\Omega^{\kappa^\cU}\simeq \Omega^{\kappa^{\cU'}},\quad f^*(\Omega^{\kappa^\cU}(-D))\simeq \Omega^{\kappa^{\cU'}}(-D).
\end{equation}

Let $\bar{M}^*(\mu_N,\frc)$ be the minimal compactification of $M(\mu_N,\frc)$. We have a natural proper map
\[
\bar{M}(\mu_N,\frc)\to \bar{M}^*(\mu_N,\frc).
\] 
Note that a sufficiently large power of the usual Hasse invariant can be considered as a global section of an ample invertible sheaf on $\bar{M}^*(\mu_N,\frc)$. Let $\bar{\SGm}^*(\mu_N,\frc)(v_\tot)$ be the normal admissible formal scheme defined similarly to $\bar{\SGm}(\mu_N,\frc)(v_\tot)$ using $\bar{M}^*(\mu_N,\frc)$ instead of $\bar{M}(\mu_N,\frc)$. Let $\bar{\cM}^*(\mu_N,\frc)(v_\tot)$ be its Raynaud generic fiber. By the above ampleness property, we see that $\bar{\cM}^*(\mu_N,\frc)(v_\tot)$ is a $K$-affinoid variety. 
We also have proper morphisms 
\begin{align*}
\rho:& \bar{\SGm}(\Gamma_1(p^n),\mu_N,\frc)(v_\tot) \to \bar{\SGm}^*(\mu_N,\frc)(v_\tot),\\
\rho^\rig:& \bar{\cM}(\Gamma_1(p^n),\mu_N,\frc)(v_\tot)\to \bar{\cM}^*(\mu_N,\frc)(v_\tot).
\end{align*}
By the base extension, these induce proper morphisms
\begin{align*}
\rho_{\OCp}:& \bar{\SGm}(\Gamma_1(p^n),\mu_N,\frc)(v_\tot)_{\OCp} \to \bar{\SGm}^*(\mu_N,\frc)(v_\tot)_{\OCp},\\
\rho^\rig_{\Cp}:& \bar{\cM}(\Gamma_1(p^n),\mu_N,\frc)(v_\tot)_{\Cp}\to \bar{\cM}^*(\mu_N,\frc)(v_\tot)_{\Cp}.
\end{align*}

\begin{lem}\label{vanishing}
Let $\cV$ be a reduced $K$-affinoid variety and $\cV\to \cW^G$ an $n$-analytic morphism. Then the natural base change map
\[
(\rho\times 1)_* (\bOmega^{\kappa^\cV}(-D))_{\OCp}\to (\rho_{\OCp}\times 1)_*(\bOmega^{\kappa^{\cV}}(-D)_{\OCp})
\]
is an isomorphism. Moreover, we have
\[
R^q (\rho_{\OCp}\times 1)_*(\bOmega^{\kappa^{\cV}}(-D)_{\OCp})=0
\]
for any $q>0$.
\end{lem}
\begin{proof}
It is enough to show the claim formal locally. Put $\cV=\Spv(A)$ and $\frV=\Spf(A^\circ)$. Let $\frY$ be a formal affine open subscheme of $\bar{\SGm}^*(\mu_N,\frc)(v_\tot)$ and put $\frX=\rho^{-1}(\frY)$. Since $\rho$ is proper of finite presentation and $\bOmega^{\kappa^\cV}(-D)$ is coherent, \cite[(2.11.8.1)]{Abbes} implies that the restriction
\[
R^q(\rho\times 1)_* (\bOmega^{\kappa^\cV}(-D))|_{\frY\times\frV}
\]
is the coherent sheaf associated to the $\cO(\frY\times \frV)$-module 
\[
H^q(\frX\times \frV,\bOmega^{\kappa^\cV}(-D)). 
\]
By \cite[Corollary 3.19]{AIP2}, we have $H^q(\frX\times \frV,\bOmega^{\kappa^\cV}(-D))=0$ for any $q>0$.

Since $\frX$ is quasi-compact, we can take a finite covering $\frX=\bigcup_{i=1}^r \frX_i$ by formal affine open subschemes $\frX_i$. Consider the \v{C}ech complex for the coherent sheaf $\bOmega^{\kappa^\cV}(-D)$
\[
0 \to H^0(\frX\times \frV, \bOmega^{\kappa^\cV}(-D))\to C^0(\bOmega^{\kappa^\cV}(-D))\to C^1(\bOmega^{\kappa^\cV}(-D))\to\cdots
\]
with respect to the covering $\frX\times \frV =\bigcup_{i=1}^r \frX_i\times \frV$, which is exact by the above vanishing. From the definition, we see that the sheaf $\bOmega^{\kappa^\cV}(-D)$ is flat over $\okey$ and each $\okey$-module $C^q(\bOmega^{\kappa^\cV}(-D))$ is also flat. By taking modulo $p^n$, tensoring $\OCp$ and taking the inverse limit, we see that the sequence is exact even after taking $-\hat{\otimes}_{\okey}\OCp$. This means that the \v{C}ech complex for the coherent sheaf $\bOmega^{\kappa^\cV}(-D)_{\OCp}$ with respect to the formal open covering $\frX_{\OCp}\times \frV_{\OCp}=\bigcup_{i=1}^r \frX_{i,\OCp}\times \frV_{\OCp}$ is exact except the zeroth degree. Taking the zeroth cohomology gives an isomorphism
\[
H^0(\frX\times \frV, \bOmega^{\kappa^\cV}(-D))\hat{\otimes}_{\okey}\OCp\to H^0(\frX_{\OCp}\times \frV_{\OCp}, \bOmega^{\kappa^\cV}(-D)_{\OCp})
\]
and the $q$-th cohomology for $q>0$ gives
\[
H^q(\frX_{\OCp}\times \frV_{\OCp}, \bOmega^{\kappa^\cV}(-D)_{\OCp})=0.
\]
This concludes the proof.
\end{proof}

\begin{lem}\label{CpBC}
Let $\cV$ be a reduced $K$-affinoid variety and $\cV\to \cW^G$ an $n$-analytic morphism. Then the natural map
\[
S^G(\mu_N,(\nu^\cV,w^\cV))(v_\tot)\hat{\otimes}_K \Cp \to S^G(\mu_N,(\nu^{\cV_{\Cp}},w^{\cV_{\Cp}}))(v_\tot)
\]
is an isomorphism.
\end{lem}
\begin{proof}
Put $\cV=\Spv(A)$. 
By taking the Raynaud generic fibers and \cite[Proposition 4.7.23 and Proposition 4.7.36]{Abbes}, we see from Lemma \ref{vanishing} that the base change map
\[
(\rho^\rig \times 1)_* (\Omega^{\kappa^\cV}(-D))_{\Cp}\to (\rho_{\Cp}\times 1)_*(\Omega^{\kappa^{\cV}}(-D)_{\Cp})
\]
is an isomorphism. By (\ref{OmegaBC}), the latter sheaf is isomorphic to the sheaf $(\rho_{\Cp}\times 1)_*(\Omega^{\kappa^{\cV_{\Cp}}}(-D))$. Since $\bar{\cM}^*(\mu_N,\frc)(v_\tot)_{\Cp}\times\cV_{\Cp}$ is a $\Cp$-affinoid variety, taking global sections yields an isomorphism
\begin{equation}\label{IsomCpBC}
\begin{aligned}
H^0(\bar{\cM}(\Gamma_1(p^n)&, \mu_N,\frc)(v_\tot)\times \cV,\Omega^{\kappa^\cV}(-D))\hat{\otimes}_K\Cp\to \\
&H^0(\bar{\cM}(\Gamma_1(p^n), \mu_N,\frc)(v_\tot)_{\Cp}\times \cV_{\Cp},\Omega^{\kappa^{\cV_{\Cp}}}(-D)).
\end{aligned}
\end{equation}
Taking the $\bT(\bZ/p^n\bZ)$-equivariant part and the $\Delta$-fixed part, we obtain the lemma.
\end{proof}

\begin{lem}\label{specializeCp}
Let $\cV=\Spv(A)$ be a reduced $K$-affinoid variety. Let $\cV\to \cW^G$ be an $n$-analytic morphism and $x\in \cV(\Cp)$. 
Let $x^*:A \to \Cp$ be the ring homomorphism defined by $x$.
Suppose that the maximal ideal $m_x$ of $A_{\Cp}=A\hat{\otimes}_K \Cp$ corresponding to $x$ is generated by a regular sequence. Put $(\nu,w)=(\nu^\cV(x),w^\cV(x))$. Then the specialization map
\[
S^G(\mu_N, (\nu^{\cV},w^{\cV}))(v_\tot)\hat{\otimes}_{A,x^*}\Cp\to S^G(\mu_N, (\nu,w))(v_\tot)
\]
is an isomorphism.
\end{lem}
\begin{proof}
This is essentially proved in \cite[Proposition 3.22]{AIP2}. 
Put $\kappa^\cV=k(\nu^\cV,w^\cV)$ and $\kappa=k(\nu,w)$. By the assumption on $m_x$, we have the Koszul resolution
\[
0\to A_{\Cp} \to A^{n_{r}}_{\Cp} \to \cdots \to A^{n_1}_{\Cp}\to A_{\Cp} \to A_{\Cp}/m_x \to 0
\]
with some non-negative integers $n_1,\ldots,n_r$, which induces a finite resolution of the sheaf $(1\times x)_*(\Omega^\kappa(-D))$ by finite direct sums of $\Omega^{\kappa^\cV}(-D)_{\Cp}$. By Lemma \ref{vanishing}, the push-forward of this resolution by the map $\rho^\rig_{\Cp}\times 1$ is exact. Since $\bar{\cM}^*(\mu_N,\frc)(v_\tot)_{\Cp}\times \cV_{\Cp}$ is a $\Cp$-affinoid variety, the sequence obtained by taking global sections is also exact. This and (\ref{IsomCpBC}) yield isomorphisms
\begin{align*}
H^0(\bar{\cM}&(\Gamma_1(p^n),\mu_N,\frc)(v_\tot)\times \cV, \Omega^{\kappa^\cV}(-D))\hat{\otimes}_{A,x^*}\Cp\\
&\simeq H^0(\bar{\cM}(\Gamma_1(p^n),\mu_N,\frc)(v_\tot)_{\Cp}\times \cV_{\Cp}, \Omega^{\kappa^{\cV_{\Cp}}}(-D))\hat{\otimes}_{A_{\Cp},x^*}\Cp\\
&\simeq H^0(\bar{\cM}(\Gamma_1(p^n),\mu_N,\frc)(v_\tot)_{\Cp}, \Omega^{\kappa}(-D)).
\end{align*}
Taking the $\bT(\bZ/p^n\bZ)$-equivariant part and the $\Delta$-fixed part shows the lemma.
\end{proof}

We can extend naturally the Hecke operators over $\Cp$: Let $\cU$ be a reduced $\Cp$-affinoid variety and $\cU\to \cW^G_{\Cp}$ an $n$-analytic morphism. Consider the base extension of the isomorphism $\pi_\frl$
\[
\pi_{\frl, \Cp}: p_2^* \cIW_{w,\frl\frc}^+(v)_{\Cp} \simeq p_1^*\cIW_{w,\frc}^+(v)_{\Cp},
\]
which defines an isomorphism
\[
\pi_{\frl,\Cp}^*: p_1^*(\Omega^{\kappa^{\cU}})\simeq p_2^*(\Omega^{\kappa^{\cU}}).
\]
We define the Hecke operator $T_\frl$ over $\Cp$ for $(\frl,p)=1$ by
\begin{align*}
H^0(\cM(\mu_N,\frl\frc)(\uv)_{\Cp}\times \cU,\Omega^{\kappa^{\cU}})& \overset{p_2^*}{\to} H^0(\cY'_{\frc,\frl}(\uv)_{\Cp}\times \cU,p_2^*\Omega^{\kappa^\cU})\\
& \overset{(\pi_{\frl,\Cp}^*)^{-1}}{\to} H^0(\cY'_{\frc,\frl}(\uv)_{\Cp}\times \cU,p_1^*\Omega^{\kappa^\cU})\\
& \overset{\Nor_{F/\bQ}(\frl)^{-1}\Tr_{p_1}}{\to} H^0(\cM(\mu_N,\frc)(\uv)_{\Cp}\times \cU,\Omega^{\kappa^\cU}).
\end{align*}
Similarly, we have Hecke operators $T_\frl$ for $(\frl,p)\neq 1$ and $S_\frl$ over $\Cp$. We can show that they are compatible with the Hecke operators over $K$ and that the specialization map in Lemma \ref{specializeCp} is $\bT$-linear.

%---------------------------------------------------------------------

%---------------------------------------------------------------------

\section{$q$-expansion principle}\label{SecQexp}

In this section, we study the $q$-expansion map for arithmetic overconvergent Hilbert modular forms. 
For any reduced $\bC_p$-affinoid variety $\cU$, any $n$-analytic map $\cU\to \cW^G_{\Cp}$ and any $v\in \bQ\cap[0,\frac{p-1}{p^{n}})$, we have isomorphisms
\begin{align*}
M^G(\mu_N, (\nu^\cU,w^\cU))(v)&\simeq \bigoplus_{\frc\in [\mathrm{Cl}^+(F)]^{(p)}} M^G(\mu_N,\frc, (\nu^\cU,w^\cU))(v),\\
S^G(\mu_N, (\nu^\cU,w^\cU))(v)&\simeq \bigoplus_{\frc\in [\mathrm{Cl}^+(F)]^{(p)}} S^G(\mu_N,\frc, (\nu^\cU,w^\cU))(v)
\end{align*}
by which we identify both sides. For any element $f\in M^G(\mu_N, (\nu^\cU,w^\cU))(v)$, we write $(f_\frc)_{\frc\in [\mathrm{Cl}^+(F)]^{(p)}}$ for the image of $f$ by the above isomorphism. We say that $f$ is an eigenform if it is an eigenvector for any element of $\bT$.

%---------------------------------------------------------------------

%---------------------------------------------------------------------

\subsection{$q$-expansion of overconvergent modular forms}\label{SecQexpDef}

For any $\frc\in [\mathrm{Cl}^+(F)]^{(p)}$, let us consider an unramified cusp $(\fra,\frb,\phi)$ of $M(\mu_N,\frc)$ as in \S \ref{SecToroidal}. Using any polyhedral cone decomposition $\sC\in \Dec(\fra,\frb)$ of $F^{*,+}_\bR$, we have the $\hat{I}_\sigma$-adically complete ring $\hat{R}_\sigma$ and the semi-abelian scheme $\Tate_{\fra,\frb}(q)$ over $\bar{S}_\sigma=\Spec(\hat{R}_\sigma)$ for any $\sigma\in \sC$. 

Let $\breve{S}_\sigma=\Spf(\breve{R}_\sigma)$ be the $(p,\hat{I}_\sigma)$-adic formal completion of $\hat{S}_\sigma$. The smoothness assumption on $\sC$ implies that there exists a basis $\xi_1,\ldots,\xi_g$ of the $\bZ$-module $\fra\frb$ satisfying
\[
(\fra\frb)\cap \sigma^\vee=\bZ_{\geq 0}\xi_1+\cdots+\bZ_{\geq 0}\xi_r+\bZ\xi_{r+1}+\cdots+\bZ\xi_g
\]
with some $r$.
For any ring $B$, we write as
\[
B[X_{\leq r}, X^\pm_{>r}]:=B[X_1,\ldots,X_r,X_{r+1}^\pm,\ldots,X_g^\pm].
\]
For any extension $L/K$ of complete valuation fields, 
we denote the $p$-adic completion of $\oel[X_{\leq r}, X^\pm_{>r}]$ by $\oel\langle X_{\leq r}, X^\pm_{>r}\rangle$ and put 
\[
L\langle X_{\leq r}, X^\pm_{>r} \rangle=\oel\langle X_{\leq r}, X^\pm_{>r}\rangle[1/p].
\]
Then the $\okey$-algebra $\hat{R}_\sigma$ is isomorphic to the completion of the ring $\okey[X_{\leq r}, X^\pm_{>r}]$
with respect to the principal ideal $(X_1\cdots X_r)$ via the map $X_i\mapsto q^{\xi_i}$, and the ring $\breve{R}_\sigma$ is isomorphic to the $p$-adic completion of $\hat{R}_\sigma$. Hence the ring $\breve{R}_\sigma$ is normal and the formal scheme $\breve{S}_{\sigma}$ is an object of the category $\mathrm{FS}_{\okey}$ of \cite[Definition 7.0.1]{deJong}. In fact, the ring $\breve{R}_\sigma$ is isomorphic to the ring
\begin{equation}\label{breveRidentify}
\okey\langle X_{\leq r}, X^\pm_{>r}\rangle [[Z]]/(Z-X_{1}\cdots X_r).
\end{equation}
Moreover, since the natural map
\begin{align*}
\cO_{K,m}[X_{\leq r}, X^\pm_{>r}]&/(X_1\cdots X_r)^n \to \\
&\cO_{K,m}[X_{r+1}^\pm,\ldots,X_g^\pm][[X_1,\ldots,X_r]]/(X_1\cdots X_r)^n 
\end{align*}
is injective for any positive integer $m$, by taking the limit we may identify the rings $\hat{R}_\sigma$ and $\breve{R}_\sigma$ with $\okey$-subalgebras of the $\okey$-algebra
\[
\okey\langle q^{\pm\xi_{r+1}},\ldots, q^{\pm\xi_{g}}\rangle[[q^{\xi_1},\ldots,q^{\xi_r}]].
\]

We denote by $\breve{S}_\sigma^\rig$ the Berthelot generic fiber of $\breve{S}_\sigma$. 
Similarly, we denote by $\breve{S}_{\sC}$ and $\breve{S}^\rig_{\sC}$ the formal completion of $\hat{S}_{\sC}$ along the boundary of the special fiber and its Berthelot generic fiber. From the definition, we have formal open and admissible coverings
\[
\breve{S}_{\sC}=\bigcup_{\sigma\in \sC} \breve{S}_{\sigma},\quad \breve{S}^\rig_{\sC}=\bigcup_{\sigma\in \sC} \breve{S}^\rig_{\sigma}.
\]
Since the quotient of $\hat{S}_{\sC}$ by the action of $U_N$ is obtained by a re-gluing, so is the quotient $\breve{S}_{\sC}/U_N$ and this coincides with the formal completion of $\hat{S}_{\sC}/U_N$ along the boundary of the special fiber.

Consider the case $\sC=\sC(\fra,\frb)$. Since the map $\bar{S}_{\sigma}\to \bar{M}(\mu_N,\frc)$ defined using $\Tate_{\fra,\frb}(q)$ induces an isomorphism
\[
\coprod \hat{S}_{\sC(\fra,\frb)}/U_N \to \bar{M}(\mu_N,\frc)|_{D}^{\wedge}
\]
to the formal completion of $\bar{M}(\mu_N,\frc)|_{D}^{\wedge}$ of $\bar{M}(\mu_N,\frc)$ along the boundary divisor $D$, taking the formal completion we obtain an isomorphism
\[
\coprod \breve{S}_{\sC(\fra,\frb)}/U_N \to \bar{\SGm}(\mu_N,\frc)|_{D_k}^{\wedge}
\]
to the formal completion $\bar{\SGm}(\mu_N,\frc)|_{D_k}^{\wedge}$ of $\bar{M}(\mu_N,\frc)$ along the boundary $D_k$ of the special fiber. Let $\spc: \bar{\cM}(\mu_N,\frc)\to \bar{M}(\mu_N,\frc)_k$ be the specialization map with respect to $\bar{\SGm}(\mu_N,\frc)$. Then \cite[Lemma 7.2.5]{deJong} implies $(\bar{\SGm}(\mu_N,\frc)|_{D_k}^{\wedge})^\rig=\spc^{-1}(D_k)$.

Let $\breve{S}^\rig_{\sigma, \Cp}$ and $\breve{S}^\rig_{\sC, \Cp}$ be the base extensions to $\Spv(\Cp)$ of $\breve{S}^\rig_{\sigma}$ and $\breve{S}^\rig_{\sC}$, respectively. Note that $\breve{S}^\rig_{\sigma,\Cp}$ can be identified with the rigid analytic variety over $\Cp$ whose set of $\Cp$-points is
\begin{equation}\label{descripSCp}
\left\{(x_1,\ldots,x_g)\in \mathbb{C}_p^g\middle| 
\begin{array}{cc}
x_i\in \OCp\ (i\leq r),\ x_i\in \cO^\times_{\Cp}\ (i > r), \\
x_1\cdots x_r\in m_{\Cp}
\end{array}
\right\}
\end{equation}
for $r$ as above. Then, with the notation of \cite[Theorem 3.1.5]{Con_MC}, we have 
\[
(\breve{S}_{\sigma})^\rig_{/\Cp}=\breve{S}^\rig_{\sigma, \Cp},\quad (\breve{S}_{\sC})^\rig_{/\Cp}=\breve{S}^\rig_{\sC, \Cp}.
\]
Since the functor $(-)^\rig_{/\Cp}$ sends formal open immersions to open immersions and formal open coverings to admissible coverings, each $\breve{S}^\rig_{\sigma,\Cp}$ is an admissible open subset of $\breve{S}^\rig_{\sC,\Cp}$ such that $\breve{S}^\rig_{\sC,\Cp}=\bigcup_{\sigma\in \sC} \breve{S}^\rig_{\sigma,\Cp}$ is an admissible covering. Moreover, we have
\[
(\breve{S}_{\sC}/U_N)^\rig_{/\Cp}=\breve{S}^\rig_{\sC,\Cp}/U_N.
\]
Note that the formation of the tube $\spc^{-1}(D_k)$ is compatible with the base extension to $\Cp$ \cite[Proposition 1.1.13]{Berthelot}. Thus, for $\sC=\sC(\fra,\frb)$, we obtain maps
\begin{equation}\label{SCopenim}
\coprod_{\sigma\in \sC} \breve{S}^\rig_{\sigma,\Cp} \to \breve{S}_{\sC,\Cp}^\rig/U_N \to \bar{\cM}(\mu_N,\frc)_{\Cp},
\end{equation}
where the first map is a surjective local isomorphism and the second map is an open immersion factoring through $\bar{\cM}(\mu_N,\frc)(0)_{\Cp}$. 

We denote by $\breve{R}_{\sigma,\OCp}$, $\breve{S}_{\sigma, \OCp}$ and $\breve{S}_{\sC, \OCp}$ the base extensions to $\Spf(\OCp)$ of $\breve{R}_{\sigma}$, $\breve{S}_{\sigma}$ and $\breve{S}_{\sC}$, respectively. From the identification (\ref{breveRidentify}), we can show
\[
\breve{R}_{\sigma,\OCp}=\OCp\langle X_{\leq r}, X^\pm_{>r}\rangle [[Z]]/(Z-X_{1}\cdots X_r).
\]
Indeed, first note that the ring $\breve{R}_{\sigma,\OCp}$ is isomorphic to
\begin{equation}\label{doublelim}
\varprojlim_{n\geq 0}\varprojlim_{m\geq 0} \cO_{\Cp,n}[X_{\leq r}, X^\pm_{>r},Z]/(Z-X_{1}\cdots X_r, Z^m).
\end{equation}
Since the ring
\[
\cO_{\Cp,n}[X_{\leq r}, X^\pm_{>r},Z]/(Z-X_{1}\cdots X_r)
\]
is $Z$-torsion free, its $Z$-adic completion is 
\[
\cO_{\Cp,n}[X_{\leq r}, X^\pm_{>r}][[Z]]/(Z-X_{1}\cdots X_r).
\]
Similarly, since an elementary argument shows that the ring
\[
\cO_{\Cp}[X_{\leq r}, X^\pm_{>r}][[Z]]/(Z-X_{1}\cdots X_r)
\]
is $p$-torsion free, taking the $p$-adic completion yields the claim (The reason of this ad hoc proof is that in general we do not know if the completion is compatible with quotients for non-quasi-idyllic rings). 

\begin{lem}\label{breveRdom}
For any extension $L/K$ of complete valuation fields with residue field $k_L$, the rings
\[
\oel\langle X_{\leq r}, X^\pm_{>r}\rangle [[Z]]/(Z-X_{1}\cdots X_r),\quad k_L[X_{\leq r}, X^\pm_{>r}] [[Z]]/(Z-X_{1}\cdots X_r)
\]
are integral domains. In particular, the ring $\breve{R}_{\sigma,\OCp}$ is an integral domain.
\end{lem}
\begin{proof}
For the former ring, we can show that it is a subring of the ring
\[
\breve{R}_L:=L\langle X_{\leq r}, X^\pm_{>r} \rangle [[Z]]/(Z-X_{1}\cdots X_r).
\]
It suffices to show that $\breve{R}_L$ is an integral domain. Since the ring $L\langle X_{\leq r}, X^\pm_{>r}\rangle$ is Noetherian and normal, the ring $\breve{R}_L$ is also normal. Since $\breve{R}_L$ is $Z$-adically complete, $Z$-torsion free and $\Spec(\breve{R}_L/(Z))$ is connected, we see that $\Spec(\breve{R}_L)$ is also connected and the lemma follows. We can show the assertion on the latter ring similarly.
\end{proof}

From the description (\ref{descripSCp}) of $\breve{S}_{\sigma,\Cp}^\rig$, we see that there exists an inclusion 
\[
\cO(\breve{S}_{\sigma,\OCp})=\breve{R}_{\sigma,\OCp}\subseteq \cO^\circ(\breve{S}^\rig_{\sigma,\Cp}). 
\]
By gluing, this yields an inclusion
\begin{equation}\label{BoundedFnS}
\cO(\breve{S}_{\sC,\OCp}) \subseteq \cO^\circ(\breve{S}_{\sC,\Cp}^\rig).
\end{equation}

By the description (\ref{doublelim}) of the ring $\breve{R}_{\sigma,\OCp}$, we have a natural inclusion
\begin{equation}\label{Rincl}
\breve{R}_{\sigma,\OCp} \subseteq \prod_{\xi\in \fra\frb} \OCp q^{\xi},
\end{equation}
which is compatible with the restriction map 
$\breve{R}_{\sigma,\OCp} \to \breve{R}_{\sigma',\OCp}$ for any $\sigma$ and $\sigma'$ such that $\sigma'$ is a face of the closure $\bar{\sigma}$. Then we have an isomorphism
\[
\cO(\breve{S}_{\sC,\OCp})\simeq \bigcap_{\sigma\in \sC} \breve{R}_{\sigma,\OCp}.
\]
Note that, if the dimension of the $\bR$-vector space $\Span_{\bR}(\sigma)$ generated by the elements of $\sigma$ is $g$, then we have
\[
(\fra\frb)\cap \sigma^\vee=\bZ_{\geq 0} \xi_1\oplus\cdots \oplus \bZ_{\geq 0} \xi_g
\]
with some $\xi_1,\ldots,\xi_g\in \fra\frb$. Thus any element of $\breve{R}_{\sigma,\OCp}$ is a formal power series of $q^{\xi_1},\ldots,q^{\xi_g}$ and the ring $\OCp[[q^\xi\mid \xi\in (\fra\frb)\cap\sigma^\vee]]$ can be identified with the subset
\[
\{(a_\xi q^\xi)_{\xi \in \fra\frb}\in \prod_{\xi\in \fra\frb} \OCp q^\xi \mid a_\xi=0\ \text{ for any } \xi\notin (\fra\frb)\cap \sigma^\vee \}. 
\]
From the equality
\[
(\fra\frb)^+\cup \{0\}=\bigcap\{(\fra\frb)\cap \sigma^\vee\mid \sigma\in\sC,\ \dim_{\bR}(\Span_{\bR}(\sigma))=g\},
\]
we have an inclusion
\[
\OCp[[q^\xi \mid \xi\in (\fra\frb)^+\cup\{0\}]]\supseteq \cO(\breve{S}_{\sC,\OCp}).
\]
On the other hand, if we identify as 
\[
F_\bR\simeq \prod_{\beta\in \Hom_{\bQ\text{-alg.}}(F,\bR)} \bR,\quad x\otimes 1\mapsto (\beta(x))_\beta,
\]
then every boundary $\tau$ of $\sigma^\vee$ is outside the closure of the positive cone $F_\bR^{\times,+}$ of $F_\bR$. Hence, for any positive real number $\rho$, the number of elements $\xi$ of $(\fra\frb)^+$ such that the distance from $\xi$ to $\tau$ is less than $\rho$ is finite. This implies that any element of $\OCp[[q^\xi\mid \xi\in (\fra\frb)^+\cup\{0\}]]$ is contained in the completion of the ring
\[
\OCp[q^{\xi_1},\ldots, q^{\xi_r},q^{\pm \xi_{r+1}},\ldots, q^{\pm \xi_g}]
\]
with respect to the $q^{\xi_1}\cdots q^{\xi_r}$-adic topology. 
We can see that this completion is contained in $\breve{R}_{\sigma,\OCp}$. Therefore, we obtain an identification
\[
\OCp[[q^\xi\mid \xi\in (\fra\frb)^+\cup\{0\}]]\simeq \cO(\breve{S}_{\sC,\OCp})
\]
which is compatible with the inclusion (\ref{Rincl}).

Let $\frc$ be an element of $[\mathrm{Cl}^+(F)]^{(p)}$ and $\fra,\frb$ fractional ideals satisfying $\fra\frb^{-1}=\frc$. Suppose $\fra\subseteq \fro$ and $(\fra,Np)=1$. Then the natural inclusion $\fro\subseteq \fra^{-1}$ induces isomorphisms
\[
\phi_{\fra,\frb}: \fra^{-1}/N\fra^{-1} \simeq \fro/N\fro,\quad \phi'_{\fra,\frb}: \fra^{-1}/p^n \fra^{-1} \simeq \fro/p^n\fro.
\]
Consider the unramified cusp $(\fra,\frb,\phi_{\fra,\frb})$ of $M(\mu_N,\frc)$. Take $\sC\in \Dec(\fra,\frb)$ and $\sigma\in \sC$ as above. By the construction of the Tate object, the map $\phi'_{\fra,\frb}$ yields a natural immersion $\cD_F^{-1}\otimes \mu_{p^n}\to \Tate_{\fra,\frb}(q)$ over $\bar{S}_\sigma$, which induces an isomorphism
\[
\omega_{\Tate_{\fra,\frb}(q)}\otimes_{\okey}\cO_{K,n}\simeq \omega_{\cD_F^{-1}\otimes \mu_{p^n}}.
\]
Note that the map $\Tr_{F/\bQ}\otimes 1: \cD_F^{-1}\otimes\Gm\to \Gm$ gives an element $(\Tr_{F/\bQ}\otimes 1)^*\frac{dT}{T}$ of the $\cO_F\otimes \cO(\bar{S}_\sigma)$-module $\omega_{\cD_F^{-1}\otimes \Gm}\simeq \omega_{\Tate_{\fra,\frb}(q)}$. By the pull-back, we obtain a Tate object over $\Spec(\breve{R}_\sigma)$ with a canonical invariant differential which are compatible with those over $\Spec(\breve{R}_\tau)$ for any $\tau\in \sC$ satisfying $\tau\subseteq \bar{\sigma}$.

We denote the $p$-adic completion of $\bar{S}_\sigma$ by $\tilde{S}_\sigma$. We have $\tilde{S}_\sigma=\Spf(\breve{R}_\sigma)$, where we consider the $p$-adic topology on $\breve{R}_\sigma$. Its base extension to $\OCp$ is denoted by $\tilde{S}_{\sigma,\OCp}=\Spf(\tilde{R}_{\sigma,\OCp})$. Here the affine algebra $\tilde{R}_{\sigma,\OCp}$ is the $p$-adic completion of the ring $\breve{R}_{\sigma}\otimes_{\okey} \OCp$. 

The identity map $\breve{R}_\sigma\to \breve{R}_\sigma$ is continuous if we consider the $p$-adic topology on the source and the $(p,\hat{I}_\sigma)$-adic topology on the target. Then, for the case of $\sC=\sC(\fra,\frb)$, its composite with the $p$-adic completion of the map $\bar{S}_{\sigma}\to \bar{M}(\mu_N,\frc)$ gives a morphism of formal schemes $\breve{S}_{\sigma} \to \tilde{S}_\sigma\to \bar{\SGm}(\mu_N,\frc)(0)$, and also a morphism $\breve{S}_{\sC} \to \bar{\SGm}(\mu_N,\frc)(0)$ by gluing. Since $\breve{R}_\sigma$ is Noetherian, the moduli interpretation of $\frIW_{w,\frc}^+(0)$ as in \S\ref{SecOCHMFDef} is also valid for $\breve{R}_\sigma$.
We have a commutative diagram
\[
\xymatrix{
\underline{\bZ/p^n\bZ}(\breve{R}_\sigma) \ar[r]\ar[d] & \underline{\cO_F/p^n\cO_F}(\breve{R}_\sigma) \ar[d] \\
(\mu_{p^n})^\vee(\breve{R}_{\sigma}) \ar[r]\ar[d]_{\HT}& (\cD_F^{-1}\otimes \mu_{p^n})^\vee(\breve{R}_{\sigma}) \ar[d]^{\HT} & \\
\omega_{\mu_{p^n}}\otimes \breve{R}_{\sigma} \ar[r]& \omega_{\cD_F^{-1}\otimes \mu_{p^n}}\otimes \breve{R}_{\sigma},
}
\]
where the top horizontal arrow is the natural inclusion and the other horizontal arrows are induced by the map $\Tr_{F/\bQ}\otimes 1$. 
Thus the above moduli interpretation and the base extension give a morphism of formal schemes over $\Spf(\OCp)$
\[
\tau_{\fra,\frb}: \breve{S}_{\sigma,\OCp}\to \tilde{S}_{\sigma,\OCp} \to \frIW_{w,\frc}^+(0)_{\OCp}.
\]
By gluing, this defines a morphism $\breve{S}_{\sC,\OCp}\to \frIW^+_{w,\frc}(0)_{\OCp}$, which we also denote by $\tau_{\fra,\frb}$.

\begin{lem}\label{pIcompl}
The natural map $\tilde{R}_{\sigma,\OCp} \to \breve{R}_{\sigma,\OCp}$ is injective. In particular, the ring $\tilde{R}_{\sigma,\OCp}$ is an integral domain.
\end{lem}
\begin{proof}
We have an isomorphism
\begin{equation}\label{doublelimp}
\tilde{R}_{\sigma,\OCp}\simeq \varprojlim_n \varinjlim_{L/K} \cO_{L,n}[X_{\leq r}, X_{>r}^\pm][[Z]]/(Z-X_1\cdots X_r),
\end{equation}
where the direct limit is taken with respect to the directed set of finite extensions $L/K$ in $\bar{\bQ}_p$. Since the map
\[
\cO_{L,n}[X_{\leq r}, X_{>r}^\pm]/(X_1\cdots X_r)^m \to \cO_{\Cp,n}[X_{\leq r}, X_{>r}^\pm]/(X_1\cdots X_r)^m 
\]
is injective for any such $L/K$, the injectivity of the lemma follows from (\ref{doublelim}). Lemma \ref{breveRdom} yields the last assertion. 
\end{proof}

For any finite extension $L/K$, we write the $p$-adic completion 
\[
\breve{R}_\sigma \hat{\otimes}_{\okey} \oel=\cO_{L}\langle X_{\leq r}, X_{>r}^\pm\rangle [[Z]]/(Z-X_1\cdots X_r)
\]
also as $\tilde{R}_{\oel}$. Let $\pi_L$ be a uniformizer of $L$. By Lemma \ref{breveRdom}, the ring $\tilde{R}_{\oel}/(\pi_L)$ is an integral domain. Since $\tilde{R}_{\oel}$ is normal, the localization $(\tilde{R}_{\oel})_{(\pi_L)}$ is a discrete valuation ring with uniformizer $\pi_L$ such that $Z$ is invertible.  Put $\tilde{R}_{\infty}=\varinjlim_{L/K} \tilde{R}_{\oel}$ and $m_\infty=\varinjlim_{L/K} (\pi_L)$, where the direct limits are taken as above. Then the localization $(\tilde{R}_{\infty})_{m_\infty}=\varinjlim_{L/K} (\tilde{R}_{\oel})_{(\pi_L)}$ is a valuation ring. Let $\cO_{\cK_\sigma}$ be its $p$-adic completion. By (\ref{doublelimp}), the ring $\tilde{R}_{\sigma,\OCp}$ coincides with the $p$-adic completion of $\tilde{R}_\infty$. Since the $p$-adic topology on $\tilde{R}_{\infty}$ is induced by that on $(\tilde{R}_{\infty})_{m_\infty}$, we obtain an injection $\tilde{R}_{\sigma,\OCp}\to \cO_{\cK_\sigma}$.
This defines a morphism of $p$-adic formal schemes $\Spf(\cO_{\cK_\sigma})\to \tilde{S}_{\sigma,\OCp}$
for any $\sigma\in \sC(\fra,\frb)$.
In particular, we have the pull-back of $\Tate_{\fra,\frb}(q)$ over $\Spec(\cO_{\cK_\sigma})$ which is a HBAV. Since $\cO_{\cK_\sigma}$ is quasi-idyllic, we have the moduli interpretation of any morphism $\Spf(\cO_{\cK_\sigma}) \to \frIW^+_{w,\frc}(0)_{\OCp}$ over $\bar{\SGm}(\mu_N,\frc)(0)_{\OCp}$ as in \S\ref{SecOCHMFDef}. The additional structures of the Tate object over $\Spec(\breve{R}_\sigma)$ defines a canonical test object
\[
(\Tate_{\fra,\frb}(q),\iota_{\fra,\frb},\lambda_{\fra,\frb},\psi_{\fra,\frb}, u_{\fra,\frb},\alpha_{\fra,\frb})
\]
over $\Spec(\cO_{\cK_\sigma})$. This corresponds via the moduli interpretation to a map
\[
\tau_{\fra,\frb,\cO_{\cK_\sigma}}: \Spf(\cO_{\cK_\sigma})\to  \frIW_{w,\frc}^+(0)_{\OCp}
\]
satisfying the following property: The composite $\breve{S}_{\sigma,\OCp}\to \breve{S}_{\sC,\OCp}\overset{\tau_{\fra,\frb}}{\to}\frIW_{w,\frc}^+(0)_{\OCp}$ factors through $\tilde{S}_{\sigma,\OCp}$ and its restriction to $\Spf(\cO_{\cK_\sigma})$ equals $\tau_{\fra,\frb,\cO_{\cK_\sigma}}$, as in the diagram
\begin{equation}\label{OLres}
\begin{gathered}
\xymatrix{
& \breve{S}_{\sigma,\OCp}\ar[r] \ar[d] & \breve{S}_{\sC,\OCp}\ar[d]^{\tau_{\fra,\frb}} \\
\Spf(\cO_{\cK_\sigma}) \ar[r] & \tilde{S}_{\sigma,\OCp}\ar[r] & \frIW_{w,\frc}^+(0)_{\OCp}.
}
\end{gathered}
\end{equation}

Let $\kappa\in \cW(\Cp)$ be any $n$-analytic weight. Since the formal scheme $\bar{\SGm}(\mu_N,\frc)(v)_{\OCp}$ is quasi-compact and the sheaf $\bOmega^\kappa$ is coherent, we have
\[
M(\mu_N,\frc,\kappa)(v)=H^0(\bar{\SGm}(\mu_N,\frc)(v)_{\OCp}, \bOmega^\kappa)[1/p]\subseteq \cO(\frIW_{w,\frc}^+(v)_{\OCp})[1/p].
\]
For any element $f_\frc$ of $M(\mu_N,\frc,\kappa)(v)$, we define the $q$-expansion $f_\frc(q)$ of $f_\frc$ by
\[
f_{\frc}(q)=\tau^*_{\fro,\frc^{-1}}(f_\frc)\in \cO(\breve{S}_{\sC,\OCp})[1/p]=\OCp[[q^\xi\mid \xi\in (\frc^{-1})^+\cup\{0\}]][1/p].
\]
Thus, for any $f=(f_\frc)_{\frc\in[\mathrm{Cl}^+(F)]^{(p)}}$, we can write as
\[
f_\frc(q)=a_{\fro,\frc^{-1}}(f,0)+\sum_{\xi\in (\frc^{-1})^+} a_{\fro,\frc^{-1}}(f,\xi)q^\xi
\]
with some $a_{\fro,\frc^{-1}}(f,\xi)\in \Cp$. For any refinement $\sC'\in \Dec(\fro,\frc^{-1})$ of $\sC$, the natural map $\breve{S}_{\sC',\OCp}\to \breve{S}_{\sC,\OCp}$ induces the identity map on the ring $\OCp[[q^\xi\mid \xi\in (\frc^{-1})^+\cup\{0\}]][1/p]$. Thus we can compute the $q$-expansion by taking any refinement of the fixed cone decomposition $\sC(\fro,\frc^{-1})$ in $\Dec(\fro,\frc^{-1})$.
We say that an eigenform $f$ is normalized if $a_{\fro,\fro}(f,1)=1$.

%---------------------------------------------------------------------

%---------------------------------------------------------------------

\subsection{Weak multiplicity one theorem}\label{SecWeakMulti}

Let $(\nu,w)\in \cW^G(\Cp)$ be an $n$-analytic weight.  Let $f=(f_\frc)_{\frc\in[\mathrm{Cl}^+(F)]^{(p)}}$ be a non-zero eigenform in $S^G(\mu_N, (\nu,w))(v)$.  
For any non-zero ideal $\frn\subseteq \fro$, let $\Lambda(\frn)$ be the eigenvalue of $T_\frn$ acting on $f$. We set $\Phi(\frn)$ to be the eigenvalue of $S_\frn$ for $(\frn,Np)=1$ and $\Phi(\frn)=0$ otherwise. We put $\mu_\frn=(\cD_F^{-1}\otimes\Gm)[\frn]$. Any element $\zeta$ of $\mu_{\frn}(L)\subseteq (\cD^{-1}_F\otimes \Gm)(L)$ with some extension $L/K$ defines a ring homomorphism $\zeta: \cO(\cD^{-1}_F\otimes \Gm)\to L$. We put $\zeta^\eta=\zeta(\frX^\eta)$ for any $\eta\in \fro$, which gives a homomorphism $\fro/\frn\to L^\times$. We fix an element $\frc\in [\mathrm{Cl}^+(F)]^{(p)}$.

%---------------------------------------------------------------------

%---------------------------------------------------------------------

\subsubsection{$q$-expansion and Hecke operators}\label{SecQExpHeckeOp}

For any $\sC\in \Dec(\fra,\frb)$ and any maximal ideal $\frem$ of $\fro$, we can find $\sC'\in \Dec(\fra,\frem^{-1}\frb)$ which is a refinement of $\sC$. For any $\sigma\in \sC$ and $\tau\in \sC'$ satisfying $\sigma\supseteq \tau$, we have natural maps $\hat{R}_\sigma\to \hat{R}_\tau$, $\hat{R}_\sigma^0\to \hat{R}_\tau^0$ and $\breve{R}_\sigma\to \breve{R}_\tau$. Consider the case $\fra=\fro$. Let $\zeta$ be an element of $\mu_\frem(K)$. Fix an isomorphism of $\fro$-modules 
\[
\rho: \frem^{-1}\frb/\frb\simeq \fro/\frem.
\]
Then we have a natural ring homomorphism 
\[
q\zeta^\rho: \hat{R}_\tau\to \hat{R}_\tau,\quad q^\xi \mapsto q^\xi \zeta^{\rho(\xi)}. 
\]
We denote by $\Tate_{\fro,\frem^{-1}\frb}(q\zeta^\rho)$ the pull-back of $\Tate_{\fro,\frem^{-1}\frb}(q)$ by this map.

On the other hand, we have $\Dec(\fra,\frb)=\Dec(\fra,\eta\frb)$ for any cusp $(\fra,\frb,\phi)$ and $\eta\in F^{\times,+}$. Thus any $\sigma\in \sC$ gives similar rings to $\hat{R}_\sigma$, $\hat{R}_\sigma^0$ and $\breve{R}_\sigma$ for the cusp $(\fra,\eta\frb, \phi)$, which are denoted by $\hat{R}_{\eta,\sigma}
$, $\hat{R}_{\eta,\sigma}^0$ and $\breve{R}_{\eta,\sigma}$, respectively. We have a natural ring homomorphism
\[
q^\eta: \hat{R}_\sigma \to \hat{R}_{\eta,\sigma},\quad q^\xi\mapsto q^{\xi\eta}.
\]
We denote by $\Tate_{\fra,\frb}(q^\eta)$ the pull-back of $\Tate_{\fra,\frb}(q)$ by this map.

We will omit entries of test objects $(A,\iota,\lambda,\psi,u,\alpha)$ for overconvergent Hilbert modular forms if they are clear from the context.

\begin{lem}\label{transform}
We have an isomorphism of test objects over $\hat{R}_{\eta,\sigma}^0$
\[
(\Tate_{\fro,\eta\frc^{-1}}(q),\lambda_{\fro,\eta\frc^{-1}})\simeq (\Tate_{\fro,\frc^{-1}}(q^\eta),\eta\lambda_{\fro,\frc^{-1}}).
\]
\end{lem}
\begin{proof}
We denote by $|_{q^\eta}$ the pull-back along the map $q^\eta$. Consider the composite
\[
\frc^{-1} \to \cD_F^{-1}\otimes \Gm(\hat{R}_\sigma^0)\to \cD_F^{-1}\otimes \Gm|_{q^\eta}(\hat{R}_{\eta,\sigma}^0)
\]
of the map $\alpha \mapsto (\frX^\xi\mapsto q^{\alpha\xi}\ (\xi\in \fro))$ and the map $q^\eta$, which we also denote by $q^\eta$. We also have a similar map $q^\eta: \eta\frc^{-1}\to \eta\cD_F^{-1}\otimes \Gm|_{q^\eta}(\hat{R}_{\eta,\sigma}^0)$. Then the following diagram over $\hat{R}_{\eta,\sigma}^0$ is commutative.
\[
\xymatrix{
\eta\frc^{-1} \ar[ddd]_{\times \eta^{-1}}\ar[rd]\ar@{=}[rrr]& & & \eta\frc^{-1}\ar[ld]\ar@{=}[ddd] \\
 & \cD_F^{-1}\otimes \Gm(\hat{R}_{\eta,\sigma}^0) \ar@{=}[r]\ar@{=}[d]& \cD_F^{-1}\otimes \Gm(\hat{R}_{\eta,\sigma}^0) \ar[d]_{\wr}^{\times \eta}& \\
 & \cD_F^{-1}\otimes \Gm|_{q^\eta}(\hat{R}_{\eta,\sigma}^0) \ar[r]^{\sim}_{\times \eta} & \eta\cD_F^{-1}\otimes \Gm|_{q^\eta}(\hat{R}_{\eta,\sigma}^0)\\
 \frc^{-1} \ar[ur]_{q^\eta}\ar[rrr]_{\times \eta} & & & \eta\frc^{-1}\ar[ul]^{q^\eta}
}
\]
This yields an isomorphism $\Tate_{\fro,\eta\frc^{-1}}(q)\to \Tate_{\fro,\frc^{-1}}(q^\eta)$ as in the lemma.
\end{proof}

\begin{lem}\label{HeckeQexpT}
Let $\frem$ be a maximal ideal of $\fro$ satisfying $\frem\nmid pN$. Let $\frc$ be an element of $[\mathrm{Cl}^+(F)]^{(p)}$. Take any elements $x,y\in F^{\times,+,(p)}$ such that $\frc'=x\frem \frc$ and $\frc''=xy^{-1}\frem^{-1}\frc$ are elements of $[\mathrm{Cl}^+(F)]^{(p)}$. Fix an isomorphism of $\fro$-modules $\rho:(x\frem\frc)^{-1}/(x\frc)^{-1}\simeq \fro/\frem$. Then we have
\[
(T_\frem f)_{\frc}(q)=\frac{\nu(x)}{\Nor_{F/\bQ}(\frem)}\left(\frac{\Nor_{F/\bQ}(\frem)^2\Phi(\frem)}{\nu(y)} f_{\frc''}(q^{xy^{-1}})+\sum_{\zeta\in \mu_\frem(\bar{\bQ}_p)}f_{\frc'}(q^x\zeta^\rho)\right).
\]
\end{lem}
\begin{proof}
For any $\sC\in \Dec(\fro,\frc^{-1})$ and $\sC'\in \Dec(\frem, \frc^{-1})$, we choose $\sC''\in \Dec(\fro, (\frem\frc)^{-1})$ such that $\sC''$ is a common refinement of $\sC$ and $\sC'$. For any $\sigma\in \sC$ and $\sigma'\in \sC'$, take $\tau\in \sC''$ satisfying $\tau\subseteq \sigma,\sigma'$. By the diagram (\ref{OLres}) and the inclusions
\[
\breve{R}_{\tau,\OCp}\supseteq \tilde{R}_{\tau,\OCp}\subseteq\cO_{\cK_\tau}, 
\]
it is enough to show the equality of the lemma after pulling back to $\Spf(\cO_{\cK_\tau})$.

Choose an element $\xi_\frem\in (x\frem\frc)^{-1}$ which gives a generator of the principal $\fro/\frem$-module $(x\frem\frc)^{-1}/(x\frc)^{-1}$. We define an element $Q\in \cD_F^{-1}\otimes\Gm(\hat{R}^0_{x^{-1},\tau})$ by $\frX^\eta\mapsto q^{\xi_\frem\eta}$ for any $\eta\in \fro$. Then, over $\Spec(\cO_{\cK_\tau})$, the $\frem$-cyclic $\cO_F$-subgroup schemes of the Tate object $\Tate_{\fro,(x\frc)^{-1}}(q)$ are exactly those induced by the closed subgroup schemes 
\[
\mu_\frem,\quad \cH_{Q,\zeta}:=(\fro/\frem) Q \zeta\quad (\zeta\in \mu_\frem(\bar{\bQ}_p))
\]
of $\cD_F^{-1}\otimes \Gm$. 
Then the pull-back of $(T_\frem f)_{\frc}(q)$ is equal to
\begin{align*}
(L_x T_\frem f_{\frc'})(\Tate_{\fro,\frc^{-1}}(q),\lambda_{\fro,\frc^{-1}})&=\nu(x) (T_\frem f_{\frc'})(\Tate_{\fro,\frc^{-1}}(q),x^{-1}\lambda_{\fro,\frc^{-1}})\\
&=\nu(x) (T_\frem f_{\frc'})(\Tate_{\fro,(x\frc)^{-1}}(q^x),\lambda_{\fro,(x\frc)^{-1}})
\end{align*}
which equals
\[
\frac{\nu(x)}{\Nor_{F/\bQ}(\frem)}\left( f_{\frc'} (\Tate_{\fro,(x\frc)^{-1}}(q^x)/\mu_\frem)+\sum_{\zeta\in\mu_\frem(\bar{\bQ}_p)}f_{\frc'}(\Tate_{\fro,(x\frc)^{-1}}(q^x)/\cH_{Q,\zeta}|_{q^x}) \right).
\]

For the first term, we have the exact sequence
\[
\xymatrix{
0\ar[r] & \mu_\frem \ar[r] & \cD_F^{-1}\otimes \Gm\ar[r] & \frem^{-1}\cD_F^{-1}\otimes \Gm\ar[r] & 0.
}
\]
For any $\xi\in (x\frc)^{-1}$, the natural map $\cD_F^{-1}\otimes \Gm\to \frem^{-1}\cD_F^{-1}\otimes \Gm$ sends the $\hat{R}_{x^{-1},\tau}^0$-valued point $(\frX^\eta\mapsto q^{\xi\eta}\ (\eta\in \fro))$ to $(\frX^\eta\mapsto q^{\xi\eta}\ (\eta\in \frem))$ and this gives an isomorphism
\[
\Tate_{\fro,(x\frc)^{-1}}(q)/\mu_\frem\simeq \Tate_{\frem,(x\frc)^{-1}}(q)
\]
compatible with natural additional structures. This implies that the evaluation $f_{\frc'} (\Tate_{\fro,(x\frc)^{-1}}(q^x)/\mu_\frem)$ equals
\begin{align*}
f_{\frc'} (\Tate_{\frem,(x\frc)^{-1}}(q^x),&\lambda_{\frem,(x\frc)^{-1}})\\
&=f_{\frc'} (\frem^{-1}\otimes_{\oef} \Tate_{\fro,\frem(x\frc)^{-1}}(q^x),\frem^2 \lambda_{\fro,\frem(x\frc)^{-1}})\\
&=\frac{\Nor_{F/\bQ}(\frem)^2}{\nu(y)}(L_y S_\frem f_{\frc'})(\Tate_{\fro,(\frc'')^{-1}}(q^{xy^{-1}}),\lambda_{\fro,(\frc'')^{-1}})\\
&=\frac{\Nor_{F/\bQ}(\frem)^2\Phi(\frem)}{\nu(y)} f_{\frc''}(\Tate_{\fro,(\frc'')^{-1}}(q^{xy^{-1}})).
\end{align*}

For the second term, the subgroup 
\[
\{(\frX^\eta\mapsto q^{\xi\eta}\zeta^{\rho(\xi\eta)}\ (\eta\in \fro))\mid \xi\in (\frc')^{-1}\}\subseteq \cD_F^{-1}\otimes\Gm(\hat{R}_{x^{-1},\tau}^0)
\]
is generated by $\cH_{Q,\zeta}$ and the image of the subgroup
\[
\{(\frX^\eta\mapsto q^{\xi\eta}\ (\eta\in \fro))\mid \xi\in (x\frc)^{-1}\}\subseteq \cD_F^{-1}\otimes\Gm(\hat{R}_{x^{-1},\sigma}^0)
\]
via the natural map $\hat{R}_{x^{-1},\sigma}^0\to \hat{R}_{x^{-1},\tau}^0$. This yields an isomorphism
\[
\Tate_{\fro,(x\frc)^{-1}}(q)/\cH_{Q,\zeta} \simeq \Tate_{\fro,(\frc')^{-1}}(q\zeta^\rho)
\]
compatible with natural additional structures.
Hence the lemma follows.
\end{proof}

A similar proof also gives the following variant for $\frem\mid Np$.

\begin{lem}\label{HeckeQexpU}
\begin{enumerate}
\item\label{HeckeQexpUN}
For any maximal ideal $\frem\mid N$, take any element $x\in F^{\times,+,(p)}$ satisfying $\frc'=x\frem \frc\in [\mathrm{Cl}^+(F)]^{(p)}$. Fix an isomorphism of $\fro$-modules $\rho:(x\frem\frc)^{-1}/(x\frc)^{-1}\simeq \fro/\frem$. Then we have
\[
(T_\frem f)_{\frc}(q)=\frac{\nu(x)}{\Nor_{F/\bQ}(\frem)}\sum_{\zeta\in \mu_\frem(\bar{\bQ}_p)}f_{\frc'}(q^x\zeta^\rho).
\]
\item\label{HeckeQexpUp}
For any maximal ideal $\frp\mid p$, take any element $x\in F^{\times,+,(p)}$ satisfying $\frc'=xx_{\frp}^{-1}\frp \frc\in [\mathrm{Cl}^+(F)]^{(p)}$. Fix an isomorphism of $\fro$-modules $\rho:(xx_{\frp}^{-1}\frp\frc)^{-1}/(xx_{\frp}^{-1}\frc)^{-1}\simeq \fro/\frp$. Then we have
\[
(U_\frp f)_{\frc}(q)=\frac{\nu(x)}{\Nor_{F/\bQ}(\frp)}\sum_{\zeta\in \mu_\frp(\bar{\bQ}_p)}f_{\frc'}(q^{xx_\frp^{-1}}\zeta^\rho).
\]
\end{enumerate}
\end{lem}

%---------------------------------------------------------------------

%---------------------------------------------------------------------

\subsubsection{$q$-expansion and Hecke eigenvalues}\label{SecQExpHeckeEV}

For any $\xi\in F^\times$, we put $\chi_p(\xi)=\prod_{\frp\mid p} x_{\frp}^{v_\frp(\xi)}$. For any non-zero ideal $\frn\subseteq \fro$, take $\eta\in F^{\times,+}$ satisfying $\frc=\eta^{-1}\frn\in  [\mathrm{Cl}^+(F)]^{(p)}$ and put
\[
C(\frn,f)=\nu(\eta^{-1}\chi_p(\eta))a_{\fro,\frc^{-1}}(f,\eta).
\]
By Lemma \ref{transform}, this is independent of the choice of $\eta$. 
Then we have the following variant of \cite[(2.23)]{Shimura} in our setting.

\begin{lem}\label{ShimuraRelation}
For any non-zero ideal $\frl,\frn$ of $\fro$, we have
\[
C(\frn,T_\frl f)=\sum_{\frl+\frn\subseteq \fra\subseteq \fro} \Nor_{F/\bQ}(\fra)\Phi(\fra) C(\fra^{-2}\frl\frn,f).
\]
\end{lem}
\begin{proof}
We can easily reduce it to the case $\frl=\frem^s$ for some maximal ideal $\frem$. Consider the case of $\frem\nmid Np$ and $s=1$. We follow the notation of Lemma \ref{HeckeQexpT}. Since $x^{-1}\eta\in (x\frc)^{-1}$, we have
\[
\sum_{\zeta\in \mu_\frem(\bar{\bQ}_p)}\zeta^{\rho(x^{-1}\eta)}= \Nor_{F/\bQ}(\frem).
\]
Moreover, $x^{-1}y\eta\in (\frc'')^{-1}$ if and only if $\frem \mid \frn$.
Thus Lemma \ref{HeckeQexpT} implies
\[
C(\frn, T_\frem f)=\left\{\begin{array}{ll} \Nor_{F/\bQ}(\frem)\Phi(\frem) C(\frem^{-1}\frn,f)+C(\frem\frn,f)& (\frem \mid \frn)\\
C(\frem\frn,f) & (\frem \nmid \frn)
\end{array} \right.
\]
and the lemma follows for this case. The case of $\frem\mid Np$ and $s=1$ can be shown similarly from Lemma \ref{HeckeQexpU}. For $s\geq 2$, using the relation (\ref{Heckes}), we can show the lemma by an induction in the same way as the classical case.
\end{proof}

\begin{prop}\label{weakmulti}
For any $\frc\in [\mathrm{Cl}^+(F)]^{(p)}$ and any $\eta\in (\frc^{-1})^+$, put $\frn=\eta\frc\subseteq \fro$.
Then we have
\[
a_{\fro,\frc^{-1}}(f,\eta)=\nu(\eta \chi_p(\eta)^{-1})  \Lambda(\frn) a_{\fro,\fro}(f,1).
\]
\end{prop}
\begin{proof}
We have $a_{\fro,\frc^{-1}}(f,\eta)=\nu(\eta\chi_p(\eta)^{-1}) C(\frn,f)$ and $C(\fro, f)=a_{\fro,\fro}(f,1)$. By Lemma \ref{ShimuraRelation}, we obtain
\begin{align*}
\Lambda(\frn)a_{\fro,\fro}(f,1)&=\Lambda(\frn)C(\fro,f)=C(\fro, T_\frn f)\\
&=C(\frn,f)=\nu(\eta^{-1}\chi_p(\eta)) a_{\fro,\frc^{-1}}(f,\eta),
\end{align*}
from which the proposition follows.
\end{proof}

%---------------------------------------------------------------------

%---------------------------------------------------------------------

\subsection{$q$-expansion and integrality}\label{SecQExpIntegral}

Let $\kappa\in \cW(\Cp)$ be any $n$-analytic weight. Put
\[
\mathbf{M}(\mu_N,\frc,\kappa)(0):=H^0(\bar{\SGm}(\mu_N,\frc)(0)_{\OCp},\bOmega^\kappa)\subseteq M(\mu_N,\frc,\kappa)(0).
\]
This is an $\OCp$-lattice of the Banach $\Cp$-module $M(\mu_N,\frc,\kappa)(0)$. Consider the cusp $(\fro,\frc^{-1},\id)$, the fixed cone decomposition $\sC=\sC(\fro,\frc^{-1})\in \Dec(\fro,\frc^{-1})$ and $\sigma\in \sC$. 
By the definition of the $q$-expansion, every coefficient of the $q$-expansion of $f\in \mathbf{M}(\mu_N,\frc,\kappa)(0)$ is an element of $\OCp$. We also have the following converse, which can be considered as a $q$-expansion principle for our setting.

\begin{prop}\label{qexpprinciple}
Let $f_\frc$ be any element of $M(\mu_N,\frc,\kappa)(0)$. If every coefficient of the $q$-expansion $f_\frc(q)$ is in $\OCp$, then we have $f\in \mathbf{M}(\mu_N,\frc,\kappa)(0)$.
\end{prop}
\begin{proof}
First we show the following lemma.

\begin{lem}\label{AnnFinGen}
Let $\frX$ be a quasi-compact separated admissible formal scheme over $\OCp$. Let $\frF$ be an invertible sheaf on $\frX$. Let $\frX_{\Fpbar}$ be the special fiber of $\frX$ and $\frF_{\Fpbar}$ the pull-back of $\frF$ to $\frX_{\Fpbar}$.

\begin{enumerate}
\item\label{AnnFinGen_Ann} Suppose that $\frX^\rig$ is reduced and $\frX$ is integrally closed in $\frX^\rig$. Then, for any non-zero element $f\in H^0(\frX,\frF)[1/p]$, the $\OCp$-submodule of $\Cp$
\[
I=\{x\in\Cp\mid x f\in H^0(\frX,\frF)\}
\]
is principal.
\item\label{AnnFinGen_Red} Let $g$ be an element of $H^0(\frX,\frF)$. Suppose that the image of $g$ by the map
\[
H^0(\frX,\frF)\to H^0(\frX_{\Fpbar},\frF_{\Fpbar})
\]
is zero. Then there exists $x\in m_{\Cp}$ satisfying $g\in xH^0(\frX,\frF)$.
\end{enumerate}
\end{lem}
\begin{proof}
For the first assertion, take a finite covering $\frX=\bigcup_{i=1}^r \frU_i$ by formal affine open subschemes $\frU_i=\Spf(\frA_i)$ such that $\frF|_{\frU_i}$ is trivial. Since $\frX$ is separated, the intersection $\frU_{i,j}=\frU_i\cap \frU_j$ is also affine. Put $A_i=\frA_i[1/p]$, $\SGm_i=\Gamma(\frU_i,\frF)$ and $\SGm_{i,j}=\Gamma(\frU_{i,j},\frF)$. Then we have a commutative digram
\[
\xymatrix{
0\ar[r] & \Gamma(\frX,\frF) \ar[r]\ar[d] & \prod_{i=1}^r \SGm_i \ar@<0.5ex>[r]\ar@<-0.5ex>[r]\ar[d] & \prod_{i,j=1}^r \SGm_{i,j} \ar[d] \\
0\ar[r] & \Gamma(\frX,\frF)[1/p] \ar[r] & \prod_{i=1}^r \SGm_i[1/p] \ar@<0.5ex>[r]\ar@<-0.5ex>[r]& \prod_{i,j=1}^r \SGm_{i,j}[1/p],
}
\]
where the rows are exact and the vertical arrows are injective. Put $I_i=\{x\in\Cp\mid x f|_{\frU_i} \in \SGm_i\}$. Note that $I_i=\Cp$ if $f|_{\frU_i}=0$. Since the above diagram implies $I=\bigcap_{i=1}^r I_i$, it is enough to show that $I_i$ is principal if $f|_{\frU_i}\neq 0$.

By choosing a trivialization, we identify $\SGm_i$ with $\frA_i$ and $f|_{\frU_i}\in \SGm_i[1/p]$ with a non-zero element $g_i\in A_i$. Note that $A_i$ is a reduced $\Cp$-affinoid algebra. Since $\frA_i$ is an admissible formal $\OCp$-algebra which is integrally closed in $A_i$, \cite[Remark after Proposition 6.3.4/1]{BGR} implies $A_i^\circ=\frA_i$. Thus, for any $x \in \Cp$, we have
\[
x g_i \in \frA_i \Leftrightarrow |x| |g_i|_{\sup}\leq 1,
\]
where $|g_i|_{\sup}$ is the supremum norm of $g_i$ on $\Spv(A_i)$. By the maximum modulus principle, there exists a non-zero element $\delta\in \Cp$ satisfying $|\delta|=|g_i|_{\sup}$. Hence we obtain
\[
I_i=\{x\in \Cp\mid |x|\leq |\delta|^{-1}\}=\delta^{-1}\OCp
\]
and the first assertion follows.

For the second assertion, consider the covering $\frX=\bigcup_{i=1}^r \frU_i$ as above. Since the reduction of $g|_{\frU_i}$ is also zero, we can write $g|_{\frU_i}=x_i h_i$ with some $x_i\in m_{\Cp}$ and $h_i\in \SGm_i$. Replacing $x_i$ by a generator $x$ of the ideal $(x_1,\ldots,x_r)$, we may assume $g|_{\frU_i}=x h_i$ for any $i$. Since $\SGm_i$ and $\SGm_{i,j}$ are torsion free $\OCp$-modules, the elements $h_i$ can be glued to define $h\in H^0(\frX,\frF)$. Then we obtain $g=xh$ and the second assertion follows. 
\end{proof}

Put $\bar{\SGm}^\ord=\bar{\SGm}(\mu_N,\frc)(0)$, $\bar{\SGm}(\Gamma_1(p^n))^\ord=\bar{\SGm}(\Gamma_1(p^n),\mu_N,\frc)(0)$ and $\frIW^{\ord}=\frIW_{w,\frc}^{+}(0)$. Recall that $\bOmega^\kappa$ is invertible on $\bar{\SGm}^\ord$. We denote the reduction of $\bar{\SGm}^\ord_{\OCp}$ by $\bar{\SGm}^\ord_{\Fpbar}$. Consider the commutative diagram
\[
\xymatrix{
\breve{S}_{\sigma,\OCp}\ar@{=}[dr]\ar[r] \ar@/^2pc/[rrr]^{\tau_{\fro,\frc^{-1}}} & \pi_w^{-1}(\breve{S}_{\sigma,\OCp}) \ar[r] \ar[d] &\pi_w^{-1}(\breve{S}_{\sC,\OCp}) \ar[r] \ar[d]& \frIW^{\ord}_{\OCp} \ar[d]^{\pi_w}\\
& \breve{S}_{\sigma,\OCp} \ar[r] &\breve{S}_{\sC,\OCp} \ar[r]  & \bar{\SGm}^\ord_{\OCp}.
}
\]
Recall that $f_\frc\in \cO(\frIW^{\ord}_{\OCp})[1/p]$. The assumption on $f_\frc(q)$ implies $\tau^*_{\fro,\frc^{-1}}(f_\frc) \in \cO(\breve{S}_{\sigma,\OCp})$.

Consider the special fiber 
\[
\bar{\pi}_w: \frIW^{\ord}_{\Fpbar}\overset{\bar{\gamma}_n}{\to} \bar{\SGm}(\Gamma_1(p^n))^\ord_{\Fpbar} \overset{\bar{h}_n}{\to}  \bar{\SGm}^\ord_{\Fpbar}
\]
of the map $\pi_w$ and the closed immersion $i:\bar{\SGm}^\ord_{\Fpbar}\to \bar{\SGm}^\ord_{\OCp}$. From the construction of the sheaf $\bOmega^\kappa|_{\bar{\SGm}^\ord_{\OCp}}$ as the fixed part of a $\bT(\bZ/p^n\bZ)$-equivariant $\OCp$-flat sheaf on a $\bT(\bZ/p^n\bZ)$-torsor, we see that the subsheaf $\bOmega^{\kappa}|_{\bar{\SGm}^\ord_{\OCp}}\subseteq (\pi_w)_*\cO_{\frIW_{w,\frc}^+(0)_{\OCp}}$ is formal locally a direct summand. Since $\pi_w$ is affine, for any morphism of formal schemes $f:S\to \bar{\SGm}^\ord_{\OCp}$, the composite of natural maps
\[
f^*(\bOmega^{\kappa}|_{\bar{\SGm}^\ord_{\OCp}})\to f^*(\pi_w)_*\cO_{\frIW^{\ord}_{\OCp}}\to (\pi_w|_{f^{-1}(S)})_*\cO_{f^{-1}(S)}
\]
is injective. This yields a commutative diagram
\begin{equation}\label{RedAroundCusp}
\begin{gathered}
\xymatrix{
\bOmega^\kappa(\bar{\SGm}^\ord_{\OCp}) \ar[r]\ar[d]& \cO(\frIW^{\ord}_{\OCp})\ar[d]\\
i^*\bOmega^\kappa(\bar{\SGm}^\ord_{\Fpbar}) \ar[r]\ar[d]& \cO(\frIW^{\ord}_{\Fpbar})\ar[d]\\
i^*\bOmega^\kappa|_{\hat{S}_{\sC,\Fpbar}}(\hat{S}_{\sC,\Fpbar}) \ar[r]& \cO(\bar{\pi}_w^{-1}(\hat{S}_{\sC,\Fpbar}))
}
\end{gathered}
\end{equation}
with injective horizontal arrows, where the base extension $\hat{S}_{\sC,\Fpbar}=\hat{S}_{\sC}\hat{\otimes}_k\Fpbar$ is equal to the special fiber of $\breve{S}_{\sC,\OCp}$.

On the Tate object $\Tate_{\fro,\frc^{-1}}(q)$ over $\Spec(\breve{R}_{\sigma})$, we defined the canonical trivialization of the canonical subgroup and that of the $\bT_w^0(\breve{S}_{\sigma})$-set $\frIW^{\ord}(\breve{S}_{\sigma})$, which are denoted by $u_{\fro,\frc^{-1}}$ and $\alpha_{\fro,\frc^{-1}}$. Since $\breve{R}_\sigma$ is Noetherian, the moduli interpretation of $\frIW^{\ord}$ is available over $\breve{R}_\sigma$ and these trivializations give isomorphisms
\[
 \breve{S}_{\sigma}\times_{\bar{\SGm}^\ord} \bar{\SGm}(\Gamma_1(p^n))^\ord\simeq \coprod_{a\in \bT(\bZ/p^n\bZ)} \breve{S}_{\sigma},\quad \pi_w^{-1}(\breve{S}_{\sigma})\simeq \coprod_{a\in \bT(\bZ/p^n\bZ)} \breve{S}_{\sigma}\times\bT_w^0,
\]
where the latter is an isomorphism of formal $\bT_w^0$-torsors. By the base extension, we also have similar isomorphisms over $\OCp$. Since the latter isomorphism is defined using the trivializations $u_{\fro,\frc^{-1}}$ and $\alpha_{\fro,\frc^{-1}}$, the unit section on the component $a=1$ coincides with the above map $\breve{S}_{\sigma,\OCp} \to \pi_w^{-1}(\breve{S}_{\sigma,\OCp})$.

The $\bT_w^0$-representation $\cO(\breve{S}_{\sigma,\OCp}\times\bT_w^0)$ is decomposed into the direct sum of the free $\breve{R}_{\sigma, \OCp}$-modules $\breve{R}_{\sigma, \OCp} s_{\chi}$ for any formal character $\chi$ of $\bT_w^0$, where $s_\chi$ is a section generating its $\chi$-part. Thus we have 
\[
f_\frc|_{\pi^{-1}_w(\breve{S}_{\sigma,\OCp})}\in \prod_{a\in \bT(\bZ/p^n\bZ)} (\breve{R}_{\sigma,\OCp}[1/p]s_{\kappa}).
\]
Write this element as $(F_a s_{\kappa})_{a\in \bT(\bZ/p^n\bZ)}$ with $F_a\in \breve{R}_{\sigma,\OCp}[1/p]$. Since $\kappa(1)=1$ and $\tau_{\fro,\frc^{-1}}^*(f_\frc)\in \breve{R}_{\sigma,\OCp}$, we obtain $F_1\in \breve{R}_{\sigma,\OCp}$. Since $f_\frc$ is $\kappa$-equivariant for the $\bT(\bZ/p^n\bZ)$-action, we have $F_a=\kappa(\hat{a})F_1$ with a lift $\hat{a}\in \bT(\bZ_p)$ of $a$. Since the image of the character $\kappa$ is contained in $\cO_{\Cp}^\times$, we see that $F_a\in \breve{R}_{\sigma,\OCp}$ for any $a\in \bT(\bZ/p^n\bZ)$. This means
\begin{equation}\label{EqnPiinv}
f_\frc|_{\pi^{-1}_w(\breve{S}_{\sigma,\OCp})}\in \cO(\pi_w^{-1}(\breve{S}_{\sigma,\OCp})).
\end{equation}

To prove the proposition, we may assume $f_\frc\neq 0$. Consider the ideal $J=\{x\in \OCp \mid x f_\frc \in \mathbf{M}(\mu_N,\frc,\kappa)(0)\}$, which is principal by Lemma \ref{AnnFinGen} (\ref{AnnFinGen_Ann}). Put $J=(x)$ and suppose $x\in m_{\Cp}$. Then the $q$-expansion $x f_\frc(q)$ is also integral, and zero modulo $m_{\Cp}$. Thus the commutative diagram (\ref{RedAroundCusp}) and (\ref{EqnPiinv}) imply that the pull-back of $x f_\frc\in \bOmega^\kappa(\bar{\SGm}^\ord_{\OCp})$ to $i^*\bOmega^\kappa|_{\hat{S}_{\sC,\Fpbar}}(\hat{S}_{\sC,\Fpbar})$ vanishes.

Note that the reduction of $\breve{S}_{\sC,\OCp}\to \bar{\SGm}^\ord_{\OCp}$ induces the map on the special fiber
\[
\hat{S}_{\sC,\Fpbar}=(\breve{S}_{\sC,\OCp})_{\Fpbar} \to \bar{M}(\mu_N,\frc)_{\Fpbar}.
\]
Let $\bar{M}(\mu_N,\frc)_{\Fpbar}|^{\wedge}_{D_{\Fpbar}}$ be the formal completion of $\bar{M}(\mu_N,\frc)_{\Fpbar}$ along its boundary $D_{\Fpbar}$. Recall that this map induces maps
\[
\hat{S}_{\sC,\Fpbar} \to \hat{S}_{\sC,\Fpbar}/U_N \to \bar{M}(\mu_N,\frc)_{\Fpbar}|^{\wedge}_{D_{\Fpbar}},
\]
where the first arrow is a surjective local isomorphism and the second arrow is an open immersion. Hence $x f_\frc$ vanishes on a formal open subscheme of the formal completion $\bar{M}(\mu_N,\frc)_{\Fpbar}|^{\wedge}_{D_{\Fpbar}}$. We know that 
the smooth scheme $\bar{\SGm}^\ord_{\Fpbar}$ is irreducible. Since the sheaf $\bOmega^\kappa$ is invertible on the ordinary locus, Krull's intersection theorem implies that $x f_\frc$ vanishes on a non-empty open subscheme of $\bar{\SGm}^\ord_{\Fpbar}$, and thus it also vanishes on $\bar{\SGm}^\ord_{\Fpbar}$. Then Lemma \ref{AnnFinGen} (\ref{AnnFinGen_Red}) implies that $x f_\frc\in y \mathbf{M}(\mu_N,\frc,\kappa)(0)$ for some $y\in m_{\Cp}$. Since the $\OCp$-module $\mathbf{M}(\mu_N,\frc,\kappa)(0)$ is torsion free, this contradicts the choice of $x$. Thus we obtain $x\in \cO_{\Cp}^\times$ and $f_\frc \in \mathbf{M}(\mu_N,\frc,\kappa)(0)$, which concludes the proof of the proposition.
\end{proof}

\begin{cor}\label{HeckeEVintegral}
Let $f=(f_\frc)_{\frc\in[\mathrm{Cl}^+(F)]^{(p)}}$ be a non-zero eigenform in the space $S^G(\mu_N,(\nu,w))(v)$ of weight $(\nu,w)\in \cW^G(\Cp)$. For any non-zero ideal $\frn$ of $\fro$, the Hecke eigenvalue $\Lambda(\frn)$ is $p$-integral.
\end{cor}
\begin{proof}
By (\ref{Heckes}), it is enough to show the case where $\frn$ is a maximal ideal $\frem$. Put $\kappa=k(\nu,w)$. 
Note that by Lemma \ref{FvPLem} and Lemma \ref{Mvconn}, the restriction map $S^G(\mu_N,(\nu,w))(v)\to S^G(\mu_N,(\nu,w))(0)$ is injective.
We consider $\Lambda(\frem)$ as an eigenvalue of the operator $T_\frem$ acting on
\[
M:=\bigoplus_{\frc\in [\mathrm{Cl}^+(F)]^{(p)}} M(\mu_N,\frc,\kappa)(0).
\]
This is a Banach $\Cp$-module with respect to the $p$-adic norm $|-|$ defined by the $\OCp$-lattice
\[
\mathbf{M}:=\bigoplus_{\frc\in [\mathrm{Cl}^+(F)]^{(p)}} \mathbf{M}(\mu_N,\frc,\kappa)(0).
\]
Namely, we put 
\[
|f|=\inf\{|x|^{-1}\mid x\in \mathbb{C}_p^\times,\ xf\in \mathbf{M}\}.
\]

By Lemma \ref{AnnFinGen} (\ref{AnnFinGen_Ann}), we can find an element $x\in \Cp$ of largest absolute value satisfying $x f_\frc\in \mathbf{M}(\mu_N,\frc,\kappa)(0)$ for any $\frc\in [\mathrm{Cl}^+(F)]^{(p)}$. The norm $|f|$ is equal to $|x|^{-1}$. Moreover, any coefficient of the $q$-expansion $x f_\frc(q)$ is contained in $\OCp$. By Lemma \ref{ShimuraRelation}, so is $x T_\frem f$. Hence Proposition \ref{qexpprinciple} shows $x T_\frem f\in \mathbf{M}$. This implies 
\[
|\Lambda(\frem)|=\frac{|T_{\frem} f|}{|f|}\leq \frac{|x|^{-1}}{|x|^{-1}}=1
\]
and the corollary follows.
\end{proof}

\begin{cor}\label{qexpintegral}
Let $f=(f_\frc)_{\frc\in[\mathrm{Cl}^+(F)]^{(p)}}$ be a normalized eigenform in $S^G(\mu_N,(\nu,w))(v)$ of weight $(\nu,w)\in \cW^G(\Cp)$. Then we have
\[
a_{\fro,\frc^{-1}}(f,\eta)\in \OCp
\]
for any $\frc\in [\mathrm{Cl}^+(F)]^{(p)}$ and any $\eta\in (\frc^{-1})^+$. 
\end{cor}
\begin{proof}
This follows from Proposition \ref{weakmulti} and Corollary \ref{HeckeEVintegral}.
\end{proof}

\begin{cor}\label{cuspintegral}
Let $(\nu,w)$ be an element of $\cW^G(\Cp)$.
\begin{enumerate}
\item\label{cuspintegral_int} For any $\frc\in[\mathrm{Cl}^+(F)]^{(p)}$, there exists an admissible affinoid open subset $\cV_\frc\subseteq \bar{\cM}(\mu_N,\frc)(v)_{\Cp}$ such that $(\pi_w^\rig)^{-1}(\cV_\frc)$ meets every connected component of $\cIW^+_{w,\frc}(v)_{\Cp}$ and, for 
any normalized eigenform $f=(f_\frc)_{\frc\in[\mathrm{Cl}^+(F)]^{(p)}}$ in $S^G(\mu_N,(\nu,w))(v)$, the restriction $f_\frc|_{(\pi_w^\rig)^{-1}(\cV_\frc)}$ has absolute value bounded by one. 
\item\label{cuspintegral_zero} Let $f=(f_\frc)_{\frc\in[\mathrm{Cl}^+(F)]^{(p)}}$ be any normalized eigenform in the space $S^G(\mu_N,(\nu,w))(v)$. If $f_\frc(q)=0$ for any $\frc\in[\mathrm{Cl}^+(F)]^{(p)}$, then $f=0$.
\item\label{cuspintegral_eq} Let $f=(f_\frc)_{\frc\in[\mathrm{Cl}^+(F)]^{(p)}}$ and $f'=(f'_\frc)_{\frc\in[\mathrm{Cl}^+(F)]^{(p)}}$ be normalized eigenforms in $S^G(\mu_N,(\nu,w))(v)$. Suppose that the eigenvalues of the Hecke operator $T_\frn$ acting on $f$ and $f'$ are the same for any non-zero ideal $\frn\subseteq \fro$. Then $f=f'$.
\end{enumerate}
\end{cor}
\begin{proof}
Let us prove the first assertion. For any $\sigma\in \sC=\sC(\fro,\frc^{-1})$, Corollary \ref{qexpintegral} and (\ref{BoundedFnS}) show that $\tau_{\fro,\frc^{-1}}^*(f_\frc)$ is a rigid analytic function on $\breve{S}^\rig_{\sigma,\Cp}$ with absolute value bounded by one. As in the proof of Proposition \ref{qexpprinciple}, we can show that $f_\frc|_{(\pi_w^\rig)^{-1}(\breve{S}^\rig_{\sC,\Cp})}$ is a rigid analytic function with absolute value bounded by one. Since the natural map $\breve{S}^\rig_{\sC,\Cp} \to \breve{S}_{\sC,\Cp}^\rig/U_N$ is a surjective local isomorphism, the restriction $f_\frc|_{(\pi_w^\rig)^{-1}(\breve{S}_{\sC,\Cp}^\rig/U_N)}$ is also with absolute value bounded by one. Thus, for any non-empty admissible affinoid open subset $\cV_\frc\subseteq \breve{S}_{\sC,\Cp}^\rig/U_N$, the absolute value of $f_\frc|_{(\pi_w^\rig)^{-1}(\cV_\frc)}$ is bounded by one. Since $\breve{S}_{\sC,\Cp}^\rig/U_N$ is an admissible open subset of $\bar{\cM}(\mu_N,\frc)(v)_{\Cp}$, we see that $\cV_\frc$ is also its admissible open subset.

On the other hand, since the rigid analytic variety $\bar{\cM}(\mu_N,\frc)(v)_{\Cp}$ is connected by Lemma \ref{Mvconn} and the map 
\[
h^\rig_n: \bar{\cM}(\Gamma_1(p^n),\mu_N,\frc)(v)_{\Cp}\to \bar{\cM}(\mu_N,\frc)(v)_{\Cp}
\]
is finitely presented and etale, it is surjective on each connected component of the rigid analytic variety $\bar{\cM}(\Gamma_1(p^n),\mu_N,\frc)(v)_{\Cp}$ and thus $(h_n^\rig)^{-1}(\cV_\frc)$ meets every connected component of it. 

We claim that the map
\[
\gamma_w^\rig: \cIW_{w,\frc}^+(v)_{\Cp}\to \bar{\cM}(\Gamma_1(p^n),\mu_N,\frc)(v)_{\Cp}
\]
induces a bijection
\[
\pi_0(\cIW_{w,\frc}^+(v)_{\Cp})\to \pi_0(\bar{\cM}(\Gamma_1(p^n),\mu_N,\frc)(v)_{\Cp})
\]
between the sets of connected components. Indeed, by \cite[Corollary 3.2.3]{Con_irr},
it is enough to show the claim with $\Cp$ replaced by a finite extension $L/K$. By a finite base extension, we may assume $L=K$. Since the formal schemes $\frIW_{w,\frc}^+(v)$ and $\bar{\SGm}(\Gamma_1(p^n),\mu_N,\frc)(v)$ are both normal, it is enough to show a similar assertion for the formal model $\gamma_w$. Since it is a formal $\bT^0_w$-torsor, it is surjective and the map between the sets of connected components is also surjective. Let $\frY$ be any connected component of $\bar{\SGm}(\Gamma_1(p^n),\mu_N,\frc)(v)$ and $\{\frX_j\}_{j\in J}$ the set of connected components of $\frIW_{w,\frc}^+(v)$ which $\gamma_w$ maps to $\frY$. Suppose $\sharp J\geq 2$. Since $\gamma_w$ is finitely presented and flat, it is open and the connectedness of $\frY$ implies that $\gamma_w(\frX_j)\cap \gamma_w(\frX_{j'})\neq \emptyset$ for some $j\neq j'$. However, for any element $y$ of this intersection, the fiber $\gamma_w^{-1}(y)$ is connected since it is isomorphic to the special fiber of $\bT^0_w$, which is a contradiction. Since $\gamma_w^\rig$ is surjective, the claim shows that every connected component of $\cIW^+_{w,\frc}(v)_{\Cp}$ meets the admissible open subset $(\pi_w^\rig)^{-1}(\cV_\frc)$ and the first assertion follows.

Now suppose that $f_\frc(q)=0$ for any $\frc\in[\mathrm{Cl}^+(F)]^{(p)}$. Then we have $f_\frc|_{(\pi_w^\rig)^{-1}(\cV_\frc)}=0$. Since the rigid analytic variety $\cIW^+_{w,\frc}(v)_{\Cp}$ is smooth over $\Cp$, the first assertion and Lemma \ref{FvPLem} show the second assertion. The third assertion follows from Proposition \ref{weakmulti} and the second one.
\end{proof}

%---------------------------------------------------------------------

%---------------------------------------------------------------------

\subsection{Normalized overconvergent modular forms in families}\label{SecNormalizedFamily}

Let $\cU=\Spv(A)$ be a reduced $\Cp$-affinoid variety and put $\frU=\Spf(A^\circ)$. Let $\cU\to \cW^G_{\Cp}$ be an $n$-analytic morphism and consider the associated weight characters $(\nu^\cU,w^\cU)$ as before. Let $f=(f_\frc)_{\frc\in [\mathrm{Cl}^+(F)]^{(p)}}$ be an eigenform in the space $S^G(\mu_N,(\nu^\cU,w^\cU))(v)$. Recall that each $f_\frc$ is an element of $\cO(\frIW_{w,\frc}^+(v)_{\OCp}\times \frU)[1/p]$. 
For the cusp $(\fro,\frc^{-1},\id)$ of $M(\mu_N,\frc)$ and any $\sigma\in \sC=\sC(\fro,\frc^{-1})$, we have the map 
\[
\tau_{\fro,\frc^{-1}}\times 1: \breve{S}_{\sigma,\OCp} \times \frU \to \frIW_{w,\frc}^+(v)_{\OCp}\times \frU 
\]
over $\bar{\SGm}(\mu_N,\frc)(v)_{\OCp}\times \frU$.

As in \S \ref{SecQexpDef}, we see that the ring $\breve{R}_{\sigma,\OCp}\hat{\otimes}_{\OCp}A^\circ$ is isomorphic to the completion of the ring
\[
A^\circ[q^{\xi_1},\ldots, q^{\xi_r}][q^{\pm \xi_{r+1}},\ldots, q^{\pm \xi_g}]
\]
with respect to the $(p, q^{\xi_1}\cdots q^{\xi_r})$-adic topology for some $\xi_1,\ldots,\xi_g\in \frc^{-1}\cap\sigma^\vee$ and thus it can be considered as a subring of the ring
\[
A^\circ\langle q^{\pm \xi_{r+1}},\ldots, q^{\pm \xi_g}\rangle [[q^{\xi_1},\ldots, q^{\xi_r}]].
\]
Hence we obtain the map of the $q^1$-coefficient
\[
\prjt_{q^1}^\cU: \cO(\breve{S}_{\sigma,\OCp}\times \frU)[1/p]\to A.
\]
For any eigenform $f\in S^G(\mu_N,(\nu^\cU,w^\cU))(v)$ as above, we put $a_{\fro,\fro}^\cU(f,1)=\prjt_{q^1}^\cU((\tau_{\fro,\fro}\times 1)^*(f_\fro))\in A$.

For any $x\in \cU(\Cp)$, put $(\nu,w)=(\nu^\cU(x),w^\cU(x))$. The specialization $f(x)=(f_\frc(x))_{\frc\in [\mathrm{Cl}^+(F)]^{(p)}}$ is an element of the space $S^G(\mu_N,(\nu,w))(v)$ over $\Cp$, and we have the usual $q^1$-coefficient $a_{\fro,\fro}(f(x),1)$ of the $q$-expansion of $f(x)$. By the commutative diagram
\[
\xymatrix{
\breve{S}_{\sigma,\OCp}\times \frU \ar[r]^-{\tau_{\fro,\fro}\times 1} & \frIW_{w,\fro}^+(v)_{\OCp}\times\frU \\
\breve{S}_{\sigma,\OCp}\ar[u]^{1\times x} \ar[r]_-{\tau_{\fro,\fro}} & \frIW_{w,\fro}^+(v)_{\OCp}, \ar[u]_{1\times x} \\
}
\]
we obtain
\begin{equation}\label{q1compat}
a_{\fro,\fro}^\cU(f,1)(x)=a_{\fro,\fro}(f(x),1).
\end{equation}

\begin{lem}\label{normalize}
Suppose that $f(x)\neq 0$ for any $x\in \cU(\Cp)$. Then we have 
\[
a_{\fro,\fro}^\cU(f,1)\in A^\times. 
\]
In particular, the specialization $f'(x)$ of $f'=a_{\fro,\fro}^\cU(f,1)^{-1} f$ is a normalized eigenform with the same eigenvalues as $f(x)$ for any $x\in \cU(\Cp)$.
\end{lem}
\begin{proof}
We claim that $a_{\fro,\fro}(f(x),1)\neq 0$ for any $x\in \cU(\Cp)$. Indeed, suppose that $a_{\fro,\fro}(f(x),1)=0$ for some $x\in \cU(\Cp)$. Since $f(x)$ is an eigenform, Proposition \ref{weakmulti} implies that the $q$-expansion $f(x)_\frc(q)$ of $f(x)$ is zero for any $\frc\in [\mathrm{Cl}^+(F)]^{(p)}$. By Corollary \ref{cuspintegral} (\ref{cuspintegral_zero}) we have $f(x)=0$, which is a contradiction. 

Now (\ref{q1compat}) implies that $a_{\fro,\fro}^\cU(f,1)(x)\neq 0$ for any $x\in \cU(\Cp)$. 
Hence we obtain $a_{\fro,\fro}^\cU(f,1)\in A^\times$.
\end{proof}

%---------------------------------------------------------------------

%---------------------------------------------------------------------

\subsection{Gluing results}

Here we prove two results on gluing overconvergent Hilbert modular forms, based on the theory of the $q$-expansion developed above. Let $\cX=\Spv(R)$ be any admissible affinoid open subset of $\cW^G$. Put $n=n(\cX)$ and $v=v_n$ as in \S\ref{SecHeckeOp}. Consider the Hilbert eigenvariety $\cE|_\cX\to \cX$, which is constructed from the input data 
\[
(R,S^G(\mu_N,(\nu^\cX,w^\cX))(v_\tot),\bT,U_p).
\]

\subsubsection{Gluing local eigenforms} 

\begin{lem}\label{HMFres}
Let $\cU=\Spv(A)$ be a $\Cp$-affinoid variety and $\cU\to \cX_{\Cp}$ a morphism of rigid analytic varieties over $\Cp$. Let $f$ be an eigenvector of the space $S^G(\mu_N,(\nu^\cX,w^\cX))(v_\tot)\hat{\otimes}_R A$ for the action of $\bT$ such that for any $x\in \cU(\Cp)$, the specialization 
\[
f(x)\in S^G(\mu_N,(\nu^\cX,w^\cX))(v_\tot)\hat{\otimes}_{R,x^*} \Cp
\]
is non-zero. 
Then the image of $f$ by the natural map
\[
S^G(\mu_N,(\nu^\cX,w^\cX))(v_\tot)\hat{\otimes}_R A\to S^G(\mu_N,(\nu^\cU,w^\cU))(v_\tot)
\]
is an eigenform with the same property.
\end{lem}
\begin{proof}
Put $(\nu,w)=(\nu^\cU(x),w^\cU(x))$. Then we have the commutative diagram
\[
\xymatrix{
S^G(\mu_N,(\nu^\cX,w^\cX))(v_\tot)\hat{\otimes}_R A\ar[r]\ar[d] & S^G(\mu_N,(\nu^\cU,w^\cU))(v_\tot) \ar[d]\\
S^G(\mu_N,(\nu^\cX,w^\cX))(v_\tot)\hat{\otimes}_{R,x^*} \Cp\ar[r]\ar[rd] &S^G(\mu_N,(\nu^\cU,w^\cU))(v_\tot)\hat{\otimes}_{A,x^*} \Cp\ar[d]\\
& S^G(\mu_N,(\nu,w))(v_\tot).
}
\]
Here the lowest two arrows are the specialization maps. Since $\cW^G$ is smooth, the maximal ideal of $R\hat{\otimes}_K\Cp$ corresponding to $x$ is generated by a regular sequence. By Lemma \ref{specializeCp}, the left oblique arrow is an isomorphism. This implies the lemma.
\end{proof}

\begin{prop}\label{GlobalEigenform}
Let $\cZ$ be a smooth rigid analytic variety over $\Cp$ which is principally refined. Let $\varphi:\cZ\to (\cE|_\cX)_{\Cp}$ be a morphism of rigid analytic varieties over $\Cp$. 
Then there exist an element 
\[
f\in \bigoplus_{\frc\in [\mathrm{Cl}^+(F)]^{(p)}} \cO(\cIW_{w,\frc}^+(v_\tot)_{\Cp}\times \cZ)
\]
and an admissible affinoid covering $\cZ=\bigcup_{i\in I} \cU_i$ such that the restriction $f|_{\cU_i}$ for each $i\in I$ is an eigenform of $S^G(\mu_N, (\nu^{\cU_i},w^{\cU_i}))(v_\tot)$ with eigensystem $\varphi^*: \bT\to \cO(\cZ)\to \cO(\cU_i)$ and $f(z)$ is normalized for any $z\in \cZ$.
\end{prop}
\begin{proof}
By Proposition \ref{DSL} (\ref{DSL2}), there exist an admissible affinoid covering $\cZ=\bigcup_{i\in I} \cU_i$, $\cU_i=\Spv(A_i)$ with a principal ideal domain $A_i$ and an eigenvector $f_i$ in the space
\[
S^G(\mu_N,(\nu^\cX,w^\cX))(v_\tot)\hat{\otimes}_R A_i
\]
such that for any $z\in \cU_i$, we have $f_i(z)\neq 0$ and
\[
(h\otimes 1)f_i=(1\otimes \varphi^*(h))f_i
\]
for any $h\in \bT$. 
By Lemma \ref{HMFres}, the image $f'_i$ of $f_i$ in the space $S^G(\mu_N,(\nu^{\cU_i},w^{\cU_i}))(v_\tot)$ is an eigenform with eigensystem $\varphi^*:\bT\to A_i$
such that $f'_i(z)\neq 0$ for any $z\in \cU_i$. Since $\cU_i$ is reduced, by Lemma \ref{normalize} we may assume that $f'_i(z)$ is a normalized eigenform for any $z\in \cU_i$. 
For any $z\in \cU_i\cap \cU_j$ and any $h\in \bT$, the $h$-eigenvalues of $f'_i(z)$ and $f'_j(z)$ are both $\varphi^*(h)(z)$. Since they are normalized eigenforms, Corollary \ref{cuspintegral} (\ref{cuspintegral_eq}) implies $f'_i(z)=f'_j(z)$.

Since the rigid analytic variety $\cIW^+_{w,\frc}(v_\tot)_{\Cp}\times \cZ$ is reduced, this equality means that $f'_i$ and $f'_j$ coincide with each other as rigid analytic functions on 
\[
\coprod_{\frc\in [\mathrm{Cl}^+(F)]^{(p)}} \cIW^+_{w,\frc}(v_\tot)_{\Cp} \times (\cU_i\cap\cU_j).
\] 
Thus we can glue $f'_i$'s to produce an element
\[
f\in \bigoplus_{\frc\in [\mathrm{Cl}^+(F)]^{(p)}}\cO(\cIW^+_{w,\frc}(v_\tot)_{\Cp}\times \cZ). 
\]
This concludes the proof.
\end{proof}

%---------------------------------------------------------------------

%---------------------------------------------------------------------

\subsubsection{Gluing around cusps} 

Consider the unit disc $\cD_{\Cp}$ over $\Spv(\Cp)$ centered at the origin $O$. Put $\cD^\times_{\Cp}=\cD_{\Cp}\setminus \{O\}$. 

\begin{lem}\label{punctdisc}
Let $\cZ$ be a quasi-compact smooth rigid analytic variety over $\Cp$. Then the ring $\cO(\cZ\times \cD^\times_{\Cp})$ can be identified with the ring of power series $\sum_{n\in\bZ}a_n T^n$ with $a_n\in \cO(\cZ)$ such that
\begin{equation}\label{convergence}
\lim_{n\to +\infty} \sup_{z\in \cZ} |a_n(z)|=0,\quad \lim_{n\to +\infty} \sup_{z\in \cZ} |a_{-n}(z)|\rho^n=0
\end{equation}
for any rational number $\rho$ satisfying $0<\rho\leq 1$.
\end{lem}
\begin{proof}
For any non-negative rational number $\rho\leq 1$, let $\cA[\rho,1]_{\Cp}$ be the closed annulus with parameter $T$ over $\Cp$ defined by $\rho \leq |T| \leq 1$. Then we have an admissible covering
\[
\cD^\times_{\Cp}=\bigcup_{\rho\to 0+} \cA[\rho,1]_{\Cp}
\]
of $\cD^\times_{\Cp}$. Note that, for any connected reduced $\Cp$-affinoid variety $\cU$, \cite[Proposition 1.1]{BLR4} implies that the rigid analytic varieties $\cU\times \cA[\rho,1]_{\Cp}$ and $\cU\times \cD^\times_{\Cp}$ are connected. This shows that, for any connected reduced rigid analytic variety $\cX$ over $\Cp$, the fiber products $\cX\times \cA[\rho,1]_{\Cp}$ and $\cX\times \cD^\times_{\Cp}$ are also connected. By Lemma \ref{FvPLem}, we have injections
\[
\cO(\cZ\times\cD^\times_{\Cp})\to \cO(\cZ\times\cA[\rho,1]_{\Cp})\to \cO(\cZ\times\cA[\rho',1]_{\Cp})
\]
for any $\rho<\rho'$ and thus
\[
\cO(\cZ\times\cD^\times_{\Cp})=\bigcap_{\rho\to 0+} \cO(\cZ\times\cA[\rho,1]_{\Cp}).
\]

Take $\varpi\in \OCp$ satisfying $|\varpi|=\rho$. We define $\cO(\cZ)\langle T, \frac{\varpi}{T}\rangle$ as the ring of formal power series $\sum_{n\in \bZ} a_n T^n$ with $a_n\in \cO(\cZ)$ satisfying (\ref{convergence}) for $\rho$.
It suffices to show
\[
\cO(\cZ\times\cA[\rho,1]_{\Cp})=\cO(\cZ)\langle T, \frac{\varpi}{T}\rangle.
\]
Take a finite admissible affinoid covering $\cZ=\bigcup_{i\in I} \cU_i$ with $\cU_i=\Spv(A_i)$. We have an inclusion
\[
\cO(\cU_i\times \cA[\rho,1]_{\Cp})=A_i\langle T, \frac{\varpi}{T}\rangle \subseteq \prod_{n\in\bZ} A_i T^n
\]
which is compatible with the restriction to any affinoid subdomain of $\cU_i$. 
Take $f\in  \cO(\cZ\times\cA[\rho,1]_{\Cp})$ and put
\[
f|_{\cU_i\times \cA[\rho,1]_{\Cp}}=\sum_{n\in \bZ} a_{i,n} T^n
\]
with $a_{i,n}\in A_i$. Then $a_{i,n}$'s can be glued to obtain an element $a_n\in \cO(\cZ)$. Put $\Phi(f)=\sum_{n\in \bZ} a_n T^n$. Since $I$ is a finite set, we can check that $a_n$'s also satisfy (\ref{convergence}) and thus $\Phi(f)\in \cO(\cZ)\langle T, \frac{\varpi}{T}\rangle$. On the other hand, for any element $g=\sum_{n\in \bZ} a_n T^n$ of $\cO(\cZ)\langle T, \frac{\varpi}{T}\rangle$, put $\Psi(g)_i=\sum_{n\in \bZ} a_n|_{\cU_i} T^n$. Then $\Psi(g)_i\in A_i\langle T, \frac{\varpi}{T}\rangle$, which can be glued to obtain $\Psi(g)\in \cO(\cZ\times\cA[\rho,1]_{\Cp})$. Then $\Phi$ and $\Psi$ are inverse to each other and the lemma follows.
\end{proof}

Next we show the following variant of \cite[Lemma 7.1]{BuzCal}.

\begin{lem}\label{gluecusplem}
Let $\cZ$ be a quasi-compact smooth rigid analytic variety over $\Cp$. Let $\cV$ be an admissible open subset of $\cZ$ which meets every connected component of $\cZ$. Let $f$ be an element of $\cO(\cZ\times\cD^\times_{\Cp})$. Suppose that $f|_{\cV\times \cD^\times_{\Cp}}$ extends to an element of $\cO(\cV \times \cD_{\Cp})$. Then $f$ extends to an element of $\cO(\cZ\times\cD_{\Cp})$.
\end{lem}
\begin{proof}
By taking an admissible affinoid open subset of the intersection of $\cV$ and each connected component of $\cZ$ and replacing $\cV$ with their union, we may assume that $\cV$ is quasi-compact.
By Lemma \ref{FvPLem}, the assumption on $\cV$ yields injections
\[
\cO(\cZ)\to \cO(\cV),\quad \cO(\cZ\times\cD^\times_{\Cp}) \to \cO(\cV\times \cD^\times_{\Cp}) \gets \cO(\cV\times \cD_{\Cp}).
\]
From Lemma \ref{punctdisc}, we see that the intersection of $\cO(\cZ\times\cD^\times_{\Cp})$ and $\cO(\cV\times \cD_{\Cp})$ inside $\cO(\cV\times \cD^\times_{\Cp})$ is the set of formal power series $\sum_{n\geq 0} a_n T^n$ with $a_n\in \cO(\cZ)$ satisfying
\[
\lim_{n\to +\infty} \sup_{z\in \cZ} |a_n(z)|=0,
\]
which is equal to $\cO(\cZ\times\cD_{\Cp})$.
\end{proof}

\begin{lem}\label{fillpunct}
\[
\cO^\circ(\cD^\times_{\Cp})\subseteq \cO(\cD_{\Cp}).
\]
\end{lem}
\begin{proof}
Let $f=\sum_{n\in \bZ} a_n T^n$ be an element of $\cO^\circ(\cD^\times_{\Cp})$. Consider the Newton polygon of $f$. Then the assumption implies that any point $(n,v_p(a_n))$ lies above the line $y=-rx$ for any non-negative rational number $r$, which forces $a_n=0$ for any $n<0$.
\end{proof}

\begin{prop}\label{gluecusp}
Let $\varphi:\cD^\times_{\Cp}\to (\cE|_\cX)_{\Cp}$ be a morphism of rigid analytic varieties over $\Cp$ such that the composite $\cD^\times_{\Cp} \to (\cE|_\cX)_{\Cp}\to \cX_{\Cp}$ extends to $\cD_{\Cp}\to \cX_{\Cp}$. Let $(\nu^{\cD_{\Cp}},w^{\cD_{\Cp}})$ be the weight associated to the map $\cD_{\Cp}\to \cX_{\Cp}$.
Suppose that, for some non-negative rational number $v'<(p-1)/p^{n}$, we are given an element
\[
f=(f_\frc)_{\frc\in[\mathrm{Cl}^+(F)]^{(p)}}\in \bigoplus_{\frc\in [\mathrm{Cl}^+(F)]^{(p)}} \cO(\cIW_{w,\frc}^+(v')_{\Cp}\times \cD^\times_{\Cp})
\]
and an admissible affinoid covering $\cD^\times_{\Cp}=\bigcup_{i\in I} \cU_i$ such that the restriction $f|_{\cU_i}$ for each $i\in I$ is an eigenform of $S^G(\mu_N, (\nu^{\cU_i},w^{\cU_i}))(v')$ with eigensystem $\varphi^*: \bT\to \cO(\cD_{\Cp}^\times)\to \cO(\cU_i)$ and $f(z)$ is normalized for any $z\in \cD^\times_{\Cp}$. 
Then there exists an eigenform $f'\in S^G(\mu_N,(\nu^{\cD_{\Cp}},w^{\cD_{\Cp}}))(v')$ such that $f'(z)$ is normalized for any $z\in \cD_{\Cp}$ and it is an eigenform with eigensystem $\varphi^*(z): \bT\to \cO(\cD^\times_{\Cp})\to \Cp$ for any $z\in \cD^\times_{\Cp}$. 
\end{prop}
\begin{proof}
Consider the map $\pi_w^\rig: \cIW_{w,\frc}^+(v')_{\Cp} \to \bar{\cM}(\mu_N,\frc)(v')_{\Cp}$ as before. Let $\cV_\frc$ be an admissible affinoid open subset of $\bar{\cM}(\mu_N,\frc)(v')_{\Cp}$ as in Corollary \ref{cuspintegral} (\ref{cuspintegral_int}). Put $\cI_\frc=(\pi_w^\rig)^{-1}(\cV_\frc)$. Then $\cI_\frc$ is an admissible open subset which meets every connected component of $\cIW_{w,\frc}^+(v')_{\Cp}$ such that $f_\frc(z)|_{\cI_\frc}$ has absolute value bounded by one for any $z\in \cD^\times_{\Cp}$. Hence $f_\frc|_{\cI_\frc\times \cD^\times_{\Cp}}$ also has absolute value bounded by one. 

Note that $\cI_\frc$ is quasi-compact, since $\pi_w$ is quasi-compact.
By Lemma \ref{punctdisc}, we can write as 
\[
f_\frc|_{\cI_\frc\times \cD^\times_{\Cp}}=\sum_{n\in \bZ} a_n T^n
\]
with some $a_n\in \cO(\cI_\frc)$. Lemma \ref{fillpunct} implies $a_n(x)=0$ for any $x\in \cI_\frc$ and any $n<0$. Since $\cI_\frc$ is reduced, we obtain $a_n=0$ for any $n<0$ and thus 
\[
f_\frc|_{\cI_\frc\times \cD^\times_{\Cp}}\in \cO(\cI_\frc\times \cD_{\Cp}).
\]
Therefore, by Lemma \ref{gluecusplem} we see that $f_\frc$ extends to an element $\tilde{f}_\frc$ of $\cO(\cIW_{w,\frc}^+(v')_{\Cp}\times \cD_{\Cp})$. 

Write as $\cD_{\Cp}=\Spv(\Cp\langle T\rangle)$. Note that the ring $\cO(\cIW_{w,\frc}^+(v')_{\Cp}\times \cD_{\Cp})$ is $T$-torsion free. We claim that, if $f_\frc\neq 0$, then there exists a non-negative integer $m_\frc$ satisfying 
\[
\tilde{f}_\frc\in T^{m_\frc} \cO(\cIW_{w,\frc}^+(v')_{\Cp}\times \cD_{\Cp})\setminus T^{m_\frc+1} \cO(\cIW_{w,\frc}^+(v')_{\Cp}\times \cD_{\Cp}).
\]
Indeed, since $\cIW_{w,\frc}^+(v')_{\Cp}$ is smooth, we can take an admissible affinoid covering 
\[
\cIW_{w,\frc}^+(v')_{\Cp}=\bigcup_{j\in J}\cV_j,\quad \cV_j=\Spv(A_j)
\]
such that every $A_j$ is a Noetherian domain. Suppose that
\[
\tilde{f}_\frc\in \bigcap_{m\geq 0} T^m \cO(\cIW_{w,\frc}^+(v')_{\Cp}\times \cD_{\Cp}).
\]
Since $A_j\langle T\rangle$ is also a Noetherian domain, Krull's intersection theorem implies $\tilde{f}_\frc|_{\cV_j\times \cD_{\Cp}}=0$ for any $j\in J$ and thus $\tilde{f}_\frc=0$, which is a contradiction.

Put $m=\min\{m_\frc\mid \frc\in[\mathrm{Cl}^+(F)]^{(p)},f_\frc\neq 0\}$. Let $\tilde{f}'_\frc$ be the unique element of $\cO(\cIW_{w,\frc}^+(v')_{\Cp}\times \cD_{\Cp})$ satisfying $\tilde{f}_\frc=T^m \tilde{f}'_\frc$. Since the maps
\begin{align*}
\cO(\cIW_{w,\frc}^+(v')_{\Cp}\times \cD_{\Cp})&\to \cO(\cIW_{w,\frc}^+(v')_{\Cp}\times \cD^\times_{\Cp})\\
&\to \prod_{i\in I} \cO(\cIW_{w,\frc}^+(v')_{\Cp}\times \cU_i)
\end{align*}
are injective by Lemma \ref{FvPLem}, the element $\tilde{f}'_\frc$ is also $\kappa^{\cD_{\Cp}}$-equivariant and $\Delta$-stable. Hence 
the collection $\tilde{f}'=(\tilde{f}'_\frc)_{\frc\in [\mathrm{Cl}^+(F)]^{(p)}}$ is an element of $S^G(\mu_N,(\nu^{\cD_{\Cp}},w^{\cD_{\Cp}}))(v')$ such that $\tilde{f}'(z)\neq 0$ for any $z\in \cD_{\Cp}$.

Let $\Lambda(\frn)$ be the image of $T_\frn$ ({\it resp.} $S_\frn$) by the map $\varphi^*: \bT\to \cO(\cD^\times_{\Cp})$. By Corollary \ref{HeckeEVintegral}, the specialization $\Lambda(\frn)(z)$ is $p$-integral for any $z\in \cD^\times_{\Cp}$. Thus Lemma \ref{fillpunct} shows $\Lambda(\frn)\in\cO(\cD_{\Cp})$. By the above injectivity, we see that $\tilde{f}'$ is an eigenform on which $T_\frn$ ({\it resp.} $S_\frn$) acts by $\Lambda(\frn)$. Now Lemma \ref{normalize} concludes the proof of the proposition.
\end{proof}

%---------------------------------------------------------------------

%---------------------------------------------------------------------

\section{Properness at integral weights}\label{mainpf}

Let $\cE\to \cW^G$ be the Hilbert eigenvariety as in \S\ref{SecHeckeOp}. In this section, we prove the following main theorem of this paper.

\begin{thm}\label{main}
Suppose that $F$ is unramified over $p$ and for any prime ideal $\frp\mid p$ of $F$, the residue degree $f_\frp$ satisfies $f_\frp\leq 2$ ({\it resp.} $p$ splits completely in $F$) for $p\geq 3$ ({\it resp.} $p=2$).
Consider a commutative diagram 
\[
\xymatrix{
\cD^\times_{\Cp} \ar[r]^-{\varphi} \ar[d]& \cE_{\Cp} \ar[d]\\
\cD_{\Cp} \ar[r]_{\psi}\ar@{.>}[ur]& \cW^G_{\Cp}
}
\]
of rigid analytic varieties over $\Cp$, where the left vertical arrow is the natural inclusion. Suppose that $\psi(O)$ is $1$-integral ({\it resp.} $1$-even) in the sense of \S\ref{SecArithOCHMF}. Then there exists a morphism $\cD_{\Cp}\to \cE_{\Cp}$ of rigid analytic varieties over $\Cp$ such that the above diagram with this morphism added is also commutative.
\end{thm}
\begin{proof}
Let $e_1,\ldots,e_g$ be a basis of the $\bZ_p$-module $2p(\cO_F\otimes \bZ_p)$ and put $E_i=\exp(e_i)\in 1+2p(\cO_F\otimes \bZ_p)$. Similarly, let $e_{g+1}$ be a basis of the $\bZ_p$-module $2p\bZ_p$ and put $E_{g+1}=\exp(e_{g+1})\in 1+2p\bZ_p$. Let $(\nu^\univ, w^\univ)$ be the universal character on $\cW^G$. Note that $\cW^G_{\Cp}$ is the disjoint union of finitely many copies of the open unit polydisc defined by
\[
|X_1|<1,\ldots, |X_{g+1}|<1
\]
with parameters $X_1,\ldots,X_{g+1}$: the connected components are parametrized by the finite order characters 
\[
\varepsilon:\bT(\bZ/2 p\bZ)\times (\bZ/2 p\bZ)^\times \to \cO_{\Cp}^\times
\]
and on each connected component, the point defined by $X_i\mapsto x_i$ corresponds to the character $(\nu,w)$ satisfying $\nu(E_i)=1+x_i$ for any $i\leq g$ and $w(E_{g+1})=1+x_{g+1}$. 

Put $q=p$ if $p\geq 3$ and $q=8$ if $p=2$. Since $\psi(O)$ is $1$-integral, it comes from a $K$-valued point of $\cW^G$, which we also denote by $\psi(O)$. This corresponds to a finite order character $\varepsilon_O$ and a map $X_i\mapsto x_i$ with some $x_i\in q\okey$. For $p=2$, the assumption that $\psi(O)$ is $1$-even implies that $\varepsilon_O$ is trivial on the torsion subgroup of $1+2(\cO_F\otimes \bZ_2)$.
Put $E'_i=(-1)^{p-1}E_i$. The group $1+p(\cO_F\otimes \bZ_p)$ is topologically generated by $E_i$'s and $E'_i$'s. We have
\[
(\nu^\univ, w^\univ)(E_i)=(\nu^\univ, w^\univ)(E'_i)=1+X_i
\]
on the $\varepsilon_O$-component of $\cW^G$. Let $\cU=\Spv(R)$ be the admissible affinoid open subset of the $\varepsilon_O$-component of $\cW^G$ defined by $|X_i-x_i|\leq |q|$ for any $i$. Then $1+X_i=1+x_i+(X_i-x_i)\in 1+qR^\circ$ and the universal character $(\nu^\univ, w^\univ)$ is $1$-analytic on $\cU$.

We denote by $\cD_{\rho,\Cp}$ the closed disc of radius $\rho$ centered at the origin over $\Cp$.
Consider the element $\psi^*(X_i)(T)$ of the ring $\cO(\cD_{\Cp})=\Cp\langle T\rangle$. Since $\psi^*(X_i)(0)=x_i$, there exists a positive rational number $\rho<1$ such that
\[
|t|\leq \rho \Rightarrow |\psi^*(X_i)(t)-x_i|\leq |q|
\]
for any $i$. This means $\psi(\cD_{\rho,\Cp})\subseteq \cU_{\Cp}$.
If we can construct a morphism $\cD_{\rho,\Cp}\to \cE_{\Cp}$ which makes the diagram in the theorem commutative, then by gluing we obtain the desired map $\cD_{\Cp}\to \cE_{\Cp}$. Thus, by shrinking the disc, we may assume that $\psi$ factors through $\cU_{\Cp}$.

Put $n=1$ and $v=v_1$. We may assume $v<1/(p+1)$ so that we have 
\[
\bar{\cM}(\mu_N,\frc)(v_\tot)\subseteq\bar{\cM}(\mu_N,\frc)(\tfrac{1}{p+1}).
\]
By Remark \ref{DSL-disc}, the rigid analytic variety $\cD^\times_{\Cp}$ is principally refined. Applying Proposition \ref{GlobalEigenform} to the map $\varphi: \cD^\times_{\Cp}\to (\cE|_{\cU})_{\Cp}$, we obtain an element
\[
f\in \bigoplus_{\frc\in [\mathrm{Cl}^+(F)]^{(p)}} \cO(\cIW_{w,\frc}^+(v_\tot)_{\Cp}\times \cD^\times_{\Cp})
\]
and an admissible affinoid covering $\cD^\times_{\Cp}=\bigcup_{i\in I} \cU_i$ such that the restriction $f|_{\cU_i}$ for each $i\in I$ is an eigenform of $S^G(\mu_N, (\nu^{\cU_i},w^{\cU_i}))(v_\tot)$ with eigensystem $\varphi^*: \bT\to \cO(\cD_{\Cp}^\times)\to \cO(\cU_i)$ and $f(z)$ is normalized for any $z\in \cD^\times_{\Cp}$.

Since $\varphi^*$ comes from the eigenvariety $\cE$, the $U_p$-eigenvalue $\varphi^*(U_p)\in \cO(\cU_i)$ of $f|_{\cU_i}$ satisfies $\varphi^*(U_p)(z)\neq 0$ for any $z\in \cU_i(\Cp)$, and thus we have $\varphi^*(U_p)\in \cO(\cU_i)^\times$. Since $U_p$ improves the overconvergence from $v$ to $pv$, taking $\varphi^*(U_p)^{-1}U_p(f|_{\cU_i})$ repeatedly, we can find an eigenform
\[
g_i \in S^G(\mu_N,(\nu^{\cU_i},w^{\cU_i}))(\tfrac{1}{p+1})
\]
with eigensystem $\varphi^*: \bT\to \cO(\cD_{\Cp}^\times)\to \cO(\cU_i)$ which extends $f|_{\cU_i}$. Note that for any $z\in \cU_i(\Cp)$ we have a commutative diagram
\[
\xymatrix{
S^G(\mu_N,(\nu^{\cU_i},w^{\cU_i}))(\tfrac{1}{p+1}) \ar[r]\ar[d] & S^G(\mu_N,(\nu^{\cU_i},w^{\cU_i}))(v_\tot)\ar[d] \\
S^G(\mu_N,(\nu^{\cU_i}(z),w^{\cU_i}(z)))(\tfrac{1}{p+1}) \ar[r] & S^G(\mu_N,(\nu^{\cU_i}(z),w^{\cU_i}(z)))(v_\tot), \\
}
\]
where the horizontal arrows are the restriction maps and the vertical arrows are the specialization maps. This implies that the specialization $g_i(z)$ is also non-zero for any $z\in \cU_i(\Cp)$. Since the $q$-expansion is determined by the restriction to the ordinary locus, $g_i(z)$ is also normalized for any $z\in \cU_i(\Cp)$. Since the Hecke eigenvalues of $g_i(z)$ are also given by the eigensystem $\varphi^*(z): \bT\to \cO(\cU_i)\overset{z^*}{\to} \Cp$, a gluing argument as in the proof of Proposition \ref{GlobalEigenform} shows that $g_i$'s can be glued. In other words, we may assume
\[
f\in \bigoplus_{\frc\in [\mathrm{Cl}^+(F)]^{(p)}} \cO(\cIW^+_{w,\frc}(\tfrac{1}{p+1})_{\Cp}\times \cD^\times_{\Cp}).
\]
By Proposition \ref{gluecusp}, we may replace $f$ by an eigenform of the space $S^G(\mu_N,(\nu^{\cD_{\Cp}},w^{\cD_{\Cp}}))(\tfrac{1}{p+1})$ such that every specialization on $\cD_{\Cp}$ is normalized, which we also denote by $f=(f_\frc)_{\frc\in [\mathrm{Cl}^+(F)]^{(p)}}$. By Lemma \ref{specializeCp}, we have an isomorphism
\[
S^G(\mu_N,(\nu^\cU,w^\cU))(v_{\mathrm{tot}})\hat{\otimes}_{R,z^*} k(z) \simeq S^G(\mu_N,(\nu^{\cD_{\Cp}}(z),w^{\cD_{\Cp}}(z)))(v_{\mathrm{tot}})
\]
for any $z\in \cD_{\Cp}$. Thus the map $\bT\to \cO(\cD_{\Cp})$ defined by the eigenvalues of $f$ is a family of eigensystems in $S^G(\mu_N,(\nu^\cU,w^\cU))(v_{\mathrm{tot}})$ over $\cD_{\Cp}$ such that its restriction to $\cD^\times_{\Cp}$ is $\varphi^*:\bT\to \cO(\cD^\times_{\Cp})$. In particular, it is of finite slopes over $\cD_{\Cp}^\times$.
If $f(O)$ is of finite slope, then Proposition \ref{Chenevier} yields a morphism $\cD_{\Cp}\to \cE|_{\cU_{\Cp}}$ with the desired property.

Let us prove that $f(O)$ is of finite slope. Put $\psi(O)=(\nu(O),w(O))$ and $\kappa=k(\nu(O),w(O))$, which are $1$-integral by assumption. Let $\kappa_1=(k_\beta)_{\beta\in \bB_F}$ be the integral weight corresponding to the restriction $\kappa|_{\bT^0_1(\bZ_p)}$. Put $\cX_\frc:=\SGm(\mu_N,\frc)^\rig$. We also write as $\cX_\frc(v')=\SGm(\mu_N,\frc)(v')^\rig$ for any $v'<1$. 
For the morphism $h_1:\bar{\SGm}(\Gamma_1(p),\mu_N,\frc)\to \bar{\SGm}(\mu_N,\frc)$, we put $h=h_1^\rig$ and $\cX_\frc^1(v')=h^{-1}(\cX_\frc(v'))$ for any $v'<(p-1)/p$.

Consider the rigid analytic variety $\cY_{\frc,p}$ as in \S\ref{SecConnCritLocus} and the natural projection $\pi:\cY_{\frc,p}\to \cX_\frc$. Put $\cY_{\frc,p}(v')=\pi^{-1}(\cX_\frc(v'))$. For the universal $p$-cyclic subgroup scheme $H^\univ$ over $\cY_{\frc,p}$, we put 
\[
\cY^1_{\frc,p}=\Isom_{\cY_{\frc,p}}(\cD_F\otimes \mu_p, H^\univ).
\]
We denote by $r$ the natural projection $\cY^1_{\frc,p}\to \cY_{\frc,p}$. Put $\pi^1=\pi\circ r$ and $\cY^1_{\frc,p}(v')=(\pi^1)^{-1}(\cX_\frc(v'))$. 
We write the base extensions to $\Cp$ of these maps also as $h$, $\pi$, $r$ and $\pi^1$, respectively.
We consider $\cU^1:=\cX_\frc^1(\tfrac{1}{p+1})$ as a Zariski open subset of $\cY^1_{\frc,p}(\tfrac{1}{p+1})$.
Then we have an isomorphism $h^*\Omega^\kappa\simeq (\pi^1)^*\Omega^\kappa|_{\cU^1_{\Cp}}$. 
Note that the sheaf $h^*\Omega^\kappa$ in this case of $1$-integral weight is isomorphic to the sheaf $h^*\Omega^{\kappa_1}\simeq h^*\omega^{\kappa_1}_{\bar{\cA}^\univ,\Cp}$ as in \S\ref{SecOCHMFDef}. The sheaf $(\pi^1)^*\Omega^{\kappa_1}$ is defined over the whole rigid analytic variety $\cY^1_{\frc,p,\Cp}$ and satisfies $(\pi^1)^*\Omega^{\kappa_1}|_{\cU^1_{\Cp}}\simeq h^*\Omega^\kappa$.
Thus each $f_\frc(O)$ defines the element
\[
g_\frc:=h^*f_\frc(O)=(\pi^1)^*f_\frc(O)|_{\cU^1_{\Cp}}\in H^0(\cU^1_{\Cp}, (\pi^1)^*\Omega^{\kappa_1}(-D)),
\]
on which any element $a$ of the Galois group $\bT(\bZ/p\bZ)$ of $h:\cU^1\to \cX_{\frc}(\tfrac{1}{p+1})$ acts via the multiplication by $\kappa(\hat{a})$ with any lift $\hat{a}\in\bT(\bZ_p)$ of $a$.

Moreover, we define $Z^1_{\frc,p}$ as the scheme over $K$ classifying triples $(A,u,D)$ consisting of a HBAV $A$ over a base scheme over $K$, an $\cO_F$-closed immersion $u: \cD_F^{-1}\otimes \mu_p\to A$ and a finite flat closed $\cO_F$-subgroup scheme $D$ which is etale locally isomorphic to $\underline{\cO_F/p\cO_F}$ satisfying $\Img(u)\cap D=0$. We denote by $\cZ^1_{\frc,p}$ the analytification of $Z^1_{\frc,p}$ restricted to $\cX_{\frc}$. We have two projections $\cZ^1_{\frc,p}\to \cY^1_{\frc,p}$ given by $(A,u,D)\mapsto (A,u)$ and $(A,u,D)\mapsto (A/D,\bar{u})$ with the image $\bar{u}$ of $u$ in $A/D$, which are denoted by $q_1$ and $q_2$, respectively. Put $\cZ^1_{\frc,p}(v')=q_1^{-1}(\cY^1_{\frc,p}(v'))$. 

We denote the restriction of the rigid analytic variety $\cY'_{\frc,p}(\tfrac{1}{p+1})$ defined in \S\ref{SecHeckeOp} to $\cX_\frc(\tfrac{1}{p+1})$ also by $\cY'_{\frc,p}(\tfrac{1}{p+1})$. 
We have a finite etale morphism
\[
\Pi: q_1^{-1}(\cU^1)\to \cY'_{\frc,p}(\tfrac{1}{p+1}),\quad (A,u,H)\mapsto (A,H).
\]
The base extensions of these maps to $\Cp$ are also denoted by $q_1$, $q_2$ and $\Pi$, respectively. 
By Corollary \ref{cncisog} (\ref{noncanisog}), we have $q_1^{-1}(\cU^1)\subseteq q_2^{-1}(\cU^1)$ and thus $q_1^{-1}(\cU^1_{\Cp})\subseteq q_2^{-1}(\cU^1_{\Cp})$. This yields commutative diagrams
\[
\xymatrix{
\cU^1_{\Cp}\ar[d]_{h} & q_1^{-1}(\cU^1_{\Cp}) \ar[l]_{q_2} \ar[d]^{\Pi} \\
\cX_{\frc,p}(\tfrac{1}{p+1})_{\Cp} & \cY'_{\frc,p}(\tfrac{1}{p+1})_{\Cp}\ar[l]^{p_2},
}
\quad
\xymatrix{
q_1^{-1}(\cU^1_{\Cp}) \ar[d]_{\Pi} \ar[r]^{q_1} & \cU^1_{\Cp}\ar[d]^{h} \\
\cY'_{\frc,p}(\tfrac{1}{p+1})_{\Cp} \ar[r]_{p_1} & \cX_{\frc,p}(\tfrac{1}{p+1})_{\Cp},
}
\]
where the latter is cartesian.

Take any point $Q=[(A,\cH)]\in Y_{\frc,p}(\oel)$ with some finite extension $L/K$ such that $\Hdg_\beta(A)=p/(p+1)$ for any $\beta\in \bB_F$, which exists by Lemma \ref{CritNonempty}. Consider the admissible open subsets $\cV_Q$, $\cV_{Q,\Cp}^0$ and $\cV_{Q,\Cp}^0(\tfrac{1}{p+1})$ defined in \S\ref{SecConnCritLocus}. 
By Corollary \ref{UpExt}, we have 
\[
q_1^{-1}(r^{-1}(\cV_Q))\subseteq q_2^{-1}(\cU^1).
\]
Taking the base extension, we also have
\[
q_1^{-1}(r^{-1}(\cV_{Q,\Cp}^0))\subseteq q_1^{-1}(r^{-1}(\cV_{Q,\Cp}))\subseteq q_2^{-1}(\cU^1_{\Cp}).
\]
Similarly, Lemma \ref{VQcan} shows $r^{-1}(\cV_{Q,\Cp}^0(\tfrac{1}{p+1}))\subseteq \cU^1_{\Cp}$.
Since the weight $\kappa_1$ is integral, we have a natural isomorphism $\pi_p^*:q_2^*(\pi^1)^*\Omega^{\kappa_1} \to q_1^*(\pi^1)^*\Omega^{\kappa_1}$ over $\cZ^1_{\frc,p}$. From these and the above commutative diagrams, we see that the operator $U_p$ extends to an operator
\[
U_Q: H^0(\cU^1_{\Cp},(\pi^1)^*\Omega^{\kappa_1})\to H^0(r^{-1}(\cV_{Q,\Cp}^0), (\pi^1)^*\Omega^{\kappa_1})
\]
which makes the following diagram commutative.
\[
\xymatrix{
H^0(\cU^1_{\Cp},(\pi^1)^*\Omega^{\kappa_1})\ar[r]^-{U_Q}&  H^0(r^{-1}(\cV_{Q,\Cp}^0),(\pi^1)^*\Omega^{\kappa_1}) \ar[d]^{\mathrm{res}}\\
 & H^0(r^{-1}(\cV_{Q,\Cp}^0(\tfrac{1}{p+1})),(\pi^1)^*\Omega^\kappa)\\
H^0(\cX_\frc(\tfrac{1}{p+1})_{\Cp},\Omega^\kappa) \ar[r]_{U_p}\ar[uu]^{h^*}  & H^0(\cX_\frc(\tfrac{1}{p+1})_{\Cp},\Omega^\kappa) \ar[u]_{h^*}
}
\]

Now suppose that $f(O)$ is of infinite slope. 
Then 
\[
(U_{Q}g_\frc)|_{r^{-1}(\cV_{Q,\Cp}^0(\tfrac{1}{p+1}))}=(h^*U_pf_\frc(O))|_{r^{-1}(\cV_{Q,\Cp}^0(\tfrac{1}{p+1}))}=0.
\]
Since $\cV_{Q,\Cp}^0$ is connected and $r$ is finitely presented and etale, the map $r$ defines a surjection from each connected component of $r^{-1}(\cV_{Q,\Cp}^0)$ to $\cV_{Q,\Cp}^0$. Since the admissible open subset $\cV_{Q,\Cp}^0(\tfrac{1}{p+1})$ is non-empty, 
we see that $r^{-1}(\cV_{Q,\Cp}^0(\tfrac{1}{p+1}))$ intersects every connected component of $r^{-1}(\cV_{Q,\Cp}^0)$. Thus
Lemma \ref{FvPLem} implies $U_{Q}g_\frc=0$. In particular, if the point $[(A,\cL)]\in Y_{\frc,p}(\cO_{\bar{\bQ}_p})$ satisfies $\Hdg_\beta(A)=p/(p+1)$ for any $\beta\in \bB_F$, then for any $\cO_F$-isomorphism $m: \cD_F^{-1}\otimes \mu_p\simeq \cL_K$, we have
\begin{equation}\label{sumD}
\sum_{\cD_K\cap \cL_K=0} g_\frc(A/\cD, \bar{m})=0,
\end{equation}
where the sum is taken over the set of finite flat closed $p$-cyclic $\cO_F$-subgroup schemes $\cD$ of $A[p]$ satisfying $\cD_K\cap \cL_K=0$.

\begin{lem}\label{combinat}
For any $p$-cyclic $\cO_F$-subgroup scheme $\cH$ of $A[p]$ and any $\cO_F$-isomorphism $u:\cD_F^{-1}\otimes \mu_p\to (A[p]/\cH)_K$, we have $g_\frc(A/\cH, u)=0$.
\end{lem}
\begin{proof}
For any $p$-cyclic $\cO_F$-subgroup scheme $\cM$ of $A[p]$, write as $\cM=\bigoplus_{\frp\mid p} \cM_\frp$. Similarly, any $\cO_F$-closed immersion $m:\cD_F^{-1}\otimes \mu_p\to A_K$ defines a closed immersion $m_\frp: \cD^{-1}_F/\frp\cD^{-1}_F \otimes \mu_p \to A[\frp]_K$ for any $\frp\mid p$. By fixing a generator of the principal $\cO_F$-module $\cD_F^{-1}/p\cD_F^{-1}$ and a primitive $p$-th root of unity in $\bar{\bQ}_p$, we identify an $\cO_F$-closed immersion $m:\cD^{-1}_F\otimes \mu_p \to A_K$ with an element of $A[p](\bar{\bQ}_p)$. Let $\frP$ be the set of maximal ideals of $\cO_F$ dividing $p$. For any subset $S\subseteq \frP$, we put $S^c=\frP\setminus S$ and
\[
\cM_S=\bigoplus_{\frp\in S}\cM_\frp,\quad \cM^S=\bigoplus_{\frp\in S^c} \cM_\frp.
\]
We define $m_S$ and $m^S$ similarly. We write $\Img(m)$ also as $\langle m \rangle$.

For any $\frp\mid p$, we fix non-zero elements $e_{\frp,1}\in \cH_\frp(\bar{\bQ_p})$ and $e_{\frp,2}\in A[\frp](\bar{\bQ}_p)$ such that $\{e_{\frp,1},e_{\frp,2}\}$ forms a basis of the $\fro/\frp$-module $A[\frp](\bar{\bQ}_p)$. 
Put $I_\frp=\{e_{\frp,1}, a_\frp e_{\frp,1}+e_{\frp,2}\mid a_\frp\in \fro/\frp\}$ and $e_{S,i}=(e_{\frp,i})_{\frp\in S}$ for $i=1,2$.
We claim that, for any element $m^S$ of $\prod_{\frp\in S^c}I_\frp$, we have
\begin{equation}\label{Svan}
\sum_{\cD^S_K\cap \langle m^S \rangle=0}g_\frc(A/(\cH_S\times \cD^S),\overline{e_{S,2}\times m^S})=0,
\end{equation}
where the sum is taken over the set of finite flat closed $(\prod_{\frp\in S^c}\frp)$-cyclic $\cO_F$-subgroup schemes $\cD^S$ of $A$ satisfying $\cD^S_K\cap \langle m^S \rangle=0$.

To show the claim, we proceed by induction on $\sharp S$. The case of $S=\emptyset$ is (\ref{sumD}). Suppose that the claim holds for some $S\neq \frP$. Take $\frp\in S^c$ and put $S'=S\cup\{\frp\}$. Fix $m^{S'}\in \prod_{\frq\in (S')^c}I_\frq$. Taking the sum of (\ref{Svan}) over the set $\{m^S=m_\frp \times m^{S'}\mid m_\frp\in I_\frp\}$, we obtain
\[
\sum_{m_\frp\in I_\frp}\sum_{\cD_{\frp,K}\cap \langle m_\frp \rangle=0} \sum_{\cD^{S'}_K\cap \langle m^{S'} \rangle=0}g_\frc(A/(\cH_S\times \cD_\frp\times \cD^{S'}),\overline{e_{S,2}\times m_\frp\times m^{S'}})=0.
\]
We compute terms in this sum for each $\cD_\frp$.
\begin{itemize}
\item If $\cD_{\frp}(\bar{\bQ}_p)=(\fro/\frp)e_{\frp,1}=\cH_\frp(\bar{\bQ}_p)$ and $\cD_{\frp,K}\cap \langle m_\frp \rangle=0$, then $m_{\frp}=a_\frp e_{\frp,1}+e_{\frp,2}$ with some $a_\frp\in \fro/\frp$. In this case, $\bar{m}_\frp$ is equal to the image $\bar{e}_{\frp,2}$ of $e_{\frp,2}$.

\item If $\cD_\frp(\bar{\bQ}_p)=(\fro/\frp)(a_\frp e_{\frp,1}+e_{\frp,2})$ and $\cD_{\frp,K}\cap \langle m_\frp \rangle=0$, then we have either $m_\frp=e_{\frp,1}$ or $m_\frp=b_\frp e_{\frp,1}+e_{\frp,2}$ with some $b_\frp\neq a_\frp\in \fro/\frp$. In each case, $\bar{m}_\frp$ is equal to the element $\bar{e}_{\frp,1}$ or $(b_\frp-a_\frp)\bar{e}_{\frp,1}$. We put
\[
s_\frp=\sum_{a\in (\fro/\frp)^\times} \kappa([a])
\]
with the Teichm\"{u}ller lift $[a]\in \cO_{F_\frp}^\times$ of $a$.
\end{itemize}
Thus the sum of the terms in which $\cD_\frp$'s of the second case appear is equal to
\begin{align*}
(1+s_\frp)\sum_{\cD^{S'}_K\cap \langle m^{S'} \rangle=0}\sum_{a_\frp\in \fro/\frp}g_\frc(A/(\cH_{S}\times (\fro/\frp)(a_\frp e_{\frp,1}+&e_{\frp,2})\times \cD^{S'}),\\
&\overline{e_{S,2}\times e_{\frp,1}\times m^{S'}}).
\end{align*}
This equals
\[
(1+s_\frp)\sum_{\cD^S_K\cap \langle e_{\frp,1}\times m^{S'}\rangle=0}g_\frc(A/(\cH_S\times \cD^S), \overline{e_{S,2}\times e_{\frp,1}\times m^{S'}}),
\]
which is zero by the induction hypothesis (\ref{Svan}). What remains is the sum of the terms of $\cD_\frp$'s of the first case, which equals
\[
p^{f_\frp}\sum_{\cD^{S'}_K\cap \langle m^{S'} \rangle=0}g_\frc(A/(\cH_{S'}\times \cD^{S'}), \overline{e_{S',2}\times m^{S'}})=0
\]
and the claim follows. Setting $S=\frP$, we obtain $g_\frc(A/\cH, \bar{e}_{\frP,2})=0$. For any $u$ as in the lemma, the map $u_\frp$ corresponds to $a_\frp \bar{e}_{\frp,2}$ for some $a_\frp\in (\fro/\frp)^\times$. Thus we have 
\[
g_\frc(A/\cH, u)=\left(\prod_{\frp\mid p}\prod_{a_\frp\in (\fro/\frp)^\times}\kappa([a_\frp])\right) g_\frc(A/\cH, \bar{e}_{\frP,2})=0
\]
and the lemma follows.
\end{proof}

Consider the admissible open subset of $\cY_{\frc,p}$ defined by
\[
\{[(A,\cH)]\mid \Hdg_\beta(A)=p/(p+1)\text{ for any }\beta\in\bB_F\}
\]
and let $\cV$ be a non-empty admissible affinoid open subset of it. 
Note that the map
\[
W: \cY_{\frc,p}\to \cY_{\frc,p},\quad (A,\cH)\mapsto (A/\cH,A[p]/\cH)
\]
is an isomorphism. By Proposition \ref{critisog}, we have $r^{-1}(W(\cV))\subseteq \cU^1$. Consider the base extensions $W_{\Cp}: \cY_{\frc,p,\Cp}\to \cY_{\frc,p,\Cp}$ and $\cV_{\Cp}$, where the latter is an admissible affinoid open subset of $\cY_{\frc,p,\Cp}$. By Lemma \ref{combinat}, $\pi^*f_\frc(O)$ vanishes on the subset $W(\cV)(\bar{\bQ}_p)$ of the admissible affinoid open subset $W_{\Cp}(\cV_{\Cp})=W(\cV)_{\Cp}$. 

\begin{lem}\label{AnalyticVanishing}
Let $A$ be a reduced $K$-affinoid algebra. Put $X=\Spv(A)$, $A_{\Cp}=A\hat{\otimes}_K \Cp$ and $X_{\Cp}=\Spv(A_{\Cp})$. We consider the set $X(\bar{\bQ}_p)$ as a subset of $X_{\Cp}(\Cp)$ by the natural inclusion $\bar{\bQ}_p\to \Cp$. Suppose that an element $f\in A_{\Cp}$ satisfies $f(x)=0$ for any $x\in X(\bar{\bQ}_p)$. Then $f=0$.
\end{lem}
\begin{proof}
For any positive rational number $\varepsilon$, we put
\[
U_{\varepsilon}=\{x\in X_{\Cp}\mid |f(x)|\leq \varepsilon\}.
\]
We can find an element $f_{\varepsilon}\in A\otimes_K \bar{\bQ}_p$ such that 
\[
|(f-f_{\varepsilon})(x)|\leq \varepsilon\text{ for any }x\in X_{\Cp}.
\]
Then we have $U_{\varepsilon}=\{x\in X_{\Cp}\mid |f_{\varepsilon}(x)|\leq \varepsilon\}$. Take a finite extension $L/K$ satisfying $f_\varepsilon\in A_L:=A\otimes_K L$.
Put $X_L=\Spv(A_L)$. The assumption implies $X(\bar{\bQ}_p)\subseteq U_\varepsilon$, namely $|f_\varepsilon(x)|\leq \varepsilon$ for any $x\in X(\bar{\bQ}_p)$. This shows $X_L=\{x\in X_L\mid |f_\varepsilon(x)|\leq \varepsilon\}$.
Since the formation of rational subsets is compatible with base extensions, we have $X_{\Cp}=U_\varepsilon$ for any $\varepsilon>0$, which implies $f(x)=0$ for any $x\in X_{\Cp}$. Since $X_{\Cp}$ is reduced, we obtain $f=0$ and the lemma follows.
\end{proof}

Since the invertible sheaf $\pi^*\Omega^\kappa$ is the base extension to $\Cp$ of a similar invertible sheaf over $K$, it is trivialized by the base extension of an admissible affinoid covering over $K$. By Lemma \ref{AnalyticVanishing}, we have $\pi^*f_\frc(O)|_{W(\cV)_{\Cp}}=0$.  
Thus $f_\frc(O)$ vanishes on the admissible open subset $\pi(W(\cV)_{\Cp})$ of $\bar{\cM}(\mu_N,\frc)(\tfrac{1}{p+1})_{\Cp}$.
By Lemma \ref{Mvconn}, $\bar{\cM}(\mu_N,\frc)(\tfrac{1}{p+1})_{\Cp}$ is connected. By Lemma \ref{FvPLem}, we obtain $f_\frc(O)=0$ for any $\frc$, which contradicts the fact that $f(O)$ is normalized. This concludes the proof of the theorem.
\end{proof}

%---------------------------------------------------------------------

%---------------------------------------------------------------------

\end{document}